%% file: main-deviation.tex
\documentclass[a4paper, a4wide, 10pt, reqno]{amsart}
\usepackage{a4wide}
\input{headers-dev}
\renewcommand{\complement}{\mathsf{c}}
\title[Large deviations of the giant in KSRGs]{Large deviations of the giant in supercritical kernel-based spatial random graphs}

\author[J.\ Jorritsma]{Joost Jorritsma$^1$}
\address{$^1$Department of Statistics, University of Oxford}

\author[J.\ Komj{\'a}thy]{J{\'u}lia Komj{\'a}thy$^2$}
\address{$^2$Delft Institute of Applied Mathematics, Delft University of Technology}
\author[D.\ Mitsche]{Dieter Mitsche$^3$}
\address{$^3$Instituto de Ingeniería Matemática y Computacional,  Pontif\'icia Univ. Cat\'olica de Chile.}
\email{joost.jorritsma@stats.ox.ac.uk, j.komjathy@tudelft.nl, dmitsche@gmail.com}
\begin{document}

\maketitle

\begin{abstract}
We study cluster sizes in supercritical $d$-dimensional inhomogeneous percolation models with long-range edges ---such as long-range percolation--- and/or heavy-tailed degree distributions ---such as geometric inhomogeneous random graphs and the age-dependent random connection model. Our focus is on large deviations of the size of the largest cluster in the graph restricted to a finite box as its volume tends to infinity. 
Compared to nearest neighbor Bernoulli bond percolation on $\Z^d$, we show that long edges can increase the exponent of the polynomial speed of the lower tail from  $(d-1)/d$ to any $\zeta_\star\in\big((d-1)/d,1\big)$.  We prove that this exponent $\zeta_\star$  also governs the size of the second-largest cluster, and the distribution of the size of the cluster  containing the origin $\CC(0)$. 

For the upper tail of large deviations, we prove that its speed is logarithmic for models with power-law degree distributions. We express the rate function 
via the generating function of $|\CC(0)|$. The upper tail in degree-homogeneous models decays much faster:  the speed in long-range percolation is \emph{linear}. 
\end{abstract}
{\footnotesize
\hspace{1em}Keywords: Cluster-size distribution, large deviation principle, second-largest component, spatial random graphs, scale-free networks.

\hspace{1em}MSC Class: 05C80, 60K35.
}

\section{Introduction}
For supercritical nearest neighbor  Bernoulli bond percolation on $\Z^d$ (NNP), it is a classical result that the graph restricted to a volume-$n$ box around the origin contains a \emph{giant}, i.e., a linear-sized component $\CC_n^\sss{(1)}$, with probability tending to one as $n\to\infty$. Its number of vertices satisfies a Law of Large Numbers (LLN):  $|\CC_n^\sss{(1)}|/n$ converges to $\theta:=\Prob_{\mathrm{NNP}}(0\leftrightarrow \infty)$ almost surely. Regarding large deviations of $|\CC_n^\sss{(1)}|$, the speed is known in all dimensions $d\ge 2$, and an interesting discrepancy occurs: the lower tail decays  slower than the upper tail.      

For the upper tail, the  event $\{|\CC_n^\sss{(1)}|>(\theta+\eps)n\}$ requires linearly more edges to be present than expected, which comes at an exponential cost. So, the speed of large deviations is \emph{linear} in the volume of the box: for each $\rho\in(\theta,1)$, there exist constants $C_\rho\ge c_\rho>0$ such that for all $n\ge 1$,
\begin{equation}\label{eq:LDP-NNP-upper}
-C_\rho\le \frac{1}{n}\log \Prob_{\mathrm{NNP}}\big(|\CC_n^\sss{(1)}|>\rho n\big)\le -c_\rho.
\end{equation} 
We refer to~\cite{durrett1988large, gandolfi1989thesis, lebowitz1988pseudo} for the speed of large deviations of a closely related quantity, and to \cref{upper-lrp-nnp} below for the precise result.
For the lower tail,  the  $\Omega(n^{(d-1)/d})$ many edges on the ``outer surface boundary'' of a linear-sized cluster have to be absent. So,  the speed of the lower tail is \emph{sublinear}~\cite{pisztora1996surface}: for each $\rho\in(0,\theta)$, there exist constants $C_\rho'\ge c_\rho'>0$ such that 
\begin{equation}\label{eq:LDP-NNP-lower}
-C_\rho' \le \frac{1}{n^{(d-1)/d}}\log \Prob_{\mathrm{NNP}}\big(|\CC_n^\sss{(1)}|<\rho n\big)\le -c_\rho'.
\end{equation}
We can interpret that the speed of the lower tail is caused by \emph{surface tension} (more precise results are known for closely related quantities, see ~\cite{alexander1990wulff, cerf2000large, durrett1988large, gandolfi1989thesis} and Section~\ref{sec:main-results} below). 

In this paper, we study large deviations of the giant in long-range percolation~\cite{schulman_1983}, and also in models with regularly-varying  edge-length-- and degree distributions. 
Our driving question is:
\begin{equation}\tag{{\bf Q}}\label{metaquestion}
     \parbox{\dimexpr\linewidth-10em}{%
    \strut
    \centering
    {\emph{How do high-degree vertices and long-range edges influence\\ the tails of 
    large deviations of the giant component?\phantom{-range}}}%
    \strut
  }
\end{equation}
We describe three new phenomena that encompass the answer.
First, high-degree vertices slow down the upper tail enormously: the speed becomes logarithmic. Second, long-range edges speed up the lower tail: the speed is still polynomial, but is no longer determined by surface tension alone. Since the lower tail decays much faster than the upper tail, the discrepancy of speeds is the opposite of the one in NNP. Third, we prove that the speed of the lower tail is closely related to the size of the second-largest connected component, and to the cluster-size distribution of the origin in the infinite model.
We explain these phenomena after motivating this topic with an application and an informal model description. 
  \begin{figure}[t]
      \centering
    \includegraphics[width=0.26\textwidth]{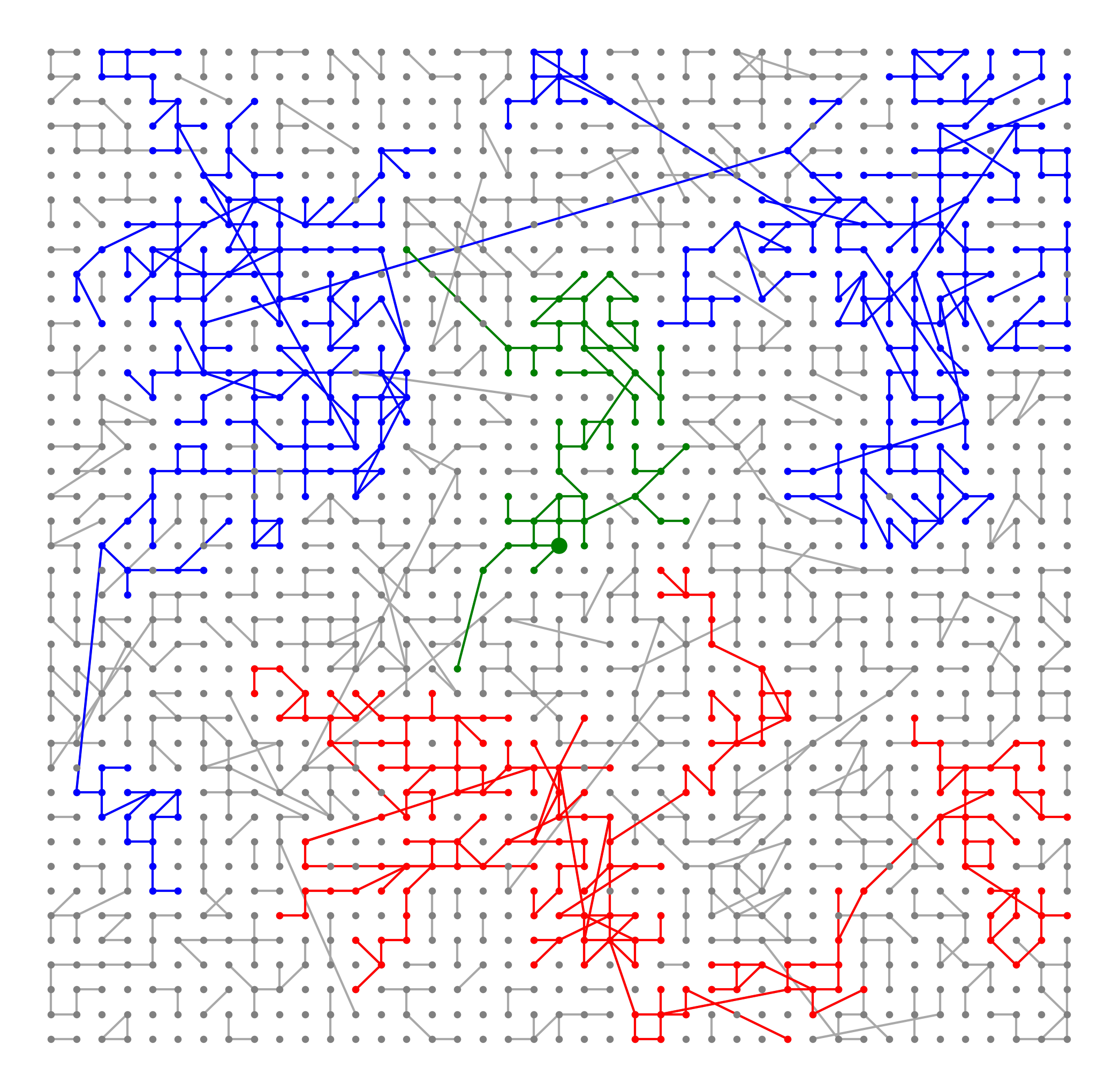}
    \includegraphics[width=0.26\textwidth]{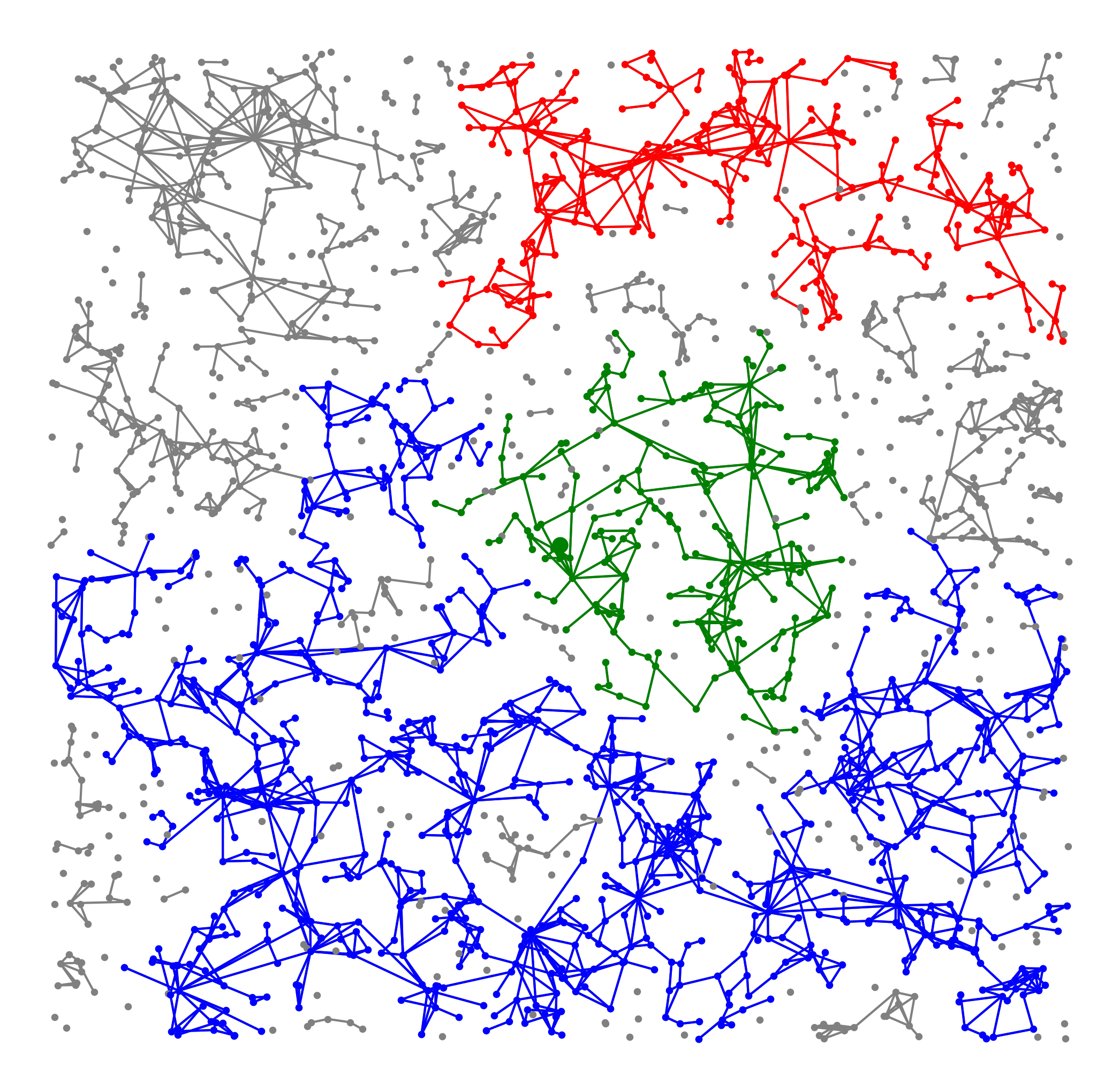}
    \includegraphics[width=0.45\textwidth, trim={0.2cm 0.2cm 0.2cm 0.2cm},clip]{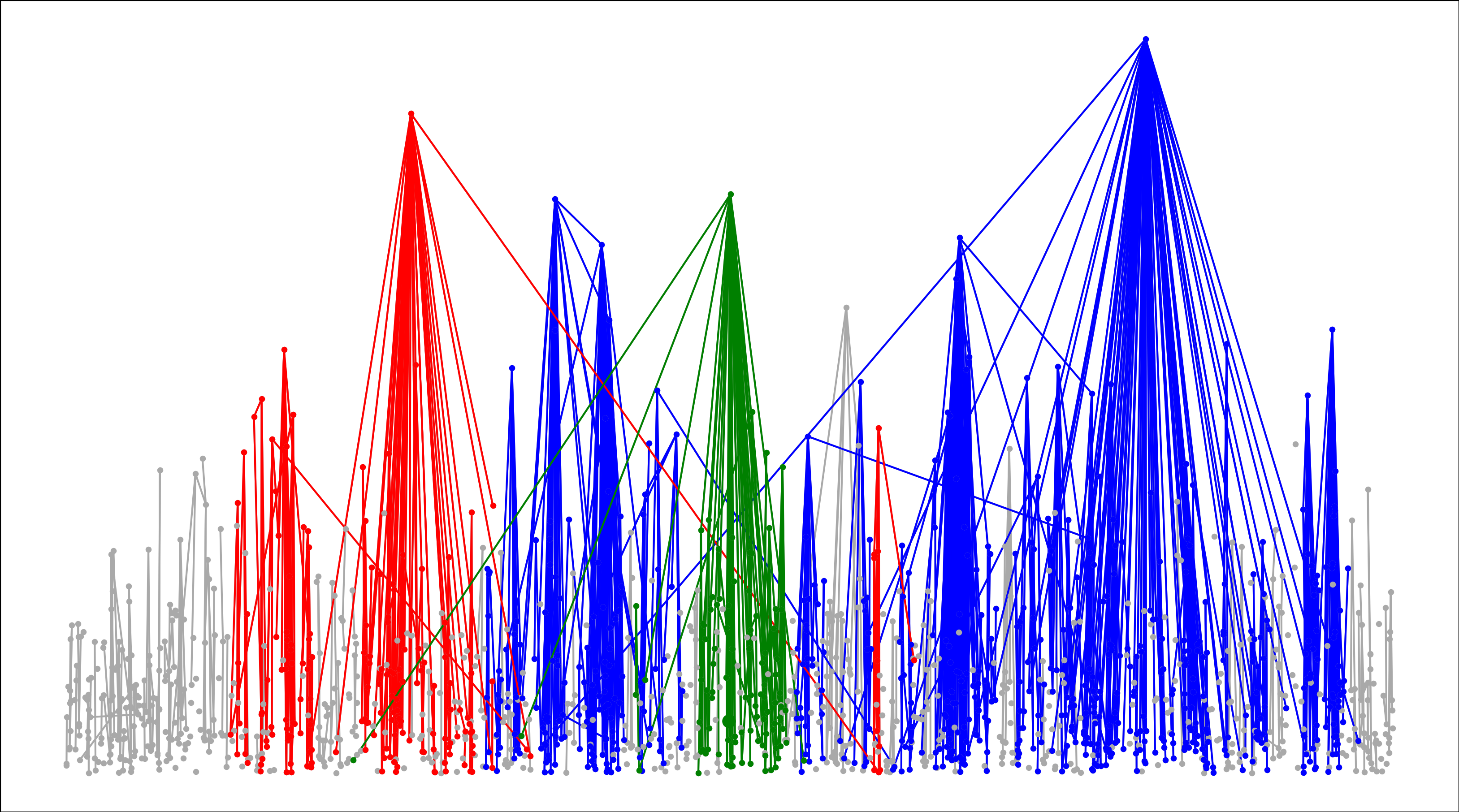}
      \captionsetup{width=0.98\textwidth}
      \caption{Simulations of 2-dimensional long-range percolation (LRP), a 2-dimensional geometric inhomogeneous random graph (GIRG), and a 1-dimensional soft Poisson--Boolean model (sPBM). In the sPBM, the $y$-axis reflects the vertex marks. The presence of long edges in these models leads to delocalized components, and affects the distributions of the sizes of three components: the speed of the lower tail of large deviations of the largest component (blue), the size of the second-largest component (red), and the distribution of the size of the cluster containing the origin (green) are all governed by the exponent $\zeta_\star$, which is defined in~\eqref{eq:zeta-star}. The upper tail of large deviations of the giant's size has linear speed in (mark)-homogeneous models as LRP, and logarithmic speed in inhomogeneous models as GIRG and sPBM.}
      \label{fig:ksrgs}
   \end{figure}

\smallskip\noindent
\emph{Motivation.} 
Many real networks have been found to have underlying geometry, inhomogeneous degrees, clustering, and long-range connections \cite{albert2002statistical, newman2012communities}.
 In biological contexts, the component sizes (under percolation of an initial graph)  correspond to the final size of a Reed--Frost or SIR epidemic on the network~\cite{abbey1952examination, barbour2004approximating, lefevre1990stochastic}, and to activity patterns in brain networks \cite{carvalho2021subsampled, del2018finding, kozma2015random}. 
 In the context of statistical physics, there are many models with long-range interactions, such as models of gravitational and charged systems, turbulent hydrodynamic systems, dipolar systems, and two-dimensional elasticity, see among others~\cite{campa2009statistical, Dauxois, Defenu} for examples.  
  In the last decade many random graph models have been invented to mimic real networks and systems with long-range interactions, see the next paragraph. For many of these models, an analysis of 
 basic properties of their component structure ---such as their approximate size-distribution---  has been lacking in the literature.
 In this paper we fill this gap and develop new, robust proof techniques that allow for the presence of long-range edges and inhomogeneous degrees. 


\smallskip\noindent
\emph{Spatial inhomogeneous percolation models.}\label{par:spatial-models} 
The most well-known spatial degree-inhomogeneous models are hyperbolic random graphs \cite{krioukov2010hyperbolic},
geometric inhomogeneous random graphs \cite{BriKeuLen19}, the (soft) Poisson--Boolean model with random radii \cite{gouere2008subcritical, hall1985continuum},  and the age-dependent random connection model \cite{gracar2019age}. The degree-inhomogeneity leads to arbitrarily long edges. While the degree distribution of long-range percolation (LRP) is light-tailed, LRP also contains long edges. The presence of long edges leads to components that might be `delocalized in space', see Figure~\ref{fig:ksrgs}.
Our formal setup, introduced informally now, captures these models all at once.

 The vertex sets are formed by a $d$-dimensional ergodic point process: either $\Z^d$ or a homogeneous Poisson point process (PPP) on $\R^d$. Each vertex $u\in\CV$ at location $x_u$ is equipped with an independent and identically distributed (iid) mark  $w_u\ge 1$, which is Pareto-distributed for degree-inhomogeneous models, and a constant for long-range percolation. The expected degree of a vertex is proportional to its mark (sometimes also called `weight'). Conditionally on all vertex locations and marks $\{(x_u, w_u)\}_{u\in\CV}$, each edge $\{u,v\}$ is present independently with probability $\varphi(\beta, w_u, w_v, \|x_u-x_v\|)$. The (model-dependent) \emph{edge-connectivity function} $\varphi$ is non-increasing in the distance, non-decreasing and symmetric in both marks. The parameter $\beta>0$ controls the edge-density. We often assume that the model is supercritical, i.e., $\beta>\beta_c$, where $\beta_c$ is the infimum over $\beta \ge 0$ such that the graph contains an infinite connected component almost surely. 

\smallskip\noindent
\emph{Law of Large Numbers for the size of the giant.} In this general model class, even the existence of a linear-sized component in finite boxes is not guaranteed/known in the whole supercritical regime.
When $\varphi$ allows for sufficiently many long edges,
we prove 
 existence and uniqueness of the giant for all $\beta>\beta_c$. Its size satisfies a Law of Large Numbers: $|\CC_n^\sss{(1)}|/n \to\theta:=\Prob(0\leftrightarrow \infty)$ in probability (for some cases the convergence holds almost surely). The long edges that are necessary for this LLN, come either from a slow decay of the connectivity function $\varphi$ as $\|x_u-x_v\|\to \infty$, or from the presence of high vertex marks, or from a combination of both. 

\medskip\noindent
\emph{Upper tail of large deviations.}
We now state informally our result on the upper tail.
\begin{metatheorem}[High-degree vertices cause slow upper tails]\label{thm:intro-upper}
Consider a spatial inhomogeneous percolation model with regularly-varying degree distribution, so that $|\CC_n^{\sss{(1)}}|$, the size of the giant, satisfies a Law of Large Numbers. The upper tail of $|\CC_n^{\sss{(1)}}|$ has logarithmic speed. We identify the non-negative rate-function $I(\cdot)$ such that for all $\rho\in(\theta,1)$,
\begin{equation}\label{eq:intro-upper}
\begin{aligned}
    -\inf_{r>\rho}I(r)&\le \liminf_{n\to\infty}\frac{1}{\log n}\log\Prob\big(|\CC_n^\sss{(1)}|>\rho n\big) \\
  &  \le 
    \limsup_{n\to\infty}\frac{1}{\log n}\log\Prob\big(|\CC_n^\sss{(1)}|>\rho n\big) 
    =-\inf_{r\ge \rho}I(r)=I(\rho).
   \end{aligned}
\end{equation}
On the contrary, the speed  in Bernoulli bond percolation or long-range percolation on $\Z^d$ is linear.
\end{metatheorem}
We formalize \cref{thm:intro-upper} in \cref{upper-lrp-nnp} and \cref{thm:large-dev-upper} below. 
\cref{thm:intro-upper} is the first result that identifies the rate function for the size of the giant in a spatial random graph model. Moreover, the rate function for the giant is not even known for many classical \emph{non-spatial} random graph models such as  the rank-one inhomogeneous random graph with iid Pareto weights \cite{ChungLu02.1, NorRei06}, and the configuration model with iid Pareto degrees \cite{MollReed98}. Some known results for non-spatial graphs are the following: The LDP of the giant in the configuration model with \emph{fixed} degree sequence has linear speed when linear-degree vertices are absent~\cite{bhamidiConfLargeDev}; in sparse inhomogeneous random graphs with bounded weights the speed is linear as well~\cite{andreis2023large}, similar to  sparse Erd{\H o}s-R\'enyi random graphs \cite{andreis2021er, oconnel1998}.
A work in progress by the first author and Zwart shows an LDP for the giant in inhomogeneous random graphs with iid heavy-tailed vertex weights~\cite{jorritsmaZwart2023}.

We briefly explain the reason  for the logarithmic speed:  with \emph{polynomially} decaying probability, the vertex set contains a few hubs: vertices with  mark $\Theta(n)$. These hubs connect to the giant and to linearly many small components. Small components do not merge with the giant with a probability that is exponentially small in their size and in the number of hubs. This allows us to express the rate-function $I(\cdot)$ via the generating function of the  cluster-size distribution of the infinite model.

 A similar phenomenon occurs for the sum of $n$ iid (fully asymmetric) Pareto random variables: this sum also satisfies an LDP with  logarithmic speed of the upper tail.  Having a single summand whose value is linear in $n$ comes at a polynomial cost. On the contrary, the lower tail for such variables has linear speed because linearly many summands need to be small \cite{nagaev1979large}. 

 \smallskip\noindent
 \emph{Lower tail of large deviations.} 
 Similar to the sum of iid Pareto random variables, the lower tail of the size of the giant decays much faster than the upper tail. However, due to the underlying geometry, the speed is not linear, but polynomial.
Contrary to NNP in \eqref{eq:LDP-NNP-lower}, in models with long-range edges the speed is not determined by surface tension alone. Instead, the speed is given by the solution of a variational problem that describes the most likely way that a box of volume $n$, denoted by $\Lambda_n$, is isolated from its complement $\Lambda_n^\complement$~\cite{clusterI}. For most edge-connectivity functions $\varphi$, the probability that $\Lambda_n$ is isolated decays as $\exp(-\Theta(n^{\zeta_\star}))$ for some exponent $\zeta_\star\in[0,1)$\footnote{For some connectivity functions $\varphi$, slowly-varying correction factors appear inside the ``$\Theta(\cdot)$'' in the exp-function.}. The exponent $\zeta_\star$ can also be determined via the expected size of the \emph{downward vertex boundary} that we define now. We say that the edge $\{u,v\}=\{(x_u, w_u), (x_v, w_v)\}$ is a `downward edge' from $u$ if $w_u\ge w_v$. We define$^{\ref{note-limit}}$
\begin{equation}\label{eq:zeta-star}
\zeta_\star:=\lim_{n\to\infty}\frac{\log \E\Big[ \big|\big\{u\in \Lambda_n: u\mbox{ has a downward edge to }\Lambda_n^\complement\big\}\big|\Big]}{\log n},
\end{equation}
The restriction of counting vertices with downward edges to $\Lambda_n^\complement$ is necessary for some (but not all) edge-connectivity functions $\varphi$. This excludes counting (potentially many) vertices on the vertex boundary that are connected by an edge to a small number of high-mark vertices spatially close to $\Lambda_n$. 
 In NNP all marks are identical and  $\zeta_\star=(d-1)/d$:
 the expectation in $\zeta_\star$ is governed by surface tension. As we will see, $\zeta_\star$ provides the right generalization of surface tension to spatial inhomogeneous percolation models with long edges.
 For these latter models, $\zeta_\star=\max(\zeta_\mathrm{long},(d-1)/d)$, where\footnote{\label{note-limit}The limits in the definitions of $\zeta_\star$ and $\zeta_\mathrm{long}$ are well-defined in the models that we study. We will never use $\zeta_\mathrm{long}$ when it equals $0$. The truncation ensures that the limit always exists.}
\begin{equation}\label{eq:zeta-long}
\zeta_\mathrm{long}:=\lim_{n\to\infty}\Bigg(0\vee\frac{\log \E\Big[ \big|\big\{u\in \Lambda_{n/2}: u\mbox{ has a downward edge to }\Lambda_n^\complement\big\}\big|\Big]}{\log n}\Bigg)
 \end{equation}
 describes the number of vertices incident to long downwards edges of length $\Omega(n^{1/d})$. The value of $\zeta_\mathrm{long}$ can be explicitly computed based on the exact model specifications. Importantly, $\zeta_\mathrm{long}$ can be larger than $(d-1)/d$ for any fixed $d\in \N$. In that case, $\zeta_\star=\zeta_{\mathrm{long}}$ and long edges speed up the decay of the lower tail of large deviations. 
\begin{metatheorem}[Long edges speed up the lower tail]\label{phen:lower}
Consider a supercritical spatial inhomogeneous percolation model so that $|\CC_n^{\sss{(1)}}|$, the size of the giant, satisfies a Law of Large Numbers, i.e., $\zeta_\mathrm{long}>0$. Assume that the expected downward vertex boundary in \eqref{eq:zeta-star} has size $\Theta(n^{\zeta_\star})$. For each $\rho\in(0,\theta)$, there exist constants $C_\rho\ge c_\rho>0$ such that for all $n\ge 1$,
      \begin{equation}
-C_\rho \le\frac{1}{n^{\zeta_\star}} \log \Prob\big(|\CC_n^\sss{(1)}|<\rho n\big)\le -c_\rho.\label{eq:lower-intro}
\end{equation}
When the expected downward vertex boundary has slowly-varying correction factors, (potentially different) slowly-varying correction factors appear in  the lower bound.
  \end{metatheorem}
   We formalize \cref{phen:lower} in~\cref{thm:large-dev2}. \cref{phen:lower} implies the existence of the giant with stretched exponential error probability, which is novel for several spatial inhomogeneous percolation models mentioned above on page  \pageref{par:spatial-models}. Moreover, this result includes long-range percolation and improves upon the currently known best bound there  as well, as we explain now. 
   
In long-range percolation, two vertices $x,y\in\Z^d$ are connected by an edge independently with probability $\Theta\big(1\wedge\|x-y\|^{-\alpha d}\big)$ as $\|x-y\|\to\infty$, for a long-range parameter $\alpha>1$. An elementary computation shows that $\zeta_\mathrm{long}=2-\alpha$, and $\zeta_\star=\max(2-\alpha, (d-1)/d)$. In~\cite{biskup2004scaling}, Biskup showed that there exists $\delta>0$ such that $\Prob\big(|\CC_n^\sss{(1)}|<\delta n\big)\le \exp\big(-n^{2-\alpha-o(1)}\big)$. \cref{phen:lower} improves upon~\cite{biskup2004scaling} in two ways when $2-\alpha=\zeta_\mathrm{long}>(d-1)/d$. First, the power of $n$ in the exponent in \eqref{eq:lower-intro} has no $o(1)$-term. Second, the decay in \eqref{eq:lower-intro} holds for arbitrary $\rho \in(0, \theta)$, instead of a small constant $\delta$.

 Returning to the general model setting, we give some intuition why $\zeta_{\mathrm{long}}$ drives the lower tail when $\zeta_{\mathrm{long}}>(d-1)/d$. 
 The definition of $\zeta_\mathrm{long}$ in \eqref{eq:zeta-long} tells us that  $\Lambda_n$ contains about  $\Theta(n^{\zeta_\mathrm{long}})$ many  vertices that are incident to long edges of length $\Omega(n^{1/d})$. 
We partition $\Lambda_n$ into subboxes of size $\Theta(n^{1-\zeta_\mathrm{long}})$, so that a subbox typically contains constantly many vertices incident to these long edges. Then, we couple the graph with an ERRG with large constant expected degree. In this coupling, each vertex corresponds to a local giant inside one of the $\Theta(n^{\zeta_\mathrm{long}})$ many subboxes, and we use the long edges to estimate the probability that two local giants are connected. Then we apply the LDP on the giant in ERRGs by O'Connell~\cite{oconnel1998}. This LDP has linear speed in the number of `Erd{\H o}s-R\'enyi vertices', and yields that $\Prob\big(|\CC_n^\sss{(1)}|<\rho  n\big)\le \exp\big(-\Theta(n^{\zeta_\mathrm{long}})\big)$. We apply this coupling \emph{after} proving the LLN, which allows us to choose any $\rho<\theta$. 

The Law of Large Numbers is a crucial ingredient to make this sketch work. The LLN ensures that local giants in subboxes exist and have the right size. We first establish the existence of the giant via two multi-scale renormalization schemes. To obtain the right size for the giant, it suffices to prove an upper bound on the cluster-size decay $\Prob( k\le  |\CC_n(0)|, 0 \notin \CC_n^{\sss{(1)}})$ in finite boxes \cite{hofstad2021giantlocal, clusterI}.

\smallskip\noindent
\emph{Cluster-size decay and second-largest component.} In our third main result, we prove that $\zeta_\star$ also governs the cluster-size decay of the origin and the size of the second-largest connected component $|\CC_n^\sss{(2)}|$. In \cite{clusterI}, we already showed that both sizes are determined by $\zeta_\star$ for some parameter settings of spatial inhomogeneous percolation models. 
 In this paper, in Theorems~\ref{thm:subexponential-decay} and~\ref{thm:second-largest}, we prove the same relation for complementary parameter settings. Combining the present paper and~\cite{clusterI}, we obtain the following result. Here, $\CC(0)$ is the cluster of the origin in the model on $\R^d$.
 \begin{metatheorem}[Cluster-size decay and the second-largest component]\label{phen:csd-c2}
     Consider a supercritical spatial inhomogeneous percolation model with $\zeta_\mathrm{long}>(d-1)/d$. Assume that the expected downward vertex boundary in \eqref{eq:zeta-star} has size $\Theta(n^{\zeta_\star})$. There exist constants $C\ge c>0$ such that for all $k\ge 1$,     \begin{equation}\label{eq:CLD-C2-intro}
\begin{aligned}
-C \le \frac{1}{k^{\zeta_\star}}\log \Prob\big(k< |\CC(0)|<\infty\big) \le - c, \qquad
\Prob\bigg(c \le \frac{|\CC_n^{\sss{(2)}}|}{(\log n)^{1/\zeta_\star}}  \le  C\bigg)\overset{n\to\infty}\longrightarrow 1. 
\end{aligned}
\end{equation}
When the expected downward vertex boundary has slowly-varying correction factors, (potentially different) slowly-varying correction factors appear in  the lower bound.
 \end{metatheorem} 
For long-range percolation, Crawford and Sly proved a polylogarithmic upper bound on $\CC_n^{\sss{(2)}}$ in \cite{crawford2012simple} with some unspecified exponent. \cref{phen:csd-c2} finds the correct exponent of this polylog.
For supercritical NNP and other homogeneous percolation models,  surface-tension-driven behavior for the speed of the lower tail, for the cluster-size decay, and for the second-largest component is known, see  Pisztora~\cite{pisztora1996surface}. 
Moreover, when $\zeta_\mathrm{long}<(d-1)/d$ (equivalently $\zeta_\star=(d-1)/d$),  we proved in~\cite{clusterIII} that~\cref{phen:csd-c2} holds for long-range percolation with high edge-density. That is, for (high edge-density) LRP with $\alpha> 1+1/d$, the cluster-size decay and the size of the second-largest component are both driven by surface tension.

  \smallskip\noindent
  \emph{Another exponent determining the presence of a supercritical phase.}
  Combined with~\cite{clusterI}, this paper completes the following picture: $\zeta_\star=\max(\zeta_{\mathrm{long}}, (d-1)/d)$ in~\eqref{eq:zeta-star} is the correct generalization of the surface-tension exponent $(d-1)/d$ for supercritical spatial inhomogeneous percolation models. It  governs the size of the second-largest component, the connected component containing a vertex at the origin, and the lower tail of large deviations of the giant component.

The value of $\zeta_\mathrm{long}$, defined in \eqref{eq:zeta-long}, is related to the \emph{effective decay exponent} $\delta_\mathrm{eff}$ defined in~\cite{jacob2023only}.  
     In words, $2-\delta_\mathrm{eff}$ measures the size of the ``long-edge'' \emph{edge}-boundary of a large box. To define it, say that $u\in \Lambda_n\times [a,b]$ if its location $x_u$ is in $\Lambda_n$, and its mark $w_u$ is in the interval $[a,b]$, and let $w_{n}^\sss{\max}$ denote the maximal \emph{expected} mark of a vertex in $\Lambda_n$. 
    \begin{equation}\label{eq:deltaeff}
    2-\delta_\mathrm{eff}=\lim_{n\to\infty}\frac{\log \E\Big[\big|\big\{\text{edges between } \Lambda_{n/2}\times[0,w_{n}^\sss{\max}]\text{ and } \Lambda_n^\complement\times[0,w_{n}^\sss{\max}]\big\}\big|\Big]}{\log n}.
    \end{equation}
  When $2-\delta_\mathrm{eff}>0$ and $d=1$,~\cite{jacob2023only} proves that the models studied here admit a supercritical phase, i.e., $\beta_c<\infty$. If $2-\delta_\mathrm{eff}<0$ and $d=1$, there is no supercritical phase \cite{jacob2023only}.
     For dimensions $d \ge 2$, $\beta_c<\infty$ holds in spatial inhomogeneous percolation models, since such models at high edge-density dominate supercritical random geometric graphs or Bernoulli bond percolation on $\Z^d$. Further, when $2-\delta_\mathrm{eff}>0$, in this model class annulus-crossing probabilities do not decay with the size of the annulus. This is denoted by $\hat \lambda_c=0$ in the recent work in \cite{ jacob2023only}. 
     
     By its definition in \eqref{eq:zeta-long}, $\zeta_\mathrm{long}$ describes the size of the ``long-edge'' downward \emph{vertex}-boundary. Hence, $2-\delta_\mathrm{eff}$ has the same sign as $\zeta_\mathrm{long}$ throughout the model class. Combining this observation with the results in \cite{jacob2023only}, positivity/negativity of $\zeta_\star=\max(\zeta_\mathrm{long}, (d-1)/d)$ also gives finiteness of $\beta_c$ and $\hat \lambda_c=0$ in this model class.

         The difference between $\zeta_\mathrm{long}$ and $2-\delta_\mathrm{eff}$ becomes apparent as soon as degrees are inhomogeneous. High-degree vertices can have many edges leaving the box. Edges leaving the maximal vertex mark in $\Lambda_n$ can form the dominating contribution to the \emph{edge}-boundary. In this case, the formula in \eqref{eq:deltaeff} is sensitive to truncation of marks and to the actual presence of the high-mark vertex. This is expressed in the non-concentration of the size of the edge-boundary around $\Theta(n^{2-\delta_\mathrm{eff}})$. On the contrary, the size of the downward vertex boundary is robust to random fluctuations of the maximal degree, and its size concentrates around $\Theta(n^{\zeta_\mathrm{long}})$. 
Summarizing, we observe that 
regarding presence/absence of supercritical phase, $\delta_{\mathrm{eff}}$ may be replaced by $\zeta_\star$ (or $\zeta_{\mathrm{long}}$), and that $\zeta_\star$ also drives the behavior of the three important distributions mentioned above in Meta-theorems~\ref{phen:lower} and \ref{phen:csd-c2}. 
         In the following section we formalize Meta-theorems~\ref{thm:intro-upper}--\ref{phen:csd-c2} that answer the question~\eqref{metaquestion}.

\section{Formal statements of model and main results}\label{sec:main-results}
\subsection{Kernel-based spatial random graphs}
We give a formal definition of a large class of spatial inhomogeneous percolation models. This class appeared under various names in the literature. The first general version of the model appeared in \cite{KomLod20}, called \emph{general geometric inhomogeneous random graphs}, and independently in \cite{GraHeyMonMor19} where it was called \emph{weight-dependent random connection model}. Afterwards, it was called \emph{spatial inhomogeneous random graphs} in~\cite{maitra2021locallim}. Due to the important role of a \emph{kernel function} driving the connection probabilities, we called them \emph{kernel-based spatial random graphs} in~\cite{clusterI}. Since this paper ties in with \cite{clusterI}, we keep this last name here. This model class includes classical homogeneous models such as Bernoulli bond percolation on $\Z^d$ \cite{broadbent1957percolation}, long-range percolation \cite{schulman_1983}, and random geometric graphs~\cite{Gilb61, Hafner72}.  It also includes (degree-) inhomogeneous models such as hyperbolic random graphs \cite{krioukov2010hyperbolic}, the Poisson-Boolean model with random radii \cite{hall1985continuum}, geometric inhomogeneous random graphs \cite{BriKeuLen19}, scale-free percolation \cite{DeiHofHoo13}, the age-dependent random connection model \cite{gracar2019age}, and the scale-free Gilbert model \cite{hirsch2017gilbertgraph}. We give now a formal definition.

\begin{definition}[Kernel-based spatial random graphs (KSRG)]\label{def:ksrg} Fix a dimension $d\ge 1$. A KSRG is then a random graph defined as follows. Let the vertex set $V$ be formed by an ergodic point process: either $\Z^d$ or a unit-intensity Poisson point process (PPP) on $\R^d$. Given $V$, we equip each vertex $u\in V$ with an independent mark $w_u\ge 1$ following distribution $F_W$. 
Let $\kappa: \R_+^2 \to \R_+$ be a symmetric function, called the kernel function. Let $\varphi: \R_+\to [0,1]$ be non-decreasing, called the profile function, and let $\beta>0$. 
Conditionally on the realization of the marked vertex set $\CV := \{(x_u, w_u)\}_{u\in V}\subset \R^d\times \R_+$, 
the edge $\{u,v\}\in V^2$ is present in the edge-set $\CE$ independently of other edges with probability
\begin{equation}\label{eq:connection-verygeneral}
\mathrm{p}(u,v):=\Prob_{W, \kappa, \varphi,  \beta}\big(u\text{ connected by an edge to }v \mid \CV\,\big) = \varphi\Big(\beta\cdot\frac{\kappa(w_u, w_v)}{\|x_u-x_v\|^d}\Big).
\end{equation}
We denote the obtained infinite graph by $\CG=(\CV,\CE)$, and  by $\CG_n=(\CV_n,\CE_n)$ the graph induced by vertices in $\Lambda_n:=[-n^{1/d}/2, n^{1/d}/2]^d$. 
We write $\CC_n^{\sss{(i)}}$ for the $i$th largest component of $\CG_n$, and $\CC_n(0)$ and $\CC(0)$ for the component containing a vertex at the origin in $\CG_n$ and in $\CG$, respectively. We write $\Prob^\sss{x}$ for the Palm-measure when the vertex set of a unit-intensity Poisson point process is conditioned to contain a vertex at location $x\in\R^d$ with unknown mark.\end{definition}

We omit the subscripts $W, \kappa, \varphi$ of $\Prob$ as they are fixed and always clear from the context.
Sometimes we emphasize the dependency of $\Prob$ on the edge-density parameter $\beta$ by writing $\Prob_\beta$. We  define
\begin{equation}\label{eq:theta}
\theta=\theta(\beta):=\Prob_\beta^{\sss{0}}(0 \leftrightarrow \infty)=\Prob_{\beta}^\sss{0}(|\CC(0)|=\infty)>0.
\end{equation}
Then, the critical edge-density $\beta_c$ is given by
\begin{equation}\label{eq:betac}
\beta_c:=\inf\big\{\beta\ge0 : \theta(\beta)>0\big\}.
\end{equation}
We call the model subcritical if $\beta<\beta_c$, critical if $\beta=\beta_c$, and supercritical if $\beta>\beta_c$. For a fixed choice of $W, \varphi, \kappa$, and $d$, some phases may not exist~\cite{gracar2022finiteness, gracar2021percolation}.

\cref{def:ksrg} allows for general kernel and profile functions. In the rest of the paper we restrict them to a few choices that are commonly used, and which cover the specific models in the introduction \cite{BriKeuLen19, DeiHofHoo13, gracar2019age, hall1985continuum, hirsch2017gilbertgraph, krioukov2010hyperbolic, schulman_1983}. 
Our results remain valid for (or can be extended to) other sufficiently similar kernels, profile functions, and mark distributions. There is no additional technicality in the proofs if one includes slowly varying functions in any of these functions, and the only change in the results is that lower order correction terms (for instance, slowly varying functions multiplying the polynomial exponent of cluster-size decay) appear in the results below. These lower-order corrections can be traced throughout the proofs. For ease of presentation, we thus work with functions satisfying the following assumption. We follow the notation of our previous work~\cite{clusterI}, and write for any $a, b, \in \R$, $a\wedge b$ for $\min(a,b)$, and $a\vee b$ for $\max(a,b)$.
 \begin{assumption}\label{assumption:main}
We assume that the mark distribution is either $W_v\equiv 1$ for all $v$, or a Pareto distribution with parameter $\tau>2$: 
\begin{equation}\label{eq:power-law}
1-F_W(w):=\Prob\big(W_v\ge w\big)=w^{-(\tau-1)}, \qquad w\ge 1.
\end{equation}
We assume that the profile function $\varphi$ is either threshold or polynomial: for constants $p\in(0,1]$, $\alpha>1$, let
\begin{equation}\label{eq:profile}
\varphi_{\mathrm{thres}}(s):=p\ind{s\ge1},\qquad\text{or }\qquad \varphi_{\mathrm{pol}}(s):=p\big(1\wedge s^{-\alpha}\big),
\end{equation}
and we assume that the kernel $\kappa$ is one of the following:
\begin{equation}\label{eq:kernels}
\kappa_\mathrm{sum}(w_1, w_2)=\big(w_1^{1/d}+w_2^{1/d}\big)^d,\qquad\text{or}\qquad
\kappa_\sigma(w_1, w_2)=(w_1\vee w_2)(w_1\wedge w_2)^\sigma,
\end{equation}
for a parameter $\sigma\ge 0$. By abuse of notation, when $W_v\equiv 1$ for all $v\in V$ we say 
that $\tau=\infty$; when $\varphi=\varphi_\mathrm{thres}$ we say that $\alpha=\infty$;  when $\kappa=\kappa_{\mathrm{sum}}$ we say that $\sigma=0$.
\end{assumption}
Assumption \ref{assumption:main} ensures that the model is parametrized so that the expected degree of a vertex is proportional to its mark when $\sigma\le \tau-1$. The restrictions $\tau>2$ and $\alpha>1$ ensure that the graph is locally finite. Increasing $\tau$ and/or $\alpha$ leads to less inhomogeneity (lighter-tailed degrees and fewer long edges, respectively). All our results below for $\tau=\infty$ except Theorem \ref{upper-lrp-nnp}   can be directly extended to all distributions $F_W$ with lighter tail than any power-law in \eqref{eq:power-law}.
The parameter $p$ in \eqref{eq:profile} makes the model closed under Bernoulli percolation of the edges, while the parameter $\sigma$ controls assortativity: increasing $\sigma$ makes it more likely that high-degree vertices are connected by an edge. The parameter $\sigma$ allows us to continuously interpolate between well-known models that are special cases. Therefore, we call $\kappa_\sigma$ the \emph{interpolation kernel}, see more in \cite{clusterI}.

Our earlier work~\cite{clusterI} as well as the papers~\cite{GraHeyMonMor19,maitra2021locallim} describe how the above models can be reconstructed as special cases. In short, in long-range percolation, the vertex set is $\Z^d$, and the marks are a constant, i.e., $\tau=\infty$, and $\alpha<\infty$. The kernel $\kappa_1$, corresponding to the product $w_1w_2$, is used to obtain hyperbolic random graphs, geometric inhomogeneous random graphs,  and scale-free percolation~\cite{BriKeuLen19, DeiHofHoo13,krioukov2010hyperbolic}, $\kappa_{\tau-2}$ to obtain the age-dependent random connection model \cite{gracar2019age}, and $\kappa_\mathrm{sum}$, with additionally $p=1$, to obtain the (soft) Poisson--Boolean model \cite{hall1985continuum}.  The soft Poisson--Boolean model uses the polynomial profile $\varphi_{\mathrm{pol}}$, while the classical Poisson--Boolean model uses the threshold profile $\varphi_{\mathrm{thres}}$. 
The kernel $\kappa_0$ is sometimes called the max-kernel, and is closely related to $\kappa_\mathrm{sum}$, as $\kappa_0\le\kappa_\mathrm{sum}\le 2\kappa_0$. Since the latter two kernels are qualitatively similar, we simply set $\sigma=0$ for $\kappa_\mathrm{sum}$; we never prove\footnote{The corresponding results are straightforward, in lower/upper bounds we use the corresponding bound from  $\kappa_0\le\kappa_\mathrm{sum}\le 2\kappa_0$, and constant factors do not matter in our calculations.} explicitly results for $\kappa_\mathrm{sum}$.

\subsection{Dominant connection types}\label{sec:dominant}
To make Meta-theorems \ref{thm:intro-upper}--\ref{phen:csd-c2} formal, we give the values of $\zeta_\star$ in~\eqref{eq:zeta-star} and $\zeta_\mathrm{long}$ in~\eqref{eq:zeta-long} in terms of the model parameters for KSRGs satisfying Assumption~\ref{assumption:main}. We follow the notation of~\cite{clusterI}, and write $\zeta_{\star}$ and $\zeta_{\mathrm{long}}$ as a maximum of  simple expressions of the parameters $d,\tau, \alpha, \sigma$.  As we will show below, they take the form
\begin{equation}\label{eq:zeta-first}
\zeta_\star = \max( \zeta_{\mathrm{short}}, \zeta_{\mathrm{ll}}, \zeta_{\mathrm{hl}}, \zeta_{\mathrm{hh}}), \qquad \zeta_{\mathrm{long}}= \max( \zeta_{\mathrm{ll}}, \zeta_{\mathrm{hl}}, \zeta_{\mathrm{hh}}).
\end{equation}
Each term inside the maxima describes the expected number of vertices with a certain type of downward edge leaving $\Lambda_n$. 
Each of these types contributes to the downward vertex boundary of $\CV\cap \Lambda_n$ in \eqref{eq:zeta-star}. If the maximum is unique, the size of the downward vertex boundary is dominated by exactly one of these four types. If the maximum is non-unique,
we say that the parameters are on a phase-transition boundary. 
We provide some back-of-the-envelope calculations to compute $\zeta_\star$ in \eqref{eq:zeta-star} and $\zeta_\mathrm{long}$ in \eqref{eq:zeta-long} when $\alpha<\infty$. When $\alpha=\infty$, one has to take the limit in the expressions below as $\alpha\to\infty$. The precise computations are given in~\cite{clusterI}. 

\emph{Short edges} -- of constant length: there are roughly $\Theta(n^{(d-1)/d})$ vertices incident to such edges crossing the boundary of $\Lambda_n$ in \eqref{eq:zeta-star}, giving the `surface-tension' exponent 
\begin{equation}\label{eq:zeta-short-ll}
    \zeta_\mathrm{short}:=(d-1)/d.
\end{equation}
Next, we start counting `long edges', that is, edges of length $\Theta(n^{1/d})$ crossing the boundary of $\Lambda_n$, and thus also contributing to $\zeta_{\mathrm{long}}$ in \eqref{eq:zeta-long}. These edges fall into three categories.

\emph{Low-low edges} connect two vertices of constant (low) mark between $\Lambda_{n/2}$ and $\Lambda_n^\complement$. The expected number of low-low edges is $\Theta(n \cdot n \cdot n^{-\alpha})$, and this is the same as the number of low-mark vertices having a low-low edge. 
Abbreviating `\emph{low}-mark to \emph{low}-mark' by ll, we obtain
\begin{equation}\label{eq:zeta-ll}
\zeta_{\mathrm{ll}}:=2-\alpha.
\end{equation}   
The other two types describe long edges incident to `high-mark' vertices in $\Lambda_{n/2}$.  We say a vertex has a high mark if its mark is at least $n^{\gamma_\mathrm{long}}$, where 
\begin{equation}\label{eq:gamma-long}
\gamma_\mathrm{long}:=\min\Big\{\gamma\ge0: \liminf_{n\to\infty}\E^\sss{0}\big[ |\{ \mbox{edges between }0\mbox{ and }\Lambda_n^\complement\}| \,\big|\, (0, n^{\gamma})\in\CV\big]>0\Big\}.
\end{equation}
Then, a constant proportion of vertices of mark at least $n^{\gamma_\mathrm{long}}$ inside $\Lambda_{n/2}$ contributes to the vertex boundary. By the Pareto distribution of the marks in~\eqref{eq:power-law}, there are $\Theta(n^{1-\gamma_\mathrm{long}(\tau-1)})$ many high-mark vertices inside $\Lambda_{n/2}$. 
The values $\tau$ in~\eqref{eq:power-law}, $\sigma$ in \eqref{eq:kernels}, and $\alpha$ in \eqref{eq:profile} jointly determine the value of $\gamma_{\mathrm{long}}$. We compute two possible values of $\gamma_{\mathrm{long}}$, that then yield the two remaining types.

\emph{High-low edges} are dominant if the other end-vertex of an edge emanating from vertex  $(0,n^{\gamma})$ to $\Lambda_n^\complement$ typically has constant mark. There are $\Theta(n)$ constant-mark vertices at distance $\Theta(n^{1/d})$. By the connection probability in~\eqref{eq:connection-verygeneral} with $\kappa_\sigma$ or $\kappa_\mathrm{sum}$ from~\eqref{eq:kernels} and $\varphi_{\mathrm{pol}}$ from~\eqref{eq:profile}, for $\gamma > 0$, the expected number of edges between vertex $(0, n^\gamma)$ and constant-mark vertices in $\Lambda_n^\complement$ is roughly $n (n^\gamma / n)^\alpha$. As required in \eqref{eq:gamma-long}, this expression is of constant order when
\begin{equation}\label{eq:gamma-lh}
\gamma=\gamma_\mathrm{hl}:=1-1/\alpha, 
\end{equation}
which leads to 
\begin{equation}\label{eq:zeta-lh}
    \zeta_\mathrm{hl}:=1-\gamma_\mathrm{hl}(\tau-1)=(\tau-1)/\alpha - (\tau-2).
\end{equation}

\emph{High-high edges} occur dominantly if the other end-vertex of an edge emanating from vertex $(0,n^\gamma)$ in \eqref{eq:gamma-long} typically has a high mark.   
There are $\Theta(n^{1-\gamma(\tau-1)})$ many vertices of mark $\Omega(n^\gamma)$ at distance $\Theta(n^{1/d})$ from $0$. Using the interpolation kernel in \eqref{eq:kernels}, the expected number of edges between $(0, n^\gamma)$ and these other vertices is roughly $n^{1-\gamma (\tau-1)} (n^{\gamma(\sigma+1)}/ n)^\alpha$. This expression is of constant order when 
\begin{equation}\label{eq:gamma-hh}
    \gamma=\gamma_\mathrm{hh}:=
    \begin{dcases}
    \frac{1-1/\alpha}{\sigma+1-(\tau-1)/\alpha},&\text{if }\tau\le \sigma+2,\\
    \frac{1}{\sigma+1},&\text{if }\tau>\sigma+2,
    \end{dcases}
\end{equation}
which leads to 
\begin{equation}\label{eq:zeta-hh}
    \zeta_\mathrm{hh}:=1-\gamma_\mathrm{hh}(\tau-1)=
    \begin{dcases}
        \frac{\sigma+2-\tau}{\sigma+1-(\tau-1)/\alpha},&\text{if }\tau\le \sigma+2,\\ 
        \frac{\sigma+2-\tau}{\sigma+1},&\text{if }\tau>\sigma+2.
    \end{dcases}
\end{equation}
We will never use the second row for $\gamma_\mathrm{hh}$ and $\zeta_\mathrm{hh}$: $\zeta_\mathrm{hh}$ is negative there, and some other connectivity type is dominant. The choice of $\gamma_\mathrm{hh}$ in this second row ensures continuity and monotonicity in the parameters. High-low connections can be dominant when $\sigma$ is small,  for example in the age-dependent random connection model and the soft Poisson--Boolean model.
The high-high type of connection can be dominant for large values of $\sigma$, e.g. in geometric inhomogeneous random graphs and random hyperbolic graphs. The infimum in \eqref{eq:gamma-long} indicates that $\gamma_{\mathrm{long}}$ is the minimum of $\gamma_{\mathrm{hl}}$ and $\gamma_{\mathrm{hh}}$, and the switch is when the denominator $\sigma+1-(\tau-1)/\alpha$ crosses $1$. 
Indeed, the following claim makes the above back-of-the-envelope computations formal: the integrals in the definitions of $\zeta_\star$ in \eqref{eq:zeta-star} and $\zeta_\mathrm{long}$ in \eqref{eq:zeta-long} are dominated  by the above four cases. Vertices incident to edges of intermediate length $\ell_n$ ($1\ll\ell_n\ll n^{1/d}$) crossing $\Lambda_n$ do not significantly contribute to the vertex boundary.
\begin{claim}[Phases of $\zeta_\star$, {\cite[Lemma 7.6]{clusterI}}]\label{claim:phases} Consider any KSRG with kernel and profile satisfying Assumption~\ref{assumption:main}. Then
\begin{align}
    \zeta_\mathrm{long}=\max(0, \zeta_\mathrm{ll}, \zeta_\mathrm{hl}, \zeta_\mathrm{hh}), \qquad 
    \zeta_\star=\max(\zeta_\mathrm{short}, \zeta_\mathrm{long}).
    \end{align} 
\end{claim}

The values of $\zeta_\star$ and $\zeta_\mathrm{long}$ do not depend on $p, \beta$, and $\zeta_\mathrm{long}$ not even on $d$.    
We refer to~\cite{clusterI} for phase diagrams that illustrate for which values of the parameters $(\tau, \alpha, \sigma, d)$ the four different connection types are dominant. 
We call the areas in the parameter space corresponding to unique dominant types the \emph{phases} of $\zeta_\star$. On the boundaries of these phases,
multiple connection types are dominant simultaneously.  
 These cause poly-logarithmic factors in the expectation in the numerator of $\zeta_\star$ in~\eqref{eq:zeta-star}. These factors vanish in the limit, but they do appear in some of our lower bounds.
So, for a given parameter setting of $\tau, \alpha, \sigma, d$, we let $\mathfrak{m}_\CZ$ count the number of dominant connection types. Formally,  we define the \emph{multiplicity}  of the maximum of $\CZ:=\{\zeta_{\mathrm{ll}},\zeta_{\mathrm{hl}}, \zeta_{\mathrm{hh}}, \zeta_{\mathrm{short}}\}$ as
\begin{equation}
 \mathfrak{m}_\CZ := \sum_{\zeta\in \CZ}\Ind{\zeta=\max(\CZ)}. \label{eq:multiplicity}
\end{equation}
We now formalize Meta-theorems \ref{thm:intro-upper},~\ref{phen:lower}, and~\ref{phen:csd-c2} to answer the question~\eqref{metaquestion}. 
\subsection{Upper tail of large deviations}\label{subsec:results}
First we state the upper tail of the giant's size for LRP on $\Z^d$ and Bernoulli bond percolation on $\Z^d$ (NNP), making~\eqref{eq:LDP-NNP-upper} precise.
\begin{theorem}[Linear speed in degree-homogeneous models]\label{upper-lrp-nnp}
Consider long-range percolation or Bernoulli bond percolation on $\Z^d$ with parameters such that $\theta<1$. For all $\rho\in(\theta, 1)$, there exists $A>0$ such that for all $n\ge 1$, 
\[
-A\le \frac{1}{n}\log \Prob\big(|\CC_n^\sss{(1)}|>\rho n\big) \le -1/A.
\]    
\end{theorem}
\cref{upper-lrp-nnp} applies for \emph{any} edge density $\beta$, i.e., regardless of whether the model is supercritical, subcritical or critical. 
Linear speed for the upper tail of large deviations of a closely related quantity is known for NNP~\cite{durrett1988large, gandolfi1989thesis}: the number of vertices in the intersection of a finite box and the infinite component.

We proceed to the upper tail for models with power-law degree distributions, when the speed is logarithmic, heading towards making Meta-theorem \ref{thm:intro-upper} precise. 
At an increased cost of technicality our subsequent results can be extended to KSRGs on $\Z^d$; in order to keep the proof simpler, we refrain from this and restrict to Poisson vertex sets for degree-inhomogeneous KSRGs, i.e., with $\tau<\infty$. 
To state the rate function, 
we introduce some notation.
We write $\CC(0)$ for the cluster of the origin in the model on $\R^d$, and recall from Definition \ref{def:ksrg} that we use the superscript $x$ when the vertex set has a vertex at location $x$ with unknown mark. 
Let $$H_{\CC}(z):=\E^\sss{0}[z^{|\CC(0)|}\ind{|\CC(0)|<\infty}]$$ be the probability generating function of the size of $\CC(0)$ restricted to be finite.  The function $H_{\CC}(z)$ is continuous and increasing, and has range $(0,1-\theta]$ for $z\in(0,1]$. 
Hence, the inverse $H_{\CC}^{(-1)}(y)$ is well-defined and increasing when $y\in(0,1-\theta]$. For $\rho\in(\theta, 1)$, define
 \begin{equation}\label{eq:hubs-gen}
 \mathrm{hubs}(\rho):=\begin{dcases}
     \frac{\log H_{\CC}^{(-1)}\big(1-\rho\big)}{\log (1-p)},&\text{if }p<1,\\
     1,&\text{if }p=1.
 \end{dcases}
 \end{equation}
 The distinction based on $p$ ensures that $\log(1-p)$ is well-defined.  
 We prove that $\lceil\mathrm{hubs}(\rho)\rceil$ many linear-mark vertices are required to increase the density of the giant from the typical density $\theta$ to $\rho$, see Section~\ref{sec:outline-upper} below for intuition. 
 The function $\mathrm{hubs}(\rho)$ is positive, increasing in $\rho$ and tends to infinity as $\rho\uparrow1$: similar to~\cite{jorritsmaZwart2023}, it can be shown that $\mathrm{hubs}(\rho)=\log(1-\rho)/\log(1-p) + o(1)$ as $\rho$ approaches $1$. 
The following theorem gives the rate function for the upper tail. Recall $\zeta_{\mathrm{long}}$ from \eqref{eq:zeta-long}.
 \begin{theorem}[Logarithmic speed in degree-inhomogeneous models]\label{thm:large-dev-upper}
 Consider a KSRG satisfying Assumption \ref{assumption:main} with a unit-intensity Poisson point process as vertex set and power-law mark distribution with $\tau<\infty$.
Fix $\rho\in(\theta, 1)$. There exists $A>0$ such that for all $n\ge 1$,
 \begin{equation}\label{eq:thm-utld}
\Prob\big(|\CC_n^\sss{(1)}|> \rho n\big)\le An^{-(\tau-2)\lceil \mathrm{hubs}(\rho)\rceil}.
 \end{equation} 
 If either the model is supercritical with $\zeta_\mathrm{long}>0$, or $\theta=0$, then there exists $A>0$ such that for all $n\ge 1$, 
 \begin{equation}\label{eq:thm-utld2}
 \Prob\big(|\CC_n^\sss{(1)}|> \rho n\big) \ge (1/A)n^{-(\tau-2)\lim_{r\downarrow\rho}\lceil\mathrm{hubs}(r)\rceil}.
 \end{equation} 
\end{theorem}
The formulation of the either/or above \eqref{eq:thm-utld2} excludes critical models from the lower bound with $\theta(\beta_c)>0$, where $\beta_c$ is the critical edge-density in \eqref{eq:betac}.  
For KSRGs in general it is unknown when $\theta(\beta_c)$ is positive: this may sensitively depend on the short edges of the model \cite{longrangeUnique1987, aizenman1986discontinuity}. The preprint \cite{monchContinuityPercolation}  by M\"{o}nch shows that $\theta(\beta_c)=0$ for KSRGs in dimensions $d\ge 2$ with $\zeta_\mathrm{long}>0$.

Combining \eqref{eq:thm-utld} and \eqref{eq:thm-utld2} gives that when $\mathrm{hubs}(\rho)\notin\N$, we identify the decay rate of  $\Prob\big(|\CC_n^\sss{(1)}|> \rho n\big)$ up to a constant factor. 
If $\mathrm{hubs}(\rho)\in\N$ and $p<1$, the lower and upper bounds are no longer of the same order, because the necessary ``wiggle room'' in the  proof of the lower bound is lost. Nevertheless, Theorem \ref{thm:large-dev-upper} gives logarithmic speed for the upper tail with 
rate function 
\[
I(\rho):= (\tau-2)\lceil\mathrm{hubs}(\rho)\rceil.
\]
\begin{corollary}[Rate function for the upper tail]\label{cor:ldp}
Consider a KSRG satisfying Assumption \ref{assumption:main} on a unit-intensity Poisson point process as vertex set and power-law mark distribution with $\tau<\infty$. If the model is either supercritical with $\zeta_{\mathrm{long}}>0$ or $\theta=0$, then for all $\rho\in(\theta, 1)$,
\begin{equation}\label{eq:cor-ldp}
\begin{aligned}
    -\inf_{r>\rho}I(r) &\le 
\liminf_{n\to\infty}\frac{1}{\log n}\log\Prob\big(|\CC_n^\sss{(1)}|>\rho n\big)\\
&\le \limsup_{n\to\infty}\frac{1}{\log n}\log\Prob\big(|\CC_n^\sss{(1)}|>\rho n\big)\le -\inf_{r\ge\rho}I(r).
\end{aligned}
\end{equation}
\end{corollary}
\begin{proof} 
Since $\mathrm{hubs}(\rho)$ in \eqref{eq:hubs-gen} is continuous and non-decreasing, both bounds follow directly from \cref{thm:large-dev-upper}.  
\end{proof}
When we compare this result to~\cref{upper-lrp-nnp} we see that the speed drops from linear to logarithmic when degree-inhomogeneity enters the models.
The assumption $\zeta_\mathrm{long}>0$ in \eqref{eq:thm-utld2} reflects the aim to study the influence of degree and edge-length inhomogeneity in~\eqref{metaquestion}. However, $\zeta_\mathrm{long}>0$ is a technical condition here. It comes from our bounds for the lower tail, which serves as a tool for proving \eqref{eq:thm-utld2}.
\subsection{Lower tail of large deviations}\label{subsec:lower-tail}
Our next result determines the speed of the lower tail of large deviations of the giant component for supercritical KSRGs, i.e., when $\beta>\beta_c$ in \eqref{eq:betac}. This will make \cref{phen:lower} precise. Contrary to Bernoulli bond percolation, the lower tail decays faster than the upper tail when $\tau<\infty$: the speed is polynomial, with exponent $\zeta_\star=\max(\zeta_\mathrm{long},\zeta_\mathrm{short})$ where $\zeta_{\mathrm{short}}=(d-1)/d$ and $\zeta_\mathrm{long}$ is defined in \eqref{eq:zeta-long}. 
Recall the definition of $\zeta_\star$ from~\eqref{eq:zeta-star}, its value from \cref{claim:phases}, and the multiplicity $\mathfrak{m}_\CZ$ from~\eqref{eq:multiplicity}.
 \begin{theorem}[Speed in the lower tail of large deviations for the giant]\label{thm:large-dev2}
 Consider a supercritical KSRG satisfying Assumption \ref{assumption:main} with a unit-intensity Poisson point process as vertex set, or consider supercritical long-range percolation on $\Z^d$. Assume that the model satisfies $\zeta_\mathrm{long}>0$.
Fix $\rho\in(0,\theta)$. There exists a constant $A>0$ such that for all $n\ge 1$,
    \begin{equation}     \log \Prob\big(|\CC_n^\sss{(1)}| < \rho n\big)\in\big(-An^{\zeta_\star}(\log n)^{\mathfrak{m}_\CZ-1}, -\tfrac{1}{A}n^{\zeta_\star}\big).\label{eq:thm-ltld}
 \end{equation} 
 \end{theorem}    
The lower bound in \cref{thm:large-dev2} is proved in our earlier work~\cite[Theorem 2.5]{clusterI}. When a single connection type is dominant (i.e., $\mathfrak{m}_\CZ=1$, see Section~\ref{sec:dominant}), the lower and upper bound on the right-hand side differ by a constant factor, matching \cref{phen:lower}. For such parameters,~\cref{thm:large-dev2} determines the exact \emph{speed} of large deviations for the lower tail. When $\mathfrak{m}_\CZ\in\{2,3,4\}$, the bounds differ by a poly-log factor. 

We discuss a few related results.
For Bernoulli nearest neighbor  percolation on $\Z^d$ (NNP), the exponent of the polynomial speed of the lower tail is $\zeta_\star=\zeta_\mathrm{short}=(d-1)/d$, see~\eqref{eq:LDP-NNP-lower} and~\cite{pisztora1996surface}. We expect that \cref{thm:large-dev2} remains true\footnote{In dimension $d=1$ there is no lower tail since the model is subcritical when $\zeta_\mathrm{long}<0$, see below \eqref{eq:deltaeff}.} also for KSRGs with $\zeta_\mathrm{long}\le0$ in $d\ge 2$. In view of \eqref{metaquestion}, this is beyond the scope of our paper. 
For Bernoulli NNP on $\Z^2$ or $\Z^3$, the works~\cite{alexander1990wulff} and~\cite{cerf2000large} identify the \emph{rate function} of the lower tail for the size of the intersection of the infinite component with $\Lambda_n$. Keep in mind that this is not necessarily the size of the largest component in $\Lambda_n$. 

For long-range percolation, Biskup proved in~\cite{biskup2004scaling} that $\Prob\big(|\CC_n^\sss{(1)}|<\delta n\big)\le \exp(-\Theta(n^{2-\alpha-o(1)}))$ for $\delta>0$ sufficiently small, which is an almost sharp upper bound, see our comment below \cref{phen:lower}.
Hyperbolic random graphs \cite{FouMul18, krioukov2010hyperbolic} are equivalent to one-dimensional threshold geometric inhomogeneous random graphs and are a special instance of the KSRGs~\cite{KomLod20} by setting $p=1, \sigma=1, \alpha=\infty$ in Assumption \ref{assumption:main}. Independent of our work, the recent preprint~\cite{girgGiantBlasius} proves the upper bound $\Prob\big(|\CC_n^\sss{(1)}|<\delta n\big)\le \exp(-\Theta(n^{(3-\tau)/2}))$ for $\delta>0$ sufficiently small. This is a special case of the upper bound here, and $(3-\tau)/2$ matches the exponent $\zeta_\mathrm{hh}$ in our earlier general lower bound in \cite{clusterI} for $\alpha=\infty$. In comparison, \cref{thm:large-dev2} allows (beyond other kernels) for $p<1$ and long-range profiles. It also shows that $\delta$ can be chosen arbitrarily close to $\theta$.

The previous three theorems give the \emph{strong} Law of Large Numbers as a corollary for KSRGs with finite-variance degree distributions, which is not known for (specific) models in the KSRG class\footnote{While the result seems folklore to us for Bernoulli NNP, we could not find a reference for the exact quantity we look at here. We are also not aware of strong LLN for other models in the class.}
. For KSRGs with an infinite-variance mark distribution, we obtain a weak LLN.
\begin{corollary}[Law of Large Numbers]\label{prop:lln}
    Consider either supercritical long-range percolation on $\Z^d$ with $\zeta_\mathrm{long}>0$, or Bernoulli nearest neighbor  bond percolation on $\Z^d$, or a supercritical KSRG satisfying Assumption \ref{assumption:main} on a unit-intensity PPP with $\tau>3$ and $\zeta_\mathrm{long}>0$. Then
    \begin{equation}\label{eq:lln-strong-main}
    \frac{|\CC_n^\sss{(1)}|}{n} \overset{\text{a.s.}}\longrightarrow \theta=\theta(\beta, p, \alpha, d),\qquad \mbox{as }n\to\infty. 
    \end{equation}
    For a supercritical KSRG satisfying Assumption \ref{assumption:main} on a unit-intensity PPP with $\zeta_\mathrm{long}>0$ and $\tau\in(2,3]$, 
    \begin{equation}\label{eq:lln-weak-main}
    \frac{|\CC_n^\sss{(1)}|}{n} \overset{\Prob}\longrightarrow \theta=\theta(\beta, p, \alpha, \tau, \sigma, d),\qquad \mbox{as }n\to\infty. 
    \end{equation}
\end{corollary}
\begin{proof}
   For NNP, a lower bound follows from the lower tail of large deviations by Pisztora in~\cite{pisztora1996surface}, an upper bound follows from \cref{upper-lrp-nnp}. The weak law of large numbers follows directly from \cref{upper-lrp-nnp}, \cref{thm:large-dev-upper} and \cref{thm:large-dev2}. The error terms are summable for LRP, NNP, and KSRGs with $\tau>3$ (in \eqref{eq:thm-utld}): an application of the  Borel--Cantelli Lemma gives the strong law, while the weak law still holds also for KSRGs with $\tau\in (2,3]$.
\end{proof}
\noindent\emph{Related work on LLN.} Fountoulakis and M{\" u}ller in  \cite{FouMul18}  proved a weak LLN for the giant in hyperbolic random graphs, which is a threshold one-dimensional KSRG with infinite variance degrees satisfying $\zeta_\mathrm{hh}>0$, and is contained here in \eqref{eq:lln-weak-main}. In \cite[Corollary 2.3]{clusterI}, we extended the weak LLN of \cite{FouMul18} to general KSRGs on PPPs with $\zeta_\mathrm{hh} > 0$.
Finally, for general long-range percolation models with $\zeta_\mathrm{long}\le 0$, B{\"a}umler in~\cite[Theorem 1.6]{Baumler} recently showed that $\Prob\big(|\CC_n^\sss{(1)}|\ge \rho n\big)\to 1$ for any $\rho\in(0,\theta)$. Below, we first prove a weak LLN when $\zeta_\mathrm{long}=\max(\zeta_\mathrm{hh}, \zeta_\mathrm{hl}, \zeta_\mathrm{ll})>0$. This weak LLN serves as a tool to proving Theorems \ref{thm:large-dev-upper} and \ref{thm:large-dev2}. When simultaneously $\zeta_\mathrm{hh}\le 0$ and $\max(\zeta_\mathrm{hl}, \zeta_\mathrm{ll})>0$, the LLN requires new methods compared to the methods in \cite{FouMul18, clusterI}.
 These theorems give the strong law in some cases.  

\subsection{Cluster-size decay and second-largest component}\label{subsec:cluster-size-decay-results}
We proceed to formalizing \cref{phen:csd-c2}. We start with the cluster-size decay of the origin. Let $\CC(0)$ denote the cluster of the origin in the infinite model. Recall $\zeta_{\mathrm{long}}$ from \eqref{eq:zeta-long}, $\zeta_{\mathrm{short}}=(d-1)/d$ and $\mathfrak{m}_\CZ$ from \eqref{eq:multiplicity}.
\begin{theorem}[Cluster-size decay]\label{thm:subexponential-decay}
Consider a supercritical KSRG satisfying Assumption \ref{assumption:main} with a unit-intensity Poisson point process as vertex set, or consider supercritical long-range percolation on $\Z^d$. Assume that the model satisfies $\zeta_\mathrm{long} > \max\{0, \zeta_\mathrm{hh}\}$. 
Then there exists a constant $A>0$ such that for all $k\ge 1$ and $n\in[Ak,\infty]$, 
 \begin{equation}\label{eq:cluster-size-decay-main}
   \log\Prob^\sss{0}\big(|\CC_n(0)|> k, \; 0\notin\CC_n^\sss{(1)} \big)\in\big(-Ak^{\max(\zeta_\mathrm{long}, \zeta_\mathrm{short})}(\log k)^{\mathfrak{m}_\CZ-1}, -\tfrac{1}{A} k^{\zeta_\mathrm{long}}\big).
 \end{equation}
The result remains valid when  
$\zeta_\mathrm{long}=\zeta_\mathrm{hh}$ and $\sigma\le \tau-1$ \cite[Theorem 2.2(i),(ii)]{clusterI}. When $\zeta_\mathrm{long}=\zeta_\mathrm{hh}$ and $\sigma> \tau-1$, the lower bound on the right-hand side remains valid while the upper bound changes to $-(1/A) k^{1/(\sigma+1-(\tau-1)/\alpha)}$ \cite[Theorem 2.2(i),(iii)]{clusterI}.
\end{theorem}
We proved the bounds for the case  $\zeta_{\mathrm{long}}=\zeta_{\mathrm{hh}}$  in our earlier work~\cite{clusterI}: then high-high connections are dominant. Theorem \ref{thm:subexponential-decay} extends the result for the phases when $\zeta_\star$ is either $\zeta_{\mathrm{ll}}$ in \eqref{eq:zeta-ll} and/or $\zeta_{\mathrm{hl}}$ in \eqref{eq:zeta-lh}. This corresponds to parameter settings where long-range connections incident to at least one vertex of small mark dominate the expectation in \eqref{eq:zeta-star}. The techniques for these latter phases are different from those in \cite{clusterI}, see Remark \ref{remark:hh-different} on the details. 

The lower and upper bound in \eqref{eq:cluster-size-decay-main} are of the same order if the maximum in $\CZ$ is unique \emph{and} $\zeta_\mathrm{long}>(d-1)/d$ (excluding the sub-optimal bound for $\zeta_\mathrm{long}=\zeta_{\mathrm{hh}}$ and $\sigma>\tau-1$). If the maximum is non-unique, i.e., $\mathfrak{m}_\CZ\in\{2,3,4\}$ and $\zeta_\mathrm{long}\ge (d-1)/d$, then the two bounds differ by a polylog factor. The exponent of $k$ in the two bounds only differs when $(d-1)/d$ is \emph{strictly larger} than $\zeta_\mathrm{long}$.  We conjecture that the lower bound is sharp, and that the correct exponent must be $\zeta_\mathrm{short}=(d-1)/d$ generally in the KSRG class in this regime. The additional restriction that $\sigma\le\tau-1$ when $\zeta_\mathrm{hh}=\zeta_\mathrm{long}$ comes as a technical assumption from~\cite{clusterI}. 

\smallskip\noindent
\emph{Related work on cluster-size decay}. There are many percolation models where the cluster-size decay is determined by surface tension.
For Bernoulli nearest neighbor  bond percolation on $\Z^d$, a series of papers \cite{aizenman1980lower,grimmett1990supercritical, kesten1990bondpercolation,kunz1978essential} showed that the exponent of the decay is $\zeta_{\mathrm{short}}$, i.e., for some $C_p\ge c_p>0$,
\begin{equation}\label{eq:csd-surface}
\begin{aligned}
-C_p\le \liminf_{k\to\infty}&\frac{1}{k^{(d-1)/d}}\log \Prob\big(k<|\CC(0)|<\infty\big)\\&\le 
\limsup_{k\to\infty}\frac{1}{k^{(d-1)/d}}\log \Prob\big(k<|\CC(0)|<\infty\big)\le -c_p. 
\end{aligned}
\end{equation}
For dimension $d\in\{2,3\}$, the constants $C_p=c_p$ above are known, and determined by the Wulff construction~\cite{alexander1990wulff,cerf2000large}. For Bernoulli percolation on more general, transitive infinite graphs with polynomial ball-growth, \cite{contreras2021supercritical, hutchcroft2022transience} obtain similar bounds.
In~\cite{clusterIII}, we derived \eqref{eq:csd-surface} for high edge-density long-range percolation when $\zeta_\star=(d-1)/d>2-\alpha$, i.e., when short-range edges are dominant. Our techniques in~\cite{clusterIII} differ from those here and in \cite{clusterI}, see more on the difficulties when $(d-1)/d=\zeta_\star$ in Remark \ref{remark:nn-difficulty} below. 

\smallskip\noindent
\emph{The second-largest component.} Our last result identifies the size  $|\CC_n^\sss{(2)}|$ of the second largest component in KSRGs in boxes of volume $n$. 
\begin{theorem}[Second-largest component]\label{thm:second-largest}
 Consider the same setting as in Theorem \ref{thm:subexponential-decay}. There exist constants $A,\delta>0$ such that  for all $n\ge1$, with probability at least $1-n^{-\delta}$,
     \begin{equation}
     \bigg(\frac{\log n}{A(\log\log n)^{\mathfrak{m}_\CZ-1}}\bigg)^{1/\max(\zeta_\mathrm{long}, \zeta_\mathrm{short})}
     \le |\CC_n^\sss{(2)}| \le (A\log n)^{1/\zeta_\mathrm{long}}.\nonumber
 \end{equation}
 The result remains valid when  
$\zeta_\mathrm{long}=\zeta_\mathrm{hh}$ and $\sigma\le \tau-1$ \cite[Theorem 2.4(i),(ii)]{clusterI}. When $\zeta_\mathrm{long}=\zeta_\mathrm{hh}$ and $\sigma> \tau-1$, the lower bound on the right-hand side remains valid while the upper bound changes to $(A \log n)^{\sigma+1-(\tau-1)/\alpha}$ \cite[Theorem 2.4(i),(iii)]{clusterI}.
\end{theorem}
\cref{thm:second-largest} accompanies the results from \cite{clusterI, clusterIII} exactly as \cref{thm:subexponential-decay}.
The lower and upper bounds are matching exactly when the bounds on the cluster-size decay are matching, and differ by a ``poly-loglog'' factor on phase transition boundaries (excluding the sub-optimal bound for $\zeta_\mathrm{long}=\zeta_{\mathrm{hh}}$ and $\sigma>\tau-1$). Since the size of $\CC_n^\sss{(2)}$ is polylogarithmic in $n$, this theorem implies uniqueness of the giant component.
Moreover, it makes a result by Crawford and Sly~\cite{crawford2012simple} sharp: that paper proves a polylogarithmic upper bound on $|\CC_n^\sss{(2)}|$ for long-range percolation with unidentified exponent. 

For readers interested in classical long-range percolation models (LRP), we state all our results valid for LRP  in a single corollary, gathered from \cite{clusterI, clusterIII} and this paper. Here, we assume that the model is on $\Z^d$ for some $d\ge 1$ and two vertices of $\Z^d$ are connected by an edge with probability that is asymptotically of order $1/\|x-y\|^{\alpha d}$, independently of all other edges.
We assume $\alpha\neq 1+1/d$ so that lower-order corrections in the lower bounds are not present (for the case $\alpha=1+1/d$ see the individual theorems above, setting $m_\mathcal Z=2$). When we talk about high edge-density below, we assume that the model follows the connectivity function
\begin{equation}\label{eq:high-edge-density}
   p(x\sim y) = p \min\big(1, \beta/\|x-y\|\big)^{d\alpha}, 
\end{equation}
where we assume for $\beta\le 1$ that  $p\beta^{d\alpha}$ is sufficiently close to $1$, and we assume for $\beta>1$ that $\beta$ is sufficiently large (given $p$). In the proof of this corollary we give the exact place of the proof of each result.
\begin{corollary}[Summary of results for LRP] Consider supercritical long-range percolation on $\Z^d$ with connection probability $p(x\sim y)$ decaying up to constant factors as $\|x-y\|^{-\alpha d}$, and let $\theta=\mathbb P(0 \in \mathcal C_\infty)>0$. In the model spanned in a box of volume $n$, let $\mathcal C_n(0), \mathcal C_n^{\sss{(1)}}, \mathcal C_n^{\sss{(2)}},$ be respectively the connected component of the origin, the largest and the second largest connected component.
  
Assume $\alpha\in(1,2)$ and  $\alpha\neq 1+1/d$.

\noindent Let $\rho\in(\theta, 1)$. There exists a constant $A_0$ such that for all $n\ge 1$, 
\begin{equation}\label{eq:upper-lrp}
-A_0n\, \le\, \log\Prob\big(|\CC_n^\sss{(1)}|>\rho n\big)\, \le\, -\tfrac1{A_0}n.
\end{equation}
Let $\rho\in(0,\theta)$. There exists a constant $A_1$ such that for all $n\ge 1$, 
\begin{equation}\label{eq:LRPcor-Cn1-lowertail}
\, -A_1n^{\max(2-\alpha, 1-1/d)}\le\,\log\Prob\big(|\CC_n^\sss{(1)}|<\rho n\big)\, \le\, -\tfrac{1}{A_1}n^{\max(2-\alpha, 1-1/d)}.
\end{equation}
There exist constants $A_2,\delta>0$ such that with probability at least $1-n^{-\delta}$,
\begin{equation}\label{eq:LRPcor-Cn2}
\tfrac{1}{A_2}(\log n)^{1/\max(2-\alpha, 1-1/d)}\,\le\, |\CC_n^\sss{(2)}|\,\le\, (A_2 \log n)^{1/(2-\alpha)}.
\end{equation}
 There exists a constant $A_3>0$ such that for all $k\ge 1$ and $n\in[A_3k, \infty]$, 
\begin{equation}\label{eq:LRPcor-Cn0}
\begin{aligned}
-A_3k^{\max(2-\alpha, 1-1/d)}\,\le\,\log \Prob\big(|\CC_n(0)|>k, 0\notin\CC_n^\sss{(1)}\big)\,\le\,-\tfrac{1}{A_3}k^{2-\alpha}.
\end{aligned}
\end{equation}
All lower bounds and the upper bound in~\eqref{eq:upper-lrp} extend to long-range percolation with $\alpha\ge 2$. 

\noindent Assuming high edge-density and $\alpha>1+1/d$, the upper bounds in \eqref{eq:LRPcor-Cn2} and \eqref{eq:LRPcor-Cn0} can be sharpened to match the lower bounds for all $\alpha>1+1/d$:
\begin{align}
\tfrac{1}{A_2}(\log n)^{1/(1-1/d)}\,\le\, &|\CC_n^\sss{(2)}|\,\le\, (A_2 \log n)^{1/(1-1/d)}\label{eq:high-density-CN2} \\
-A_3k^{1-1/d}\,\le\,&\log \Prob\big(|\CC_n(0)|>k, 0\notin\CC_n^\sss{(1)}\big)\,\le\,-\tfrac{1}{A_3}k^{1-1/d}. \label{eq:high-density-CN0}
\end{align}
\end{corollary} 
\begin{proof} 

\emph{Large deviations of the giant.\ }The proof of the upper tail in \eqref{eq:upper-lrp} can be found in Section \ref{sec:upper-tail} of this paper, and the statement is present in Theorem \ref{upper-lrp-nnp} (note that both bounds are valid for all $\alpha > 1$). The statement of the upper bound on the lower tail in \eqref{eq:LRPcor-Cn1-lowertail} can be found in Theorem \ref{thm:large-dev2}, where the condition $\zeta_{\mathrm{long}}>0$ translates to $\alpha\in(1,2)$ for LRP. For the proof, see Section \ref{sec:bootstrap} (Proposition \ref{lemma:bootstrap} in particular).  When the exponent is $2-\alpha$, the proof is in Section \ref{sec:bootstrap-long}, while for the exponent $1-1/d$, we refer to Section \ref{sec:bootstrap-short}. The lower bound of the lower tail in \eqref{eq:LRPcor-Cn1-lowertail} is \cite[Theorem 2.4]{clusterI}, while its proof can be found in Section 7 of the same paper. 

\emph{Cluster-size decay.\ }The upper bound in \eqref{eq:LRPcor-Cn0} follows from Theorem \ref{thm:subexponential-decay} of this paper, and its proof can be found in Section \ref{sec:cluster-size-lln}. Its proof heavily relies on \cite[Section 6]{clusterI}, where we developed a general technique to extend results on the decay of small components in finite boxes to the cluster-size decay of the infinite model. Moving on to the lower bound on the cluster-size decay in \eqref{eq:LRPcor-Cn0}, the statement itself is in Theorem \ref{thm:subexponential-decay}, while its proof is in Section \ref{sec:second-upper}, after Proposition \ref{prop:prer-lower} of this paper. The proof there heavily relies on \cite[Proposition 7.1, Section 7]{clusterI} where the actual lower bound appears as a localized event:  we find, with probability given in the lower bound, a localized component in a box of volume order $k$ centered at $0$. 

\emph{Second-largest component.\ }The bounds on the second-largest component are a consequence of Theorem \ref{thm:second-largest}. The proof of the lower bound is in Section \ref{sec:cluster-size-lln}, and relies on the same proof as for the cluster-size-decay, namely the localized component in \cite[Section 7, in particular Section 7.3]{clusterI}. 
The proof of the upper bound \eqref{eq:LRPcor-Cn2} on the second-largest component  follows a four-step revealment scheme described in the proof of Proposition \ref{prop:second-upper}.  
As the lower bound on the localized component event in \cite[Section 7]{clusterI} is valid for all $\alpha> 1$, all lower bounds remain valid for $\alpha\ge 2$. 

\emph{Upper bounds for $\alpha> 1+1/d$.\ }Finally, when we assume high edge density and $\alpha>1+1/d$ (which is actually below $2$ for dimensions $2$ and higher), the proofs of the upper bounds in \eqref{eq:high-density-CN0} and \eqref{eq:high-density-CN2} are the main results of \cite{clusterIII}. The proof of \eqref{eq:high-density-CN2} uses a decomposition of any component into nearest-neighbor connected blocks that are connected by edges longer than $2$, then carries out a spanning tree counting method combined with estimates based on isoperimetry. The extension to the upper bound on the cluster-size decay in \eqref{eq:high-density-CN0} follows the same method as here in Section \ref{sec:cluster-size-lln}, which is essentially \cite[Section 6]{clusterI}.
\end{proof}

 We believe that many of our proof techniques can be extended to other models with long-range interaction, for instance the random cluster model (FK-percolation) with $q\ge1$. To deal with the additional dependencies coming from boundary conditions in renormalization arguments, techniques from~\cite{pisztora1996surface} and~\cite{lrpRevisited20} could serve as a starting point. Extension to FK percolation would match the picture for large deviations of the largest component  for Bernoulli bond percolation on $\Z^d$ obtained by Pisztora~\cite{pisztora1996surface}. 

\medskip
\textbf{Notation.} Let $G=(V,E)$ be a graph. We  write $u\leftrightarrow_G v$ if there exists a path between $u,v\in V$, and $|A|$ for the number of vertices in a set $A$ of vertices, whereas for two subsets $V_1, V_2$ containing (marked) vertices we write $V_1 \sim_G V_2$ if there exist vertices $v_1 \in V_1, v_2 \in V_2$ such that $v_1$ is connected by an edge to $v_2$ (in case $V_1=\{v_1\}$ and/or $V_2=\{v_2\}$ we write simply $v_1 \sim_G V_2$, or $v_1 \sim_G v_2$, respectively). We also write $V_1 \nsim_G V_2$ if there is no edge between vertex pairs in $V_1\times V_2$. If the graph $G$ is clear from the context we omit the subscript. If $\CH=(\CV_\CH, \CE_\CH)$ is (a subgraph of) a KSRG $\CG=(\CV,\CE)$, we write $\CH[a,b)$ for the induced subgraph of $\CH$ on vertices with mark in the half-open interval $[a,b)$, denoted by $\CV_\CH[a,b)$. For a set $B\subseteq\R^d$, we write $\CV_B$ for the vertices with location in $B$. If $B=\Lambda_n=[-n^{1/d}/2, n^{1/d}/2]^d$, we simply write $\CV_n$.
For two random variables $X$ and $Y$, we write $Y \preccurlyeq X$, if $X$ stochastically dominates $Y$, that is, $\Prob(X \ge x) \ge \Prob(Y \ge x)$ for all $x \in \R$. For two random graphs $G_1=(V_1, E_1), G_2=(V_2, E_2)$ we say that $G_1$ stochastically dominates $G_2$ if there exists a coupling such that $\Prob\big(E_1\supseteq E_2, V_1\supseteq V_2\big)=1$. A sequence of events $(\CA_n)_{n\ge 1}$ holds with high probability (whp) if $\Prob(\CA_n)\to1$ as $n\to\infty$.
 We use standard Landau notation for asymptotics of functions: for two not necessarily positive real functions $f(n), g(n)$ we say that $f=O(g)$ if $\limsup_{n\to\infty}|f(n)/g(n)|<\infty$, $f=o(g)$ if $\lim_{n\to\infty}|f(n)/g(n)|=0$, $f=\Omega(g)$ if $g=O(f)$, $f=\omega(g)$ if $g=o(f)$, and $f=\Theta(g)$ if both $f=O(g)$ and $g=O(f)$.

\section{Roadmap of the paper}\label{sec:outline}
 Theorems~\ref{thm:large-dev-upper}--\ref{thm:second-largest} are closely linked together.  We prove them in (almost) reverse order, in six main steps.  
All except Step 4 rely on (often multi-scale) renormalization techniques. 
Here we sketch these steps and explain their relation, the role of the condition $\zeta_\mathrm{long}>0$, 
and the connection to the methods used in our earlier works~\cite{clusterI, clusterIII}. 
This sketch assumes a unit-intensity Poisson point process and power-law vertex marks. In the sections below we make the adaptation for the other settings. We start with the lower tail of large deviations.  
\subsection{Route to the lower tail of large deviations} The lower bound is given in~\cite{clusterI}, so we focus on the more involved upper bound.
Using $\zeta_\mathrm{long}>0$, for $\rho\in(0,\theta)$ we sketch how to prove 
\begin{equation}\label{eq:ltld-upper}
\Prob\big(|\CC_n^\sss{(1)}|<\rho n\big)\le \exp\big(-\Theta\big(n^{\zeta_\star}\big)\big).
\end{equation}

\medskip
\noindent\emph{Step 1. Existence of a giant (Section~\ref{sec:biskup}).} For supercritical settings with $\max(\zeta_\mathrm{hl}, \zeta_\mathrm{ll})>0$, we prove for a small constant $\varepsilon>0$ that  $\{|\CC_n^\sss{(1)}|\ge \varepsilon n\}$ whp. 
For $\zeta_\mathrm{hh}>0$ this step is already given by the weak law of large numbers for $|\CC_n^\sss{(1)}|$ in our earlier work~\cite{clusterI}. 
To prove $\Prob(|\CC_n^\sss{(1)}|\ge \varepsilon n)\to 1$, we first adapt a multi-scale renormalization technique by Berger~\cite{berger2002transience} which relies on ergodicity and uniqueness of the infinite component.  This technique yields a polynomial lower bound on $|\CC_n^\sss{(1)}|$. We  use this lower bound as the initialization of another multi-scale renormalization scheme, inspired by a method that Biskup introduced for long-range percolation~\cite{biskup2004scaling}. The two techniques combined yield that for some $\varepsilon>0$, $A>0$, and $\gamma\in(0,1)$ satisfying $\zeta=1-\gamma(\tau-1)$, 
\begin{equation} \label{eq:biskup-outline}
    \Prob\big(|\CC_n^\sss{(1)}[1, An^{\gamma})|< \varepsilon n\big)\le \exp\big(-n^{\max(\zeta_\mathrm{ll}, \zeta_\mathrm{hl})-o(1)}\big),
    \end{equation}
where $\CC_n^\sss{(1)}[1, An^{\gamma})$ 
denotes the largest connected component in the graph induced on vertices with location in $\Lambda_n$ and mark in $[1, An^{\gamma})$.

\medskip
\noindent\emph{Step 2. Sharp bounds on the giant's existence (Section~\ref{sec:bootstrap}).}
The inequality \eqref{eq:biskup-outline} is worse than~\eqref{eq:ltld-upper}. In the next step we improve it: 
 we show the implication that \emph{if} 
 \begin{equation}\label{eq:density-condition}\tag{Req}
      \parbox{\dimexpr\linewidth-17em}{%
    \strut
    $\Lambda_n$ has whp a giant component of density at least $\rho>0$\\
    containing all vertices with mark at least $ n^{(1-o(1))/(\tau-1)}$,
    \strut
  }
 \end{equation}
 \emph{then}
 for any $\varepsilon'>0$ and some $A'=A'(\varepsilon')>0$,
\begin{equation}\label{eq:outline-er}
\Prob\big(|\CC_n^\sss{(1)}[1, A'n^{\gamma})|< (\rho-\varepsilon') n\big) \le \exp\big(-\Theta(n^{\max(\zeta_\mathrm{ll}, \zeta_\mathrm{hl}, \zeta_\mathrm{hh},(d-1)/d)})\big).
\end{equation}
Compared to \eqref{eq:biskup-outline}, the $o(1)$ disappeared from the speed-exponent while $(d-1)/d$ and $\zeta_\mathrm{hh}$ entered the maximum.
To prove~\eqref{eq:outline-er} when $\zeta_\mathrm{long}>(d-1)/d$, we reduce the event $\{|\CC_n^\sss{(1)}|<(\rho-\varepsilon')n\}$ to a large-deviation event for Erd\H{o}s-R\'enyi random graphs, see below \cref{phen:lower}.
When $(d-1)/d$ is the unique maximum and $\zeta_\mathrm{long}>0$, we prove \eqref{eq:outline-er} with a different renormalization. Using boxes of constant volume $M$, the renormalized model dominates supercritical site-bond percolation on $\Z^d$. Using long-edges present due to $\zeta_\mathrm{long}>0$, two neighboring boxes of volume $M$ are connected by an edge with probability tending to $1$ as $M\to\infty$. 

We also prove the condition \eqref{eq:density-condition} of \eqref{eq:outline-er}: all vertices with mark at least $(\log n)^{\eta'}$ for an explicit $\eta'>0$ are in the giant component: using an `extra layer' on the renormalization leading to \eqref{eq:biskup-outline} we prove in the case $\max(\zeta_\mathrm{hl}, \zeta_\mathrm{ll})>0$ that
\begin{equation}\label{eq:high-mark-outline}
\Prob\big(\exists \text{ component }\CC\text{ in }\CG_n: \CV_n[(\log n)^{\eta'}, \infty)\supseteq \CC, |\CC|\ge \varepsilon n\big)\longrightarrow 1.
\end{equation}
We  combine this with~\eqref{eq:outline-er}, and set $\rho=\eps, \eps'=\eps/2$  when $\zeta_\mathrm{long}=\max(\zeta_\mathrm{hl}, \zeta_\mathrm{ll})$. When $\zeta_\mathrm{long}=\zeta_\mathrm{hh}$, we use the results from~\cite{clusterI} and set $\rho=\theta-\eps$. In Section \ref{sec:bootstrap} we thus prove that for some $\eps>0$, 
\begin{equation}\label{eq:outline-right-prob}
    \Prob\big(|\CC_n^\sss{(1)}[1,A'n^\gamma)|< \varepsilon n\big)\le \exp\big(-\Theta(n^{\zeta_\star})\big).
\end{equation}
After having proven an improved value of $\rho=\theta-\eps$ in Step 5 below, requirement \eqref{eq:density-condition} shall imply an improved version of  \eqref{eq:outline-right-prob} with the new density.

\medskip
\noindent\emph{Step 3. Non-giant components are small (Section~\ref{sec:second-sub}).}
In this step we prove the upper bound on the second-largest component in Theorem \ref{thm:second-largest}, using~\eqref{eq:outline-right-prob}. For KSRGs with PPP vertex sets, for any $k=k(n)$, we prove
\begin{equation}\label{eq:second-outline}
\Prob\big(|\CC_n^\sss{(2)}|> k\big)\le 3(n/k)\exp\big(-\Theta( k^{\zeta_\mathrm{long}})\big)=:\mathrm{err}_{n,k}.
\end{equation}
Given \eqref{eq:second-outline}, substitute $k=(A\log n)^{1/\zeta_\mathrm{long}}$ for a large constant $A=A(\delta)$ to obtain the upper bound in \cref{thm:second-largest}. The lower bound there follows by verifying a prerequisite from our earlier work~\cite{clusterI}. There we
also proved \eqref{eq:second-outline} when $\zeta_\mathrm{long}=\zeta_{\mathrm{hh}}$. The method in~\cite{clusterI} breaks down for the cases when $\zeta_\mathrm{long}\in\{\zeta_\mathrm{ll}, \zeta_\mathrm{hl}\}$, see \cref{remark:hh-different}. So, we use a new method that we explain now. 

We sequentially reveal the graph. 
After the last revealment stage a \emph{non-largest} component of size at least $k$ is present in $\Lambda_n$ with probability at most $\mathrm{err}_{n,k}$. 
We partition $\Lambda_n$ into boxes of volume $k$. The probability that all boxes contain a linear-sized component on vertices with mark at most $A'k^\gamma$ is at least $1-\tfrac{1}{3}\mathrm{err}_{n,k}$ by \eqref{eq:outline-right-prob}. When $\zeta_\mathrm{long}=\zeta_\mathrm{hl}$, we sprinkle in vertices of mark larger than $A' k^{\gamma_\mathrm{hl}}$, forming the PPP $\CV^{\sss{\mathrm{high}}}$. We say that a vertex $v\in \CV^{\sss{\mathrm{high}}}$ is \emph{box-wedging} if it connects to the local giant in its own box and also to the local giant in the neighboring box. We prove that the number of box-wedging vertices in each box is $\Theta(k^{\zeta_\mathrm{long}})$: see the back-of-the-envelope calculations in Section~\ref{sec:dominant} for intuition on this. 
 So,  all local giants exist and are in the same connected component in $\CG_n$ via box-wedging vertices with probability at least  $1-\tfrac23\mathrm{err}_{n,k}$. We call this component the \emph{backbone}.

We now take care of small components: in the third stage we reveal edges between vertices that are not in local giants but in different boxes, obtaining $\CG^{\sss{3\mathrm{rd}}}_n$. In the fourth revealment stage we reveal edges between the components of $\CG^{\sss{3\mathrm{rd}}}_n$ and the backbone.  If a component of $\CG^{\sss{3\mathrm{rd}}}_n$ is larger than $k$, then we show that with probability at least $1-\exp\big(-\Theta(k^{\zeta_\mathrm{long}})\big)$ it connects to a box-wedging vertex in the backbone. So, all components of size at least $k$ merge with the backbone with probability at least $1-\mathrm{err}_{n,k}$. 
\medskip 
\noindent\emph{Step 4. LLN for the giant, cluster-size decay (Section~\ref{sec:cluster-size-lln}).}
In \cite[Proposition 6.1]{clusterI}, we showed how the combination of the bounds \eqref{eq:high-mark-outline},  \eqref{eq:second-outline}, and the whp existence of a polynomially-sized largest component (a weaker bound than \eqref{eq:outline-right-prob}), leads to the upper bound on the cluster-size decay in Theorem \ref{thm:subexponential-decay} and the weak LLN
\begin{equation}\label{eq:lln}
 |\CC_n^\sss{(1)}|/n \overset{\Prob}\longrightarrow \theta(\beta, p, \alpha, \tau, \sigma),\qquad \mbox{as }n\to\infty.
 \end{equation}

\medskip 
\noindent\emph{Step 5. Lower tail of large deviations of the giant (Section~\ref{sec:lower-tail}).}\label{roadmapuppertail}
The LLN in~\eqref{eq:lln} gives that $|\CC_n^\sss{(1)}|/n\ge \theta-\varepsilon$ holds whp for any $\varepsilon>0$. So the condition \eqref{eq:density-condition} holds with $\rho=\theta-\eps$. We can reuse~\eqref{eq:outline-er}, combine it with \eqref{eq:high-mark-outline}, so that the unique giant contains all highest-mark vertices whp. Then combination with \eqref{eq:second-outline} results in~\eqref{eq:ltld-upper}.
\smallskip

\begin{remark}[Difficulties when surface tension is dominant]\label{remark:nn-difficulty}
When $(d-1)/d>\zeta_\mathrm{long}$, Theorem \ref{thm:large-dev2} on the lower tail of LDP is sharp, while Theorems \ref{thm:subexponential-decay} and \ref{thm:second-largest} are not sharp. 
In the proofs we connect local giants in subboxes via edges with length of the same order as the box-diameter. These edges are present under the assumption that $\zeta_\mathrm{long}>0$. 
This is enough to give Theorem \ref{thm:large-dev2}: to be able to estimate the giant's size, not every subbox needs to contribute to the global giant. 
The condition $\zeta_{\mathrm{long}}>0$ could only be dropped by developing more information about the position of local giants inside subboxes. They must come close to the boundary to be able to connect them without using long edges. This is beyond the scope of this paper.

For proving Theorems \ref{thm:subexponential-decay} and \ref{thm:second-largest} with $\zeta=(d-1)/d$ in the upper bound, we need sharper estimates. Between the giant component and any other component of size at least $k$, we need to find at least $\Omega(k^{(d-1)/d})$ non-connected vertex pairs within constant distance, with probability at least $1-n\exp(-\Omega(k^{(d-1)/d}))$. Our techniques do not focus on the geometry of short missing edges around components, so this case is beyond the scope of this paper. For NNP, such geometric information is known, see \cite{Grim99} for references. For high edge-density long-range percolation our recent work~\cite{clusterIII} solves this issue, and obtains matching upper and lower bounds for $|\CC_n^\sss{(2)}|$ and the cluster-size decay. For long-range percolation, we refer also to the recent work by B\"aumler~\cite{Baumler} for related results.
\end{remark}
\begin{remark}[High-high regime is different]\label{remark:hh-different}
In \cite{clusterI}, we proved the bound on $|\CC_n^\sss{(2)}|$ in~\eqref{eq:second-outline} when $\zeta_{\mathrm{hh}}$ is maximal. When high-high type connections dominate, the likeliest way to connect two boxes is via edges between two high-mark vertices.
So, building a backbone with error probability $\mathrm{err}_{n,k}$ in \eqref{eq:second-outline} is more straightforward than the multi-scale renormalization techniques here. 
However, the sequential revealment method sketched in Step 3 above breaks down when we look at small components of size at least $k$. When $\zeta_{\mathrm{hh}}=\zeta_\star$, we also need to show that the giant has $\Theta(k^{\zeta_\mathrm{hh}})$ many high-mark vertices within ``connection probability'' $\Theta(1)$ to every one of these components, with probability $\mathrm{err}_{n,k}$. The revealment method from Step 3 does not guarantee that, since the diameter of the box $\Theta(k^{1/d})$ is a too crude bound on the distance between previously revealed small components and high-mark vertices. So, in \cite{clusterI} we developed a geometric method that we call \emph{cover-expansion}, which allows us to control this distance based on the geometry of the cluster. See \cite[Section 3]{clusterI} for a comparison. 
\end{remark}

\subsection{The upper tail of large deviations (Section~\ref{sec:upper-tail})}\label{sec:outline-upper}
First we present the intuition behind the formula of $\mathrm{hubs}(\rho)$ in~\eqref{eq:hubs-gen}, sketched in Figure~\ref{fig:hubs}. Afterwards, we outline the proof of \cref{thm:large-dev-upper} and \cref{upper-lrp-nnp}.  

\begin{figure}
    \centering
    \includegraphics[width=0.8\linewidth]{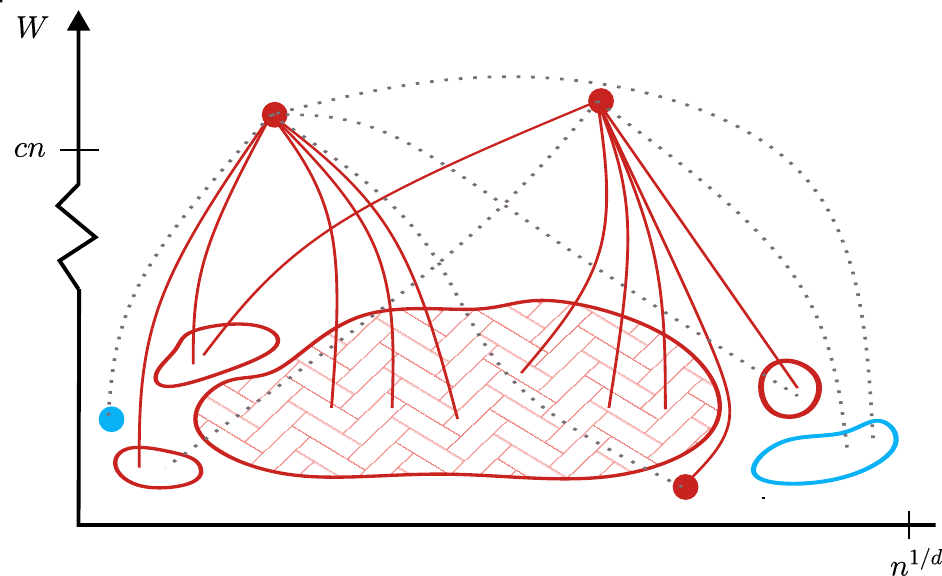}
    \captionsetup{width=0.95\linewidth}
    \caption{Strategy for the upper tail. The giant induced on lower-mark vertices (filled component) increases in size:  hubs (vertices of mark $\Omega(n)$) connect to it by an edge, and also to smaller-size components. The resulting giant component in the entire graph is colored red. If $p<1$, the hubs do not connect by an edge to all vertices (missing edges are dashed in gray), leaving some small components or isolated vertices (blue). The probability for a component of size exactly $k$ to have no edge to one of at least $h$ many hubs is $(1-p)^{kh}$, so the giant's size increases by roughly $\sum_{k} \mathbb P(\mathcal C(0)=k)(1-(1-p)^{kh})$. The function $\mathrm{hubs}$ in~\eqref{eq:hubs-gen} gives the minimal number of hubs, so that the size of the giant increases above $\rho n$. }
    \label{fig:hubs}
\end{figure}

\medskip
\noindent \emph{Hubs increase the giant linearly.}
We investigate an unlikely event. Take $R$ to be a large constant. When $h$ vertices with mark at least $Rn$ -- so called \emph{hubs} -- are present in a box of volume $n$, the size of the giant increases linearly in $n$, since each hub connects by an edge to every other vertex in the box with probability $p$ independently. By the Pareto distribution of the marks in~\eqref{eq:power-law}, the probability of having $h$ hubs is $\Theta(n^{-h(\tau-2)})$. For the exact rate function, we need to compute how the giant's size increases when there are $h$ hubs present.
We reveal the edges in two stages: first all edges between non-hub vertices, obtaining the graph $\CG_n[1,Rn)$. Then we reveal all edges incident to hubs, obtaining $\CG_n$. 
We capture $|\CC_n^{\sss{(1)}}|$ via its complement: 
small components of $\CG_n[1,Rn)$ do not merge with the giant in the second step if none of the vertices in the component connect to the hubs.  Distinguishing all component sizes, we write
\begin{equation}\label{eq:random-clusters}
\begin{aligned}
1-\frac{|\CC_n^{\sss{(1)}}|}{n}\approx\frac{|\CV_n\setminus\CC_n^\sss{(1)}|}{n}&\approx \sum_{v\in\CV_n[1,Rn)}
\!\!\!\!\frac{\ind{\CC_n(v)[1,Rn)\nsim \mathrm{hubs}}}{n} 
=\sum_{\ell\ge 1}\sum_{\substack{\CC\in \CG_n[1,Rn):\\ |\CC|=\ell}}\frac{\ell\cdot\ind{\CC\nsim \mathrm{hubs}}}{n}.
\end{aligned}
\end{equation}
The probability that no edge exists between a size-$\ell$ component and the $h$ hubs is $(1-p)^{h\ell}$. Moreover, by local convergence (see~\cite{Hofbook2} for a detailed explanation and references), the proportion of vertices in components of size-$\ell$ converges in probability to $\Prob^\sss{0}(|\CC(0)|=\ell)$. Therefore, the total number of size-$\ell$ components in $\Lambda_n$ concentrates around $(n/\ell)\cdot\Prob^\sss{0}(|\CC(0)|=\ell)$. Hence, 
\begin{equation}\label{eq:concentrated-clusters}
1-\frac{|\CC_n^{\sss{(1)}}|}{n}\approx\sum_{\ell\ge 1}\Prob^\sss{0}\big(|\CC(0)|=\ell\big)\cdot (1-p)^{h\ell}=\E^\sss{0}\big[(1-p)^{h|\CC(0)|}\big].
\end{equation}
On the event $\{|\CC_n^\sss{(1)}|/n>\rho\}$,
the right-hand side above should be at most $1-\rho$. Thus, letting $h$ to be the smallest integer satisfying
\begin{equation}\label{eq:hub-informal}
\E^\sss{0}\big[(1-p)^{h|\CC(0)|}\big] \le 1-\rho,
\end{equation}
the presence of $h$ hubs implies the event $\{|\CC_n^\sss{(1)}|/n>\rho\}$.
The solution to \eqref{eq:hub-informal} with equality is exactly $\mathrm{hubs}(\rho)$  in~\eqref{eq:hubs-gen}. 

\medskip
\noindent\emph{Concentration on the number of size-$\ell$ components.} 
To move from \eqref{eq:random-clusters} to \eqref{eq:concentrated-clusters}, we need that the rate of concentration of the number of size-$\ell$ components is $o(n^{-(\tau-2)h}))$. Local convergence (LC) does not provide any convergence rate. We boost the performance of LC by applying it to disjoint subboxes of $\Lambda_n$, binomial concentration, and a bound on the number of components containing edges that cross box-boundaries.
When vertex-marks are at most a linear function in $n$, this yields a polynomial convergence rate with a tunable exponent.

\medskip\noindent
\emph{Linear-scale hubs are required to increase the giant's size.}
For the lower bound in \eqref{eq:thm-utld2}, we use that the graph without hubs contains a giant component of density almost $\theta$ by the lower tail of large deviations. The hubs connect to this giant with error probability exponential in $n$. When $\mathrm{hubs}(\rho)\notin\N$, our concentration bound on size-$\ell$-component is strong enough to guarantee that the presence of  $\lceil\mathrm{hubs}(\rho)\rceil$ hubs increases the giant's size to a value above $\rho n$. We also show that any event that has fewer than $\lceil\mathrm{hubs}(\rho)\rceil$ hubs means that too many size-$\ell$ components remain outside of the giant with probability at least $\Theta\big(n^{-(\tau-2)h}\big)$. So, 
the complement of the giant remains larger than $(1-\rho)n$ on such events. This gives the upper bound.

\medskip\noindent
\emph{Homogeneous percolation models.} In Bernoulli bond percolation or long-range percolation on $\Z^d$, the degree-distribution has (super-)exponential tails, so no linear-scale hubs are present. We show that the number of size-$\ell$ components concentrate with error probability exponentially small in $n$ for each constant $\ell\in\N$. We capture the giant's size by its complement, and \cref{upper-lrp-nnp} follows.

\section{Existence of a giant}\label{sec:biskup}
The main goal of this section is to prove the existence of a giant, i.e., a linear-sized component, for the whole supercritical regime when $\max(\zeta_\mathrm{hl},\zeta_\mathrm{ll})>0$. For $\zeta_\mathrm{hh}>0$, we already proved existence in our recent work \cite{clusterI}.
The first proposition corresponds to \eqref{eq:biskup-outline} in the roadmap in Section \ref{sec:outline}. 
\begin{proposition}[Existence of a giant]\label{proposition:ltld-weaker}
Consider a supercritical KSRG satisfying Assumption \ref{assumption:main} with a Poisson point process as vertex set. Assume that the parameters satisfy $\max(\zeta_\mathrm{hl}, \zeta_\mathrm{ll})>0$.
For each constant $\delta>0$, there exists a constant $A>0$ such that for all $n\ge 1$,
    \begin{itemize}
        \item[(hl)] when $\zeta_\mathrm{hl}>0$,
\begin{equation}\label{eq:ltld-weaker}
     \Prob\big(|\CC_n^\sss{(1)}[1, An^{\gamma_\mathrm{hl}})|\le \tfrac{1}{A} n \big) \le \exp\big(- \tfrac{1}{A} n^{\zeta_\mathrm{hl}-\delta}\big) ,
 \end{equation}
 \item[(ll)] when $\zeta_\mathrm{ll}>0$, 
 \begin{equation}\label{eq:ltld-weaker-ll}
     \Prob\big(|\CC_n^\sss{(1)}[1, A)|\le \tfrac{1}{A} n \big) \le \exp\big(- \tfrac{1}{A} n^{\zeta_\mathrm{ll}-\delta}\big).
 \end{equation}    
    \end{itemize}    
 The statements~\eqref{eq:ltld-weaker} and~\eqref{eq:ltld-weaker-ll} remain valid for the Palm version $\Prob^\sss{x}$ of $\Prob$ for any $x\in\R^d$.
\end{proposition}
\cref{proposition:ltld-weaker} generalizes~\cite[Theorem 3.2]{biskup2004scaling} by Biskup from long-range percolation on $\Z^d$ to a wider class of KSRGs. The presence of vertex marks in KSRGs,  and the truncation of them in the statement of Proposition~\ref{proposition:ltld-weaker} require new ideas. 
We will prove Proposition~\ref{proposition:ltld-weaker} via a multi-scale renormalization scheme. To initialize this scheme, we need an initial  lower bound on $|\CC_n^{\sss{(1)}}|$.
\begin{lemma}[Existence of an ``almost''-giant]\label{lemma:ltld-weakest}
Consider a supercritical KSRG under the same setting as in Proposition \ref{proposition:ltld-weaker}. For all $\varepsilon>0$
 \begin{equation}\label{eq:ltld-weakest}
     \Prob\big(|\CC_n^\sss{(1)}|\le n^{1-\varepsilon} \big) \longrightarrow 0,\qquad\mbox{ as }n\to\infty .
 \end{equation}
 The statement remains valid for the Palm version $\Prob^\sss{x}$ of $\Prob$ for any $x\in\R^d$.
\end{lemma}
The proof of this lemma is inspired by a multi-scale renormalization technique introduced for long-range percolation by Berger in~\cite{berger2002transience}. When $\zeta_\mathrm{ll}=2-\alpha>0$ and the vertex set is formed by a PPP, \cite[Proposition 3.9]{GraHeyMonMor19} already proves \eqref{eq:ltld-weakest} using an adaptation of~\cite{berger2002transience}.
For the case $\zeta_\mathrm{hl}>0$ we give a new proof  in  Appendix~\ref{app:linear-sized} that explicitly relies on the vertex-marks.  The rest of this section is devoted to proving Proposition~\ref{proposition:ltld-weaker}. We first introduce some additional preliminaries.

\subsection{Preliminaries}\label{sec:sprinkle}
We describe a sprinkling technique that enables us to reveal the vertex set of a KSRG on a PPP in multiple stages, so that the graph is  supercritical already before sprinkling.  
\begin{definition}[Scaled-thinned KSRG on a PPP]\label{def:scale-thin} Consider a KSRG on a PPP.
    Let $q\in[0,1]$ and $\eta>0$. We write $\eta\CG_n^{\beta, q}$ for a scaled-thinned KSRG as follows: first, we keep each vertex of the PPP $\CV$ independently with probability $q$. Then we sample the edges according to \cref{def:ksrg} with edge-density parameter $\beta>0$. Finally, we we assign a new location to each retained vertex, i.e., we map $u=(x_u, w_u)\in\R^d\times[1,\infty)$ to $(\eta x_u, w_u)$. 
\end{definition}
It is not hard to see that the scaled-thinned KSRG with $\eta=q$ is distributed like the original graph with a smaller parameter $\beta$ in a smaller box. This is the content of the following observation.
\begin{observation}[Sprinkling trick-A]\label{obs:sprinkle}
Consider a supercritical KSRG on a PPP, i.e., with $\beta>\beta_c(p, \alpha, \tau, \sigma)$. Let $q\in(0,1]$. Then the following distributional identity holds:
\begin{equation}\label{eq:sprinkling-dist}
    q\CG_n^{\beta, q}\overset{d}{=}\CG_{qn}^{\beta q^d, 1} \overset{d}{=}  \CG_{qn}^{\beta q^d}
\end{equation}
In particular, the graph $q\CG_n^{\beta, q}$ is supercritical whenever the new edge-intensity $\beta q^d>\beta_c$.
\begin{proof}
Denote by $\CV^{(q,q)}$ the marked vertex set of $q\CG_n^{\beta, q}$. Since the projection of $\CV$ onto the spatial dimensions in $\R^d$ is a PPP with intensity being the Lebesgue measure, thinning $\CV$ gives intensity $q\mathrm{Leb}(\cdot)$. Then, re-scaling with $\eta=q$ maps the box $\Lambda_n$ to $\Lambda_{qn}$. For any set $A\subseteq \Lambda_{qn}$, the expected number of points in it is $q \mathrm{Leb}(\eta^{-1}A)= \mathrm{Leb}(A)$. Hence $\CV^{(q,q)}$ has the same distribution as $\CV\cap \Lambda_{qn}$, since the spatial intensity of the new point process $\CV^{(q,q)}$ is $\mathrm{Leb}(\cdot)$ again, and the distribution of the mark-coordinate is unchanged. 
We now inspect the connection probabilities in $q\CG_n^{\beta, q}$. Let us denote the new connectivity function by $\mathrm{p}^{(q,q)}$. Consider two points $u=(x_u, w_u)$ and  $v=(x_v, w_v)$ in $\CV^{(q,q)}$. Since we connect $u$ to $v$ by an edge in $q\CG_n^{\beta, q}$ iff we connect their inverse points $(q^{-1}x_u, w_u)$ and $(q^{-1}x_v, w_v)$ in $\CG_n$ by an edge, the new connection probability $\mathrm{p}^{(q,q)}$ is given by
\[ 
\mathrm{p}^{(q,q)}\big((x_u, w_u),  (x_v, w_v)\big) = \mathrm{p}\big( (q^{-1}x_u, w_u), (q^{-1}x_v, w_v)\big)
= 
 \varphi\Big(\beta\frac{\kappa(w_u, w_v)}{\|q^{-1}(x_u-x_v)\|^d}\Big).\]
 Thus, the new graph has the same distribution as a graph on  $\CV\cap \Lambda_{qn}$ with the same connection probabilities but new edge-density parameter $\beta'=\beta q^d$. This graph is denoted by $\CG_{qn}^{\beta',1}$, since in this latter notation the thinning factor is $q=1$, and the box is $\Lambda_{qn}$. The last identity in \eqref{eq:sprinkling-dist} says that omitting $q=1$ from the superscript is consistent with the notation. 
 Finally, the new graph is supercritical whenever the new edge-intensity $\beta q^d$ is at least $\beta_c$ as well.
\end{proof}
\end{observation}
Throughout the paper we will use the following construction of KSRGs when we consider PPPs as the underlying vertex set, following the setting of Assumption~\ref{assumption:main}.
\begin{definition}[Alternative KSRG construction on a PPP]\label{def:ksrg-alt}
    Consider a KSRG on a PPP with parameters and kernel $\kappa$ from Assumption~\ref{assumption:main}. If $\beta>\beta_c$, let $q\in(0,1]$ be such that $\beta q^d >\beta_c$, i.e., the scaled-thinned graph $q\CG_{n}^{\beta,q}$ in \eqref{eq:sprinkling-dist} is supercritical. If $\beta\le \beta_c$, set $q=1$.
    Let $\CV^\sss{\mathrm{base}}$ and $\CV^\sss{\mathrm{spr}}$ be two independent PPPs with intensity $q\mathrm{Leb}\!\times\! F_W(\rd w)$ and $(1-q)\mathrm{Leb}\!\times\! F_W(\rd w)$, respectively, with $F_W$ as in Assumption~\ref{assumption:main}. Define 
    \begin{equation*}
        \overline\CV:=\CV^\sss{\mathrm{base}}\cup\CV^\sss{\mathrm{spr}},
    \end{equation*} 
    which is equal in distribution  to a Poisson point process on $\R^d\times[1,\infty)$ with intensity 
    \begin{equation}\label{eq:poisson-intensity}
    \mathrm{Leb}\times F_W(\rd w) = \begin{dcases}
        \rd x\times (\tau-1)w^{-\tau}\rd w,&\text{if }\tau<\infty,\\
        \rd x\times \delta_1(\rd w),&\text{if }\tau=\infty,
    \end{dcases} 
    \end{equation}
    where $\delta_1(\cdot)$ denotes a Dirac measure at one.
    Conditionally on the realization of $\overline\CV$, let $\overline\CG_n^\sss{\beta}$ be the graph on $\overline \CV$ where two vertices $u, v\in\overline\CV$ are connected by an edge independently with probability 
    \begin{equation}
        \mathrm{p}(u,v)=
        \begin{dcases}
            p\Big(\beta\cdot\frac{\kappa(w_u, w_v)}{\|x_u-x_v\|^d}\wedge 1\Big)^\alpha,&\text{if }\alpha<\infty,\\
            p\mathds{1}\big\{\beta\kappa(w_u, w_v)\ge \|x_u-x_v\|^d
            \big\},&\text{if }\alpha=\infty.
            \end{dcases}\label{eq:connection-prob-gen}
    \end{equation}
    Let $\CG_n^{\beta,\sss{\mathrm{base}}}\ {\buildrel d \over =}\ \CG_n^{\beta,q}$ be the induced subgraph on $\CV^\sss{\mathrm{base}}$. If $q<1$, we call $\CG_n^\sss{\beta,\mathrm{base}}$ the \emph{thinned} KSRG.
  We omit the superscript $\beta$ if it is clear from the context.
\end{definition}
\begin{observation}[Sprinkling trick-B]\label{obs:sprinkle-2}
    Consider a supercritical KSRG on a PPP, and let $\overline \CG_n^{\sss{\beta}}$ and $\CG_n^{\beta,\sss{\mathrm{base}}}$ be as in \cref{def:ksrg-alt}. Write $\beta':=\beta q^d$.
    Let $\CC^\sss{(1)}_{n,\sss{\mathrm{base}}}$ and $\CC^\sss{(1)}_{qn,\beta'}$ be the largest component in $\CG_n^{\beta, \sss{\mathrm{base}}}$ and $\CG_{qn}^{\beta'}$, respectively. Then 
  \[
    \overline \CG_n^\sss{\beta}\overset{d}=\CG_n^\sss{\beta},
    \qquad 
    \CG_n^{\beta,\sss{\mathrm{base}}} \overset{d}= \tfrac1q \CG_{qn}^{\beta'}, \qquad     |\CC_{qn, \beta'}^\sss{(1)}|
    \, \overset{d}=\,
    |\CC_{n, \sss{\mathrm{base}}}^\sss{(1)}|. 
    \]
    \begin{proof} $ \overline \CG_n^\sss{\beta}$ has the same law as $\CG_n^\sss{\beta}$ since $\overline \CV$ has the same law as $\CV$ and the connection probabilities are the same. $\CG_n^{\beta,\sss{\mathrm{base}}}$ has the same law as $1\CG_{n}^{\beta, q}$ in \cref{def:scale-thin}, i.e., we just apply thinning but do not rescale. We write this now as $q^{-1} (q\CG_{n}^{\beta, q})$ and apply the identity \eqref{eq:sprinkling-dist} to the scaled-thinned graph, and we get the second distributional identity above. The last identity about the largest component holds since edges (and thus connected components) do not change under rescaling.
    \end{proof}
\end{observation}
Let us also state a well-known statement about the connectivity of Erd\H{o}s-R\'enyi random graphs. We give a short (suboptimal) proof for the sake of completeness.
\begin{claim}[Connectivity of $G(n,p_n)$]\label{claim:er-conn}
    Consider the Erd\H{o}s-R\'enyi random graph $G(n,p_n)$ with $np_n=\omega(\log n)$. For all sufficiently large $n$, 
    \begin{equation}
        \Prob\big(\mbox{$G(n,p_n)$ is not connected}\big)\le 3(\re n)^2(1-p_n)^{n/2}. 
    \end{equation}
    \end{claim}
    \begin{proof}If $G(n,p_n)$ is not connected, then there must exist a connected component of size $k$, for some $1\le k \le \lfloor n/2 \rfloor$. For $k=1$, this probability is at most $n(1-p_n)^{n-1}\le (\re n)^2 (1-p_n)^{n/2}$ by a union bound over all vertices. For $k\ge 2$, by Cayley's formula, there are $k^{k-2}$ spanning trees on $k$ vertices, and a connected component of size $k$ has to contain such a tree. So by a union bound over the spanning trees, the probability that a connected component of size between $2$ and $n/2$ exists without being connected to the other $n-k$ vertices, is at most
    \[
    \sum_{k=2}^{\lfloor n/2 \rfloor} \binom{n}{k} k^{k-2}(1-p_n)^{k(n-k)} \le \sum_{k=2}^{\lfloor n/2 \rfloor}\left(\frac{\re n}{k}\right)^k k^{k-2}(1-p_n)^{k(n-k)}\le \sum_{k=2}^{\lfloor n/2 \rfloor} \big(\re n(1-p_n)^{n-k}\big)^k.
    \]
     Using $n-k\ge \lceil n/2\rceil$, we get that
    the right-hand side is bounded from above by 
    \[
    \sum_{k=2}^{\lfloor n/2 \rfloor} \big(\re n(1-p_n)^{\lceil n/2\rceil}\big)^k\le \frac{(\re n)^2(1-p_n)^{2\lceil n/2\rceil }}{1-\re n(1-p_n)^{\lceil n/2\rceil }}\le 2(\re n)^2(1-p_n)^{n/2}.
    \]
 Combining the cases of $k=1$ and $k\ge 2$ we obtain the desired statement. 
    \end{proof}

\subsection{A linear-sized component (high-low regime)}
We turn to the proof of Proposition~\ref{proposition:ltld-weaker} for the hl-regime, so we restrict to KSRGs on PPPs. We describe the adjustments for the ll-regime at the end of this section.  A sketch of the argument can be found in Figure~\ref{fig:multi-scale}.
\begin{figure}
    \centering
    \includegraphics[width=0.95\textwidth]{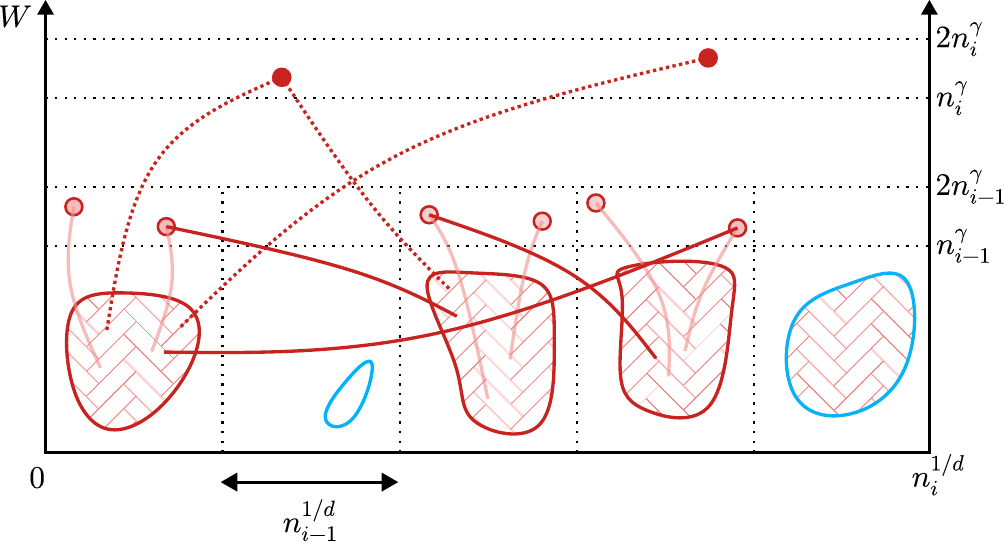}
    \captionsetup{width=0.95\linewidth}
    \caption{Sketch of the multi-scale renormalization for the hl-regime. The level-$(i-1)$ box situated second from the left is bad because its largest connected component is too small (e.g.\ due to failure of \emph{Con} at a previous level). The fifth box is bad because the \emph{Mark} step failed at level $i-1$, even though it contains a large component. The large components in the first, third, and fourth box are connected using the connector vertices, forming the darkred component. At the end, we reveal the high-mark vertices in $[n_i^\gamma, 2n_i^\gamma)$ and their edges to the local giant in the level-i box. If there are sufficiently many high-mark vertices, the box passes the \emph{Mark}-step. We will only use these vertices in the next level.}
    \label{fig:multi-scale}
\end{figure}

\subsubsection{Multi-scale renormalization -- sketch}\label{sec:outline-biskup-renorm} We set up an iterative boxing scheme. The initial box size is a constant $n_0$. We check if the largest component of a thinned KSRGs inside a level-$0$ box has size at least $\rho_0n_0:=n_0^{1-\varepsilon}$.  Next, we sprinkle vertices of sufficiently high (but constant) mark that connect to this largest component. If both of these steps succeed, we call a level-$0$ box good. This happens with constant probability using \cref{lemma:ltld-weakest}. 

For each iteration $i\ge1$, we group $m_{i}$ level-{$(i-1)$} boxes of volume $n_{i-1}$ into a larger level-$i$ box $\mathrm{B}$ of volume $n_i:=m_{i}n_{i-1}$. We always work with only the `good' subboxes of the previous level $i-1$ inside $\mathrm{B}$ to determine whether the level-$i$ box $\mathrm{B}$ is good,
and  control the probabilities of the following three events:
\begin{enumerate}
    \item[\emph{(Sub)}]Using independence across subboxes, at least $(1-\varepsilon_{i})m_{i}$ of the level-$(i-1)$ subboxes of volume $n_{i-1}$ of $\mathrm{B}$ are level-$(i-1)$-good with sufficiently large probability. 
    That is, the subboxes satisfy the next  events \emph{(Con), (Mark)} at level $i-1$: they have a local component of density at least $\rho_{i-1}$, which contains enough level-$(i-1)$ high-mark vertices. 
    \item[\emph{(Con)}] Each good level-$(i\!-\!1)$ subbox of $\mathrm{B}$  has enough high-mark vertices inside a local  component by the event \emph{(Mark)} on the previous level. We use these high-mark vertices to connect the level-$(i\!-\!1)$ subboxes \emph{to each other}. This is done by a stochastic domination of an Erd\H{o}s-R\'enyi random graph $G\big((1-\varepsilon_i)m_i, q_i\big)$: the giant component of each good level-$(i-1)$ subbox is represented by a vertex in $G$. We ensure that $q_i$ is above the connectivity threshold of Erd\H{o}s-R\'enyi random graphs. The outcome is that the level-$i$ box $\mathrm{B}$ contains a connected component of density at least $\rho_{i}:=\rho_{i-1}(1-\varepsilon_{i})$ with error probability stretched exponential in $n_i$ (and not $n_{i-1}$ that we had on level $i-1$).
    \item[\emph{(Mark)}]The sprinkling trick from Section~\ref{sec:sprinkle} allows us to sprinkle level-$i$ high-mark vertices that were not exposed in previous levels. We make sure enough of them connect to the large connected component in $\mathrm{B}$ built in the \emph{Con}-step. We will then use these vertices in the \emph{Con}-step of level $i\!+\!1$ later.
\end{enumerate}
If the steps \emph{Con} and \emph{Mark} succeed then we call $\mathrm{B}$ level-$i$ good. (The first step \emph{Sub} is only a tool to control error probabilities of these latter steps).
We will ensure that $\inf_i\rho_i>0$, so the size of the connected component we construct in the \emph{Con}-step is linear in $n_i$. Currently we use \cref{lemma:ltld-weakest} for initializing the level-$0$ boxes: this guarantees local largest components of size $n_0^{1-\eps}$, i.e., density $n_0^{-\eps}$, and error probability tending to $0$ at an unspecified rate. This proof of Proposition~\ref{proposition:ltld-weaker} here will `boost' the giant to have density at least of order $n_0^{-\varepsilon}n_i$, which is linear in $n_i$ (as $i\to \infty$, because $n_0$ is a constant), and the error probability becomes stretched exponential in $n_i$. This renormalization gives Proposition~\ref{proposition:ltld-weaker} along the subsequence $(n_i)_{i\ge 0}$ of box sizes.
A final boxing argument will ensure that Proposition~\ref{proposition:ltld-weaker} holds for all $n\ge 1$.

\subsubsection{Formal setup of the renormalization}\label{sec:formal-setup-biskup} We formalize the boxing scheme with all its parameters.
Fix any small constant $\delta>0$ that appears in the statement of Proposition~\ref{proposition:ltld-weaker}. We define an extra constant, that  shall describe the growth rate of the box sizes $(n_i)_{i\ge 1}$:
\begin{equation} 
\xi_\delta:=\xi_\delta(\zeta)=\frac{\delta/2}{\zeta+(\alpha-1)-\delta/2}.\label{eq:xi-delta}
\end{equation}
The value of $\xi_\delta$ is defined both for $\zeta_{\mathrm{hl}}=(\tau-1)/\alpha-(\tau-2)$ and for $\zeta_{\mathrm{ll}}=2-\alpha$ whenever these are positive.
Since $\alpha>1$, $\xi_\delta>0$ for all $\delta<\zeta$.
 Unless explicitly stated otherwise, we abbreviate in the rest of the section
\begin{equation}
    \gamma=\gamma_\mathrm{hl}=1-1/\alpha,\qquad 
    \zeta=\zeta_\mathrm{hl}=(\tau-1)/\alpha-(\tau-2).
\end{equation}
We recall the PPPs $\CV^\sss{\mathrm{base}}$ and $\CV^\sss{\mathrm{spr}}$ with thinning constant $q<1$ from \cref{def:ksrg-alt}. We will choose any $q$ with $\beta q^d>\beta_c$. We recall the parameters $p, \beta$ from the connectivity function in \eqref{eq:connection-prob-gen}. We define one extra constant depending on $q$:
    \begin{equation}
    \delta':=\tfrac{1}{4}(1-q)(1-2^{-(\tau-1)})\cdot p\beta^\alpha d^{-\alpha d/2}.\label{eq:delta-prime}
    \end{equation}
    Slightly abusing notation, if $\CV$ is a set of vertices and $\CC$ is a component, we write $\CV\cap\CC$ for the set of vertices in $\CV$ that belong to the component $\CC$.
\begin{definition}[renormalization scheme]\label{def:boxing-biskup}
Consider a KSRG under the same setting as in Proposition \ref{proposition:ltld-weaker}.
    Fix $\zeta>0$, and $\delta\in (0, \zeta/2)$, and $\xi_\delta(\zeta) $ as in \eqref{eq:xi-delta}, and a thinning constant $q<1$ with $\beta q^d>\beta_c$. 
    Let $n_0$ be a large constant. We iteratively define for $i\ge 1$ 
        \begin{equation}\label{eq:def-mi-ni}
    m_{i}:=\left\lfloor n_{i-1}^{\xi_\delta/d}\right\rfloor^d, \qquad 
    n_{i}:=m_{i}n_{i-1}.
    \end{equation}    
    Let $\CQ_i(x)$ denote a box centered at $x\in\R^d$ of volume $n_i$. We define the level-$i$ vertex set of $\CQ_i(x)$ to be 
    \begin{equation}
        \widehat\CV_{i}:= \CV^\sss{\mathrm{base}}_{\CQ_i(x)}\big[1, n_0^{2/(\tau-1)}\big)\cup\CV^\sss{\mathrm{spr}}_{\CQ_i(x)}[1, 2n_i^\gamma).\label{eq:def-xi-biskup}
    \end{equation}
    We write $\widehat \CG_{n_i}$ for the induced subgraph on $\widehat\CV_{i}$.
    Define the sequences 
    \begin{align}
    \varepsilon_i&:=
    \begin{dcases}
        1-n_0^{-\delta/4},&\text{if }i=0,\\
        (i+1)^{-2},&\text{if }i\ge 1,
    \end{dcases}\label{eq:eps-i}\\
    \rho_i&:=\prod_{j=0}^i(1-\varepsilon_j)=n_0^{-\delta/4}\prod_{j=2}^{i+1} (1-j^{-2}), 
    \qquad 
    \rho:=\lim_{j\to\infty}\rho_j.\label{eq:rho-i}
    \end{align}
    We call $\CQ_i(x)$ $i$-good (and $i$-bad otherwise) if it satisfies the event (with $\delta'$ in \eqref{eq:delta-prime})
    \begin{equation}\label{eq:i-good-def}
    \begin{aligned}
       \CA_{\mathrm{i\textnormal{-}good}}(x):= \Big\{ \exists&\text{comp.\ $\CC_{i}$ in }\widehat\CG_{n_i}: |\CC_{i}| \ge \rho_i n_i, \\
       &\qquad|\CV^\sss{\mathrm{spr}}_{\CQ_i(x)}[n_i^\gamma, 2n_i^\gamma)\cap\CC_{i}| \ge \delta'\rho_{i} n_i^{\zeta}\Big\}.
    \end{aligned}
    \end{equation}
\end{definition}
The definitions of $(m_i)_{i\ge 0}$ and $(n_i)_{i\ge 0}$ in \eqref{eq:def-mi-ni} ensure that they  increase doubly exponentially in $i$: indeed, ignoring integer parts, $n_i \approx n_{i-1}^{1+\xi_\delta}$ so $n_i\approx n_0^{(1+\xi_\delta)^i}$. This also implies that $2n_{i-1}^\gamma< n_{i}^\gamma$, and so for any level-$(i\!-\!1)$ subbox centered at any $y$, $\widehat\CV_{i-1, y}$ is disjoint of $\CV^\sss{\mathrm{spr}}_{\CQ_i(x)}[n_i^\gamma, 2n_i^\gamma)$. Thus, in $\CQ_{i}(x)$ we sprinkle in \emph{new} vertices with mark in $[n_{i}^\gamma, 2n_{i}^\gamma)$ in level-$i$ compared to previous levels. Summability of $(\varepsilon_i)_{i\ge 0}$ in \eqref{eq:eps-i} implies that $\rho>0$ for any starting value $\eps_0(n_0)$ in (\ref{eq:eps-i}--\ref{eq:rho-i}).

The next claim provides useful bounds on the sequences that will appear in subsequent computations. Actually, $\xi_\delta$ in~\eqref{eq:xi-delta} is chosen precisely such that~\eqref{eq:mi-ni-sub}--\eqref{eq:mi-ni-mark} are satisfied below. The label on the inequalities indicate where we will use the bounds later. 
\begin{claim}[Useful bounds]\label{claim:useful}
   Fix \emph{any} $\zeta\in(0,1)$ and then $\xi_{\delta}(\zeta)$ in \eqref{eq:xi-delta}.
    Assume $\delta\in(0, \zeta/2)$ and consider the quantities in \cref{def:boxing-biskup}. For each constant $C>0$, there exists a constant $n_0^\star>0$ such that for all $n_0\ge n_0^\star$ and all $i\ge 1$, the following hold:
    \begin{align}
    \varepsilon_{i}m_{i}n_{i-1}^{\zeta-\delta}&\overset{\mathrm{(Sub)}}\ge C n_{i}^{\zeta-\delta}, \label{eq:mi-ni-sub}\\
    \rho_{i-1}^2 m_{i}^{1-\alpha}n_{i-1}^\zeta &\overset{\mathrm{(Con)}}\ge Cn_{i}^{\zeta-\delta}, \label{eq:mi-ni-con}\\
      \rho_in_i^\zeta&\overset{\mathrm{(Mark)}}\ge Cn_i^{\zeta-\delta}.\label{eq:mi-ni-mark}
    \end{align}
 Further,
    \begin{align}
    n_i\in \Big[n_0^{(1+\xi_\delta/2)^i}, n_0^{(1+\xi_\delta)^i}\Big], \qquad
    m_{i} \in \Big[\tfrac{1}{2}n_0^{\xi_\delta(1+\xi_\delta/2)^{i-1}}, n_0^{\xi_\delta(1+\xi_\delta)^{i-1}}\Big],\label{eq:mi-ni-doubly-exp}\\
    \varepsilon_{i}m_{i}^{1-\zeta+\delta}\ge C, \qquad
     (n_{i-1}/n_0)^{\zeta-\delta}\ge C\log(\re/\varepsilon_{i}). 
    \label{eq:mi-ni-1}    
    \end{align}
    Moreover, for all $m\le m_{i+1}$ and $n:=n_im$,
    \begin{equation}
        mn_i^{\zeta+\alpha} n^{-\alpha}\ge n^{\zeta-\delta}.\label{eq:useful-boxing}
    \end{equation}
    Finally, when $\zeta=\zeta_{\mathrm{hl}}$ and $\gamma=1-1/\alpha$, then
    \begin{equation}\label{eq:mi-ni-1-gamma}
     n_{i}^\gamma > 2n_{i-1}^\gamma.
    \end{equation}
    \end{claim}
    We postpone the proof to Appendix~\ref{app:useful}.

\subsubsection{Induction base of the renormalization}\label{sec:induction-base-biskup} We now carry out the multi-scale renormalization outlined in Section \ref{sec:outline-biskup-renorm} by induction. We first constitute the induction base. 
\begin{claim}[Induction base]\label{claim:induction-base}Consider a supercritical KSRG under the setting as in Proposition \ref{proposition:ltld-weaker} and consider Definition \ref{def:boxing-biskup}. 
There exists a constant $n_0^\star>0$ such that for all $n_0\ge n_0^\star$ in \cref{def:boxing-biskup}, and all $x\in\R^d$,
\[
\Prob\big(\CQ_0(x)\mbox{ is $0$-bad}\big)\le 1/\re^{6}.
\]
\begin{proof}
By translation invariance, we assume w.l.o.g.\ that $x=0$, and abbreviate $\CQ_0:=\CQ_0(0)$.
We use the the definition of $i$-goodness in \eqref{eq:i-good-def} for $i=0$ in \cref{def:boxing-biskup}, and substitute $\rho_0=n_0^{-\delta/4}$ by \eqref{eq:rho-i}. Then,
\begin{equation}\label{eq:0-good-crit}
\begin{aligned}
\Prob\big(\CQ_0\mbox{ is $0$-good}\big)=
\Prob\big(\exists&\text{ component $\CC$ in $\widehat\CG_{n_0}$: } |\CC|\ge n_0^{1-\delta/4}, \\
&\qquad\hspace{15pt} |\CC\cap\CV^\sss{\mathrm{spr}}[n_0^\gamma, 2n_0^\gamma)|\ge \delta'n_0^{-\delta/4}n_0^{\zeta}\big).
\end{aligned}
\end{equation}
Following \cref{def:ksrg-alt}, we use the alternative KSRG construction with the $q$-thinning of the vertex set. Let $\CG^\sss{\mathrm{base}}_{n_0}$ be the thinned KSRG restricted to $\CQ_0$, which is a subgraph of $\widehat\CG_{n_0}$; let $\CC^\sss{(1)}_{n_0, \mathrm{base}}[1, n_0^{2/(\tau-1)})$ be the largest component in the induced subgraph on the  restricted marks $[1, n_0^{2/(\tau-1)})$ in $\CG^\sss{\mathrm{base}}_{n_0}$; and let $\CC_{0}\supseteq \CC^\sss{(1)}_{n_0, \mathrm{base}}[1, n_0^{2/(\tau-1)})$ be the component in the whole $\widehat\CG_{n_0}$ containing this component. We define the event that this latter component is large enough:
\begin{equation}\label{eq:large-0-giant}
\CA_{\mathrm{large}}^{n_0} := \big\{|\CC^\sss{(1)}_{n_0, \mathrm{base}}[1, n_0^{2/(\tau-1)})|\ge n_0^{1-\delta/4}\big\}.
\end{equation}
Conditionally on $\CA_{\mathrm{large}}^{n_0}$, the component $\CC_{0}$ is large enough to play the role of $\CC$ in \eqref{eq:0-good-crit}. So, by the law of total probability,
\begin{equation}\label{eq:two-factors-0-good}
\Prob\big(\CQ_0\mbox{ is $0$-good}\big)\ge
\Prob\big(\CA_{\mathrm{large}}^{n_0}\big)\cdot \Prob\big(|\CC_{0}\cap\CV^\sss{\mathrm{spr}}_{\CQ_0}[n_0^\gamma, 2n_0^\gamma)|\ge \delta'n_0^{-\delta/4}n_0^{\zeta}\mid \CA_{\mathrm{large}}^{n_0}\big).
\end{equation}
We show that both factors on the right-hand side tend to one as $n_0\to\infty$. 

\emph{(Con)-step of level-$0$.}
We first consider $\Prob(\CA_{\mathrm{large}}^{n_0})$ with the event defined in \eqref{eq:large-0-giant}. If the thinned PPP  $\CV^\sss{\mathrm{base}}_{n_0}$ contains no vertices of mark at least $n_0^{2/(\tau-1)}$, then $\CC^\sss{(1)}_{n_0, \mathrm{base}}[1, n_0^{2/(\tau-1)})$ and $\CC^\sss{(1)}_{n_0,\mathrm{base}}$ coincide. Here $\CC^\sss{(1)}_{n_0,\mathrm{base}}$ is the largest component of the thinned PPP in $\CQ_0$. By a union bound over the complementary events, it follows that 
\begin{align}
\Prob(\CA_{\mathrm{large}}^{n_0}) &\ge \Prob\big( |\CV^\sss{\mathrm{base}}_{\CQ_0}[n_0^{2/(\tau-1)}, \infty)|=0, \ |\CC^\sss{(1)}_{n_0, \mathrm{base}}|\ge n_0^{1-\delta/4}, \big)\nonumber\\ &\ge 1-\Prob\big(|\CC^\sss{(1)}_{n_0, \mathrm{base}}|\le n_0^{1-\delta/4}\big) - \Prob\big(|\CV^\sss{\mathrm{base}}_{\CQ_0}[n_0^{2/(\tau-1)}, \infty)|>0\big).\label{eq:base-1}
\end{align}
By \cref{obs:sprinkle-2}, $\CC^\sss{(1)}_{n_0, \mathrm{base}}$ has the same distribution as the largest component of a supercritical KSRG with edge density  $\beta q^d>\beta_c$ inside the box $\Lambda_{qn_0}$, with $q\in(0,1)$ from \cref{def:ksrg-alt}. Hence, by \cref{lemma:ltld-weakest}, the first probability on the right-hand side in~\eqref{eq:base-1} tends to zero as $n_0q \to \infty$ (we absorb the factor $q$ into the exponent $1-\delta/4$ by choosing $\eps<\delta/4$ in \cref{lemma:ltld-weakest}). The intensity of $\CV^\sss{\mathrm{base}}$ is $q\mathrm{Leb}\times F_W(\rd w)=q\mathrm{Leb}\times (\tau-1)w^{-\tau}\rd w$ and the box has volume $n_0$. So, the expected number of vertices in $\CV^\sss{\mathrm{base}}_{\CQ_0}[n_0^{2/(\tau-1)},\infty)$ is $\Theta(n_0^1n_0^{-2})$, which tends to zero. Hence, the second error term in \eqref{eq:base-1} also tends to zero.
This shows that $\Prob(\CA_{\mathrm{large}}^{n_0})\to 1$ in \eqref{eq:two-factors-0-good}. 

\emph{(Mark) step of level-$0$.}
We now carry out the \emph{(Mark)}-step of level $0$: We inspect the second factor in \eqref{eq:two-factors-0-good}.
The vertices $\CV^\sss{\mathrm{spr}}_{\CQ_0}[n_0^\gamma, 2n_0^\gamma)$ form a PPP with intensity $(1-q)\mathrm{Leb}\times (\tau-1)w^{-\tau}\rd w$ in a box with volume $n_0$. Its expected size is  
    \begin{equation}\label{eq:poisson-mean-0}
    \E\big[|\CV^\sss{\mathrm{spr}}_{\CQ_0}[n_0^\gamma, 2n_0^\gamma)|\big]=(1-q) (1-2^{-(\tau-1)})n_0^{1-\gamma(\tau-1)}=(1-q) (1-2^{-(\tau-1)})n_0^{\zeta}.
    \end{equation}
Further, $\CV^\sss{\mathrm{spr}}$ is independent of $\CV^{\sss{\mathrm{base}}}$ and thus also of the graph $\CG^\sss{\mathrm{base}}_{n_0}[1, n_0^{2/(\tau-1)})$, the graph induced by vertices whose marks are in the interval $[1, n_0^{2/(\tau-1)}$) and the event $\CA_{\mathrm{large}}^{n_0}$.
 On this event, there are at least $n_0^{1-\delta/4}$ vertices in the above mentioned component, and each vertex there has mark at least one. We use that the maximal distance between two vertices in $\CQ_0$ is $\sqrt{d}n_0^{1/d}$.  Given $\CA_{\mathrm{large}}^{n_0}$, for each vertex $v$ in $\CV^\sss{\mathrm{spr}}_{\CQ_0}[n_0^\gamma, 2n_0^\gamma)$,
     \[
    \begin{aligned}
    \Prob\Big( v \sim \CC_{n_0, \sss{\mathrm{base}}}^\sss{(1)}[1,n_0^{2/(\tau-1)}) \mid \CA_{\mathrm{large}}^{n_0}\Big) &\ge
    1-\Big(1-p\big(1\wedge \frac{\beta n_0^{\gamma}}{d^{d/2} n_0}\big)^\alpha\Big)^{n_0^{1-\delta/4}} \\
    &\ge 
    1-\exp\big(-p\beta^\alpha d^{-\alpha d/2} n_0^{-\delta/4} n_0^{\gamma\alpha-(\alpha-1)}\big)\\
    &=1-\exp\big(-p\beta^\alpha d^{-\alpha d/2}n_0^{-\delta/4}\big).
    \end{aligned}
    \]
 We used that $\gamma<1$, so $n_0^{\gamma-1}=o(1)$ and the minimum is attained at the second term for $n_0$ sufficiently large to obtain the first inequality. Then we used that $\gamma=(\alpha-1)/\alpha$ to obtain the second row.
    For $n_0$ sufficiently large, we use the bound $1-\re^{-x}\ge x/2$ for $x<1$. Therefore, the right-hand side is at least 
    $p\beta^\alpha d^{-\alpha d/2}n_0^{-\delta/4}/2$.   This is the connectivity estimate for a single vertex. 
    Since each vertex in $\CV^\sss{\mathrm{spr}}_{\CQ_0}[n_0^\gamma, 2n_0^\gamma)$ connects independently to $\CC_{n_0, \mathrm{base}}^\sss{(1)}[1,n_0^{2/(\tau-1)})$, the number of connecting vertices is Poisson with mean at least the mean in \eqref{eq:poisson-mean-0} multiplied with this connectivity estimate. So, using $\rho_0=n_0^{-\delta/4}$ in \eqref{eq:rho-i},
        \[
    \begin{aligned}
    \Prob\big(&|\CV^\sss{\mathrm{spr}}_{\CQ_0}[n_0^\gamma, 2n_0^\gamma)\cap\CC_{0}| < \delta'\rho_{0} n_0^{\zeta} \mid \CA_{\mathrm{large}}^{n_0}\big)
    \\
    &\le \Prob\big(\mathrm{Poi}\big((1-q)(1-2^{-(\tau-1)})\cdot \tfrac{1}2p\beta^\alpha d^{-\alpha d/2}\cdot n_0^{\zeta-\delta/4}\big)< \delta' n_0^{\zeta-\delta/4}\big).
    \end{aligned}
    \]
    By the choice of $\delta'$ in~\eqref{eq:delta-prime} and concentration of Poisson variables in \cref{lemma:poisson-1}, the right-hand side tends to zero as $n_0\to\infty$. This proves that the second factor in \eqref{eq:two-factors-0-good} tends to $1$ as well. This finishes the proof of \cref{claim:induction-base}.
\end{proof}
\end{claim}

\subsubsection{Advancing the induction: the iterative step.} \cref{claim:induction-base} shows that level-$0$ boxes are good with constant probability. In the next claim we advance the induction by bounding the probabilities of the steps \emph{(Sub), (Con), (Mark)} in Section \ref{sec:outline-biskup-renorm} for level-$i$ boxes.
\begin{claim}[Advancing the induction]\label{claim:induction-advance}
Consider a supercritical KSRG on a PPP with parameters such that $\zeta_\mathrm{hl}>0$.  Assume $\delta\in(0,\zeta/2)$. 
There exists a constant $n_0^\star>0$ such that for all $n_0\ge n_0^\star$ in \cref{def:boxing-biskup}, for all $i\ge 0$, 
    \begin{equation}
    \Prob\big(\CQ_i(x)\mbox{ is $i$-bad}\big) \le \exp\big(-6(n_i/n_0)^{\zeta-\delta}\big):=\mathrm{err}_i.\label{eq:induction-advance}
    \end{equation}
    \begin{proof}
    When $i=0$ the statement follows from \cref{claim:induction-base}.     
    Let $i\ge 1$ and assume that~\eqref{eq:induction-advance} holds for level-$(i-1)$ boxes. We now follow the steps \emph{(Sub), (Con), (Mark)} in Section \ref{sec:outline-biskup-renorm} and bound their error probabilities. Like in \cref{claim:induction-base}, we assume w.l.o.g.\ that $x=0$, and write $\CQ_i:=\CQ_i(0)$.

    \emph{Enough subboxes are good.}
    We start with the \emph{(Sub)} step.
    Since $n_i=m_in_{i-1}$ by \eqref{eq:def-mi-ni}, we can partition $\CQ_{i}(x)$ into $m_{i}$ subboxes of volume $n_{i-1}$. Let $(x_i)_{i\le m_i}$ denote the centers of these boxes. Define the event 
    \begin{equation}
    \CA_\mathrm{sub}:=\Big\{\sum_{j\in [m_i]}\ind{\CQ_{i-1}(x_j)\text{ is $(i-1)$-bad}} < \varepsilon_{i}m_{i}\Big\}.\label{eq:sub-event}
    \end{equation}
    In words, $\CQ_{i}(x)$ contains at least $(1-\eps_i)m_i$ many $(i-1)$-good subboxes. 
    Using that disjoint boxes are i.i.d., the induction hypothesis yields that the number of bad boxes is stochastically dominated by a $\mathrm{Bin}(m_i, \mathrm{err}_{i-1})$ random variable, i.e.,
    \begin{align*}
        \Prob\big(\neg \CA_\mathrm{sub})\le \Prob\Big(\mathrm{Bin}\big(m_{i}, \mathrm{err}_{i-1})\big) \ge \varepsilon_{i}m_i\Big)
        \le 
        \sum_{k= \lceil\varepsilon_{i}m_{i}\rceil}^{m_{i}}\binom{m_{i}}k \mathrm{err}_{i-1}^k.
    \end{align*}
    We use that $\binom{m_i}k\le (\re m_i/k)^k$, so the binomial coefficient is at most  $(\re/\varepsilon_i)^k$ for all $k\ge \varepsilon_i m_i$. 
    Hence, using the definition of $\mathrm{err}_{i-1}$ from  \eqref{eq:induction-advance},
    \begin{equation*}
        \Prob\big(\neg \CA_\mathrm{sub})\le  \sum_{k= \lceil\varepsilon_{i}m_{i}\rceil}^{\infty}\exp\big(-k(6(n_{i-1}/n_0)^{\zeta-\delta}-\log(\re/\varepsilon_{i}))\big). 
        \end{equation*}
        We use that $\varepsilon_i=(i+1)^{-2}$ by definition in~\eqref{eq:eps-i}: when $i=1$ we compute $6-\log(\re/\varepsilon_1)=6-\log(4\re)=5-\log(4)>3$; for $i\ge 2$, by the second inequality in \eqref{eq:mi-ni-1} for $C=4/6$, we obtain that $6(n_{i-1}/n_0)^{\zeta-\delta}-\log(\re/\varepsilon_{i})>3(n_{i-1}/n_0)^{\zeta-\delta}$ when $n_0$ is sufficiently large. 
        We obtain a geometric series with base $\exp(-3(n_{i-1}/n_0)^{\zeta-\delta})$, which is at most $1/2$ for all $i\ge 2$ by choosing $n_0$ sufficiently large. Evaluating the geometric sum and then using the bound (Sub) in \eqref{eq:mi-ni-sub} yields
      \begin{equation}\label{eq:not-a-sub}
                \begin{aligned}
        \Prob\big(\neg \CA_\mathrm{sub})&
        \le 
        \sum_{k= \lceil\varepsilon_{i}m_{i}\rceil}^{\infty}\exp\big(-3(n_{i-1}/n_0)^{\zeta-\delta}\big)^k    \le \exp\big(-3\varepsilon_{i}m_{i}(n_{i-1}/n_0)^{\zeta-\delta}\big)/2\\
        &\le \tfrac13\exp\big(-6(n_{i}/n_0)^{\zeta-\delta}\big).
        \end{aligned}
        \end{equation} 
This finishes the \emph{(Sub)}-step of level-$i$. 

\emph{Connectivity: the local giant is big enough.}
We now move on to the \emph{(Con)}-step in Section \ref{sec:outline-biskup-renorm}. 
In level $i-1$, we used the vertex sets $\CV^\sss{\mathrm{base}}_{\CQ_{i-1}(x_j)}\big[1, n_0^{2/(\tau-1)}\big)\cup\CV^\sss{\mathrm{spr}}_{\CQ_{i-1}(x_j)}[1, 2n_{i-1}^\gamma)$ in each subbox $j\in[m_i]$, see~\eqref{eq:def-xi-biskup}. We denote the level-$(i-1)$ graphs on these vertex sets by $\widehat\CG_{n_{i-1}, x_j}$.
We now reveal the realization of $(\widehat\CG_{n_{i-1}, x_j})_{j\le m_i}$ for all $j\le m_i$ in the level-$(i-1)$ subboxes of $\CQ_{i}(x)$, and assume that the realization satisfies the event $\CA_\mathrm{sub}$ for $\CQ_i(x)$.
We write 
\begin{equation}\label{eq:pre-graph}
\begin{aligned}
\mathrm{pre\textnormal{-}}\widehat\CG_{n_{i}}&:=\cup_{j\le m_i}\widehat\CG_{n_{i-1}, x_j} \quad \mbox{on}\quad \\
\CV_{\mathrm{pre}}&:=\cup_{j\le m_i}\CV^\sss{\mathrm{base}}_{\CQ_{i-1}(x_j)}\big[1, n_0^{2/(\tau-1)}\big)\cup\CV^\sss{\mathrm{spr}}_{\CQ_{i-1}(x_j)}[1, 2n_{i-1}^\gamma).
\end{aligned}
\end{equation}
Since the realization satisfies $\CA_{\mathrm{sub}}$, at least $(1-\eps_i)m_i$ many subboxes are $(i-1)$-good.
    We re-label boxes so that the boxes with indices $j\le (1-\eps_i)m_i$ are all $(i-1)$-good. In $\mathrm{pre\textnormal{-}}\widehat\CG_{n_{i}}$ we did not reveal yet whether vertices in \emph{different} level-$(i-1)$ subboxes are connected to each other or not. The graph $\widehat\CG_{n_{i}}$ which contains  $\mathrm{pre\textnormal{-}}\widehat\CG_{n_{i}} $
         additionally contains these edges between subboxes. 
    
   \emph{The Erd{\H o}s-R\'enyi type auxiliary graph.} The definition of $(i-1)$-goodness in \eqref{eq:i-good-def} gives us components $\CC_{i-1, x_j}$ in each good subbox with size at least $\rho_{i-1}n_{i-1}$. Consider the auxiliary graph $\CH=(\CV_\CH, \CE_\CH)$ in which the vertex $j\in\CV_\CH:=[\lceil (1-\varepsilon_i)m_i\rceil]$ corresponds to $\CC_{i-1, x_j}$, and two vertices $j,k\in\CV_\CH$ are connected by an edge if the components $\CC_{i-1, x_j}$ and $\CC_{i-1, x_{k}}$ are connected by an edge in $\widehat\CG_{n_i}$. If the auxiliary graph $\CH$ is connected, then $\CQ_i(x)$ contains a component that has size at least $(1-\eps_i)m_i \cdot \rho_{i-1}n_{i-1}$. Since $m_{i}n_{i-1}=n_i$ and $\rho_{i-1}(1-\eps_i)=\rho_i$ by~\eqref{eq:def-mi-ni} and~\eqref{eq:rho-i}, respectively, this size is at least $\rho_{i}n_i$.
    Therefore, it satisfies the first criterion for being $i$-good in \eqref{eq:i-good-def}. 
    Hence, for all realizations $\mathrm{pre\textnormal{-}}\widehat\CG_{n_{i}}$ satisfying $\CA_{\mathrm{sub}}$,
\begin{equation}\label{eq:first-good-crit}
    \Prob(\exists \mbox{ component } \CC_{i}: |\CC_{i}|\ge \rho_i n_i \mid \mathrm{pre\textnormal{-}}\widehat\CG_{n_{i}}) \ge \Prob(\CH \mbox{ connected} \mid \mathrm{pre\textnormal{-}}\widehat\CG_{n_{i}} ). \end{equation}
We now prove that $\CH$ stochastically dominates an Erd\H{o}s-R\'enyi graph. The criteria for being $(i-1)$-good in \eqref{eq:i-good-def} guarantee that the level-$(i-1)$ good subboxes contain a component $\CC_{i-1, x_j}$ with size at least $\rho_{i-1}n_{i-1}$, \emph{and} this component contains at least $\delta'\rho_{i-1}n_{i-1}^{\zeta}$ vertices of mark in $[n_{i-1}^{\gamma}, 2n_{i-1}^\gamma)$.
   We bound the probability that there is an edge between two such components in different good subboxes by checking whether any of the potential edges between the $\delta'\rho_{i-1}n_{i-1}^{\zeta}$ many level-$(i-1)$ high mark vertices in one box and the $\rho_{i-1}n_{i-1}$ many vertices in the component of mark at least $1$ in the other box is present. 
    Any two vertices in $Q_i(x)$ are at distance at most $\sqrt{d}n_i^{1/d}$, so using the connection probability in \eqref{eq:connection-prob-gen} we compute for any \emph{good} subboxes with index $j,k$ that
    \begin{align}
    \Prob\big(\CC_{i-1, x_j}\nsim_{\widehat\CG_{n_i}}\CC_{i-1, x_{k}}\mid \text{pre-}\widehat\CG_{n_{i}}\big)
    &\le \Big(1-p\big(1\wedge \tfrac{\beta n_{i-1}^{\gamma}}{\sqrt{d}^d n_i}\big)^\alpha\Big)^{\delta'\rho_{i-1}^2n_{i-1}^{1+\zeta}} \nonumber\\
    &\le\exp\big(-p\delta'\beta^{\alpha}d^{-\alpha d/2}\rho_{i-1}^2n_{i-1}^{1+\zeta+\gamma\alpha} n_i^{-\alpha}\big)\nonumber \\
    &=
    \exp\big(-p\delta'\beta^{\alpha}d^{-\alpha d/2}\rho_{i-1}^2n_{i-1}^{\zeta+\alpha} n_i^{-\alpha}\big)
    =: 1-q_i,\label{eq:biskup-qi}
    \end{align}
    where we used that $\gamma=1-1/\alpha$ to get the last row, and also to see that the minimum is at the second term in the first row, since $n_{i-1}^\gamma/n_i$ can be made arbitrarily small by choosing $n_0$ sufficiently large, see \eqref{eq:def-mi-ni}. The connectivity between different local large components is conditionally independent given the realization of $\mathrm{pre}\textnormal{-}\widehat\CG_{n_{i}}$, so the auxiliary graph $\CH$ dominates an Erd\H{o}s-R\'enyi graph $G(\lceil(1-\varepsilon_i)m_i\rceil, q_i)$. 
    The reader may verify that $m_iq_i\gg \log m_i$ for all $i\ge 1$ if $n_0$ is sufficiently large, so by \cref{claim:er-conn} on the connectivity of $G$, 
    \begin{align}
    \ind{\CA_\mathrm{sub}}
    \Prob\big(\CH\text{ not conn.}\mid \text{pre-}\widehat\CG_{n_i}\big)
   & \le 
    3(\re m_i)^2(1-q_i)^{(1-\varepsilon_i)m_i/2 } \nonumber \\
    &\le 
    3(\re m_i)^2\exp\big(-\tfrac{3p\delta'\beta^{\alpha}}{8d^{\alpha d/2}}\rho_{i-1}^2n_{i-1}^{\zeta+\alpha} n_i^{-\alpha}m_i\big),\label{eq:repeat-at-n-con}
    \end{align}
    where we substituted $1-q_i$ from~\eqref{eq:biskup-qi}, and used  that $1-\varepsilon_i\ge 3/4$ for all $i\ge 1$. 
    We substitute $n_i=m_in_{i-1}$ to obtain the first row below, and then use the inequality (Con) in~\eqref{eq:mi-ni-con} to see that the $i$-dependent factors together can be bounded from below by $Cn_i^{\zeta-\delta}$ for any constant $C>0$ to get the second row. Here, we can choose $n_0$ sufficiently large so that $C$ cancels out the constant ($\delta', \beta, p, d$-dependent) prefactors:
    \begin{align}
    \ind{\CA_\mathrm{sub}}
    \Prob\big(\CH\text{ not conn.}\mid \text{pre-}\widehat\CG_{n_i}\big)
    &\le 3(\re m_i)^2\exp\big(-\tfrac{3p\delta'\beta^{\alpha}}{8d^{\alpha d/2}}\rho_{i-1}^2n_{i-1}^{\zeta} m_i^{1-\alpha}\big).\nonumber\\
    &\le 3(\re m_i)^2\exp\big(-7n_i^{\zeta-\delta}\big) \le \tfrac{1}{3}\exp\big(-6n_i^{\zeta-\delta}\big).\label{eq:ind-pr-con}
    \end{align}
       Returning to \eqref{eq:first-good-crit}, this shows that the first condition of being level-$i$ good in \eqref{eq:i-good-def}, i.e., containing a connected component with size at least $\rho_in_i$, holds with probability at least $1-\exp(-6n_i^{\zeta-\delta})/3$, given $\CA_{\mathrm{sub}}$. This finishes the \emph{(Con)}-step. We summarize it as follows:
   since $\Prob\big(\neg \CA_{\mathrm{sub}})\le \exp\big(-6(n_i/n_0)^{\zeta-\delta}\big)/3$ by \eqref{eq:not-a-sub},
       \begin{align}
\Prob\big(\CQ_i(x)\text{ is $i$-bad}\big)&\le 
     \Prob\big(\neg \CA_\mathrm{sub}\big)
     +
\E\big[\ind{\CA_\mathrm{sub}}
\Prob\big(\CQ_i(x)\text{ is $i$-bad}\mid \text{pre-}\widehat\CG_{n_i}\big)\big] \nonumber\\
     &
     \begin{aligned}
     &\le\Prob\big(\neg \CA_\mathrm{sub}\big)
     +
     \E\big[\ind{\CA_\mathrm{sub}}
     \Prob\big(\CH\text{ not conn.}\mid \text{pre-}\widehat\CG_{n_i}\big)\big] \\
     &\quad+\E\big[\ind{\CA_\mathrm{sub}}
     \Prob\big(\CQ_i(x)\text{ is $i$-bad}\mid \text{pre-}\widehat\CG_{n_i}, \CH\text{ connected}\big)\big]
\end{aligned}\label{eq:biskup-giant-pr}\\
&\le \tfrac{2}{3}\exp\big(-6(n_i/n_0)^{\zeta-\delta}\big) + \E\big[\ind{\CA_\mathrm{sub}}
     \Prob\big(\CQ_i(x)\text{ is $i$-bad}\mid \text{pre-}\widehat\CG_{n_i}, \CH\text{ conn.}\big)\big],
\nonumber\label{eq:biskup-giant-pr}
    \end{align}
   where the error of the \emph{(Con)}-step is the first term in the last row, and the second term still needs to be bounded during  the \emph{(Mark)} step. 
    
    \emph{Enough high-mark vertices in level-$i$.}
    Finally, we treat the \emph{(Mark)}-step in Section \ref{sec:outline-biskup-renorm}. That is, we bound the conditional probability in the last row above. This corresponds to the second criterion in \eqref{eq:i-good-def}, i.e., the requirement that $\CC_{i}$ also contains enough level-$i$ high mark vertices. 
    On the event $\CA_\mathrm{sub}\cap\{\CH\text{ is connected}\}$, 
    we set $\CC_{i}$ to be the connected component in $\widehat\CG_{n_i}$ that contains $\cup_{j\in\CV_\CH}\CC_{i-1, x_j}=:\mathrm{pre\textnormal{-}}\CC_{i}$. By the argument leading to \eqref{eq:first-good-crit},  $\mathrm{pre\textnormal{-}}\CC_{i}$ has already size at least $\rho_in_i$.
    Hence, by the criterion in \eqref{eq:i-good-def}, we need to bound 
    \begin{equation}
    \begin{aligned}
\ind{\CA_\mathrm{sub}}\Prob\big(|\CV^\sss{\mathrm{spr}}_{\CQ_i(x)}[n_i^\gamma, 2n_i^\gamma)\cap\CC_{i}| < \delta'\rho_{i} n_i^{1-\gamma(\tau-1)} &\mid \text{pre-}\widehat\CG_{n_i}, \CH\text{ connected}\big).
    \end{aligned}\label{eq:mark-to-show} 
    \end{equation}
    We argue similarly as for $i=0$ around and below \eqref{eq:poisson-mean-0}. Since $\mathrm{pre\textnormal{-}}\CC_{i}\subseteq\CV_{\mathrm{pre}}$ in \eqref{eq:pre-graph}, and $2n_{i-1}^\gamma<n_i^\gamma$ by \eqref{eq:mi-ni-1-gamma}, the already revealed  vertex set $\CV_{\mathrm{pre}}$ is disjoint and thus independent of $\CV^\sss{\mathrm{spr}}_{\CQ_i(x)}[n_i^\gamma, 2n_i^\gamma)$. 
     Hence, the PPP $\CV^\sss{\mathrm{spr}}_{\CQ_i(x)}[n_i^\gamma, 2n_i^\gamma)$ itself is independent of the conditioning in \eqref{eq:mark-to-show}. Using the intensity in~\eqref{eq:poisson-intensity}, its expected size is 
\begin{equation}\label{eq:poisson-mean-i}
\E\big[|\CV^\sss{\mathrm{spr}}_{\CQ_i(x)}[n_i^\gamma, 2n_i^\gamma)|\big]=(1-q) (1-2^{-(\tau-1)})n_i^{1-\gamma(\tau-1)}=(1-q) (1-2^{-(\tau-1)})n_i^{\zeta},
    \end{equation} with $q\in(0,1)$ from the \cref{def:ksrg-alt}.  There are at least $\rho_i n_i$ vertices in  $\mathrm{pre\textnormal{-}}\CC_{i}$, each with mark at least one, and each pair of vertices inside $\CQ_{i} $ is within distance $\sqrt{d}n_i^{1/d}$. So, we get using the connection probability in \eqref{eq:connection-prob-gen} that each individual vertex in $\CV^\sss{\mathrm{spr}}_{\CQ_i(x)}[n_i^\gamma, 2n_i^\gamma)$ connects by an edge to  $\mathrm{pre\textnormal{-}}\CC_{i}\subseteq \CC_{i}$ with probability at least 
    \[
    \begin{aligned}
    1-\Big(1-p\big(1\wedge  \tfrac{\beta n_i^\gamma}{\sqrt{d}^d n_i}\big)^\alpha\Big)^{\rho_i n_i} &\ge 
    1-\exp\big(-p\beta^\alpha d^{-\alpha d/2}\rho_in_i^{\gamma\alpha-(\alpha-1)}\big)\\
    &=1-\exp\big(-p\beta^\alpha d^{-\alpha d/2}\rho_i\big).
    \end{aligned}
    \]
    To obtain the second expression, we used that the minimum is at the second term for $n_0$ sufficiently large, and $\gamma=(\alpha-1)/\alpha$ to get the second row. 
   Since $\rho_i=\Theta(n_0^{-\delta/4})$ as $n_0$ tends to infinity by \eqref{eq:rho-i}, we may assume that $n_0$ is so large so that the bound $1-e^{-x}\ge x/2$ applies (for all $x\le 1$), and the right-hand side is at least 
    $p\beta^\alpha d^{-\alpha d/2}\rho_i/2$.    
    Since each vertex in $\CV^\sss{\mathrm{spr}}_{\CQ_i(x)}[n_i^\gamma, 2n_i^\gamma)$ connects independently to $\mathrm{pre\textnormal{-}}\CC_{i}$, the number of connecting vertices forms a PPP with mean at least the product of \eqref{eq:poisson-mean-i} multiplied with this bound. So,
        \[
    \begin{aligned}
    \ind{\CA_\mathrm{sub}}\Prob\big(&|\CV^\sss{\mathrm{spr}}_{\CQ_i(x)}[n_i^\gamma, 2n_i^\gamma)\cap\CC_{i}| < \delta'\rho_{i} n_i^{\zeta} \mid \text{pre-}\widehat\CG_{n_i}, \CH\text{ conn.}\big)
    \\
    &\le \Prob\big(\mathrm{Poi}\big((1-q)(1-2^{-(\tau-1)})\cdot \tfrac{1}2p\beta^\alpha d^{-\alpha d/2}\cdot\rho_i n_i^{\zeta}\big)< \delta'\rho_{i} n_i^{\zeta}\big).
    \end{aligned}
    \]
    The value of $\delta'$ in \eqref{eq:delta-prime} is chosen exactly so that $2\delta'$ is the prefactor of $\rho_in_i^\zeta$ in the mean of the Poisson variable. Hence, 
by the concentration of Poisson variables in \cref{lemma:poisson-1}, there exists $c=c(\delta')>0$ such that 
    \[
\ind{\CA_\mathrm{sub}}\Prob\big(|\CV^\sss{\mathrm{spr}}_{\CQ_i(x)}[n_i^\gamma, 2n_i^\gamma)\cap\CC_{i}| 
< \delta'\rho_{i} n_i^{\zeta} \mid \text{pre-}\widehat\CG_{n_i}, \CH\text{ conn.}\big)
    < \exp\big(-c\rho_in_i^{\zeta}\big),
    \]
     Finally, by the bound (Mark) in~\eqref{eq:mi-ni-mark}, the right-hand side is at most $\exp\big(-6n_i^{\zeta-\delta}\big)/3$ for $n_0$ sufficiently large. Combining this bound with  \eqref{eq:biskup-giant-pr} finishes the proof of \cref{claim:induction-advance}.
     \end{proof}    
\end{claim}

If the box $\CQ_{i}(x)$ is $i$-good, then it contains a giant component on the restricted marks $[1, 2n_i^{\gamma})$. Thus, Claims \ref{claim:induction-base} and \ref{claim:induction-advance} together prove Proposition~\ref{proposition:ltld-weaker} \emph{along the subsequence} $n_i$, and when $\zeta=\zeta_{\mathrm{hl}}$. We now prove  Proposition~\ref{proposition:ltld-weaker} for arbitrary  $n$.

\begin{proof}[Proof of Proposition~\ref{proposition:ltld-weaker} \label{proof:prop-ltld-weaker}
when $\zeta_\mathrm{hl}>0$]
    W.l.o.g.\ we may assume that $\delta\in(0,\zeta/2\wedge 1-\zeta)$ in the statement of the lemma. Let $n_0$ be a constant such that Claims~\ref{claim:induction-base}--\ref{claim:induction-advance} hold for all $i\ge 0$.
    Fix now the sequence $n_i$ in~\eqref{eq:def-mi-ni}. 
    We may assume that $n>n_0$, since the constant $A$ in the statement of Proposition~\ref{proposition:ltld-weaker} can be increased to account for values $n\le n_0$.  We
    define
    \begin{equation}\label{eq:boxing-tilde-n}
    \begin{aligned}
    i_0&:=\min\{i: 2n_i^\gamma > n_0^{2/(\tau-1)}\}, &
    i_\ast&:=\max\{i: n_i<n\}, \\
    m&:=\lfloor (n/n_{i_\ast})^{1/d}\rfloor^d,
    &\tilde n&:=n_{i_\ast}m.
    \end{aligned}
    \end{equation}
    By increasing the constant $A$ in the statement of Proposition~\ref{proposition:ltld-weaker}, we may assume w.l.o.g.\ that $i_\ast>i_0$.
   Then, there exist $\nu$ such that
   \begin{equation}\label{def:ratio-nu}
   m\ge \nu (n/n_{i_\ast}).
    \end{equation}
    By the definition of $m$ and $\tilde n$, $\Lambda_{\tilde n}\subseteq \Lambda_{n}$ can be partitioned into exactly $m\ge 1$ disjoint boxes of volume $n_{i_\ast}$. We denote the boxes of this partition by $\CQ_{i_\ast}(x_1),\ldots,\CQ_{i_\ast}(x_m)$. From here on we argue similarly as in the proof of \cref{claim:induction-advance}: we use \emph{(Sub)} and \emph{Con}-type steps. 
       We will show that at least half of the boxes are $i_\ast$-good, and afterwards connect the large components in the first $\lceil m/2\rceil$ many $i_\ast$-good boxes via a domination to an Erd\H{o}s-R\'enyi random graph. 

 \emph{Enough subboxes are good.}
    Define 
    \[ 
    \CA_\mathrm{sub}(n):=
    \bigg\{\sum_{j\in [m]}\ind{\CQ_{i_\ast}(x_j)\text{ is $i_\ast$-bad}} < m/2\bigg\}.
    \]
   We follow the calculations from~\eqref{eq:sub-event} until~\eqref{eq:not-a-sub} and replace $\varepsilon_i$ by $1/2$, $m_i$ by $m$ and $n_{i-1}$ by $n_{i_\ast}$ there. The bound ${m \choose k}\le (m/\re k)^k$ works for all $m\ge1$ and $k\le m$. We also use there the second inequality in \eqref{eq:mi-ni-1} for $C=4/6$, which also holds for $n_{i_\ast}$. So all bounds except the last row in \eqref{eq:not-a-sub} remain valid. Using \eqref{def:ratio-nu} we obtain the bound
    \begin{equation}\label{eq:not-a-sub-mmm}
                \begin{aligned}
        \Prob\big(\neg \CA_\mathrm{sub}(n))&
        \le 
        \sum_{k= \lceil m/2\rceil}^{\infty}\exp\big(-3(n_{i_\ast}/n_0)^{\zeta-\delta}\big)^k    \le \exp\big(-3(m/2)(n_{i_\ast}/n_0)^{\zeta-\delta}\big)/2\\
        &\le  
        \exp\Big(-(\nu/2n_0^{\zeta-\delta })\cdot nn_{i_\ast}^{\zeta-\delta-1}\Big).
        \end{aligned}
        \end{equation} 
  As $n>n_{i_\ast}$ and $\zeta<1$, it follows that $n_{i_\ast}^{\zeta-\delta-1}>n^{\zeta-\delta-1}$, which, when rearranged yields that $n n_{i_\ast}^{\zeta-\delta-1} > n^{\zeta-\delta}$. Hence, 
    \begin{equation}
        \Prob\big(\neg\CA_\mathrm{sub}(n)\big) 
        \le 
        \exp\Big(-(\nu/2n_0^{\zeta-\delta }) \cdot n^{\zeta-\delta}\Big).\label{eq:boxing-sub}
    \end{equation}
    This finishes the \emph{(Sub)}-step for general $n$. 
    
    \emph{Connecting local-giants of subboxes.} Now we move on to the \emph{(Con)}-step in Section \ref{sec:outline-biskup-renorm}. We follow the argument around \eqref{eq:pre-graph}, and define the pre-vertex set $\CV_{\mathrm{pre}}$ as there, with $n_i$ replaced by $n$, $i-1$ replaced by $i_\ast$, and $m_i$ replaced by $m$ there: 
\begin{equation}\label{eq:pre-graph-n}
\begin{aligned}
\mathrm{pre\textnormal{-}}\widehat\CG_{\tilde n}&:=\cup_{j\le m}\widehat\CG_{n_{i_\ast}, x_j} \quad \mbox{on}\quad \\
\CV_{\mathrm{pre}}(\tilde n)&:=\cup_{j\le m}\CV^\sss{\mathrm{base}}_{\CQ_{i_\ast}(x_j)}\big[1, n_0^{2/(\tau-1)}\big)\cup\CV^\sss{\mathrm{spr}}_{\CQ_{i_\ast}(x_j)}[1, 2n_{i_\ast}^\gamma).
\end{aligned}
\end{equation}
We reveal the realization of $\mathrm{pre\textnormal{-}}\widehat\CG_{n}$, satisfying the $\CA_\mathrm{sub}(n)$. Following the argument above \eqref{eq:first-good-crit}, we construct the auxiliary graph $\CH_n$ on vertices representing the first $[\lceil m/2\rceil]$ many $i_\ast$-good subboxes. Each vertex $j$ represents a large component $\CC_{i_\ast, x_j}$ guaranteed by the first condition in $i$-goodness in \eqref{eq:i-good-def}. Two vertices $j,k$ are connected by an edge in $\CH_n$  if there is an edge between $\CC_{i_\ast, x_j}$ and $\CC_{i_\ast, x_{k}}$. We follow the computations between~\eqref{eq:first-good-crit} and~\eqref{eq:repeat-at-n-con}, and replace there $(1-\eps_i)m_i$ with $m/2$, replace $n_{i-1}, \rho_{i-1}$ with $n_{i_\ast}, \rho_{i_\ast}$, and replace $n_i$ by $\tilde n$. The latter holds because $m n_{i_\ast}=\tilde n$ by \eqref{eq:boxing-tilde-n}. Then, for some constant $c>0$, \eqref{eq:repeat-at-n-con} turns into
\begin{align}
    \ind{\CA_\mathrm{sub}(n)}
    \Prob\big(\CH_n\text{ not conn.}\mid \text{pre-}\widehat\CG_{\tilde n}\big)
    &\le A'm\exp\big(-c\rho_{i_\ast}^2n_{i_\ast}^{\zeta+\alpha} \tilde n^{-\alpha}m\big).\nonumber
    \end{align}
    By \eqref{eq:boxing-tilde-n}, $m\le m_{i_\ast+1}$, so the bound~\eqref{eq:useful-boxing} holds for $\tilde n=m n_{i_\ast}$, which gives the first row below. Then, $\tilde n\ge \nu n$ follows by multiplying both sides of \eqref{def:ratio-nu}. Moreover, the prefactor $A'm < A'n$ can be absorbed by decreasing the constant prefactor in the exponent. We obtain for some $c'>c''>0$
    \begin{align} 
    \ind{\CA_\mathrm{sub}(n)}
    \Prob\big(\CH_n\text{ not connected}\mid \text{pre-}\widehat\CG_{\tilde n}\big)
    &\le A'm\exp\big(-c'\rho_{i_\ast}^2\tilde n^{\zeta-\delta}\big)\nonumber\\
    &\le \tfrac{1}{2}\exp\big(-c''\rho_{i_\ast}^2 n^{\zeta-\delta}\big).\nonumber
    \end{align}
    The definition of $\rho_i$ in \eqref{eq:rho-i} gives that $\rho_{i_\ast}>\rho>0$ for all $i_\ast$, which gives our desired bound for the \emph{(Con)}-step. Combining    the error bound in~\eqref{eq:boxing-sub} with this, we obtain that for any $A>0$ sufficiently large that only depends on the fixed value of the constant $n_0$, 
    \[ 
    \Prob\big(\CA_\mathrm{sub}(n)\cap\{\CH_n\mbox{ connected}\}\big)\ge 1-\exp\big(-\tfrac{1}{A}n^{\zeta-\delta}\big).
    \]
    By \eqref{eq:pre-graph-n} and the arguments below, connectivity of the auxiliary graph $\CH_n$ means that the large components $\CC_{i_\ast, x_j}$ in the first $m/2$ $i_\ast$-good boxes form a connected graph. In pre-$\widehat\CG_{\tilde n}$ we only revealed edges between vertices of mark at most $2n_{i_\ast}^\gamma<2n^\gamma$ by \eqref{eq:pre-graph-n}. By \eqref{eq:boxing-tilde-n} and the assumption that $i_\ast>i_0$ below \eqref{eq:boxing-tilde-n}, we also have $2n_{i_\ast}^\gamma>n_0^{2/(\tau-1)}$.
   Moreover, on $\CA_\mathrm{sub}(n)\cap\{\CH_n\mbox{ connected}\}$, the $\lceil m/2\rceil$ components of size at least $\rho_{i_\ast} n_{i_\ast}$ form a connected graph, and thus $\CG_n[1, 2n_{i_\ast}^\gamma)$ contains a component of size at least 
    \[
    (m/2)\cdot \rho_{i_\ast}n_{i_\ast}= \rho_{i_\ast} \tilde n /2 \ge (\nu/2) \rho_{i_\ast}n \ge (\nu/2)\rho n.
    \]
    This implies that the largest component must be at least this large as well. So, 
    \begin{equation}\CA_\mathrm{sub}(n)\cap\{\CH_n\mbox{ connected}\}
    \,
    \subseteq
    \,
    \big\{|\CC_n^\sss{(1)}[1,An^{\gamma})|\ge \tfrac{1}{A}n\big\}\label{eq:boxing-imp}
    \end{equation} when $A$ is sufficiently large.
 This finishes the proof of~\eqref{eq:ltld-weaker}. We do not need the \emph{(Mark)} step here since there is no further iteration. For the Palm version $\Prob^\sss{x}$ the proof is exactly the same. 
    \end{proof}

\begin{proof}[Proof of Proposition~\ref{proposition:ltld-weaker} when $\zeta_\mathrm{ll}>0$]
When the vertex is formed by a PPP, we obtain the second statement \eqref{eq:ltld-weaker-ll} by a straightforward adaptation of the above proofs. We provide a sketch.   
When $\zeta_{\mathrm{ll}}=2-\alpha>0$, we re-define what we call an $i$-good box by leaving the requirement on the marks in \eqref{eq:i-good-def} out. 
 We now call $\CQ_i(x)$ $i$-good (and $i$-bad otherwise) if it satisfies the event 
    \begin{equation}\label{eq:i-good-def-ll}
       \CA_{\mathrm{i\textnormal{-}good}}(x):= \Big\{ \exists\text{ component $\CC_{i}$ in }\widehat\CG_{n_i}: |\CC_{i}| \ge \rho_i n_i \Big\}.
    \end{equation}
Then, we replace $\zeta_\mathrm{hl}$ by $\zeta_\mathrm{ll}=2-\alpha$ in all computations and definitions,  and we skip the \emph{(Mark)}-step on each level. We do this by setting the level-$i$ vertex set to be $\widehat\CV_{i}:=\CV_{\CQ_i(x)}[1, n_0^{2/(\tau-1)})$. That is, only the size of the box changes across levels, the mark-truncation stays put. The \emph{(Sub)}-step leading to \eqref{eq:not-a-sub} carries through unchanged by the inductive assumption and the concentration of binomial random variables on the number of level-$(i-1)$ good subboxes. For the \emph{(Con)}-step, instead of using high-mark vertices, we estimate the mark of each vertex from below by $1$ in the large components $\CC_{i-1, x_j}$ and $\CC_{i-1, x_k}$. Using this lower bound, the number of edges between the two components dominates a binomial random variable. Both components have size at least $\rho_{i-1}n_{i-1}$ because they are contained in $(i-1)$-good boxes, and their distance is at most $\sqrt{d}n_i^{1/d}$, so \eqref{eq:biskup-qi} simplifies to 
\[ 
\begin{aligned}
\Prob\big(\CC_{i-1, x_j}\nsim_{\widehat\CG_{n_i}}\CC_{i-1, x_{k}}\mid \text{pre-}\widehat\CG_{n_{i}}) &\le \Big(1-p\big(1\wedge \tfrac{\beta}{\sqrt{d}^d n_i}\big)^\alpha\Big)^{\rho_{i-1}^2n_{i-1}^{2}} \\
&\le \exp\big(  - (p \beta d^{-\alpha d/2}) \cdot \rho_{i-1}^2n_{i-1}^{2} n_i^{-\alpha}\big)=:1-q_i. 
\end{aligned}
\]  
Observe that the exponent $n_{i-1}^2=n_{i-1}^{\zeta+\alpha}$ since here $\zeta=2-\alpha$. From here on the computations are identical to those between \eqref{eq:repeat-at-n-con} and \eqref{eq:ind-pr-con}, and since the last term in \eqref{eq:biskup-giant-pr} is not present, we conclude the proof at \eqref{eq:biskup-giant-pr} for $\zeta_{\mathrm{ll}}>0$.
Along the lines we also use that all statements of \cref{claim:useful} remain valid also for $\zeta=\zeta_{\mathrm{ll}}$, except \eqref{eq:mi-ni-1-gamma}. But, \eqref{eq:mi-ni-1-gamma} is only used in the \emph{(Mark)}-step which we omit  when we consider the case $\zeta_{\mathrm{ll}}>0$.
\end{proof}

\section{Sharp bounds on the existence of a giant}\label{sec:bootstrap}
The goal of this section is to improve upon the probability of the existence of the giant component compared to Proposition~\ref{proposition:ltld-weaker}, which is formalized in the next proposition. In Proposition~\ref{proposition:ltld-weaker}, the exponent of $n$ in the stretched exponential decay is $\max(\zeta_\mathrm{hl}, \zeta_\mathrm{ll})-\delta$. Here we remove the error $\delta$ and `append' $(d-1)/d, \zeta_\mathrm{hh}$ to the list in the maximum.

For a regime $\mathrm{type}\in\{\mathrm{ll}, \mathrm{hl}, \mathrm{hh}\}$, and constants $\rho,\delta\in(0,1)$, we define the event 
\begin{equation}\label{eq:a-type-giant}
    \CA_\mathrm{giant}^\mathrm{type}(\rho):=
    \begin{dcases}
        \big\{\exists\mbox{ comp. $\CC$ in $\CG_n$}: |\CC|\ge\rho n, \CC\supseteq\CV[n^{(1-\delta)/(\tau-1)},\infty)\big\},&\text{if }\mathrm{type}\in\{\mathrm{hl}, \mathrm{hh}\}, \\
        \big\{\exists\mbox{ comp. $\CC$ in $\CG_n$}:  |\CC|\ge\rho n\big\},&\text{if }\mathrm{type}=\mathrm{ll}.
    \end{dcases}
\end{equation}
Recall the values $\gamma_\mathrm{hl}$ and $\gamma_\mathrm{hh}$ from~\eqref{eq:gamma-lh} and~\eqref{eq:gamma-hh}, respectively. Define 
\begin{equation}
    \gamma_\mathrm{ll}:=(\alpha-1)/(\tau-1),\label{eq:gamma-ll}
\end{equation}
which we will only use when $\zeta_\mathrm{ll}=2-\alpha>0$. In particular, we will not use it when $\alpha=\infty$.
\begin{proposition}[Sharp bounds for the giant]\label{lemma:bootstrap}
Consider a supercritical KSRG satisfying Assumption \ref{assumption:main} on a PPP or consider supercritical long-range percolation on $\Z^d$. Assume that there is a $\mathrm{type}\in\{\mathrm{ll}, \mathrm{hl}, \mathrm{hh}\}$ so that  $\zeta_\mathrm{type}>0$ and that for a constant $\rho>0$, 
\begin{equation}\label{eq:bootstrap-condition}
\Prob\big(\CA_\mathrm{giant}^\mathrm{type}(\rho)\big)\longrightarrow 1,\qquad\text{as }n\to\infty.
\end{equation}
\begin{enumerate}
    \item[(long)] Then for all $\varepsilon>0$, there exists $A>0$ such that for all $n\ge 1$
    \begin{equation}
\Prob\big(|\CC_n^\sss{(1)}[1, An^{\gamma_\mathrm{type}})|<(1-\varepsilon)\rho n\big)\le \exp\big(-\tfrac{1}{A}n^{\zeta_\mathrm{type}}\big).\label{eq:bootstrap-result-long}
\end{equation}
    \item[(short)] Also, for all $\varepsilon>0$, there exists $A>0$ such that for all $n\ge 1$
    \begin{equation}
\Prob\big(|\CC_n^\sss{(1)}[1, A)|<(1-\varepsilon)\rho n\big)\le \exp\big(-\tfrac{1}{A}n^{\zeta_\mathrm{short}}\big).\label{eq:bootstrap-result-short}
\end{equation}
\end{enumerate}
\end{proposition}
\cref{lemma:bootstrap} gives sharp decay for the value of $\zeta_\mathrm{type}$ that maximizes $\{\zeta_\mathrm{ll}, \zeta_\mathrm{hl}, \zeta_\mathrm{hh}, \zeta_\mathrm{short}\}$, which is equal to $\zeta_\star$ by \cref{claim:phases}.  When $\zeta_{\mathrm{long}}=\max\{\zeta_{\mathrm{ll}}, \zeta_{\mathrm{hl}}, \zeta_{\mathrm{hh}}\}>0$ and $(d-1)/d>0$, the second statement shows that the giant still exists with stretched exponential error probability even when all vertices with mark above a constant are deleted from the graph. When $\zeta_\star=(d-1)/d$, this decay is even stronger than \eqref{eq:bootstrap-result-long}.

The requirement~\eqref{eq:bootstrap-condition} corresponds to \eqref{eq:density-condition} in Section \ref{sec:outline}. It demands \emph{some} initial density  bound $\rho$ on the size of the giant, and for the (hh) and (hl)-regimes it also demands  that the highest-mark vertices are in the giant. Via edges incident to these high-mark vertices we can connect large components in subboxes in renormalization schemes below. We guarantee the presence of these long edges via the reasoning in \cref{sec:dominant}. The requirement on high-mark vertices is not needed in the (ll) regime in \eqref{eq:a-type-giant}, since we bound all marks from below by one. Thus, in the (ll)-regime, \eqref{eq:bootstrap-condition} is already satisfied for a small $\rho$ by \cref{proposition:ltld-weaker}.
At the end of this section we verify \eqref{eq:bootstrap-condition} also for the (hl) and (hh)-regimes for a small $\rho$ using Proposition~\ref{proposition:ltld-weaker}, obtaining \cref{cor:initial-upper}. 

We will apply \cref{lemma:bootstrap} again in Section~\ref{sec:lower-tail} with density $\rho$ close to $\theta$, as by then we have proved the weak LLN for $|\CC_n^\sss{(1)}|$ for all parameters satisfying $\zeta_\mathrm{long}>0$. This will result in the upper bound on the right-hand side of \eqref{eq:thm-ltld} in \cref{thm:large-dev2}.
The exponent $\zeta_\star$ of the stretched exponential decay in \cref{lemma:bootstrap}  will also appear in our upper bounds on $|\CC_n^\sss{(2)}|$ and the cluster-size decay for $\zeta_\mathrm{type}\in\{\zeta_\mathrm{hl}, \zeta_\mathrm{ll}\}$ in Theorems \ref{thm:subexponential-decay} and \ref{thm:second-largest}. Instead of \cref{lemma:bootstrap}, \cref{proposition:ltld-weaker} could instead be used as input to obtain (weaker) upper bounds on $|\CC_n^\sss{(2)}|$ and the cluster-size decay for $\zeta_\mathrm{type}\in\{\zeta_\mathrm{hl}, \zeta_\mathrm{ll}\}$. That strategy would still give a weak LLN and also Theorem \ref{thm:large-dev2}, but would lead to the appearance of an error $\delta$ in exponents in Theorems \ref{thm:subexponential-decay} and \ref{thm:second-largest}, which would then never match the lower bounds in Theorems \ref{thm:subexponential-decay} and \ref{thm:second-largest}. This explains the benefit of \cref{lemma:bootstrap}.  
\begin{remark}[High-level comparison of proofs in Sections \ref{sec:biskup} and \ref{sec:bootstrap}]\label{remark:biskup-comparison}
Section \ref{sec:bootstrap} sets up a single-layer renormalization with carefully chosen number and size of boxes that lead exactly to \eqref{eq:bootstrap-result-long} and \eqref{eq:bootstrap-result-short}, respectively. It needs Section \ref{sec:biskup} as an input, because in each of these boxes a linear-sized component has to be present whp on truncated mark vertices. 

It is natural to ask why one could not use `more carefully chosen' box sizes already in Section \ref{sec:biskup} and obtain better bounds already there? The rough answer is the following: here in Section \ref{sec:bootstrap} we shall shortly perform a single-layer renormalization. Since it is a single layer, we \emph{can} afford to lose a small linear portion of the giant. Thus, we can renormalize to a \emph{supercritical} Erd{\H os}-R\'enyi random graph (ERRG), and use large deviation results for those. Then, we have $\Theta(n^{\zeta})$ many boxes, and the probability that a supercritical ERRG has a too small giant decays exponentially in the the number of boxes.

On the contrary, Section \ref{sec:biskup} uses an iterative scheme in order to improve the initial sublinear bound on the largest component to a linear bound. That is not possible in a single-layer renormalization. In an iterative scheme, one \emph{cannot} allow to lose a small constant fraction of the largest component in each step. That would (again) result in a sublinear bound on the largest component. Therefore, in Section \ref{sec:biskup} we renormalized to a \emph{connected} ERRG. For that we need slightly larger boxes, and thus we have $o(n^{\zeta})$ boxes when reaching total size $n$. Moreover, the probability of the rare event that such an ERRG is not connected, decays exponentially in the expected degree, which is smaller than the total number of boxes. These reasons explain the $-\delta$ term in the exponent in Section \ref{sec:biskup}. 

\end{remark}
 We proceed to the proofs. We start with a claim that allows us to truncate the vertex marks from above; then we separately present the proof of the two statements in \cref{lemma:bootstrap}. 
\begin{claim}\label{claim:bootstrap-cond-strong}
    Consider a supercritical KSRG under the setting of Proposition \ref{lemma:bootstrap}. Assume that $\zeta_{\mathrm{type}}>0$ and ~\eqref{eq:bootstrap-condition} holds for some $\rho>0$, $\mathrm{type}\in\{\mathrm{hl}, \mathrm{hh}\}$, and $\delta\in (0,1)$. Then, as $n$ tends to infinity,
\begin{equation}
\Prob\Big(\exists\mbox{comp. $\CC$ in $\CG_n\Big[1,n^{\tfrac1{\tau-1}}\Big)$}: \CC\supseteq\CV\Big[n^{\tfrac{1-\delta}{\tau-1}},n^{\tfrac1{\tau-1}}\Big), |\CC|>\rho n\Big)\longrightarrow 1. \label{eq:bootstrap-condition-cor}
\end{equation}
\begin{proof}
    We abbreviate $\underline w=n^{(1-\delta)/(\tau-1)}$, $\overline w=n^{1/(\tau-1)}$ throughout the proof.
     We start from \eqref{eq:bootstrap-condition}
    and distinguish whether $\CG_n$ contains vertices of mark at least $\overline w$: if there are no vertices of mark at least $\overline w$, then $\CG_n[1,\overline w)=\CG_n$. Slightly abusing standard notation, we write here $\CC\subset \CV_n[1,\underline w)\cup \CV_n[\underline w, \infty)$ for a set of vertices that is a \emph{component} when restricted to marks below $\underline w$. Then we get
    \begin{align*}
        \Prob&\big(\exists\mbox{ comp. $\CC$ in $\CG_n$}: \CC\supseteq\CV[\underline w,\infty), |\CC|>\rho n\big) \\
        &=
        \Prob\big(\{\exists\mbox{ comp. $\CC$ in $\CG_n[1,\overline w)$}: \CC\supseteq\CV[\underline w,\overline w), |\CC|>\rho n\}\cap \{\CV[\overline w,\infty)=\emptyset\}\big) \\
        &\hspace{15pt}
        +
        \Prob\big(\{\exists\mbox{ comp. $\CC$ in $\CG_n$}: \CC\supseteq\CV[\underline w,\overline w), |\CC|>\rho n\}\cap\{\CV[\overline w,\infty)\neq \emptyset\}\big).
    \end{align*}
    Since the vertices are formed by a PPP, $\CV[\overline w,\infty)$ is independent of $\CV[1, \overline w)$ and thus also of the graph $\CG_n[1,\overline w)$.  So, in the first term on the right-hand side we can take the product of the probabilities of the two events. Using the intensity of $\CV$ in~\eqref{eq:poisson-intensity}, $\Prob(\CV[\overline w,\infty)=\emptyset)=1/\re$. Hence, conditioning on $\{\CV[\overline w,\infty)\neq\emptyset\}$ in the second term on the right-hand side yields
    \begin{align*}
        \Prob&\big(\exists\mbox{ comp. $\CC$ in $\CG_n$}: \CC\supseteq\CV[\underline w,\overline w), |\CC|>\rho n\big) \\
        &=
        \frac{1}{\re}\Prob\big(\exists\mbox{ comp. $\CC$ in $\CG_n[1,\overline w)$}: \CC\supseteq\CV[\underline w,\overline w), |\CC|>\rho n\big) \\
        &\hspace{15pt}
        +
        \Big(1-\frac{1}{\re}\Big)\Prob\big(\exists\mbox{ comp. $\CC$ in $\CG_n$}: \CC\supseteq\CV[\underline w,\overline w), |\CC|>\rho n\mid\CV[\overline w,\infty)\neq\emptyset\big).
    \end{align*}
    Since the left-hand side converges to $1$ as $n$ tends to infinity by the assumed condition~\eqref{eq:bootstrap-condition}, the two probabilities on the right-hand side must converge to $1$ as well. Since the event in the first row on the right-hand side coincides with \eqref{eq:bootstrap-condition-cor}, the claim follows.
\end{proof}
\end{claim}

\subsection{When relying on long edges}\label{sec:bootstrap-long}
Before the proof of \cref{lemma:bootstrap}(long), we state a result about the giant in Erd\H{o}s-R\'enyi random graph $G(n,p)$. The result will `replace' \cref{claim:er-conn} in the renormalization scheme of this section compared to the scheme in Section \ref{sec:biskup}. The claim follows from a large-deviation principle for the size of the giant component obtained by O'Connell~\cite{oconnel1998}, see also~\cite{andreis2021er}. For our result, we do not need the exact rate function of this LDP. 
 
\begin{claim}[Large deviations in $G(n,\lambda/n)$]\label{prop:er-ltld}
Consider the largest connected component $C_n^\sss{(1)}(p_n)$ of the Erd\H{o}s-R\'enyi random graph $G(n,p_n)$. Fix $\varepsilon>0$. There exist constants $A>0$, $\lambda=\lambda(\varepsilon)>1$ such that for all $n\ge 1$
\[ 
\Prob\big(|C_n^\sss{(1)}(\lambda/n)|<(1-\varepsilon)n\big)\le \exp\big(-\tfrac{1}{A}n\big).
\]
\begin{proof}
Denote by $\varrho_\lambda$ the survival probability of a Bienaym\'e-Galton-Watson branching process with $\mathrm{Poi}(\lambda)$ offspring. By~\cite[Theorem 3.1]{oconnel1998}, for every $\lambda > 1$ and $\tilde \varepsilon > 0$, there exists a constant $A>0$ such that
    \begin{equation}\nonumber
\Prob\big(|C_n^\sss{(1)}(\lambda/n)|<(1-\tilde \varepsilon)\varrho_\lambda n\big)\le \exp\big(-\tfrac{1}{A}n\big).
\end{equation}
It is well known that the survival probability $\varrho_\lambda$ tends to one as $\lambda \to \infty$, as $\varrho_\lambda$ is the largest solution in $(0,1)$ to the equation $\varrho_\lambda=1-\exp(-\lambda\varrho_\lambda)$ (see e.g.\ \cite[Example 6.1.10]{mdp}).
     Hence, given $\varepsilon > 0$, we may choose $\lambda$ large enough, and $\tilde\varepsilon>0$ small enough, so that $(1-\tilde \varepsilon)\varrho_\lambda\ge(1-\varepsilon)$. The statement follows. 
\end{proof}
\end{claim}

\begin{proof}[Proof of \cref{lemma:bootstrap}(long)]
We give the proof for KSRGs on a Poisson point process. The result for long-range percolation on $\Z^d$ follows analogously to the case $\zeta_\mathrm{ll}>0$.

Fix $\zeta_\mathrm{type}\in\{\zeta_\mathrm{hl}, \zeta_\mathrm{hh}, \zeta_\mathrm{ll}\}$ such that $\zeta_\mathrm{type}>0$. Let $\gamma_\mathrm{type}\in\{\gamma_\mathrm{hl}, \gamma_\mathrm{hh}, \gamma_\mathrm{ll}\}$ be the corresponding value that solves $\zeta_\mathrm{type}=1-\gamma_\mathrm{type}(\tau-1)$ by the definitions of $\gamma_\mathrm{hl}$, $\gamma_\mathrm{hh}$, and $\gamma_\mathrm{ll}$ in~\eqref{eq:gamma-lh},~\eqref{eq:gamma-hh}, and~\eqref{eq:gamma-ll}, respectively.
For readability, we omit the subscript $\mathrm{type}$ throughout the proof. 

\smallskip\noindent
\emph{Definitions of constants and boxing scheme.}
Below we will suitably define a constant 
$M>0$ depending on the regime (i.e., high-low, high-high, or low-low). 
For simplicity, we assume that $(n^{\zeta}/M)^{1/d}\in\N$, so that we can partition $\Lambda_n$ into 
\begin{equation}
m_n:=n^{\zeta}/M\label{eq:bootstrap-mn}
\end{equation} boxes $\CQ_1,\ldots,\CQ_{m_n}$. Using that $(1-{\zeta})/(\tau-1)=\gamma$, each box has volume 
\begin{equation}\label{eq:bootstrap-kn}
k_n:=n/m_n = Mn^{1-\zeta}= \big(M^{1/(\tau-1)}n^{\gamma}\big)^{\tau-1}.
    \end{equation}
Given any $\varepsilon>0$ in \cref{lemma:bootstrap}(long), define the small $\varepsilon_1, \varepsilon_2>0$ as the solutions of the following equations:
\begin{equation}
    1-\varepsilon = (1-\varepsilon_1)^2, 
    \qquad 
    \Prob\big(\CV[\varepsilon_2k_n^{1/(\tau-1)},\infty)=\emptyset\big)=\varepsilon_1/4.\label{eq:bootstrap-eps}
\end{equation}
In a box of volume $k_n$, the expected number of vertices of mark at least $\varepsilon_2k_n^{1/(\tau-1)}$ equals $\varepsilon_2^{-(\tau-1)}$ by \eqref{eq:power-law}, so $\varepsilon_2$ is a constant.

\emph{Good boxes and auxiliary graph.} We write $\CV_{k_n, j}:=\CV\cap(\CQ_j\times[1,\infty))$ for the vertices in $\CQ_j$, $\CG_{k_n,j}$ for the subgraph of $\CG_n$ induced by the vertices in $\CQ_j$, and similar to \eqref{eq:pre-graph}, we set
\begin{equation}\label{eq:pre-gn-333}
\text{pre-}\CG_{n}:=\cup_{j\le m_n} \CG_{k_n,j}.
\end{equation}
We call a box $\CQ_j, j\in[m_n]$ \emph{good} if the following event holds depending on the regime $\mathrm{type}\in\{\mathrm{hl}, \mathrm{hh}, \mathrm{ll}\}$.
When $\mathrm{type}\in\{\mathrm{hl}, \mathrm{hh}\}$,
\begin{equation}\label{eq:bootstrap-box}
    \CA_\mathrm{box}(j)\!:=
    \!\left\{ \,\begin{aligned}\exists \mbox{comp. $\CC_j$ in $\CG_{k_n, j}\big[1,k_n^{1/(\tau-1)}\big)$}: 
    |\CC_{j}|&\ge \rho k_n, \\
 \CC_{j}&\supseteq\CV_{k_n, j}[\varepsilon_2k_n^{1/(\tau-1)}, k_n^{1/(\tau-1)})\neq\emptyset
    \end{aligned}
    \right\}.
\end{equation}
When $\mathrm{type}=\mathrm{ll}$, 
\begin{equation*}
\CA_\mathrm{box}(j) :=
    \!\left\{ \,\exists \mbox{comp. $\CC_j$ in $\CG_{k_n, j}\big[1,k_n^{1/(\tau-1)}\big)$}: 
    |\CC_{j}|\ge \rho k_n
    \right\}.
    \end{equation*}
We define the event that there are enough good boxes in $\Lambda_n$ (similar to $\CA_{\mathrm{sub}}$ in 
\eqref{eq:sub-event}):
\begin{equation}\label{eq:a-sub-bootstrap}
\CA_{\mathrm{sub}}:=\bigg\{\sum_{j\in[m_n]}\ind{\CA_\mathrm{box}(j)}\ge (1-\varepsilon_1)m_n\bigg\},
\end{equation}
The event $\CA_\mathrm{sub}$ is measurable with respect to pre-$\CG_{n}$ in \eqref{eq:pre-gn-333}.

\emph{Enough subboxes are good.} We estimate $\Prob\big(\neg \CA_\mathrm{sub}\big)$, with $\CA_\mathrm{sub}$ defined in \eqref{eq:a-sub-bootstrap}. 
For that we analyze the individual probabilities $\Prob\big(\CA_\mathrm{box}(j)\big)$. By the definition of $\varepsilon_1$ in~\eqref{eq:bootstrap-eps} and the hypothesis~\eqref{eq:bootstrap-condition}, for all $n$ sufficiently large and any $j\in[m_n]$
\begin{equation}
\Prob\big(\CA_\mathrm{box}(j)\big)\ge 1-\varepsilon_1/2.
\label{eq:a-box-i-wp}
\end{equation}
Since the graphs $\CG_{k_n,j}$ are  iid across $j\in[m_n]$, by a standard Chernoff bound on binomial random variables, there exists a constant $c>0$ such that 
\begin{equation}
\Prob\big(\neg\CA_\mathrm{sub}\big)\le \exp\big(-cm_n\big).\label{eq:bootstrap-firstev}
\end{equation}

\emph{Introducing the auxiliary graph.}
Similar to below \eqref{eq:pre-graph} and \eqref{eq:pre-graph-n}, on any realization of pre-$\CG_n$ satisfying $\CA_\mathrm{sub}$, we re-label boxes so that the first $\lceil (1-\eps_1) m_n\rceil$ many boxes are good. 
Then we define the auxiliary graph $\CH_n=(\CV_{\CH_n},\CE_{\CH_n})$ on $\lceil (1-\eps_1) m_n\rceil$ many vertices, so that vertex $j\in \CV_{\CH_n}$ corresponds to the component $\CC_j$ inside $\CQ_j$ in \eqref{eq:bootstrap-box}, 
and two vertices $j, \ell$ are connected in $\CH_n$ if the components $\CC_{j}$ and $\CC_{\ell}$ are connected by an edge in $\CG_n$.
We write $\CC_{\CH_n}^\sss{(1)}$ for the largest component in $\CH_n$.
By construction, the induced connected components $(\CC_j)_{j\in \CC_{\CH_n}^\sss{(1)}}$ in \eqref{eq:bootstrap-box} are all contained in the same component in $\CG_n$. So if the event $\CA_\mathrm{sub}\cap\{|\CC_{\CH_n}^\sss{(1)}|\ge (1-\varepsilon_1) |\CV_{\CH_n}|\}$ holds, then the graph $\CG_n\big[1,k_n^{1/(\tau-1)}\big)$ contains a connected component $\CC$ with
\[
|\CC| \ge \sum_{j\in\CC_{\CH_n}^\sss{(1)}}|\CC_j| \ge (1-\varepsilon_1)^2m_n\cdot \rho k_n = (1-\varepsilon_1)^2\rho n = (1-\varepsilon)\rho n.
\]
Here we used that $|\CC_j|\ge \rho k_n$ by the event $\CA_\mathrm{box}(j)$, that $k_n=n/m_n$ by~\eqref{eq:bootstrap-kn}, and $(1-\varepsilon_1)^2=1-\varepsilon$ by \eqref{eq:bootstrap-eps}.
 Clearly, the largest component of $\CG_n[1, k_n^{1/(\tau-1)})$ is at least as large as $\CC$, so 
\begin{align}
    \Prob\big(|\CC_n^\sss{(1)}[1,k_n^{1/(\tau-1)})|\le (1-\varepsilon)\rho n\big)&\le 
    \Prob\big(\neg \CA_\mathrm{sub}\big) + \Prob\big(\CA_\mathrm{sub}\cap\{|\CC_{\CH_n}^\sss{(1)}|\le (1-\varepsilon)|\CV_{\CH_n}|\}\big).\label{eq:bootstrap-twoterms}
\end{align}
We study the two probabilities on the right-hand side separately.

\emph{Large deviation for the giant of ERRGs.} We turn to the second term on the right-hand side in~\eqref{eq:bootstrap-twoterms}. We rewrite this second term using the law of total probability: we first reveal pre-$\CG_n$, then integrate over realizations satisfying $\CA_\mathrm{sub}$: 
\begin{equation}\label{eq:bootstrap-total-prob}
\Prob\big(\CA_\mathrm{sub}\cap\{|\CC_{\CH_n}^\sss{(1)}|\le(1-\varepsilon)^2m_n\big)
=
\E\big[\ind{\CA_\mathrm{sub}}\Prob\big(|\CC_{\CH_n}^\sss{(1)}|\le(1-\varepsilon_1)|\CV_{\CH_n}|\mid \text{pre-}\CG_n\big)\big].
\end{equation}
Let $\lambda=\lambda(\eps_1),A>0$ be the two constants from \cref{prop:er-ltld} so that for all $m\ge 1$, 
\begin{equation}
    \Prob\big(|C_m^\sss{(1)}(\lambda/m)|< (1-\varepsilon_1)m\big)\le A\exp\big(-\tfrac{1}{A}m\big).\label{eq:bootstrap-lambda}
\end{equation}
We will show below that conditionally on the realization pre-$\widehat \CG_n$ satisfying $\CA_\mathrm{sub}$,
\begin{equation}
    \CH_n \succcurlyeq G(|\CV_{\CH_n}|, \lambda(\eps_1)/|\CV_{\CH_n}|).\label{eq:bootstrap-domination}
\end{equation}
Suppose we have proven this stochastic domination. Then we can apply \eqref{eq:bootstrap-lambda}, and use that $|\CV_{\CH_n}|=\lceil (1-\eps_1)m_n\rceil=:\widetilde m_n$:
\begin{equation*}
\begin{aligned}
\ind{\CA_\mathrm{sub}}\Prob\big(|\CC_{\CH_n}^\sss{(1)}|<(1-\varepsilon_1)|\CV_{\CH_n}|\mid \widehat\CG_n\big)
&\le \Prob\big(|C_{\widetilde m_n}^\sss{(1)}(\lambda/\widetilde m_n)|< (1-\varepsilon_1)\widetilde m_n\big)\\
& \le \exp\big(-\tfrac{1}{A}|\CV_{\CH_n}|\big)\le  \exp(-\tfrac{1}{A(1-\varepsilon_1)}m_n)
\end{aligned}
\end{equation*}
 This bound holds uniformly for all realizations of pre-$\CG_n$ with $\CA_\mathrm{sub}$ in \eqref{eq:bootstrap-total-prob}. So,  combined with~\eqref{eq:bootstrap-firstev}, this bounds the two terms on the right-hand side in~\eqref{eq:bootstrap-twoterms}. Recalling that $m_n=n^{\zeta}/M$,
\begin{equation*}
    \Prob\big(|\CC_n^\sss{(1)}[1, k_n^{1/(\tau-1)})|<(1-\varepsilon) \rho n\big)\le \exp\big(-\tfrac{c}{M}n^{\zeta}\big) + \exp\big(-\tfrac{1}{AM(1-\varepsilon_1)}n^{\zeta}\big).
\end{equation*}
The mark bound $k_n^{1/(\tau-1)}$ in the largest component on the left-hand side is equal to $M^{1/(\tau-1)}n^{\gamma}$ by~\eqref{eq:bootstrap-kn}. 
Thus, to finish the proof of \eqref{eq:bootstrap-result-long}, it remains to prove the stochastic domination in~\eqref{eq:bootstrap-domination}. 

\emph{Showing that the auxiliary graph dominates ERRG.}
We now prove~\eqref{eq:bootstrap-domination} case-by-case for the three regimes, for $\zeta$ being either $\zeta_{\mathrm{ll}}$, $\zeta_{\mathrm{hl}}$, or  $\zeta_{\mathrm{hh}}$. We reveal the  subgraph pre-$\CG_n$ in \eqref{eq:pre-gn-333}, i.e., all induced subgraphs inside boxes $(\CQ_j)_{j\in[m_n]}$, and assume that the realization satisfies $\CA_\mathrm{sub}$. The first $\lceil (1-\eps_1) m_n\rceil$ boxes are good and these labels form the vertex set $\CV_{\CH_n}$. Each label corresponds to a component $\CC_j$ in $\CQ_j$. Edges between these components $(\CC_j)_{j\in\CV_{\CH_n}}$ have not been revealed yet, and  conditional on pre-$\CG_n$, they are present independently of each other. Therefore, it is sufficient to show that for all $j,\ell\in\CV_{\CH_n}$,
\begin{equation}
\Prob\big(\CC_j\sim_{\CG_n}\CC_{\ell}\mid \text{pre-}\CG_n\big)\ge \frac{\lambda}{|\CV_{\CH_n}|}.\label{eq:bootstrap-domination-toshow}
\end{equation}
We will prove this by suitably choosing the constant $M$ in~\eqref{eq:bootstrap-mn}, case-by-case for the possible values of $\zeta=\zeta_\mathrm{type}\in\{\zeta_\mathrm{hl}, \zeta_\mathrm{hh}, \zeta_\mathrm{ll}\}$.

\smallskip
\noindent
\emph{1. High-low edges, $\zeta_{\mathrm{type}}=\zeta_{\mathrm{hl}}>0$.}
We assume that $\alpha<\infty$, since $\zeta_\mathrm{hl}=(\tau-1)/\alpha-(\tau-2)<0$ otherwise.
If a box $\CQ_j$ is good, see \eqref{eq:bootstrap-box}, the large component $\CC_j$ with  size at least $\rho k_n$ contains at least one vertex $v_j^\star$ of mark at least $\varepsilon_2k_n^{1/(\tau-1)}$. If there are more such  vertices, choose one arbitrarily.   
 We estimate the probability that $\CC_j$ and $\CC_{\ell}$ are connected from below by the probability that either $v_j^\star$ connects to any vertex in $\CC_{\ell}$ or $v_{\ell}^\star$ connects to any vertex in $\CC_{j}$. We bound marks in the other component from below by $1$. Any two vertices in $\Lambda_n$ are within distance $\sqrt{d}n^{1/d}$.  Since vertex weights are revealed in pre-$\CG_n$, edges are present independently. So, we obtain from the connectivity probability $\mathrm{p}$ in~\eqref{eq:connection-prob-gen}: 
 \begin{equation}\label{eq:compute-lh-ER}
\begin{aligned}    \Prob\big(\CC_j\nsim_{\CG_n}\CC_{\ell}\mid \text{pre-}\CG_n\big) &\le \Prob(\{v_j^\star \nsim_{\CG_n} \CC_{\ell} \}\cap \{v_{\ell}^\star \nsim_{\CG_n} \CC_{j}\}\mid \text{pre-}\CG_n\big)\\
&=
\prod_{u\in\CC_j}(1-\mathrm{p}(u,v_{\ell}^\star))\cdot \prod_{\tilde u\in\CC_{\ell}}(1-\mathrm{p}(\tilde u,v_{j}^\star)) 
\\
&\le
\Big(1-p\Big(1\wedge \frac{\beta \eps_2 k_n^{1/(\tau-1)}}{d^{d/2} n}\Big)^\alpha \Big)^{2\rho k_n} \\ 
&\le 
\exp\Big(-2\rho k_n p( 1\wedge  \beta^\alpha d^{-\alpha d/2}\varepsilon_2^\alpha k_n^{\alpha/(\tau-1)}n^{-\alpha}\big) \Big).
\end{aligned}
\end{equation}
There are now two cases. If the minimum is attained at $1$ in the last row, then the right-hand side tends to $0$ since $k_n=M n^{1-\zeta}\to \infty$ since $\zeta<1$. This means that~\eqref{eq:bootstrap-domination-toshow} holds. 
If the minimum is not attained at $1$, some further calculations are required. 
We set $M$ as the solution of the equation 
\begin{equation}\label{eq:M-def-ll}
\rho p\beta^\alpha d^{-\alpha d/2}\varepsilon_2^\alpha =M^{-\alpha/(\tau-1)}\lambda
/(1-\varepsilon_1).
\end{equation}
Then the last row in \eqref{eq:compute-lh-ER} equals
\[
\exp\big(-2\rho p\beta^\alpha d^{-\alpha d/2}\varepsilon_2^\alpha k_n^{1+\alpha/(\tau-1)}n^{-\alpha}\big) = \exp\big(-\tfrac{2\lambda}{1-\varepsilon_1}k_n(k_n/M)^{\alpha/(\tau-1)}n^{-\alpha}\big).
\]
We now use that $k_n=n/m_n$ for the first appearance of $k_n$ and that $k_n/M=n^{1-\zeta_\mathrm{hl}}$ in the second appearance, see \eqref{eq:bootstrap-kn}. The exponent there is $1-\zeta_\mathrm{hl}=(1-1/\alpha)(\tau-1)$, so rewriting the right-hand side yields that the exponents of $n$ cancel each other.  Using also $|\CV_{\CH_n}|=\lceil(1-\varepsilon_1)m_n\rceil$, 
\[
\exp\big(-\tfrac{2\lambda}{1-\varepsilon_1}(n/m_n)n^{(1-\zeta_\mathrm{hl})\alpha/(\tau-1)}n^{-\alpha}\big)=\exp\Big(-\frac{2\lambda}{(1-\varepsilon_1)m_n}\Big) < 1-\frac{\lambda}{|\CV_{\CH_n}|}
\]
for all $n$ sufficiently large. This proves the desired bound~\eqref{eq:bootstrap-domination-toshow} in the high-low regime.

\smallskip
\noindent
\emph{2. High-high edges, $\zeta_{\mathrm{type}}=\zeta_{\mathrm{hl}}>0$, and $\alpha<\infty$.}
 As in the high-low regime, we reveal pre-$\CG_n$, and choose in each good box $\CQ_j$ a vertex $v_j^\star$ in $\CC_j$ with mark at least $\varepsilon_2k_n^{1/(\tau-1)}$, see \eqref{eq:bootstrap-box}. We directly estimate the probability that $v_j^\star$ is connected to $v_{\ell}^\star$ for $j, \ell\in \CV_{\CH_n}$:
\begin{equation}\label{eq:compute-hh-ER}
\begin{aligned}    \Prob\big(\CC_j\sim_{\CG_n}\CC_{\ell}\mid \text{pre-}\CG_n\big) &\ge \Prob(v_j^\star \sim_{\CG_n}v_{\ell}^\star\mid \text{pre-}\CG_n )\\
&\ge  p\Big(1\wedge \frac{\beta \varepsilon_2^{\sigma+1} k_n^{(\sigma+1)/(\tau-1)}}{d^{ d/2}n}\Big)^\alpha.
\end{aligned}
\end{equation}
If the minimum is attained at $1$ on the right-hand side, then \eqref{eq:bootstrap-domination-toshow} follows, so we assume it is attained at the second term. We use that $k_n=n/m_n$ by \eqref{eq:bootstrap-kn}. The right-hand side equals 
\begin{align}\nonumber
    p\beta^\alpha d^{-\alpha d/2}\varepsilon_2^{(\sigma+1)\alpha}n^{\alpha(\sigma+1)/(\tau-1)-\alpha}
    m_n^{-\alpha(\sigma+1)/(\tau-1)+1}
    \frac{1}{m_n},
\end{align}
where we multiplied by a factor $m_n/m_n$. We use now that $m_n=n^{\zeta_\mathrm{hh}}/M$ by \eqref{eq:bootstrap-mn} in the first occurrence of $m_n$ above to compute the power of $n$. The above expression equals
\begin{equation*}
    p\beta^\alpha d^{-\alpha d/2}\varepsilon_2^{(\sigma+1)\alpha}M^{\alpha(\sigma+1)/(\tau-1)-1}n^{(1-\zeta_\mathrm{hh})\alpha(\sigma+1)/(\tau-1)-\alpha+\zeta_\mathrm{hh}} \cdot \frac{1}{m_n}. 
\end{equation*}
It is an elementary computation to check that the power of $n$ is $0$ using that
$\zeta_\mathrm{hh}=(\sigma +1 - (\tau-1))/(\sigma+1-(\tau-1)/\alpha)$
 in \eqref{eq:zeta-hh}.
Further, $\zeta_{\mathrm{hh}}$ is only positive when $\sigma> \tau-2$. This latter inequality implies that $(\sigma+1)/(\tau-1)> 1$ and so the exponent of $M$ is strictly positive for all $\alpha>1$. 
So, if we choose $M$ large enough then the constant multiplying $1/m_n$ in the above formula becomes arbitrarily large. So, returning to \eqref{eq:compute-hh-ER},
\[ \Prob\big(\CC_j\sim_{\CG_n}\CC_{\ell}\mid \text{pre-}\CG_n\big)\ge \frac{1}{m_n}\cdot p\beta^\alpha d^{-\alpha d/2}\varepsilon_2^{(\sigma+1)\alpha}M^{\alpha(\sigma+1)/(\tau-1)-1} \ge \frac{\lambda}{|\CV_{\CH_n}|}
\]
holds for sufficiently large $M$, since $|\CV_{\CH_n}|=\lceil (1-\eps_1)m_n\rceil$. This finishes the proof of \eqref{eq:bootstrap-domination} for the high-high regime with $\alpha<\infty$.

\smallskip\noindent
\emph{3. High-high edges, $\zeta_{\mathrm{type}}=\zeta_{\mathrm{hh}}=(\sigma+2-\tau)/(\sigma+1)>0$, and $\alpha=\infty$.} This case is almost the same as the $\alpha<\infty$ case. The only difference is that in \eqref{eq:compute-hh-ER} the minimum on the right-hand side \emph{must} be attained at $1$, and then the connection probability is $p$, otherwise it is $0$. 
So we check whether 
\begin{equation}\nonumber
\beta\varepsilon_2^{\sigma+1}k_n^{(1+\sigma)/(\tau-1)} / (d^{d/2}n) \ge 1
\end{equation}
holds.
We use $k_n=Mn^{1-\zeta_\mathrm{hh}}$ in \eqref{eq:bootstrap-kn}
and $1-\zeta_\mathrm{hh}=\gamma_\mathrm{hh}(\tau-1)=(\tau-1)/(\sigma+1)$ when $\alpha=\infty$. Thus, the powers of $n$ cancel each other again. Since $\tau<2+\sigma$, we can choose $M$ large enough so that
\[
\beta\varepsilon_2^{\sigma+1}M^{(1+\sigma)/(\tau-1)}/ d^{d/2} \ge 1.
\]
The bound~\eqref{eq:bootstrap-domination-toshow} also follows in this regime.

\smallskip
\noindent\emph{4. Low-low edges, $\zeta_{\mathrm{type}}=\zeta_{\mathrm{ll}}>0$.}
We may assume that $\alpha<\infty$, since otherwise $\zeta_\mathrm{ll}=2-\alpha<0$. By the definition of a good box in~\eqref{eq:bootstrap-box},
$|\CC_j|\ge \rho k_n$ and each vertex has mark at least $1$. So we compute using the connection probability $\mathrm{p}$ in~\eqref{eq:connection-prob-gen} for $j, \ell\in \CV_{\CH_n}$:
\begin{equation}\label{eq:computation-ll-ER}
\begin{aligned}
\Prob\big(\CC_j\nsim_{\CG_n}\CC_{\ell}\mid \text{pre-}\CG_n\big)
&\le
\big(1-p\beta^\alpha d^{-\alpha d/2}n^{-\alpha}\big)^{|\CC_j|\cdot|\CC_\ell|}\\
&\le 
\exp\big(-p\beta^\alpha d^{-\alpha d/2}\rho^2 k_n^2n^{-\alpha}\big).
\end{aligned}
\end{equation}
Since $k_n=Mn^{1-\zeta}=M n^{\alpha-1}$ in \eqref{eq:bootstrap-kn}, $k_n^2n^{-\alpha}=M^2 n^{\alpha-2}$. With  $m_n=n^{2-\alpha}/M$, we may write $k_n^2n^{-\alpha}=M/m_n$ on the right-hand side above. We choose $M$ as the solution of the equation 
\begin{equation*}
    p\beta^\alpha d^{-\alpha d/2}\rho^2 M=2\lambda/(1-\varepsilon_1).
\end{equation*} Now we obtain for $n$ large that, using $|\CV_{\CH_n}|=\lceil (1-\varepsilon_1)m_n\rceil$,
\begin{equation} \label{eq:computation-ll-ER-end} 
\begin{aligned}     \Prob\big(\CC_j\nsim_{\CG_n}\CC_{\ell}\mid \text{pre-}\CG_n\big) 
    &=\exp\big(-p\beta^\alpha d^{-\alpha d/2}\rho^2 M\tfrac{1}{m_n}\big) \le
    \exp\bigg(-\frac{2\lambda}{|\CV_{\CH_n}|}\bigg)< 1-\frac{\lambda}{|\CV_{\CH_n}|},
    \end{aligned}
\end{equation}
showing \eqref{eq:bootstrap-domination-toshow}.
The same calculation works for long-range percolation on $\Z^d$.

\smallskip
\noindent
\emph{Conclusion for proving \eqref{eq:bootstrap-result-short}.} We conclude that the bound \eqref{eq:bootstrap-domination-toshow} holds whenever $\zeta$ equals either $\zeta_{\mathrm{ll}}, \zeta_{\mathrm{hl}}, \zeta_{\mathrm{hh}}$ for a suitable choice of $M$ in~\eqref{eq:bootstrap-mn}. This proves the stochastic domination stated in~\eqref{eq:bootstrap-domination} and finishes the proof of \eqref{eq:bootstrap-result-short} in \cref{lemma:bootstrap}.
\end{proof}

\subsection{When relying on constant-length edges}\label{sec:bootstrap-short}
We proceed to proving \cref{lemma:bootstrap}(short). Here, instead of using long edges to connect boxes and a domination towards an Erd{\H o}s-R\'enyi random graph, we use ``relatively short'' edges and stochastic domination towards nearest neighbor  site-bond percolation on $\Z^d$. In that model, 
 each vertex $x\in\Z^d$ is independently \emph{active} with probability $q$, and each edge between active neighboring vertices is \emph{open}  with probability $r$, independently again. We denote the resulting graph formed by open edges between active vertices by $\mathrm{NNP}(q,r)$, and write $L_n^\sss{(1)}(q,r)$ for the largest component  inside $\Lambda_n$. We state a combination of two classical results \cite[Theorem 1.1]{deuschel1996surface, liggett1997domination} that will replace~\cref{prop:er-ltld} in this regime. 
\begin{proposition}[High-density NNP($q, r$) has a large giant]\label{prop:nn-hasgiant}
    Consider nearest neighbor  site-bond percolation $\mathrm{NNP}(q,r)$ on $\Z^d$ with probabilities $q,r\in(0,1]$. For each $\varepsilon>0$, there exist constants $q_0,r_0\in(0,1)$, and $A>0$ such that for all $q\ge q_0$, $r\ge r_0$, $n\ge 1$,
    \[
    \Prob\big(|L_n^\sss{(1)}(q,r)|<(1-\varepsilon) n\big)\le \exp\big(-\tfrac{1}{A}n^{(d-1)/d}\big).
    \]
    \begin{proof}
    Fix $\varepsilon > 0$. We shall define an $\eps_1=\eps_1(\eps)$ and  set $q_0:=1-\eps_1$ and $r_0=1-\eps_1$. Let  $\omega\sim\mathrm{NNP}(q_0, r_0)$. 
       We  couple $\omega$ to iid Bernoulli bond percolation $\omega^\star\sim\mathrm{NNP}(1, 1-12 d\eps_1)$ on $\Z^d$, similarly to   \cite{andjel1993characteristic, liggett1997domination}. By a union bound over the  endpoints of an edge and the edge itself, each edge in is open in $\omega$ with \emph{marginal} probability at least $1-3\eps_1$. Consider an edge with endpoints $(x,y)$ in $\omega$, and let $N_{(x,y)} := \{\text{edges incident to $x$ or $y$ in $\Z^d$}\} \setminus \{(x,y)\}$. Clearly  $|N_{(x,y)}| = 4d-2$.
    For any set of edges $S \subseteq N_{(x,y)}$, by a union bound,
    \begin{align*}
    \Prob((x,y) \text{ open in } \omega\mid \text{all edges in $S$ open in $\omega$}) \nonumber & \geq \Prob((x,y) \text{ open} \text{ and } \text{all edges in $S$ open}) \\ 
    & \geq 1-(4d-1)\cdot 3\eps_1 \geq 1-12d\eps_1=:\tilde p.
    \end{align*}
    Edges that do not share a vertex are independently open, so one can iterate this argument:  for every finite set $E$ of edges in $\Z^d$, $\Prob(\text{all edges in $E$ are open}) \geq \tilde p^{|E|}$. Hence, by Strassen's theorem~\cite{lindvall1999strassen} we can couple $\omega$ with an independent bond percolation $\omega^{\star}$ on $\Z^d$ with retention-probability $\tilde p$, where every open edge in $\omega^{\star}$ is open in $\omega$.
   Let us write $C_{B_n}^\sss{(1)}(\tilde p)$ for the largest component of $\omega^\star\cap\Lambda_n$. Given $\eps>0$, by~\cite[Theorem 1.1]{deuschel1996surface}, there exists a $p_\star(\eps)<1$ so that whenever   $\tilde p >p_\star(\eps)$, for some $A > 0$,
$$ \Prob\big(|C_{B_n}^\sss{(1)}(\tilde p)| < (1-\varepsilon) n\big) \le \exp\big(-\tfrac{1}{A}n^{(d-1)/d}\big).
    $$
    Hence, setting $\eps_1> (1-p_\star(\eps))/12d$ yields $\tilde p>p(\eps)$, and we obtain by the coupling $\omega \succcurlyeq \omega_\star$:
     \[
    \Prob\big(|L_n^\sss{(1)}(q,r)|<(1-\varepsilon) n\big) \le 
    \Prob\big(|C_{B_n}^\sss{(1)}(\tilde p)| < (1-\varepsilon) n\big) \le \exp\big(-\tfrac{1}{A}n^{(d-1)/d}\big),
    \]
    and the result follows for  $q_0=r_0=1-\eps_1$, and by monotonicity for all values $q \ge q_0$ and $r \ge r_0$.
    \end{proof}
\end{proposition}

\begin{proof}[Proof of \cref{lemma:bootstrap}(short)] We give the proof for KSRGs on a Poisson point process. The result for long-range percolation on $\Z^d$ follows analogously to the case $\zeta_\mathrm{ll}>0$. By the assumptions in \cref{lemma:bootstrap}), there is a $\mathrm{type}\in\{\mathrm{hl}, \mathrm{hh},\mathrm{ll}\}$ so that $\zeta_{\mathrm{type}}>0$. 
  Positivity will ensure that ``relatively long'' edges are  present to keep the giant together.  
Let $M$ be a sufficiently large constant. We assume for simplicity that $(n/M)^{1/d}\in\N$. 
Tessellate $\Lambda_n$ into $m_n:=n/M$ may boxes $\CQ_1,\ldots, \CQ_{m_n}$ of volume $M$. Similarly as before, we write $\CV_{M, j}:=\CV\cap(\CQ_j\times[1,\infty))$ for the vertices in $\CQ_j$, $\CG_{M,j}$ for the subgraph of $\CG_n$ induced by these vertices in $\CQ_j$, and pre-$\CG_{n}$ for the union of the subgraphs $\CG_{M,j}$, $j\in [m_n]$.
Fix $\delta > 0$ small as in \eqref{eq:bootstrap-condition}.

\emph{Good subboxes.} We call a box $\CQ_j, j\in[m_n]$ \emph{good} if the following event holds. 
When $\mathrm{type}\in\{\mathrm{hl}, \mathrm{hh}\}$, we set
\begin{equation}\label{eq:bootstrap-boxNN}
    \CA_\mathrm{box}(j)\!:=
    \!\left\{ \,\begin{aligned}\exists \mbox{comp. $\CC_j$ in $\CG_{M, j}\big[1,M^{1/(\tau-1)}\big)$}: 
    |\CC_{j}|&\ge \rho M, \\
 \CC_{j}&\supseteq\CV_{M, j}[M^{(1-\delta)/(\tau-1)}, M^{1/(\tau-1)})\\    
     \tfrac{1}{2}M^{\delta}&\le |\CV_{M, j}[M^{(1-\delta)/(\tau-1)}, M^{1/(\tau-1)})|
    \end{aligned}
    \right\}.
\end{equation}
When $\mathrm{type}=\mathrm{ll}$, we set
\begin{equation}
\CA_\mathrm{box}(j) :=
    \!\left\{ \,\exists \mbox{comp. $\CC_j$ in $\CG_{M, j}\big[1,M^{1/(\tau-1)}\big)$}: 
    |\CC_{j}|\ge \rho M
    \right\}.\label{eq:bootstrap-boxNN-2}
    \end{equation}
This event is measurable with respect to pre-$\CG_{n}$. Our assumption \eqref{eq:bootstrap-condition} directly implies \cref{claim:bootstrap-cond-strong}, which in turn implies that the first two requirements in $\CA_\mathrm{box}(j)$ ---concerning $\CC_j$--- are satisfied with probability arbitrarily close to  $1$ by choosing $M$ sufficiently large.
For the third requirement in \eqref{eq:bootstrap-boxNN} we use the intensity of the PPP~\eqref{eq:poisson-intensity}, and compute for fixed $\delta>0$ 
\[
\Prob\big(|\CV_{M, j}[M^{(1-\delta)/(\tau-1)},M^{1/(\tau-1)})|\ge \tfrac12 M^{\delta}\big) \longrightarrow 1,
\]
as $M$ tends to infinity.  Since the three conditions on the right-hand side hold with probability arbitrarily close to $1$, we may assume that $M$ is so large that for a fixed $q\in(0,1)$ and any $j\in[m_n]$, independently of all other boxes,
\begin{equation}
\Prob\big(\CA_\mathrm{box}(j)\big)>q.\label{eq:domination-nn-active}
\end{equation}
 \emph{The auxiliary graph and LDP for site-bond percolation.} Consider then the auxiliary graph $\CH_n$, with vertex set $\CV_{\CH_n}\subset [m_n]$. By definition of the boxing, the set of vertices $[m_n]$ can be mapped to the integer points in a box of side-length $m_n^{1/d}$, such that any two adjacent boxes are mapped to integer points at distance one from each other.
 We call $j\in\CV_{\CH_n}$ active if the box $\CQ_j$ satisfies $\CA_\mathrm{box}(j)$. Two active vertices $j, \ell\in \CV_{\CH_n}$ are connected in $\CH_n$ if they are neighboring vertices in $B_{m_n}$, and  there exist vertices $v \in \CC_{j}, \tilde v \in \CC_{\ell}$ such that $v$ is connected by an edge to $\tilde v$ in $\CG_n$. Let us write  $\CC_{\CH_n}^\sss{(1)}$ for the largest component in $\CH_n$. Since each $\CC_j$ in a good box $\CQ_j$ has size at least $\rho n$,  the event $\{|\CC_{\CH_n}^\sss{(1)}|\ge (1-\varepsilon)m_n\}$ implies  $\{|\CC_n^\sss{(1)}[1,M^{1/(\tau-1)})|\ge(1-\varepsilon)\rho n\}$. 
Let $q, r>0$ be such that by Proposition~\ref{prop:nn-hasgiant} for a fixed $\varepsilon>0$, $A>0$ and all $n\ge 1$, \[
\Prob(|L_{m_n}^\sss{(1)}(q,r)|<(1-\varepsilon) m_n\big)\le \exp\big(-\tfrac{1}{A}m_n^{(d-1)/d}\big).
\]
We will show that we can choose $M$ sufficiently large (depending on $\varepsilon$) so that 
\begin{equation}
\CH_n\succcurlyeq \mathrm{NNP}(q,r).\label{eq:domination-nn}
\end{equation}
This domination then implies that $|\CC_{\CH_n}^\sss{(1)}|\succcurlyeq |L_{m_n}^\sss{(1)}(q,r)|$, and so
\begin{equation}\label{eq:finish-nn}
\begin{aligned}
\Prob\big(|\CC_n^\sss{(1)}[1,M^{1/(\tau-1)})|<(1-\varepsilon)\rho n\big)
    &\le 
\Prob\big(|L_{m_n}^\sss{(1)}(q,r)|<(1-\varepsilon) m_n\big)\\
    &\le 
    \exp\big(-\tfrac{1}{A}m_n^{(d-1)/d}\big)\\
    &=
    \exp\big(-\tfrac{1}{A M^{(d-1)/d}}    n^{(d-1)/d}\big),
\end{aligned}
\end{equation}
implying the statement of \cref{lemma:bootstrap} for a larger value $A$.

\emph{The auxiliary graph dominates site-bond percolation.} It remains to  prove the domination~\eqref{eq:domination-nn}.
By~\eqref{eq:domination-nn-active} and the definition of $\CH_n$ below~\eqref{eq:domination-nn-active}, each vertex $j\in[m_n]$ is independently \emph{active} with probability at least $q$ for $M$ sufficiently large. Thus, it remains to show that for two large components $\CC_j, \CC_{\ell}$ in adjacent good boxes $\CQ_j, \CQ_{\ell}$,
\begin{equation}
\Prob\big(\CC_j\nsim_{\CG_n}\CC_{\ell}\mid \text{pre-}\CG_n\big) = \prod_{u\in\CC_j, v\in\CC_{\ell}}(1-\mathrm{p}(u,v)) \longrightarrow 0,\qquad\text{as }M\to\infty.\label{eq:domination-nn-edge}
\end{equation}
We write $|\CE(\CC_j,\CC_{\ell})|$ for the number of edges between the two components. Then
 \[
\Prob\big(\CC_j\nsim_{\CG_n}\CC_{\ell}\mid \text{pre-}\CG_n\big) \le  \exp\bigg(-\sum_{u\in\CC_j, v\in\CC_{\ell}}\mathrm{p}(u,v)\bigg)
=
\exp\big(-\E[|\CE(\CC_j,\CC_{\ell})|\mid\text{pre-}\CG_n]\big).
\]
 We will show case-by-case for $\zeta\in\{\zeta_\mathrm{hl}, \zeta_\mathrm{hh}, \zeta_\mathrm{ll}\}$ that $\E[|\CE(\CC_j,\CC_{\ell})|\mid\text{pre-}\CG_n]\to\infty$ using either high-low, high-high, or low-low connections between two active neighboring boxes $j,\ell$. The maximal distance between any two vertices in neighboring boxes is  $2\sqrt{d}M^{1/d}$. 

 \smallskip\noindent
 \emph{1. Low-low connections help: $\zeta_{\mathrm{ll}}>0$.} We may assume $\alpha < \infty$ otherwise $\zeta_{\mathrm{ll}}=2-\alpha$ is negative. Every vertex has mark at least $1$. Thus, from the connection probability $\mathrm{p}$ in~\eqref{eq:connection-prob-gen} it follows that 
\begin{equation}\label{eq:expected-edges-ll}
\E[|\CE(\CC_j,\CC_{\ell})|\mid\text{pre-}\CG_n]
= \Omega\big(
M^{-\alpha}|\CC_j|\cdot|\CC_{\ell}|\big)=\Omega\big(
M^{2-\alpha}\big).
\end{equation}
 Since by assumption $\zeta_\mathrm{ll}=2-\alpha > 0$, \eqref{eq:domination-nn-edge} follows for fixed $\varepsilon$. The same calculation works for long-range percolation on $\Z^d$.

\smallskip\noindent
\emph{2. High-low connections help: $\zeta_{\mathrm{hl}}>0$.}
We may again assume $\alpha<\infty$ otherwise $\zeta_{\mathrm{hl}}$ is negative. 
On the event $\CA_\mathrm{box}(j)$, defined  in~\eqref{eq:bootstrap-boxNN}, the component $\CC_j$ contains at least $M^{\delta}/2$ vertices with mark at least $M^{(1-\delta)/(\tau-1)}$.
Then, using possible edges between a high-mark vertex in $\CC_j$ and the at least $\rho M$ vertices in $\CC_{\ell}$ that have mark at least $1$, we obtain
\begin{align}
    \E[|\CE(\CC_j,\CC_{\ell})|\mid\text{pre-}\CG_n]
    &= \Omega(M^{\delta+1})\cdot\Omega\bigg(1 \wedge \Big(\frac{M^{(1-\delta)/(\tau-1)}}{M}\Big)^\alpha\bigg)\nonumber\\&= 
\Omega(M^{1+\delta})\wedge \Omega(M^{1-\alpha(\tau-2)/(\tau-1)+\delta(1-\alpha/(\tau-1))})\label{eq:bootstrap-nn-lh}
    .
\end{align}
When $\zeta_\mathrm{hl}=1-(1-1/\alpha)(\tau-1)=(\tau-1)/\alpha-(\tau-2)>0$,  
the exponent of $M$ in~\eqref{eq:bootstrap-nn-lh} is positive for sufficiently small $\delta>0$. So \eqref{eq:domination-nn-edge} follows.

\smallskip\noindent
\emph{3. High-high connections help, $\zeta_{\mathrm{hh}}>0$.} We first assume that  $\alpha<\infty$.
Using that we only consider good boxes, and that there are at least $M^{\delta}/2 $ many high-mark vertices with mark in the interval $[M^{(1-\delta)/(\tau-1)}, M^{1/(\tau-1)})$ in both $\CC_j$ and $\CC_{\ell}$, we compute
\[
    \E[|\CE(\CC_j,\CC_{\ell})|\mid\text{pre-}\CG_n]
    =\Omega(M^{2\delta})\cdot\Omega\bigg(1 \wedge \Big(\frac{M^{(\sigma+1)(1-\delta)/(\tau-1)}}{M}\Big)^\alpha\bigg).
\]
From \eqref{eq:zeta-hh},  $\zeta_\mathrm{hh}>0$ only when $\sigma>\tau-2$. So, for $\delta$ is sufficiently small, the second factor is $\Theta(1)$. Thus, the right-hand side tends to infinity with $M$ and \eqref{eq:domination-nn-edge} follows. 
When $\alpha=\infty$, it follows analogously that $\E[|\CE(\CC_j,\CC_{\ell})|\mid\widehat\CG_n]=\Omega(M^{2\delta})$ if $\delta$ is sufficiently small since the minimum is attained at $1$ on the right-hand side above.

\smallskip\noindent
\emph{Conclusion.} We conclude that~\eqref{eq:domination-nn-edge} follows in all three cases, proving the stochastic domination~\eqref{eq:domination-nn} and finishing the proof of \eqref{eq:bootstrap-result-short}.
\end{proof}

\subsection{Verification of the requirement}
We now verify the requirement~\eqref{eq:bootstrap-condition}. To do so, we prove the following lemma that we only need for models with non-trivial vertex marks.
 We define one more exponent for each regime $\mathrm{hl, ll, hh}$:
\begin{equation}\label{eq:eta-def}
    \eta_\mathrm{hl}:=
    1-\gamma_\mathrm{hl}(\tau-1)/\alpha, \qquad 
        \eta_\mathrm{ll}:=1/\alpha, 
        \qquad 
        \eta_\mathrm{hh}:=1-\sigma\gamma_\mathrm{hh}.
\end{equation}
Using the respective values of $\gamma_\mathrm{ll}, \gamma_\mathrm{hl}, \gamma_\mathrm{hh}$  in~\eqref{eq:gamma-ll},~\eqref{eq:gamma-lh}, and~\eqref{eq:gamma-hh}, it is elementary to verify that if $\zeta_{\mathrm{type}}=\zeta_\star$ for some $\mathrm{type}\in\{\mathrm{ll},\mathrm{hl}, \mathrm{hh}\}$, then $\eta_\mathrm{type}>\gamma_\mathrm{type}$ holds (see the proof of~\cref{lemma:prerequisite-verif}). Then there are values of $\alpha$ such that $\gamma_\mathrm{ll}>\eta_\mathrm{ll}$.
 We write $\CC_n(0)[1, w)$ for the component of the origin spanned on vertices in $\CV_n[1,w)$.
\begin{lemma}[Poly-logarithmic mark-thresholds]\label{lemma:prerequisite-verif}
Consider a supercritical KSRG satisfying Assumption \ref{assumption:main} on a PPP. Assume for a  $\mathrm{type}\in\{\mathrm{ll}, \mathrm{hl}\}$ that $\zeta_\mathrm{type}=\max(\zeta_\mathrm{hl}, \zeta_\mathrm{ll})>0$.
    Whenever there exist constants $\delta \in[0, \zeta)$, $A',\rho> 0$, such that for all $n \ge 1$,  
        \begin{equation}\label{eq:hyp-lh-verif}
    \Prob\big(|\CC_n^\sss{(1)}[1, A'n^{\gamma_\mathrm{type}})|< \rho n \big) \le \exp\big(- \tfrac{1}{A'} n^{\zeta_\mathrm{type}-\delta}\big),
    \end{equation}
    then, for all $C>0$, there exists a constant $A>0$ such that for all $n$  sufficiently large,
    \begin{align}
        \Prob\big(\big|\CC_n^\sss{(1)}\big[1, (A\log n)^{\gamma_\mathrm{type}/(\zeta_\mathrm{type}-\delta)}\big)\big|\ge \rho n\big) &\ge 1-n^{-C},\label{eq:lemma-prer-ver-2}
    \end{align}
    and 
    \begin{equation}\label{eq:lemma-prer-ver-1}
        \Prob\big(\exists \mbox{ comp. }\CC\mbox{ in }\CG_n: \CC\supseteq \CV\big[(A\log n)^{\eta_\mathrm{type}/(\zeta_\mathrm{type}-\delta)}, \infty\big), |\CC|\ge \rho n\big) \ge 1-n^{-C}.
    \end{equation}
    The statements remain valid for the Palm version $\Prob^\sss{x}$ of $\Prob$ for any $x\in\R^d$.  
\end{lemma}
The proof is an adaptation of~\cite[Proposition 5.12]{clusterI} and we present it in Appendix~\ref{app:verif}. 
In a nutshell, we apply a boxing scheme ---with box sizes $\Theta((\log n)^{1/(\zeta-\delta)})$--- and apply \eqref{eq:hyp-lh-verif} to each box. To form $\CC$, we connect boxes to each other using high- and/or low-mark vertices depending on the regime. Then, we connect vertices of mark at least $(A\log n)^{\eta/(\zeta-\delta)}$ to the constant-mark vertices in  $\CC$. Computations are similar to those in \cref{sec:second-sub}.

The corollaries of 
\cref{lemma:bootstrap,lemma:prerequisite-verif}  motivate the somewhat counterintuitive formulation of \cref{lemma:prerequisite-verif}, i.e., the requirement~\eqref{eq:hyp-lh-verif} with both $\delta>0$ and $\delta=0$ allowed. 
\begin{corollary}[A giant with sharper probability]\label{cor:initial-upper}
Consider a supercritical KSRG satisfying Assumption \ref{assumption:main} on a PPP or consider supercritical long-range percolation on $\Z^d$. Assume that $\zeta_\mathrm{long}>0$, and $\zeta_\star=\zeta_{\mathrm{type}}$ for some $\mathrm{type}\in \{\mathrm{ll}, \mathrm{hl}, \mathrm{hh}, \mathrm{short}\}$. Set $\gamma_{\mathrm{short}}:=0$. There exist constants $A_1>0$ such that for all $n\ge 1$,
\begin{equation}\label{eq:ltld-weak-lh}
\Prob\big(|\CC_n^\sss{(1)}[1,A_1n^{\gamma_{\mathrm{type}}})|<\tfrac{1}{A_1} n\big)\le \exp\big(-\tfrac{1}{A_1}n^{\max(\zeta_\mathrm{long}, (d-1)/d)}\big).
\end{equation}
\begin{proof}
We first consider KSRGs on a PPP. By \cref{claim:phases}, $\zeta_\mathrm{long}=\max(\zeta_\mathrm{ll}, \zeta_\mathrm{hl}, \zeta_\mathrm{hh})$.
    When $\zeta_\mathrm{long}\in\{\zeta_\mathrm{ll},\zeta_\mathrm{hl}\}$, Proposition~\ref{proposition:ltld-weaker} proves the requirement \eqref{eq:hyp-lh-verif} of \cref{lemma:prerequisite-verif} for some $\delta>0$ and a small value of $\rho>0$. Hence,~\eqref{eq:lemma-prer-ver-2}-\eqref{eq:lemma-prer-ver-1} holds for these values of $\delta$ and $\rho$. This implies the requirement \eqref{eq:bootstrap-condition} of \cref{lemma:bootstrap}. When $\zeta_\mathrm{\star}=\zeta_\mathrm{hh}>0$, \eqref{eq:bootstrap-condition} follows from the combination of the law of large numbers \cite[Corollary 2.3]{clusterI}, and that $|\CC_n^\sss{(1)}|$ contains all vertices of at least some polylogarithmic weight in $n$ by \cite[Proposition 5.12]{clusterI}, which then gives \eqref{eq:hyp-lh-verif} and thus Lemma \ref{lemma:prerequisite-verif}. This proves \eqref{eq:ltld-weak-lh} when $\zeta_{\star}\in\{\zeta_{\mathrm{ll}}, \zeta_{\mathrm{hl}}, \zeta_{\mathrm{hh}}\}$, and gives the requirement \eqref{eq:bootstrap-condition} of \cref{lemma:bootstrap} whenever $\zeta_{\mathrm{long}}>0$. Further, the case for $\zeta_\star=\zeta_{\mathrm{short}}$ follows from \eqref{eq:bootstrap-result-short}. 
    For long-range percolation on $\Z^d$, $\zeta_\mathrm{hl}$ and $\zeta_\mathrm{hh}$ are negative, so $\zeta_\mathrm{long}=\zeta_\mathrm{ll}=2-\alpha$. The requirement \eqref{eq:hyp-lh-verif} of \cref{lemma:prerequisite-verif} is verified by~\cite[Theorem 3.2]{biskup2004scaling}. 
\end{proof}
 \end{corollary}
The actual mark-threshold so that \emph{all} vertices above that mark are in the giant is as follows.
\begin{corollary}[Poly-logarithmic mark-thresholds]\label{cor:polylog-thresholds}
    Consider a supercritical KSRG satisfying Assumption \ref{assumption:main} on a PPP or consider long-range percolation on $\Z^d$. Assume $\zeta_\mathrm{type}=\max(\zeta_\mathrm{ll}, \zeta_\mathrm{hl})>0$ for some $\mathrm{type}\in\{\mathrm{ll}, \mathrm{hl}\}$. 
    For all $C>0$, there exist constants $A, \varepsilon>0$ such that for all $n$  sufficiently large,
    \begin{align}\label{eq:giant-mark-cor}
        \Prob\big(\big|\CC_n^\sss{(1)}\big[1, (A\log n)^{\gamma_\mathrm{type}/\zeta_\mathrm{type}}\big)\big|\ge \varepsilon n\big) &\ge 1-n^{-C}, 
    \end{align}
    and
    \begin{equation}
        \Prob\big(\exists \mbox{ comp. }\CC\mbox{ in }\CG_n: \CC\supseteq \CV_n\big[(A\log n)^{\eta_\mathrm{type}/\zeta_\mathrm{type}}, \infty\big), |\CC|\ge \varepsilon n\big) \ge 1-n^{-C}.\label{eq:threshold-mark-cor}
    \end{equation}
     \end{corollary}
    \begin{proof}
        In \cref{cor:initial-upper} we have just shown that the requirement \eqref{eq:hyp-lh-verif} in \cref{lemma:prerequisite-verif} is met with $\delta=0$ for the case $\zeta_{\mathrm{type}}=\max(\zeta_{\mathrm{ll}}, \zeta_{\mathrm{hl}})$. 
        This yields the first two bounds \eqref{eq:giant-mark-cor}--\eqref{eq:threshold-mark-cor} for KSRGs on a PPP. In long-range percolation, all marks are equal to one, so that $\CV[(A\log n)^{\eta_\mathrm{type}/\zeta_\mathrm{type}},\infty)=\emptyset$, and the two bounds are implied by \cref{cor:initial-upper}.\qedhere 
    \end{proof}

In \cref{cor:initial-upper} we obtain the sharp decay on the probability of having a giant, i.e., we obtain the target decay of the lower tail in \cref{thm:large-dev2}. However, the proportion of vertices in the giant is not yet close to $\theta$. 

\section{Lower tail of large deviations}\label{sec:second-upper}
In this section we prove \cref{thm:large-dev2}. We first show uniqueness of the giant via an upper bound on the size of the second-largest component $\CC_n^\sss{(2)}$. In \cref{sec:cluster-size-lln}, we extend this upper bound on $|\CC_n^\sss{(2)}|$ to the upper bound on the cluster-size decay in \cref{thm:subexponential-decay} using our previous work~\cite{clusterI}. A Law of Large Numbers for $|\CC_n^\sss{(1)}|$ will follow. Eventually we prove the lower tail of large deviations.

\subsection{Non-giant components are small}\label{sec:second-sub}
We prove an upper bound on $|\CC_n^\sss{(2)}|$ when $\max(\zeta_\mathrm{hl}, \zeta_\mathrm{ll})>0$. When $\zeta_\mathrm{hh}>0$, we proved a similar upper bound in~\cite{clusterI} using different techniques, see \cref{remark:hh-different}. See also \cref{remark:nn-difficulty} on the difficulties one needs to overcome to extend this result to the case when surface tension dominates connections, i.e., when $\zeta_\star=\zeta_\mathrm{short}$.

\begin{proposition}[Upper bound second-largest component]\label{prop:second-upper}
Consider a supercritical KSRG on a PPP or supercritical long-range percolation on $\Z^d$. Assume $\zeta_\mathrm{type}=\max(\zeta_\mathrm{ll}, \zeta_\mathrm{hl})>0$ for some $\mathrm{type}\in\{\mathrm{ll}, \mathrm{hl}\}$.    There exists a constant $A>0$ such that for all $n>k$,
 \begin{equation}
     \Prob\big(|\CC_n^\sss{(2)}|> k \big) \le (n/k)\exp\big(-\tfrac{1}{A} k^{\zeta_\mathrm{type}}\big) .\label{eq:second-upper}
 \end{equation}
 The statement remains valid for the Palm version $\Prob^\sss{x}$ of $\Prob$ for any $x\in\R^d$.
\end{proposition}
The proof of Proposition \ref{prop:second-upper} is inspired by Sly and Crawford~\cite{crawford2012simple} who obtain a polylogarithmic upper bound with unidentified exponent on the second-largest component for long-range percolation. Contrary to long-range percolation, we have to use the vertex-marks as well when $\tau<\infty$.
We present the proof for KSRGs on Poisson point processes satisfying $\zeta_\mathrm{hl}>0$, and explain the needed changes for  $\zeta_\mathrm{ll}>0$ and for KSRGs on $\Z^d$ at the end. 
We abbreviate $\gamma=\gamma_\mathrm{hl}=1-1/\alpha$ and $\zeta=\zeta_\mathrm{hl}=1-\gamma_\mathrm{hl}(\tau-1)=(\tau-1)/\alpha-(\tau-2)$. Clearly $\zeta_\mathrm{hl}>0$ implies $\alpha<\infty$.

We let $A_{\ref{cor:initial-upper}}$ be the constant from \eqref{eq:ltld-weak-lh} in \cref{cor:initial-upper} so that for all $n\ge 1$,
\begin{equation}\label{eq:concentration-right-lh}
\Prob\big(|\CC_n^\sss{(1)}[1,A_{\ref{cor:initial-upper}}n^\gamma)|\le \tfrac{1}{A_{\ref{cor:initial-upper}}} n\big)\le \exp\big(-\tfrac{1}{A_{\ref{cor:initial-upper}}}n^\zeta\big).
\end{equation}
For simplicity we assume that $(n/k)^{1/d}\in\N$, so that $\Lambda_n$ can be partitioned into volume-$2^{-d}k$ boxes.

\begin{definition}[Boxing scheme]\label{def:second-boxing}
Fix $k\ge 1$ sufficiently large. Consider the boxes $\CQ_1, \ldots, \CQ_{2^dn/k}$ that partition $\Lambda_n$, and that are labeled so that $\CQ_j$ is adjacent to $\CQ_{j+1}$ for all $j\in [2^dn/k-1]$. We write $\CV_{k,j}$ for the PPP restricted to the box $\CQ_j$, $\CG_{k,j}$ for the graph $\CG$ restricted to $\CQ_j$.
We call vertices with mark respectively below and at least $A_{\ref{cor:initial-upper}}k^\gamma$ low-mark and high-mark vertices. We call
$\CC_{k,j}^\sss{(1)}[1,A_{\ref{cor:initial-upper}}k^\gamma)$ the local low-largest component of $\CQ_j$. We also define
\begin{equation}\label{eq:v-j-def}          \CU_{j}:=\{v\in\CV_{k,j}[A_{\ref{cor:initial-upper}}k^\gamma,\infty): v  \sim \CC_{k,j}^\sss{(1)}[1,A_{\ref{cor:initial-upper}}k^\gamma)\}.
    \end{equation}
    We then define the local giant component of box $\CQ_j$ as
    \begin{equation}\label{eq:local-giant-def}
\mathrm{LG}_j:=\CU_j\cup\CC_{k,j}^\sss{(1)}[1,A_{\ref{cor:initial-upper}}k^\gamma).
   \end{equation}
    \end{definition}
   In words, $\CU_j$ is the set of high-mark vertices in $\CQ_j$ that connect by at least one edge to the local low-largest component in the same box. So, together they form the local giant component $\mathrm{LG}_j$.
We mention that $\CU_j$ typically does not contain all high-mark vertices. 
    \begin{definition}[Revealment stages]\label{def:revealment-stages} Consider a KSRG under the same setting as Proposition~\ref{prop:second-upper} such that $\zeta_\mathrm{hl}>0$, with the boxing scheme from \cref{def:second-boxing}. Initially, we reveal the realization of vertex set $\CV$, and none of the edges. We define the following four edge-revealment stages:
    \begin{enumerate}[(1)]
        \item\label{item:stage1} Reveal edges between pairs of vertices both in the same box $\CQ_j$.
\item\label{item:stage2}
    Reveal edges between pairs of vertices both in $\cup_{j\le 2^d(n/k)}\mathrm{LG}_j$. 
    \item\label{item:stage3} Reveal remaining edges between pairs of vertices both in $\CV_{n}\setminus \big(\cup_{j\le 2^d(n/k)}\mathrm{LG}_j\big)$
    \item\label{item:stage4} Reveal all remaining edges.
    \end{enumerate}
    We write \emph{stage-$j$-$\CG_n$}    for the graph obtained after the $j$th revealment stage for $j\in\{1,2,3,4\}$.
    \end{definition}
   Thus, in stage \emph{(\ref{item:stage1})} we reveal all `local' edges between vertices inside the same box. In stage \emph{(\ref{item:stage2})}, edges between local giants of different boxes. In stage \emph{(\ref{item:stage3})}, we reveal edges with both endpoints outside local giants in different boxes. Finally, in stage \emph{(\ref{item:stage4})} we reveal all remaining edges: edges in different boxes where one endpoint is in a local giant, and the other endpoint is outside a local giant. The set $\mathrm{LG}_j$, (the local-low largest component and high-mark vertices connecting to it in the same box) is determined during stage \emph{(\ref{item:stage1})}. So stages \emph{(\ref{item:stage2}--\ref{item:stage4})} are well-defined. These latter stages reveal edges between \emph{different} boxes. 
    
    \begin{definition}[Stage\! \emph{(1)} and \emph{(2)} good events]\label{def:second-events}
    Let $\delta\in(0,1/(2^dA_{\ref{cor:initial-upper}}))$ be a small constant. 
    We call a box $\CQ_j$ 1-good if it satisfies the event
    \begin{equation}\label{eq:1-good-def}
    \CA_{\mathrm{1\textnormal{-}good}}(\CQ_j):=\{\CQ_j \ 1\text{-good}\}:=\big\{ |\CC_{k,j}^\sss{(1)}[1,A_{\ref{cor:initial-upper}}k^\gamma)|>\delta k\big\} \cap \big\{|\CU_j|> \delta k^{\zeta}\big\}.
    \end{equation}
For $j<2^d(n/k)$, we call $u\in \CU_j$ box-wedging if $u\sim \mathrm{LG}_{j+1}$.
 We say that stage-2-$\CG_n$ is $2$-good if 
        \begin{equation}\label{eq:2-good-def}
    \{\text{stage-2-$\CG_n$ 2-good}\}:=\bigcap_{j\le 2^d(n/k)}\CA_{\mathrm{1\textnormal{-}good}}(\CQ_j) \cap \bigcap_{j< 2^d(n/k)} \big\{\exists u\in\CU_j: u\sim \mathrm{LG}_{j+1}\big)\big\}.
    \end{equation}
\end{definition}
Whether $\CQ_j$ is $1$-good is entirely determined by the revealed edges in stage \emph{(1)}, and whether stage-$2$-$\CG_n$ is $2$-good, is entirely determined by edges revealed by the end of stage \emph{(2)}. The second intersection in \eqref{eq:2-good-def} 
means that there is at least one box-wedging vertex in each local giant, implying that all local giants are part of the same component in $\CG_n$ that we will call the backbone.\begin{claim}[Stage \emph{(\ref{item:stage1})} error bound]\label{claim:1good}
    Consider a KSRG under the same setting as Proposition~\ref{prop:second-upper} such that $\zeta_\mathrm{hl}>0$. There exists a constant $\varepsilon>0$ such that for all $k$ sufficiently large, for all $j\le 2^d(n/k)$,
    \[
        \Prob\big(\CQ_j\mbox{ is 1-good}\big) \ge 1-\exp(-\varepsilon k^\zeta).
    \]
    \end{claim}
    \begin{proof}  
        We estimate the probabilities of the two events in \eqref{eq:1-good-def}. By \cref{cor:initial-upper},
        \begin{equation}
        \Prob\big(\CC_{k,j}[1,A_{\ref{cor:initial-upper}}k^\gamma) \le \delta k\big)\le \exp\big(-2^{-d\zeta_{\mathrm{hh}}}k^{\zeta_\mathrm{hh}}/A_{\ref{cor:initial-upper}}\big).\label{eq:1good-pr1}
        \end{equation}
        Next, consider $v\in\CV_{n}[A_{\ref{cor:initial-upper}}k^\gamma,\infty)$ in box $\CQ_j$.
        We use  that each vertex in $\CC_{k,j}[1,A_{\ref{cor:initial-upper}}k^\gamma)$ has mark at least $1$, and that the distance between any two vertices in $\CQ_j$ is at most $\sqrt{d}k^{1/d}$. Thus, by the connection probability $\mathrm{p}$ in~\eqref{eq:connection-prob-gen} (with $\alpha<\infty$), conditional on the vertex set $\CV_n$ and the subgraph $\CC_{k,j}[1,A_{\ref{cor:initial-upper}}k^\gamma)$ having size at least $\delta k$,
        $$
        \begin{aligned}
        \Prob\big(v \nsim \CC_{k,j}[1,A_{\ref{cor:initial-upper}}k^\gamma) \, \big| \, | \CC_{k,j}[1,A_{\ref{cor:initial-upper}}k^\gamma)|& \ge \delta k, \CV_{n}\big) \le \Big(1-p\Big(1 \wedge \beta A_{\ref{cor:initial-upper}}\frac{k^\gamma}{d^{d/2}k} \Big)^\alpha\Big)^{\delta k}.
        \end{aligned}
        $$
        If the minimum is at $1$, then this probability is at most $e^{-p\delta k}$. If the minimum is at the second expression, we substitute $\gamma=(\alpha-1)/\alpha$. When we bound $1-x\le e^{-x}$, the powers of $k$ cancel, and  the right-hand side is at most
        \begin{equation}\label{eq:constconnectionproba}
        \Prob\big(v \notin \CU_j\mid | \CC_{k,j}[1,A_{\ref{cor:initial-upper}}k^\gamma)| \ge \delta k, \CV_{n}\big)\le \exp(-p\beta A_{\ref{cor:initial-upper}}^\alpha/d^{d/2}).
        \end{equation}
        In both cases, there is $\delta_1 > 0$ such that any such $v$ connects by an edge to $\CC_{k,1}^\sss{(1)}[1,A_{\ref{cor:initial-upper}}k^\gamma)$ with probability at least $\delta_1$, independently of all other vertices.
       Now, $\CV_{k,j} \cap [A_{\ref{cor:initial-upper}}k^{\gamma},\infty)$ is by~\eqref{eq:poisson-intensity} a PPP with intensity $\mathrm{Leb}\times F_W(\mathrm dw)$. Therefore, 
       the vertices in $\CV_{k,1} \cap [A_{\ref{cor:initial-upper}}k^{\gamma},\infty)$ that connect by an edge to $\CC_{k,1}^\sss{(1)}[1,A_{\ref{cor:initial-upper}}k^\gamma)$ dominate a Poisson random variable with mean 
       \[
      \delta_1\cdot \mathrm{Vol}(\CQ_j)  (A_{\ref{cor:initial-upper}}k^\gamma)^{-(\tau-1)}= \delta_1 2^{-d} A_{\ref{cor:initial-upper}}^{-(\tau-1)}k^{1-\gamma(\tau-1)}=\Theta(k^\zeta).
       \]
       Thus, by \cref{lemma:poisson-1}, there exists $\delta=\delta(A_{\ref{cor:initial-upper}}) > 0$ small enough so that 
       \[
        \begin{aligned}
        \Prob\big(|\CU_j| \le \delta k \, \big| \,  |\CC_{k,j}[1,A_{\ref{cor:initial-upper}}k^\gamma)| \ge \delta k\big) \le \exp(-\varepsilon k^{\zeta}).
        \end{aligned}
       \]
       Combined with~\eqref{eq:1good-pr1} and the definition of $\{\CQ_1\mbox{ is $1$-good}\}$ in \cref{def:second-events} this finishes the proof for $k$ sufficiently large and $\delta, \varepsilon>0$ sufficiently small.        
    \end{proof}

\begin{claim}[Stage \emph{(\ref{item:stage2})} error bound]\label{claim:second-2-good}
    Consider a KSRG under the same setting as Proposition~\ref{prop:second-upper}  such that $\zeta_\mathrm{hl}>0$. There exists a  constant $\varepsilon>0$ such that for all $k$ sufficiently large,
    \[
        \Prob\big(\mbox{stage-2-}\CG_n\mbox{ is 2-good}\big) \ge 1-2^{d+1}(n/k)\exp(-\varepsilon k^\zeta).
    \]
    \end{claim}
    \begin{proof}
By \cref{claim:1good}, each box $\CQ_i$ is $1$-good with probability at least $1-\exp(-\varepsilon k^{\zeta})$.  Denote the first event in \eqref{eq:2-good-def} that all boxes $\CQ_j$ are $1$-good by $\CA_\mathrm{sub}$. By a union bound, \begin{equation}\label{firstpart}
\Prob(\CA_{\mathrm{sub}})
\ge 1-2^d(n/k)\exp(-\varepsilon k^{\zeta}).
\end{equation}
We move on to estimating the probability of the second event in \eqref{eq:2-good-def} conditioned on $\CA_\mathrm{sub}$. Fix some $j<2^d(n/k)$. Before Stage~\emph{(\ref{item:stage2})}, none of the edges of $\CU_j$ towards vertices in $\CQ_{j+1}$ have been revealed yet. The maximal Euclidean distance between two neighboring boxes is $2\sqrt{d}k^{1/d}/2$, so we can bound the probability that a single high-mark vertex is box-wedging using \eqref{eq:connection-prob-gen}. When $\CQ_{j+1}$ is $1$-good, by \eqref{eq:1-good-def} there are at least $\delta n$ vertices of mark at least $1$ in $\mathrm{LG}_{j+1}$, so 
\begin{equation}\label{eq:box-wedging-one}
\Prob\big(u \in \CU_j \text{ not box-wedging } \big| \, \CQ_{j+1}\mbox{ $1$-good}, \CV_n \big) \le \Big(1-p\Big(1\wedge \frac{\beta A_{\ref{cor:initial-upper}}k^\gamma}{d^{d/2}k}\Big)^\alpha\Big)^{\delta k}.
\end{equation}
Then, we can use that there are at least $\delta k^\zeta$ candidate vertices to be box-wedging in a $1$-good $\CQ_j$, and that edges are present independently given the vertex set $\CV$. Thus,   
\begin{equation}\label{eq:box-wedging-two}
\begin{aligned}
\Prob\big(\nexists u\in\CU_j: u \text{ box-wedging}\, \big| \,  \CQ_{j},\CQ_{j+1}\mbox{ $1$-good}, \CV_n\big) 
&\le \Big(1-p\Big(1\wedge \frac{\beta A_{\ref{cor:initial-upper}}k^\gamma}{d^{d/2}k}\Big)^\alpha\Big)^{(\delta k)\cdot(\delta k^{\zeta})}\hspace{-6pt}.
\end{aligned}
\end{equation}
  If the minimum on the right-hand side is at $1$, this probability is at most $\exp(-\varepsilon k^{1+\zeta})$. If the minimum is at the second expression, using  that $1-x\le \exp(-x)$ and $\gamma=1-1/\alpha$, the powers of $k$ simplify since $(k^{\gamma}/k)^{\alpha} \cdot k=1$, so this probability is at most
  $
  \exp(-\varepsilon k^{\zeta})
  $
  for $\varepsilon > 0$ small enough. 
  So, integrating over the realizations $\CV_n$ and a union bound over all $2^d(n/k)$ boxes, the probability that there is at least one box $\CQ_j$ with no box-wedging vertices is at most  $2^d(n/k)  \exp(-\varepsilon k^{\zeta})$. Combined with~\eqref{eq:2-good-def} and \eqref{firstpart} the claim follows.
    \end{proof}

\begin{claim}[Stage \emph{(\ref{item:stage3}--\ref{item:stage4})} error bound: second-largest component] \label{claim:second-4-good}
    Consider a KSRG under the same setting as Proposition~\ref{prop:second-upper}  such that $\zeta_\mathrm{hl}>0$. There exists a constant $\varepsilon>0$ such that for all $k$ sufficiently large and $n>k$
    \begin{equation}\label{eq:second-largest-est}
               \Prob\big(\{|\CC_n^\sss{(2)}|> k\}\cap \mbox{stage-2-}\CG_n\mbox{ is 2-good}\big) \le (n/k)\exp(-\varepsilon k^\zeta).
    \end{equation}
    \end{claim}
    \begin{proof}
We condition on  the full realization of the stage-$2$-$\CG_n$ that is $2$-good. 
    By \eqref{eq:2-good-def}, any such realization satisfies that  all local giants $(\mathrm{LG}_j)_{j\le 2^d(n/k)}$ have size at least $\delta k$ and they are all connected by box-wedging vertices. This gives a component in $\CG_n$ of size at least $\delta n$ that contains these local giants.  We call this  component the backbone $\mathrm{BB}=\mathrm{BB}_{n,k}$.
      In Stage~\emph{(\ref{item:stage3})} we expose all edges with both endpoints outside the local giants.  If no component of size larger than $k$ except for the backbone remains after Stage~\emph{(\ref{item:stage3})}, then we are done:  In stage~\emph{(\ref{item:stage4})} all remaining edges to be revealed have one endpoint in local giants, hence in $\mathrm{BB}$, and one outside. 

We may interpret now \eqref{eq:second-largest-est} as follows: the second-largest component being larger than $k$ means that there is at least one component above size $k$ that is not $\mathrm{BB}$.
      If after Stage~\emph{(\ref{item:stage3})} there is a component $\CC\neq \mathrm{BB}$ above size $k$, then we aim to show that this component merges with $\mathrm{BB}$ in Stage~\emph{(\ref{item:stage4})} with probability at least $1-\exp( - \eps k^{\zeta})$. Since there are at most $(n/k)$ components above size $k$ after Stage~\emph{(\ref{item:stage3})}, a union bound over all such components will yield \eqref{eq:second-largest-est}. 
      
      So, suppose that there exists a component $\CC\neq \mathrm{BB}$ of size larger than $k$ after Stage~\emph{(\ref{item:stage3})} and fix a realization of stage-$2$-$\CG_n$.  After Stage~\emph{(\ref{item:stage3})} we revealed already all within-box edges, but we did not reveal edges between $\CC$ and local giants in neighboring boxes `around' vertices of $\CC$. Since the number of boxes is at least  $2^d(n/k)$ by \cref{def:second-boxing}, a subbox always has a neighboring box. Let $j(v)$ denote index of the box that vertex $v\in \CC$ belongs to, which is a.s.\ unique.
      Since all boxes are $1$-good, by \eqref{eq:1-good-def}, for each $v\in \CC$, there are at least $\delta k^{\zeta}$ high-mark vertices  at distance at most $\sqrt{d}k^{1/d}$ from $v$ in the vertex set $\CU_{j(v)+1}\subseteq \mathrm{LG}_{j(v)+1}$ in the neighboring boxes $\CQ_{j(v)-1}$ and $\CQ_{j(v)+1}$. Each vertex of $v$ has mark at least $1$. Hence, for a given component $\CC$, by the same calculations as in the proof of \cref{claim:second-2-good} around~\eqref{eq:box-wedging-one}--\eqref{eq:box-wedging-two},
      \begin{align}
      \Prob\big(\CC &\nsim \mathrm{BB} \, \big| \,  \mbox{2-good-}\CG_n, , \CV_n, \mbox{stage-3-}\CG_n\big)\nonumber\le 
       \Big(1-p\Big(1\wedge \beta A_{\ref{cor:initial-upper}}\frac{k^\gamma}{d^{d/2}k}\Big)^\alpha\Big)^{k\cdot (\delta k^{\zeta})}
       \le\exp\big(-\varepsilon k^\zeta\big).\nonumber
       \end{align}
      for $\varepsilon>0$ sufficiently small. Here we used that $\gamma=1-1/\alpha$ and thus some powers of $k$ cancel.
      By a union bound over all at most $|\CV_n|/k$ components of size at least $k$, 
      \begin{align*}
       \Prob\big(\{|\CC_n^\sss{(2)}|> k\}\cap\{\mbox{stage-2-}\CG_n\mbox{ is 2-good}\}\big)&\le        
      \E\big[\Prob\big(|\CC_n^\sss{(2)}|\ge k\, \big| \,  \mbox{2-good-}\CG_n, \CV_n, \mbox{stage-3-}\CG_n\big)\big] \\&\le \E\big[|\CV_n|/k]\exp\big(-\varepsilon k^\zeta\big)\le (n/k)\exp\big(-\varepsilon k^\zeta\big).\end{align*} 
      This finishes the proof of \eqref{eq:second-largest-est}.
    \end{proof}

We are ready to prove Proposition~\ref{prop:second-upper} that bounds $\Prob(|\CC_n^\sss{(2)}|\ge k)$.
\begin{proof}[Proof of Proposition~\ref{prop:second-upper} when $\zeta_{\mathrm{hl}}>0$]
    By the law of total probability,
    \[ 
    \Prob(|\CC_n^\sss{(2)}|\ge k)
    \le  
    \Prob\big(\{|\CC_n^\sss{(2)}|\ge k\} \cap  \mbox{stage-2-}\CG_n\mbox{ is 2-good}\big)+\Prob\big(\mbox{stage-2-}\CG_n\mbox{ is \emph{not} 2-good}\big).   
    \]
    Proposition~\ref{prop:second-upper} follows from Claims~\ref{claim:second-2-good}--\ref{claim:second-4-good}.
\end{proof}
We now give a sketch on how to modify the above proof when $\zeta_{\mathrm{ll}}>0$.
\begin{proof}[Proof of Proposition~\ref{prop:second-upper} when $\zeta_{\mathrm{ll}}>0$]\label{proof:lowlow-second} The proofs are identical for KSRGs on PPPs and on $\Z^d$.
 We follow the $\zeta_\mathrm{hl}$-computations. 
  By \cref{cor:initial-upper},
 \begin{equation}     \Prob\big(|\CC_n^\sss{(1)}\big[1,A_{\ref{cor:initial-upper}})|\le \tfrac{1}{A_{\ref{cor:initial-upper}}} n \big) \le \exp\big(-\tfrac{1}{A_{\ref{cor:initial-upper}}} n^{\zeta_\mathrm{ll}}\big).
 \end{equation}
 We set $\mathrm{LG}_j:=\CC_n^\sss{(1)}\big[1,A_{\ref{cor:initial-upper}})$ to be the local giant of box $\CQ_j$. We ignore high-mark vertices here.
 The revealment scheme in \cref{def:revealment-stages} is then well-defined with this $\mathrm{LG}_j$. We modify the box being $1$-good in \eqref{eq:1-good-def} by only requesting the first event $|\mathrm{LG}_j|=|\CC_n^\sss{(1)}\big[1,A_{\ref{cor:initial-upper}})|\ge \delta k$ there. We define $u\in \mathrm{LG}_j$ to be box-wedging if $u\sim \mathrm{LG}_{j+1}$.  
 In the definition of the stage $2$-$\CG_n$ being $2$-good in \eqref{eq:2-good-def}, we then require at least one box-wedging vertex in each box:
  $$
 \bigcap_{j<2^d(n/k)} \{\exists u\in \mathrm{LG}_j: u\sim \mathrm{LG}_{j+1}\}.$$
  \cref{claim:1good} follows then from \cref{cor:initial-upper} as before. We describe how to prove the equivalent of \cref{claim:second-2-good} for this regime. Estimating that all boxes are $1$-good goes as before in \eqref{firstpart}. For the second part, estimating that each box has a box-wedging vertex, we use only the fact that vertices have weight at least $1$. We get
  \begin{align*}
  \Prob\big(\nexists u\in\mathrm{LG}_j: u\sim \mathrm{LG}_{j+1} \, &\big| \,  \CQ_{j},\CQ_{j+1}\mbox{ are $1$-good}\big)  \le  \Big(1-p \Big(1\wedge \beta \frac{1}{d^{d/2}k} \Big)^\alpha\Big)^{(\delta k)^2}.
  \end{align*}
 If the first term is the minimum, this probability is at most $\exp(-\varepsilon k^2) \le \exp(-\varepsilon k^{2-\alpha})=\exp(-\varepsilon k^{\zeta_\mathrm{ll}})$ for $\varepsilon>0$ sufficiently small. If the second term is the minimum, this probability is also at most $\exp(-\varepsilon k^{2-\alpha})=\exp(-\varepsilon k^{\zeta_\mathrm{ll}})$. The rest of the claim is as for the hl-regime. For \cref{claim:second-4-good}, we use again that each local giant contains at least $\delta k$ vertices of mark at least $1$. Thus,
   $$
      \Prob\big(\CC \nsim \mathrm{BB} \, \big| \, \mbox{2-good-$\CG_n$}, \CV_n\big)\le 
       \Big(1-p\Big(1\wedge \beta \frac{1}{d^{d/2}k}\Big)^\alpha\Big)^{(\delta k)\cdot k}.
      $$     
     As before, the right-hand side is at most $\exp(-\varepsilon k^{2-\alpha})= \exp(-\varepsilon k^{\zeta_\mathrm{ll}})$. The proof of the claim is finished as before.
\end{proof}

We move towards the proofs of \cref{thm:second-largest}. We recall a proposition from our recent paper to prove the lower bounds.
To make notation lighter, we state it only for the case that $\zeta_{\mathrm{long}}>0$. We recall from \eqref{eq:multiplicity} that $\max(\CZ)$ is the maximum and $\mathfrak{m}_{\CZ}$ is the multiplicity of the maximum in the set $\CZ=\{\zeta_{\mathrm{short}}, \zeta_{\mathrm{ll}}, \zeta_{\mathrm{hl}}, \zeta_{\mathrm{hh}}\}$. We write here $\CC_{n}(0)[1,A)$  for the component of a vertex with unknown weight at $0\in \R^d$ inside the vertex set on $\Lambda_n \times [1, A)$, with the idea that if $w_0>A$ then  $\CC_n(0)[1,A):=\emptyset$.   

\begin{proposition}[Lower bound prerequisite {\cite[Prop.~7.1]{clusterI}}]\label{prop:prer-lower}
 Consider a supercritical KSRG on a PPP or on $\Z^d$ satisfying Assumption \ref{assumption:main} with $\zeta_\mathrm{long}>0$. Assume that there exist constants $\eta', \rho>0$ such that for all $n$ sufficiently large,
 \begin{align}
  \Prob^\sss{0}\big(\big|\CC_{n}(0)\big[1,(\log n)^{\eta'}\big)\big|\ge \rho n\big)
  \ge
  \rho.
  \label{eq:lower-sub-cond}
 \end{align}
 Then there exists $A>0$ such that for all $n\in [Ak,\infty]$,
 \begin{equation}
  \Prob^\sss{0}\big(|\CC_n(0)|> k, 0\notin\CC_n^\sss{(1)}\big) \ge
  \exp\big(-Ak^{\max(\CZ)}(\log k)^{\mathfrak{m}_\CZ-1}\big).\label{eq:lower-sub-iso-bound}
 \end{equation}
 Moreover, there exist $ \delta, \varepsilon>0$, such that for all $n$ sufficiently large,
 \begin{equation}
  \Prob\big(|\CC_{n}^\sss{(2)}|> \big((\varepsilon\log n)/(\log\log n)^{\mathfrak{m}_\CZ-1}\big)^{1/\max(\CZ)}\big)\ge 1-n^{-\delta}.\label{eq:lower-second-bound}
 \end{equation}
\end{proposition}
The proof of this proposition can be found in \cite[Section 7]{clusterI}. The key idea is to find a localized component containing $0$ in the ball of radius $\Theta(k)$ around $0$, simultaneously ensuring that no edges leave this ball. The probabilistic cost of this event gives the lower bound on the cluster size-decay.  The lower bound on the second-largest component can be derived from there, using a tesselation of the original space with boxes of volume $\Theta(n^{\varepsilon})$. Each tile corresponds to a repeated trial of the local event, and we prove that the small balls inside the tiles do not connect to other tiles with sufficiently high probability. See \cite[Section 7.3]{clusterI} for more details.
The next claim verifies the requirement~\eqref{eq:lower-sub-cond} for applying \cref{prop:prer-lower}.
\begin{claim}[The prerequisite holds]\label{claim:lower-bound-prer}
Consider a KSRG under the same setting as \cref{thm:second-largest} and \cref{thm:subexponential-decay}. There exist constants $\eta',\rho>0$ such that 
\begin{equation}\label{eq:origin-giant}
\Prob^\sss{0}\big(|\CC_n(0)[1,(\log n)^{\eta'}|\ge \rho n\big)\ge \rho.
\end{equation}
\end{claim}
\begin{proof}
    We prove the bound \eqref{eq:origin-giant} for KSRGs on a Poisson point process and leave it to the reader to adapt it to long-range percolation.
    We rely on Corollary \ref{cor:polylog-thresholds}. Fix $(\zeta,\gamma)\in\{(\zeta_\mathrm{hl},\gamma_\mathrm{hl}), (\zeta_\mathrm{ll},\gamma_\mathrm{ll})\}$ such that $\zeta>0$.
        Let $A,\eps>0$  be so that~\eqref{eq:giant-mark-cor} holds with $C=2$. 
For any $u\in \CV_{n}$, let $\widetilde\CC_{n}(u)$ be the component of $u$ in the induced graph in $\Lambda_n \times[1, (A\log n)^{\gamma/\zeta})$; and let $\widetilde\CC_{2^dn,u}(u)$ be the connected component of $u$ in $\Lambda_{2^dn}(u)\times  [1, (A\log n)^{\gamma/\zeta})$. Since $\Lambda_{n}\subseteq \Lambda_{2^dn}(u)$, clearly  $\widetilde\CC_{n}(u) \subseteq \widetilde\CC_{2^dn,u}(u)$.
 If the giant of $\Lambda_n \times[1, (A\log n)^{\gamma/\zeta})$ has size at least $\eps n$, then there are at least $\eps n$ vertices in a component of size at least $\eps n$ in  $\Lambda_n \times[1, (A\log n)^{\gamma/\zeta})$. This happens with probability at least $1-n^{-2}$ by \eqref{eq:giant-mark-cor}. So,
\begin{equation}\label{eq:before-mecke}
(1-n^{-2})\eps n\le \E\bigg[\sum_{u\in\CV_{n}}\ind{|\widetilde\CC_{n}(u)|\ge \varepsilon n}\bigg] \le \E\bigg[\sum_{u\in\CV_{n}}\ind{|\widetilde\CC_{2^dn,u}(u)|\ge \varepsilon n}\bigg].
\end{equation}
We now show that 
\begin{equation}
\E\bigg[\sum_{u\in\CV_{n}}\ind{\widetilde\CC_{2^dn,u}(u)|\ge \varepsilon n}\bigg] = n\Prob^\sss{0}\big(\CC_{2^dn}(0)[1, (A\log n)^{\gamma/\zeta})|\ge \varepsilon n\big).\label{eq:mecke-to-show}
\end{equation}
For this, we use an alternative KSRG construction. Consider the collection $\bm{U}=(U_{ij})_{j> i\ge 1}$ of independent $\mathrm{Unif}[0,1]$ random variables, independent of the PPP $\CV_n$. We can use $\bm{U}$ to encode the presence of edges: given the PPP $\CV$, order the points with respect to their distance to the origin, and include an edge between the $i$th and $j$th vertex $u_i$ and $u_j$ if $U_{ij}\le \mathrm{p}(u_i,u_j)$. The function $\ind{|\widetilde\CC_{2^dn,u}|\ge \varepsilon n}$ is determined by the realization of the PPP $\CV_n$ and $\bm{U}$ (i.e., it is measurable w.r.t.\ the sigma-algebras generated by these processes/collections). By Mecke's formula and Fubini's theorem, 
\begin{align*} 
 \E\bigg[\sum_{u\in\CV_{n}}\ind{\widetilde\CC_{2^dn, u}(u)|\ge \varepsilon n}\bigg] &= \E\bigg[\E\bigg[\sum_{u\in\CV_{n}}\ind{\widetilde\CC_{2^dn, u}(u)|\ge \varepsilon n}\mid \bm{U}\bigg]\bigg] \\&= \E\bigg[\int_{x\in\Lambda_{n}}\Prob^\sss{x}\big(\widetilde\CC_{2^dn, u}(u)|\ge \varepsilon n\mid \bm{U}, x_u=x\big)\rd x\bigg] \\
 &=\int_{x\in\Lambda_{n}}\Prob^\sss{x}\big(\widetilde\CC_{2^dn, u}(u)|\ge \varepsilon n\mid x_u=x\big)\rd x.
\end{align*}
The box $\Lambda_{2^dn}(u)$ is centered at $u$. So, by translation invariance, the integrand is constant and equal to $\Prob^\sss{0}\big(\widetilde\CC_{2^dn}(0)|\ge \varepsilon n\big)$. Since $\widetilde\CC_{2^dn}(0)=\CC_{2^dn}(0)[1, (A\log n)^{\gamma/\zeta})$ by definition above~\eqref{eq:before-mecke}, this proves~\eqref{eq:mecke-to-show}. Combined with~\eqref{eq:before-mecke}, this proves~\eqref{eq:origin-giant}  for any $\eta'>\gamma/\zeta$,  $\rho=2^{-d}\eps$, and any $n$ sufficiently large. The statement is trivial for smaller $n$ by decreasing $\rho$, finishing the proof.
\end{proof}

We are ready to prove \cref{thm:second-largest} on the second-largest component.
\begin{proof}[Proof of \cref{thm:second-largest}] We give the proof for KSRGs on PPPs only: the proof for long-range percolation $\Z^d$ is identical to the proof of the $\zeta_\mathrm{ll}$-regime for KSRGs.  Assume $\zeta_\mathrm{hl}$ or $\zeta_\mathrm{ll}$ is strictly positive and maximal in $\CZ=\{\zeta_{\mathrm{short}}, \zeta_{\mathrm{ll}}, \zeta_{\mathrm{hl}}, \zeta_{\mathrm{hh}}\}$. The lower bound holds by Proposition \ref{prop:prer-lower}, since its prerequisite~\eqref{eq:lower-sub-cond} is satisfied by \cref{claim:lower-bound-prer}. 
The upper bound follows from Proposition~\ref{prop:second-upper}, substituting $k=k_n=(A\log n)^{1/\max(\zeta_\mathrm{hl}, \zeta_\mathrm{ll})}$ for a  large constant $A=A(\delta)>0$.  
\qedhere
\end{proof}

\subsection{Cluster-size decay and law of large numbers}\label{sec:cluster-size-lln}
We proceed to the proof of \cref{thm:subexponential-decay} on the cluster-size decay. The upper bound~\eqref{eq:second-upper} is the main prerequisite~\eqref{eq:prer-second} of the following general proposition. We proved this proposition for KSRGs in general in \cite{clusterI}. We recall that $\Prob^\sss{x}$ is the Palm measure of the PPP with a vertex at $x$.
\begin{proposition}[Prerequisites for cluster-size decay {\cite[Proposition~6.1]{clusterI}}]\label{prop:prer-upper}
   Consider a supercritical KSRG on a PPP or on $\Z^d$ satisfying Assumption~\ref{assumption:main}. Assume that there exist $\zeta$, $\eta'$, $c$, $c'$, $M_c>0$, and a function $n_0(k)=O(k^{1+c'})$ such that for all $k$ sufficiently large and whenever $n\in[n_0(k),\infty)$, with $\overline{w}(n, c):=M_c(\log n)^{\eta'}$ it holds for all $x\in\Lambda_n$,
    \begin{align}
        \Prob^\sss{x}\big(|\CC_n^\sss{(2)}|> k\big)&\le n^{c'}\exp\big(-ck^\zeta\big), \label{eq:prer-second}\\
        \Prob^\sss{x}\big(|\CC_n^\sss{(1)}|< n^c\big)&\le n^{-1-c},\label{eq:prer-largest}\\
        \Prob^\sss{x}\big(\exists v\in\CV_n[\overline{w}(n, c),\infty): v\notin\CC_n^\sss{(1)}\big)&\le n^{-c}.\label{eq:prer-mark}
    \end{align}
    Then there exists a constant $A>0$ such that for all $k$ sufficiently large and $n\in[n_0(k),\infty]$, 
    \begin{equation}\label{eq:repeat-CLD} 
    \Prob^\sss{0} \big(|\CC_n(0)|> k, 0\notin\CC_n^\sss{(1)}\big)\le \exp\big(-\tfrac{1}{A}k^\zeta\big),
    \end{equation}
    and 
    \begin{equation} \label{eq:repeat-LLN}
    \frac{|\CC_n^\sss{(1)}|}{n} \overset{\Prob}\longrightarrow \theta(\beta, p, \alpha, \tau, \sigma),\qquad \mbox{as }n\to\infty. 
    \end{equation}
\end{proposition}
In our current setting, \eqref{eq:second-upper} in Proposition \ref{prop:second-upper} proves the first prerequisite and Corollary \ref{cor:polylog-thresholds} proves the second, and implies the third prerequisite. We spell out the details. This is similar to \cite[Proposition 5.12]{clusterI}.
\begin{claim}\label{claim:prer-mark}
    Consider a supercritical KSRG on a PPP satisfying Assumption \ref{assumption:main} or consider supercritical long-range percolation on $\Z^d$. Assume $\zeta_\mathrm{type}=\max(\zeta_\mathrm{ll}, \zeta_\mathrm{hl})>0$ for some $\mathrm{type}\in\{\mathrm{ll}, \mathrm{hl}\}$.  For each constant $C>0$, there exists a constant $A>0$ such that all $n$ sufficiently large, 
    \[
    \Prob\big(\exists v\in\CV_n[(A \log n)^{\eta_\mathrm{type}/\zeta_\mathrm{type}},\infty): v\notin\CC_n^\sss{(1)}\big)\le n^{-C}.
    \]
    The statement remains true for the Palm-version $\Prob^\sss{x}$ for any $x\in\R^d$.
    \begin{proof}
        We omit the subscript $\mathrm{type}$ in the following proof for $\mathrm{type}\in\{\mathrm{ll}, \mathrm{hl}\}$.
        Let $A=A(C)>0$ be sufficiently large. Using the upper bound on $|\CC_n^\sss{(2)}|$ in  \cref{prop:second-upper},
        and the lower bound on $|\CC_n^\sss{(1)}|$ from \cref{cor:polylog-thresholds}, we obtain 
    \begin{align}
    \Prob^\sss{x}\big(\{|\CC_n^\sss{(2)}|\le(A \log n)^{1/\zeta}\}\cap\{|\CC_n^\sss{(1)}|\ge (1/A) n\}\big)\ge 1-o(n^{-C}).
    \end{align}
    Hence,  
    \begin{align*}
    \Prob^\sss{x}\big(\exists v\in\CV_n[(A\log n)^{\eta/\zeta},\infty): v\notin\CC_n^\sss{(1)}\big)\le 
    \Prob^\sss{x}\left(\begin{aligned}
    &\{\exists v\in\CV_n[(A\log n)^{\eta/\zeta},\infty): v\notin\CC_n^\sss{(1)}\}\\ 
    &\cap \{|\CC_n^\sss{(2)}|\le(A\log n)^{1/\zeta}\}\\ 
    &\cap\{|\CC_n^\sss{(1)}|\ge \tfrac{1}{A} n\}\end{aligned}\right) + o(n^{-C}).
    \end{align*}
    On the intersection of the three events on the right-hand side, there is no linear-sized component in $\CG_n$  that contains all vertices with mark above $(A\log n)^{\eta/\zeta}$, which is exactly the negation of the event in \eqref{eq:threshold-mark-cor}. This has probability at most $o(n^{-C})$ by \eqref{eq:threshold-mark-cor} in \cref{cor:polylog-thresholds}, finishing the proof.
    \end{proof}
\end{claim}

Propositions \ref{prop:prer-lower} and \ref{prop:prer-upper}  give the respective lower and upper bounds for \cref{thm:subexponential-decay}.
\begin{proof}[Proof of \cref{thm:subexponential-decay}] We prove \cref{thm:subexponential-decay} when $\max(\zeta_\mathrm{hl}, \zeta_\mathrm{ll})>\max(\zeta_\mathrm{hh}, 0)$.     The lower bound follows from \eqref{eq:lower-sub-iso-bound} in \cref{prop:prer-lower}, since the prerequisite \eqref{eq:lower-sub-cond} is satisfied by \cref{claim:lower-bound-prer}. For the upper bound we verify the three prerequisites needed in Proposition~\ref{prop:prer-upper}. Prerequisite~\eqref{eq:prer-second} about the second-largest component is proven in Proposition~\ref{prop:second-upper}; prerequisite~\eqref{eq:prer-largest} requiring that the largest component is at least polynomial in size follows from \cref{cor:polylog-thresholds}; prerequisite~\eqref{eq:prer-mark} is proven in \cref{claim:prer-mark}. The upper bound in \cref{thm:subexponential-decay} then follows from \eqref{eq:repeat-CLD}. 
\end{proof}

 \subsection{Lower tail of Large Deviations}\label{sec:lower-tail}
 We are ready to combine everything and give the proof of the lower tail of large deviations for the giant in~\eqref{eq:thm-ltld}.
We cite the lower bound from \cite{clusterI}. We write $\CZ=\{\zeta_\mathrm{hh},\zeta_\mathrm{hl}, \zeta_\mathrm{ll}, \zeta_\mathrm{short}\}$, with maximum value $\zeta_{\star}$, and $\mathfrak{m}_{\CZ}$ the multiplicity of the maximum, see \eqref{eq:multiplicity}.

\begin{theorem}[Lower bound for the lower tail {\cite[Theorem 2.7]{clusterI}}] \label{thm:large-dev-lower}Consider a supercritical KSRG on a PPP satisfying Assumption \ref{assumption:main} or consider supercritical long-range percolation on $\Z^d$. There exists a constant $A>0$  such that for all $\rho>0$, and $n$ sufficiently large, 
    \[ 
    \Prob\big(|\CC_n^\sss{(1)}|<\rho n\big)\ge \exp\big(-\tfrac{1}{A\rho}n^{\zeta_{\star}}(\log n)^{\mathfrak{m}_\CZ-1}\big).
    \]
\end{theorem}
\begin{proof}[Proof of  \cref{thm:large-dev2}]
The lower bound follows from \cref{thm:large-dev-lower}, since the above theorem is valid for all $\rho>0$ (but is far from being sharp when $\rho\ge\theta$). When $\max(\zeta_\mathrm{hl}, \zeta_\mathrm{ll})>0$ we argue as follows for the upper bound: since we verified the three prerequisites needed in Proposition~\ref{prop:prer-upper} already in the proof of \cref{thm:subexponential-decay} above, the second conclusion of Proposition~\ref{prop:prer-upper} gives that 
\[
|\CC_n^\sss{(1)}| \big/n\overset\Prob\longrightarrow \theta.
\]
Hence, we may choose any $\widetilde\rho\in(\rho, \theta)$, and it follows that $\{|\CC_n^\sss{(1)}|\ge\widetilde \rho n\}$  whp.
 All highest-mark vertices are in the giant whp by \cref{claim:prer-mark}. This proves the requirement~\eqref{eq:bootstrap-condition} in Proposition \ref{lemma:bootstrap}, i.e.,  $\Prob(A_{\mathrm{giant}}^{\mathrm{type}}(\widetilde \rho))\to 1$ holds.
Applying \cref{lemma:bootstrap} for  $\varepsilon>0$ such that $\rho=(1-\varepsilon)\widetilde\rho$ proves the upper bound when $\max(\zeta_\mathrm{hl}, \zeta_\mathrm{ll})>0$.

When $\zeta_\mathrm{hh}>0$, $|\CC_n^\sss{(1)}|$ contains all vertices of mark above some polylog of $n$ by \cite[Proposition 5.12]{clusterI}, and an LLN for $|\CC_n^\sss{(1)}|$ also holds, see \cite[Corollary 2.3]{clusterI}. So, the same reasoning as above can be applied also when $\zeta_\mathrm{hh}>0$. Since $\zeta_\mathrm{long}=\max(\zeta_\mathrm{hl}, \zeta_\mathrm{hh}, \zeta_\mathrm{ll})$ by \cref{claim:phases}, this concludes the proof of \cref{thm:large-dev2}.
\end{proof}

\section{Upper tail of large deviations}\label{sec:upper-tail-ldp}
In this section we deal with the upper tail, that is, with the event of the giant being too large, and prove \cref{upper-lrp-nnp} and \cref{thm:large-dev-upper} following the roadmap outlined in Section~\ref{roadmapuppertail}.
We first show a Law of Large Numbers (LLN) for components of size $\ell$ that applies to both homogeneous models, i.e., $\tau=\infty$, and degree-inhomogeneous models, i.e., $\tau<\infty$. Afterwards, we obtain bounds on the convergence rate of this LLN. We prove that the concentration is exponentially fast when $\tau=\infty$, and show that \cref{upper-lrp-nnp} follows. Then we  analyze the effect of the presence of high-mark vertices (when $\tau<\infty$) to prove \cref{thm:large-dev-upper}.

    \subsection{Concentration of size-$\ell$ components}\label{sec:upper-tail}
        We introduce notation for the probability of the component of the origin being of size $\ell$, and for the total number of size-$\ell$ components in the graph induced by vertices with mark at most $\overline w_n$:
\begin{equation}
    \theta_\ell := \Prob^\sss{0}\big(|\CC(0)|=\ell\big), \quad 
    S_{n,\ell}(\overline{w}_n):=\big|\{\text{comp. }\CC\text{ in }\CG_n[1,\overline w_n): |\CC|=\ell\}\big|, 
    \quad S_{n,\ell}:=S_{n,\ell}(\infty).\label{eq:thetak}
\end{equation}
Clearly $\sum_{1\le \ell<\infty}\theta_\ell=\Prob^\sss{0}(|\CC(0)|<\infty)=1-\theta$. 
In models with all marks equal $1$ (e.g.~long-range percolation) it holds $S_{\ell, n}(\overline w_n)=S_{\ell,n}$ for any $\overline w_n>1$. 
\begin{lemma}[LLN for size-$\ell$ components]\label{lemma:size-k-local}
Consider a KSRG on a PPP or on $\Z^d$ satisfying Assumption \ref{assumption:main}.
    Let $\underline w_n=\omega(n^{1/(\tau-1)})$ be a sequence. Then, for every fixed $\ell\in\N$,
    \begin{equation}
    \label{eq:claim-lwl}
    \frac{\ell S_{n,\ell}(\underline w_n)}{n}\overset{\Prob}\longrightarrow \theta_\ell,\qquad\text{as }n\to\infty.
    \end{equation}
\end{lemma}
\begin{proof}
We give the proof for KSRGs on a PPP with $\tau<\infty$, i.e., non-constant marks. The proof is analogous when $\tau=\infty$ or when vertex locations are given by $\Z^d$.
By the intensity of $\CV_n$ in~\eqref{eq:poisson-intensity}, $\Lambda_n$ contains no vertices of mark at least $\underline w_n$ whp. If such vertices are absent, then $S_{n,\ell}(\underline w_n)=S_{n,\ell}(\infty)$. Therefore, for any $\psi>0$,
    \begin{align}
    \Prob\big(|\ell S_{n,\ell}(\underline w_n)/n - \theta_\ell|>\psi\big)&\le \Prob\big(\{|\ell S_{n,\ell}(\infty)/n - \theta_\ell|>\psi\}\cap\CV_n[\underline w_n,\infty)=\emptyset\big) \nonumber\\&\hspace{15pt}+ \Prob\big(\CV_n[\underline w_n,\infty)\neq\emptyset\big)\nonumber\\
    &\le \Prob\big(|\ell S_{n,\ell}(\infty)/n - \theta_\ell|>\psi\big) +o(1).\label{eq:lwl-comp-size}
    \end{align}
    We use local convergence to show that the probability on the right-hand side vanishes. We refer to~\cite{Hofbook2} for an elaborate discussion on local convergence and give a short definition. 

    A rooted graph is a couple $(G, o)$ of a graph $G$ and some, possibly random, distinguished vertex $o$ of $G$, which we call the root of $G$. A finite rooted graph $(G, o)$ is said to be uniformly rooted, if $o$ is chosen uniformly at random among the vertices of $G$. Let $\CG_\star$ be the space of all rooted locally finite graphs. Let $B_G(v, r)$ denote the induced subgraph of $G$ on all vertices that are at graph distance at most $r$ from a vertex $v$.  For $(G_1, o_1), (G_2, o_2)\in \CG_\star$, we define $R(G_1, G_2)$ to be the largest value $r\in\N$ such that $B_{G_1}(o_1, r)$ is isomorphic to $B_{G_2}(o_2, r)$: there exists a bijection $\phi: V(B_{G_1}(o_1, r))\mapsto V(B_{G_2}(o_2, r)$ such that $\phi(o_1)=o_2$ and $\{u,v\}$ is an edge in $B_{G_1}(o_1, r)$ if and only if $\{\phi(u),\phi(v)\}$ is an edge in $B_{G_2}(o_1, r)$. We set $d((G_1, o_1), (G_2, o_2)):=1/(1+R(G_1, o_1), (G_2,o_2)))$ for the distance between two rooted graphs.
    A sequence of uniformly rooted graphs $(G_n, o_n)_{n \ge 1}$ converges locally in probability towards $(G_{\infty}, o)$ having law $\mu$, if for every bounded and continuous function $h:\CG_\star\mapsto\R$,
    \begin{equation}\label{eq:local-l1}
    \E\big[h(G_n, o_n) \mid G_n\big]\overset\Prob\longrightarrow\E_\mu[h(G_\infty, o)],\qquad \text{as }n\to\infty,
    \end{equation}   
    where the expectation on the left-hand side is only with respect to the uniform root. (The limit object $(G_\infty, o)$ may be random and infinite, we here omit this discussion).
    
Considering KSRGs in Definition \ref{def:ksrg}, the sequence $(\CG_n)_{n \ge 1}$ of finite KSRGs converges locally in probability to the infinite graph $\CG_\infty$ conditioned to contain a vertex at the origin, the root, see~\cite{maitra2021locallim}. We let $h(G_n,o_n)=\ind{|\CC(o_n)|=\ell}$ be the indicator that the component of $o_n$ has size $\ell$.
Then, $h$ is bounded, and continuous on the Polish space $(\CG_\star, d)$: it only depends on the induced subgraph up to distance $\ell+1$ from the root, so $h(G_1, o_1)=h(G_2, o_2)$ for any two rooted graphs $(G_1, o_1), (G_2,o_2)$ within distance at most $1/(\ell+2)$. So \eqref{eq:local-l1} applies and the right-hand side there equals $\Prob^\sss{0}(|\CC(0)|=\ell)=\theta_\ell$. Instead of explicitly writing out the left-hand side, we first notice that $\sum_{o_n\in\CV_n}h(G_n,o_n)=\ell S_{n,\ell}(\infty)$. 
Therefore, 
\begin{equation}\label{eq:local-l2}
\frac{\ell S_{n,\ell}(\infty)}{n}=\frac{|\CV_n|}{n}\frac{1}{|\CV_n|}\sum_{o_n\in\CV_n}h(G_n, o_n)=\frac{|\CV_n|}{n}\cdot \E[h(G_n, o_n)\mid G_n] \overset\Prob\longrightarrow 1\cdot \theta_\ell,
\end{equation}
where we used that the ratio $|\CV_n|/n$ tends to one in probability. Hence, the probability on the right-hand side in~\eqref{eq:lwl-comp-size} tends to zero as $n$ tends to infinity. This finishes the proof for PPP vertex sets. When the locations are given by $\Z^d$, or $\tau=\infty$, the result follows analogously.
\end{proof}

\subsection{Homogeneous models}
We prove a bound on the convergence rate in \cref{lemma:size-k-local} using a boxing argument. 
        
\begin{lemma}[Concentration for size-$\ell$ components]\label{lemma:size-k-nnp-lrp}
    Consider long-range percolation with $\alpha>1$ or Bernoulli bond percolation on $\Z^d$. For all fixed $\ell\in\N$ and $\psi>0$, there exists a constant $\varepsilon_{\ell,\psi}>0$ such that for all $n\ge 1$,
    \begin{equation}\label{eq:local-l-rate}
    \Prob\Big( \big| S_{n,\ell}(\infty)/(n\theta_\ell/\ell)-1 \big| \le \psi \Big)\ge 1- \exp\big(-\varepsilon_{\ell,\psi} n\big).
    \end{equation}
\end{lemma}
\begin{proof}
Let $k$ be a large constant, and w.l.o.g.\ assume 
that $\Lambda_n$ can be partitioned into $(n/k)$ volume-$k$ subboxes. Let $S_{k,\ell}^\sss{(j)}$ denote the number of size-$\ell$ components in the induced subgraph inside the $j$th box. Then, by \cref{lemma:size-k-local}, for a small constant $\psi_1>0$, independently across boxes it holds that
\[
\Prob(\CA_\ell(j, k)):=\Prob\Big(|S_{k,\ell}^\sss{(j)}/ (k\theta_\ell/\ell)-1|\le \psi_1 \Big) \ge 1-\psi_1/2.
\]
The number of boxes for which this event holds dominates a binomial  random variable with parameters $n/k $ and $1-\psi_1/2$. So if we set
\[
\CA_\mathrm{sub}(k):=\bigg\{\Big|\sum_{j\in[n/k]}\ind{\CA_\ell(j,k)} -n/k\Big|\le  \psi_1 n/k\bigg\},
\]
then by a Chernoff bound, there exists a constant $c=c(\psi_1)>0$ such that 
\begin{equation}
\Prob\big(\neg\CA_\mathrm{sub}(k)\big)\le \exp\big(-c (n/k)\big).\label{eq:sub-lrp}
\end{equation}
On $\CA_\mathrm{sub}(k)$, since $(n/k) k \theta_\ell/\ell=n \theta_\ell/\ell$, the number of size-$\ell$ components of in the union of the induced subgraphs satisfies
\begin{equation}\label{eq:snell-bound-1}
(1-\psi_1)^2n\theta_\ell/\ell\le \sum_{j\in[n/k]}S_{k,\ell}^\sss{(j)}   \le (1+\psi_1)^2 n\theta_\ell/\ell \qquad \mbox{on} \  \CA_\mathrm{sub}(k).
\end{equation}
We now look at edges across subboxes. Let $\CV_n^\mathrm{cross}(k)$ denote the set of vertices having an edge to a vertex in a different subbox.  
Each vertex in $\CV_n^\mathrm{cross}(k)$ changes the number of size-$\ell$ components by at most $1$: indeed, this number increases by one when its crossing edge(s) form a new component of size-$\ell$, and it decreases by one if without crossing edges its size was $\ell$.
Therefore, 
\begin{equation}\label{eq:snell-bound-2}
\Big|S_{n,\ell}(\infty)-\sum_{j\in[n/k]}S_{k,\ell}^\sss{(j)}\Big| \le  |\CV_n^\mathrm{cross}(k)|. 
\end{equation}
Hence, for sufficiently small $\psi_1$ depending on $\psi$ and $\ell$, on the event on $ \{|\CV_n^\mathrm{cross}(k)|<\psi_1n\big\}\cap\CA_\mathrm{sub}(k)$, 
\[
(1-\psi)n\theta_\ell/\ell\le (1-\psi_1)^2n\theta_\ell/\ell -\psi_1n\le S_{n,\ell}(\infty) \le (1+\psi_1)^2n\theta_\ell/\ell +\psi_1n\le (1+\psi)n\theta_\ell/\ell.
\]
Combining this with~\eqref{eq:sub-lrp}, the exponential decay in \eqref{eq:local-l-rate}  holds if 
\begin{equation}\label{eq:upper-lrp-pr1}
\forall \psi_1>0: \exists k, n_0\in\N, \eps>0\quad\mbox{s.t.}\quad \forall n\ge n_0: \qquad \Prob\big(|\CV_n^\mathrm{cross}(k)|\ge \psi_1 n\big)\le \exp\big(-\eps n\big).
\end{equation}
For Bernoulli bond percolation on $\Z^d$, the number of vertices with a crossing edge in each subbox is deterministically at most $O(k^{(d-1)/d})$ as $k\to\infty$, so the total number of vertices with a crossing edge across the $(n/k)$ many subboxes is $nO(k^{-1/d})$. So, given $\psi_1$, we choose $k$ large enough so that $O(k^{-1/d})<\psi_1$, and then \eqref{eq:upper-lrp-pr1} holds (with error probability $0$).

Consider long-range percolation on $\Z^d$ with $\alpha>1$. 
 Let $\CV_n^\mathrm{surf}(t)$ be the set of vertices within distance $t$ from the union of the boundaries of all subboxes $j\le n/k$; and let  $\CE_n^\mathrm{long}(t)$ denote the  edges of length at least $t$ inside $\Lambda_n$. Clearly, $|\CV_n^\mathrm{surf}(t)|\le (n/k) 2^d k^{(d-1)/d} (2t+1)$. 
 For any $\psi_1>0$, fix $\psi_2>0$ such that $|\CV_n^\mathrm{surf}(\psi_2k^{1/d})|\le \psi_1 n/2$ holds deterministically for all $k$ sufficiently large (simultaneously). 
 Then, since each edge has two endpoints,
\[
|\CV_n^\mathrm{cross}(k)|\le |\CV_n^\mathrm{surf}(\psi_2k^{1/d})|+2|\CE_n^\mathrm{long}(\psi_2k^{1/d})| \le \psi_1 n/2 + 2|\CE_n^\mathrm{long}(\psi_2k^{1/d})|. 
\]
Therefore, 
\begin{equation}\label{eq:elong}
\Prob\big(|\CV_n^\mathrm{cross}(k)|\ge \psi_1 n\big)\le \Prob\Big(2\big|\CE_n^\mathrm{long}(\psi_2k^{1/d})\big|\ge \psi_1 n/2\Big).
\end{equation}
 By a standard isoperimetric inequality, the total number of possible edges inside $\Lambda_n$ with  length in the interval $[r_{i,k}, r_{i+1,k}):=[2^{i}\psi_2k^{1/d}, 2^{i+1}\psi_2k^{1/d})$  is at most $n\cdot O(r_{i,k}^{d})$ as $i\to\infty$. 
Each such edge is present independently with probability at most $\beta ^\alpha r_{i,k}^{-\alpha d}$ by the connection probability~\eqref{eq:connection-prob-gen}. Hence, there exists a constant $C>0$ such that for $k$ sufficiently large
\[
\begin{aligned}
\big|\CE_n^\mathrm{long}(\psi_2k^{1/d})\big|&\preccurlyeq \sum_{i=0}^{\infty} \mathrm{Bin}\big(C\psi_2^d n2^{id}k, \beta^\alpha \psi_2^{-\alpha d} 2^{-i\alpha d} k^{-\alpha}\big),
\end{aligned}
\]
where all binomial random variables in the sum are independent of each other. 
 Thus, 
 \[
 \Prob\big(|\CV_n^\mathrm{cross}(k)|\ge \psi_1 n\big)\le \Prob\bigg(2\sum_{i=0}^{\infty} \mathrm{Bin}\big(C\psi_2^d n2^{id}k,  \beta^\alpha \psi_2^{-\alpha d}(2^{i d}k)^{-\alpha}\big)\ge \psi_1 n/2\bigg).
 \]
First we apply $x\mapsto \exp(x/2)$ to both sides of the inequality inside the probability on the right-hand side, then we use Markov's inequality and independence of the random variables:
\[
\Prob\big(|\CV_n^\mathrm{cross}(k)|\ge \psi_1 n\big)\le 
\re^{-\psi_1 n /4}\prod_{i=0}^\infty\E\Big[\exp\Big(\mathrm{Bin}\big(C\psi_2^d n2^{id}k, \beta^\alpha \psi_2^{-\alpha d} (2^{i d}k)^{-\alpha}\big)\Big)\Big].
\]
Each factor corresponds to the moment-generating function (MGF) of a binomial random variable evaluated at 1. The MGF of $\mathrm{Bin}(n,p)$ at 1 equals $(1+p(\re-1))^n\le \exp\big((\re-1)np\big)$. Applying this bound to each of the factors yields 
\[
\Prob\big(|\CV_n^\mathrm{cross}(k)|\ge \psi_1 n\big)\le 
\exp\bigg(-\psi_1 n/4 + (\re-1)C\beta^\alpha\psi_2^{d(1-\alpha)}nk^{1-\alpha}\sum_{i=0}^\infty 2^{i d(1-\alpha)}\bigg).
\]
Since $\alpha>1$ by assumption, the sum converges. When $k$ is sufficiently large, the right-hand side is at most $\exp(-\psi_1 n/5)$. This proves \eqref{eq:upper-lrp-pr1} for long-range percolation and hence \eqref{eq:local-l-rate}.
\end{proof}
We are ready to prove  \cref{upper-lrp-nnp}.
\begin{proof}[Proof of \cref{upper-lrp-nnp}]\label{proof:upper-lrp-nnp}
    For the lower bound, we observe that $|\CC_n^\sss{(1)}|=n$ when all nearest neighbor  edges are present. Its probability decays as $\exp(-\Theta(n))$. 
  For the upper bound, we argue as follows. We fix $\rho\in(\theta,1)$, and a small $\psi<1-(1-\rho)/(1-\theta)$. Since $\sum_{\ell=1}^\infty \theta_\ell = 1-\theta$, this gives us a constant $\ell^\star=\ell^\star(\rho,\psi)$ that satisfies  $(1-\psi)\sum_{\ell=1}^{\ell^\ast}\theta_\ell> 1-\rho$. We shall shortly show that 
    \begin{equation}\label{eq:giant-in-acomp}
    \{ \CC_n^\sss{(1)} > \rho n\} \subseteq \big\{\exists \ell \le \ell^\star+1:  S_{n,\ell}(\infty)\le (1-\psi)n\theta_\ell/\ell\big\}=:\CA_{\mathrm{comp}}.
    \end{equation}
    If \eqref{eq:giant-in-acomp} holds, then \cref{lemma:size-k-nnp-lrp} and a union bound over the finite set $\ell\le \ell^\ast+1$ gives that $ \CA_\mathrm{comp}$ occurs with probability exponentially small in $n$, finishing the proof of the upper bound. To prove \eqref{eq:giant-in-acomp}, we show that $\neg \CA_{\mathrm{comp}}\subseteq \{ \CC_n^\sss{(1)} \le \rho n\}$.
    Since we excluded the case in LRP that nearest neighbor  edges are present with probability one, $\theta_\ell=\Prob\big(|\CC(0)|=\ell\big)>0$ for all $\ell\ge 1$. So, for all $n$ sufficiently large, on  $\neg\CA_\mathrm{comp}$, there is at least one component of size $\ell^\star+1$; i.e., $|\CC_n^{\sss{(1)}}|\ge\ell^\ast+1$. Hence, vertices in components of size $\ell\le \ell^\star$ are not in the giant. Thus,
         \[
 n-|\CC_n^\sss{(1)}| =  |\CV_n\setminus\CV(\CC_n^\sss{(1)})| \ge \sum_{\ell=1}^{\ell^\ast}\ell S_{n,\ell}(\infty) \ge \sum_{\ell=1}^{\ell^\ast}(1-\psi)n\theta_\ell>n(1-\rho). 
    \]
    Rearranging yields that $|\CC_{n}^{\sss{(1)}}|< \rho n$ on $\neg\CA_{\mathrm{comp}}$, finishing the proof. 
\end{proof}

\subsection{Inhomogeneous models}
Here we start again with the convergence rate of the number of size-$\ell$ components. Since the degree distribution is heavy-tailed, this is more involved than Lemma \ref{lemma:size-k-nnp-lrp}.
We write $\mathrm{deg}_v$ for the degree of vertex $v$ in $\CG_n$, and $\mathrm{deg}_v[a,b)$ for the number of neighbors of $v$ in $\CG_n$ with mark in $[a,b)$. 
\begin{claim}[Total degree from high-mark vertices]\label{claim:total-deg}
        Consider a KSRG on a PPP satisfying Assumption \ref{assumption:main} with $\tau<\infty$. For any constants $\psi, C>0$, there exists $\phi_0>0$ such that for all $n$ sufficiently large and $\phi\in(0,\phi_0)$, writing $ \underline w_n=n^{(1+\psi)/(\tau-1)}$,
    \begin{equation}\label{eq:total-deg}
    \Prob\Bigg(\sum_{v\in\CV_n[\underline{w}_n, \phi n)}\mathrm{deg}_v[1,\phi n) \ge \psi n\Bigg)\le n^{-C}.
    \end{equation}
\end{claim}

\begin{claim}[High-degree vertices]\label{claim:high-deg}Consider a KSRG on a PPP satisfying Assumption \ref{assumption:main} with $\tau<\infty$. For any constants $\ell, C>0$, there exist constants $n_0$, $C_1$ such that for all $n\ge n_0$, and any $\underline w_n> C_1\log n$,  
    \[\Prob\big(\exists v\in \CV_n[\underline w_n, \infty): \mathrm{deg}_v[1,\underline w_n) <\ell\big)\le n^{-C}.
    \]
\end{claim}
\begin{claim}[Crossing edges]\label{claim:crossing-edges}
Consider a  KSRG on a PPP satisfying Assumption \ref{assumption:main} with   $\tau<\infty$. Let $\psi>0$ be a constant such that $(1+\psi)/(\tau-1) < 1$. Then, for any constant $C>0$, there exists $n_0$ such that for all $n\ge n_0$ and any $\underline w_n\le n^{(1+\psi)/(\tau-1)}$,  
    \begin{equation}\label{eq:crossing-edges}
    \Prob\big(|\{v\in \CV_{n}[1,\underline w_n): v\sim \Lambda_n^{\complement}\times[1,\underline w_n)\}|\ge \psi n\big)\le n^{-C}.
    \end{equation}
\end{claim}
The proofs of the three claims are based on concentration inequalities for Poisson/binomial random variables and postponed to Appendix~\ref{sec:upper-deg-cross}. Their combination gives a \emph{polynomial} convergence rate for the size-$\ell$ components in \cref{lemma:size-k-local}, which is in stark contrast to the exponential convergence rate for degree-homogeneous models in \cref{lemma:size-k-nnp-lrp}.
\begin{lemma}[Concentration for size-$\ell$ components]\label{lemma:size-k-comp}
    Consider a KSRG on a PPP satisfying Assumption \ref{assumption:main} with $\tau<\infty$. For all constants $\ell, \psi, C>0$, there exist constants $\phi_0>0$ and $n_0$ such that for all $n\ge n_0$ and all $\phi\in(0, \phi_0)$,
    \begin{equation}
        \Prob\big(\big|S_{n,\ell}(\phi n)/n - \theta_\ell/\ell\big| >\psi \big)\le n^{-C}.\label{eq:size-k}
    \end{equation}
\end{lemma}
\begin{proof}
    We first reduce the upper mark threshold $\overline w_n:=\phi n$ in the event in~\eqref{eq:size-k} to a smaller value $\underline w_n=n^{(1+\psi)/(\tau-1)}$ (we assume w.l.o.g.\ that $(1+\psi)/(\tau-1)<1$). We use that each vertex with degree at least $\ell$ is in a component of size at least $\ell+1$, and distinguish the case that each vertex of mark $\underline w_n$ is in a component of size larger than $\ell$, i.e.,

    \begin{align}
    \Prob\big(\big|S_{n,\ell}(\overline w_n )/n - \theta_\ell/\ell\big| >\psi \big)&\le 
    \Prob\bigg(
    \begin{aligned}
    &\{|S_{n,\ell}(\overline w_n)/n - \theta_\ell/\ell| >\psi\} \\&\cap \{\forall v\in\CV_n[\underline w_n, \overline w_n): |\CC_n(v)[1,\overline w_n)|>\ell\} \end{aligned}\bigg) \label{eq:intersect} \\&\hspace{15pt}+ 
    \Prob\big(\exists v\in\CV_n[\underline w_n, \overline w_n): \mathrm{deg}_v[1, \overline w_n) \le \ell\big).\label{eq:term-degree}
    \end{align}
   Now we `switch' $\overline w_n$ to $\underline w_n$ in $S_{n,\ell}(\cdot)$ on the right-hand side. Let $T_{n,\ell}$ denote the number of size-$\ell$ components in $\CG_n[1,\underline w_n)$ that have an edge towards a vertex with mark in $[\underline w_n, \overline w_n)$. 
    The events on the right-hand side of \eqref{eq:intersect} imply that the size-$\ell$ components of $\CG_n[1,\overline w_n)$ do not contain any vertices with mark at least $\overline{w}_n$.
   So, on those events, $S_{n,\ell}(\overline w_n)=S_{n,\ell}(\underline w_n)- T_{n,\ell}$. Distinguishing also whether $T_{n,\ell}/n\le \psi/2$ or not yields
  \begin{align}
    \Prob\big(\big|S_{n,\ell}(\overline w_n)/n - \theta_\ell/\ell\big| >\psi \big)&\le 
    \Prob\bigg(
    \begin{aligned}
    &\{|S_{n,\ell}(\underline w_n)/n - T_{n,\ell}/n - \theta_\ell/\ell| >\psi\} \cap\{T_{n,\ell}/n \le \psi/2\}\\&\cap \{\forall v\in\CV_n[\underline w_n, \overline w_n): |\CC_n(v)[1,\overline w_n)|>\ell\} \end{aligned}\bigg) \label{eq:intersect-2}
    \\&\hspace{15pt}+ 
    \Prob\big(T_{n,\ell}/n\ge \psi /2\big)+
    \Prob\big(\exists v\in\CV_n[\underline w_n, \overline w_n): \mathrm{deg}_v[1, \overline w_n) \le \ell\big).
    \end{align}
    The first two events in \eqref{eq:intersect-2} imply that  $\{|S_{n,\ell}(\underline w_n)/n -\theta_\ell/\ell|>\psi/2\}$. Moreover, $T_{n,\ell}$ is at most the total number of edges between $\CV[\underline{w}_n, \overline w_n)$ and $\CV_n[1,\underline{w}_n)$, which is at most the total degree of vertices in  $\CV[\underline{w}_n, \overline w_n)$. Hence, 
    \begin{equation}
    \begin{aligned}
    \Prob\big(\big|S_{n,\ell}(\overline w_n)/n - \theta_\ell/\ell\big| >\psi \big)&\le 
    \Prob\big(|S_{n,\ell}(\underline w_n)/n -  \theta_\ell/\ell| >\psi/2 \big)\\&\hspace{15pt}+ 
    \Prob\bigg(\sum_{v\in\CV_n[\underline w_n, \overline w_n)}\mathrm{deg}_v[1,\overline w_n)\ge n \psi /2\bigg)\\&\hspace{15pt}+
    \Prob\big(\exists v\in\CV_n[\underline w_n, \overline w_n): \mathrm{deg}_v[1, \overline w_n) \le \ell\big).
    \end{aligned}\nonumber
    \end{equation}
     By Claim~\ref{claim:total-deg}, the second term can be made at most $n^{-2C}$ if $\phi$ in the definition of $\overline w_n=\phi n$ is sufficiently small. Since $\underline w_n=n^{(1+\psi)/(\tau-1)}$, the third term is also at most $n^{-2C}$ by Claim~\ref{claim:high-deg}. So,
     \begin{equation}\Prob\big(\big|S_{n,\ell}(\overline w_n)/n - \theta_\ell/\ell\big| >\psi \big)\le 
    \Prob\big(|S_{n,\ell}(\underline w_n)/n -  \theta_\ell/\ell| >\psi/2 \big) + o\big(n^{-C}\big).
    \label{eq:size-k-three-events}
    \end{equation} 
    It remains to bound the first term on the right-hand side with  $\underline w_n=n^{(1+\psi)/(\tau-1)}$. 
    Let $A$ be a sufficiently large constant with $(A\log n)^{1/d}\in\N$ and partition $\Lambda_n$ into boxes $\CQ_1,\ldots,\CQ_{A\log n}$ of volume $k_n:=n/(A\log n)$. Let $S_{k_n,\ell}^\sss{(j)}(\underline w_n)$ denote the number of components of size exactly $\ell$ in the $j$th box, and $\CV^\sss{(j)}_{k_n}$ the set of vertices in the $j$th box. 
    \begin{equation}\label{eq:size-ell-lln-pr}
    \CA^\sss{(j)}:=\{|S_{k_n,\ell}^\sss{(j)}(\underline w_n)/k_n-\theta_\ell/\ell|\le\psi/12\}, 
    \end{equation}
 and otherwise we call it bad. The probability of being good tends to $1$ with $k_n$ by \cref{lemma:size-k-local}, and so it is at least $1-\psi/32$ for all $n$ sufficiently large.
 Let $
X:=\sum_{j\in[k_n]}\ind{\CA^\sss{(j)}}
    $
    be the number of good boxes. Independence across boxes gives that $X$ dominates a binomial random variable with parameters $A\log n$ and $1-\psi/32$. So, by a Chernoff bound there exists a constant $c=c(\psi)>0$ such that 
    \begin{equation}
    \Prob\big(X\ge (1-\psi/16)A\log n\big)\ge 1-\exp\big(-cA\log n\big)=1-n^{-cA}.\label{eq:size-k-bad}
    \end{equation}
    Since $k_n=n/A\log n$, using concentration for Poisson random variables, for all $n$ sufficiently large,
    \begin{equation}
    \Prob\big(\forall j: |\CV_{k_n}^\sss{(j)}[1,\underline w_n)|/k_n<4/3 \big)
    \ge 1-(A\log n)\exp\big(-\Omega(n/\log n)\big)\ge 1-n^{-cA}.\label{eq:size-k-poi}\end{equation}
    Thus, in the following, we will find bounds on $S_{n,\ell}(\underline w_n)$ conditional on the event 
    \begin{equation}
  \CA_{\mathrm{good}}:=  \{\forall i: |\CV_{k_n}^\sss{(j)}[1,\underline w_n)|/k_n<4/3\}\cap \{X\ge (1-\psi/16)A \log n\}.\label{eq:size-k-intersect}
    \end{equation}
    We write $S^\sss{\mathrm{good}}_{n,\ell}$ for the sum of $S_{k_n,\ell}^{\sss{(j)}}$ over $j$ for which the $j$th box is good, and similarly $S^\sss{\mathrm{bad}}_{k_n,\ell}$ for the same sum over bad boxes. 
    Finally, let $\CV^\sss{\mathrm{cross}}_n[1,\underline w_n)$ denote the set of vertices incident to an edge that crosses a box-boundary. Since each edge changes the number of size-$\ell$ components by at most $2$, similarly to \eqref{eq:snell-bound-2}, it holds that 
    \begin{equation}
    |S_{n,\ell}(\underline w_n)-(S^\sss{\mathrm{good}}_{n,\ell}+S^\sss{\mathrm{bad}}_{n,\ell})|\le  |\CV^\sss{\mathrm{cross}}_n[1,\underline w_n)|.
     \end{equation} 
On the event $\CA_{\mathrm{good}}$ in~\eqref{eq:size-k-intersect}, the at most $A\log n$ good boxes contain together at most $(\theta_\ell/\ell+\psi/12)n$ components of size $\ell$. There are at most  $n(\psi/16)4/3=n\psi/12$ vertices in bad boxes. Assuming the worst case that all of them are in size-$\ell$ components inside their boxes, we obtain the upper bound $S^\sss{\mathrm{bad}}_{n,\ell}\le n\psi/(12\ell)$. So, by a triangle inequality on the left-hand side above,
\[ S_{n,\ell}(\underline w_n)-n \theta_\ell/\ell \le n\psi/12+n\psi/(12\ell)+|\CV^\sss{\mathrm{cross}}_n[1,\underline w_n)|
    \le  n\psi/6  +|\CV^\sss{\mathrm{cross}}_n[1,\underline w_n)|.\]
 Similarly, for a lower bound we may assume that there are no size-$\ell$ components in bad boxes. On the event~\eqref{eq:size-k-intersect}, this still gives
 \[ S_{n,\ell}(\underline w_n)-n \theta_\ell/\ell \ge n (1-\psi/16)\big( \theta_\ell/\ell -\psi/12\big) - |\CV_n^\sss{\mathrm{cross}}[1,\underline w_n)| \ge -n\psi7/48- |\CV_n^\sss{\mathrm{cross}}[1,\underline w_n)| \]   
    Recalling the bounds~\eqref{eq:size-k-three-events},~\eqref{eq:size-k-bad}, and~\eqref{eq:size-k-poi}, we obtain 
    \[
    \begin{aligned}
    \Prob\big(\big|S_{n,\ell}(\overline w_n)/n - \theta_\ell/\ell\big| >\psi \big)&\le 
    2n^{-cA}+o(n^{-C})+\Prob\big(|\CV^\sss{\mathrm{cross}}_{n,\ell}|/n \ge\psi/3 \big).
    \end{aligned}    
    \]
    We bound the probability concerning the crossing edges. By the pigeon-hole principle,  $|\CV^\sss{\mathrm{cross}}_{n,\ell}| \ge n  \psi/3 $ implies that there exists a box that contains at least $(\psi/3)k_n$ vertices with an outgoing edge towards the complement of the box. Hence, by translation invariance of the model on $\R^d$,
    \[
    \Prob\big(|\CV_{n,\ell}^\sss{\mathrm{cross}}|/n\ge \psi/3\big)\le (A\log n)\Prob\big(|\{v\in\CV_{k_n}[1, \underline w_n): v\sim \Lambda_{k_n}^\complement\times[1,\underline w_n)\}|\ge (\psi/3)k_n\big).
    \]
    \cref{claim:crossing-edges} is applicable here for $n_{\ref{claim:crossing-edges}}=k_n=n/(A\log n)$, since $\underline w_n=n^{(1+\psi)/(\tau-1)}\le k_n^{(1+\psi')/(\tau-1)}$ for any $\psi'>\psi$. So the right-hand side is smaller than $n^{-cA}$ for $n$ sufficiently large. We conclude that~\eqref{eq:size-k} follows for $\phi$ small, setting $A\ge C/c$ such that $(A\log n)^{1/d}\in\N$.
\end{proof}

We need a last preliminary claim. 
Recall $\mathrm{hubs}(\rho)$ from \eqref{eq:hubs-gen}. We introduce notation for the functions appearing in the upper and lower bound of \cref{thm:large-dev-upper}: 
\begin{equation}\label{eq:underover-h}
    h_\mathrm{up}:=\lceil\mathrm{hubs}(\rho)\rceil, 
    \qquad 
    h_\mathrm{lo}:=\lim_{r\downarrow \rho}\lceil\mathrm{hubs}(r)\rceil=
    \begin{dcases}
        \lceil\mathrm{hubs}(\rho)\rceil,&\text{if }\mathrm{hubs}(\rho)\notin\N\text{ or }p=1,\\
        \mathrm{hubs}(\rho) + 1,&\text{if }\mathrm{hubs}(\rho)\in\N\text{ and }p<1.
    \end{dcases}
    \end{equation}
    Clearly $h_\mathrm{lo} \ge h_\mathrm{up} \ge \mathrm{hubs}(\rho)$, and for all $\rho\in(\theta, 1)$ and  $p<1$ also $h_\mathrm{lo} > \mathrm{hubs}(\rho)$ strictly.
\begin{claim}[Truncation of the generating function]\label{claim:hubs-lower}
     Consider a KSRG on a PPP satisfying Assumption \ref{assumption:main} with $\tau<\infty$. Let $\rho\in(\theta,1)$. There exist constants $\psi>0$ and $\ell_0\in\N$ such that for all $\ell^\ast>\ell_0$,
    \begin{equation}
    (1-\psi)^2 \bigg(\theta + \sum_{\ell=1}^{\ell^\ast} \big(1-(1-p)^{h_\mathrm{lo} \ell}\big) \theta_\ell\bigg) > \rho,\label{eq:hubs-lower-trunc}
    \end{equation}
    and
    \begin{equation}
    (1-\psi)^2\sum_{\ell=1}^{\ell^\ast}\theta_\ell(1-p)^{(h_\mathrm{up}-1)\ell}>
    1-\rho+\psi.
\label{eq:hubs-upper-trunc}
    \end{equation}
   \end{claim}
   We postpone the proof to Appendix~\ref{sec:app-trunc-gen} and proceed to the lower bound of \cref{thm:large-dev-upper}.
\begin{proposition}[Lower bound of the upper tail]\label{prop:upper-lower}
  Consider a KSRG on a PPP satisfying Assumption \ref{assumption:main} with $\tau<\infty$.
Fix $\rho\in(\theta, 1)$. 
Assume either supercriticality and $\zeta_\mathrm{long}>0$, or assume $\theta=0$. Then there exists $A=A(\rho)>0$ such that for all $n\ge 1$, 
 \begin{equation}\nonumber
 \Prob\big(|\CC_n^\sss{(1)}|> \rho n\big) \ge (1/A)n^{-(\tau-2)h_\mathrm{lo}}.
 \end{equation} 
\end{proposition}
\begin{proof} The proof is similar in spirit to the proof of \cref{upper-lrp-nnp} on page \pageref{proof:upper-lrp-nnp}, but we also need to deal with high-mark vertices. See Section \ref{sec:outline-upper} for an outline that we make formal now. 
    Let $R,\ell^\star$ be two sufficiently large constants, and let $\phi,\psi>0$ be two small constants so that  given $\rho\in (\theta,1)$, \eqref{eq:hubs-lower-trunc} holds.
    Write $T_{n,\ell}$ for the number of components of size $\ell$ in $\CG_n[1,\phi n)$ that connect by an edge to the vertices with mark in $\CV_n[Rn,\infty)$. Define the events 
   
    \begin{align}
    \CA_\mathrm{comp}&:=
    \big\{|\CC_n^\sss{(1)}[1,\phi n)|/n\ge \theta(1-\psi)\big\}\cap \big\{\forall \ell\le \ell^\ast: S_{n,\ell}(\phi n)\ge (1-\psi)n\theta_\ell/\ell\big\}, \label{eq:a-comp-7}\\
    \CA_\mathrm{hub}&:=\big\{|\CV_{n}[Rn, \infty)|= h_\mathrm{lo}\big\}\cap\big\{\forall v\in \CV_{n}[Rn, \infty): v\sim \CC_n^\sss{(1)}[1, \phi n)\big\},\label{eq:a-hub-7}\\
    \CA_\mathrm{conn}&:=\big\{\forall \ell\le \ell^\ast: T_{n,\ell}\ge (1-\psi) \big(1-(1-p)^{h_\mathrm{lo} \ell}\big) \big((1-\psi)n \theta_\ell/\ell\big) \big\}.\label{eq:a-conn-7}
    \end{align} 
    We shall show that 
    \begin{equation}\label{eq:comp-implication}
     \CA_\mathrm{comp} \cap \CA_\mathrm{hub} \cap \CA_\mathrm{conn} \subseteq \{ |\CC_n^\sss{(1)}|>\rho n\}.
    \end{equation}
   By its definition, $T_{n,\ell}$ many clusters of size $\ell$ connect to the hubs listed in $\CA_\mathrm{hub}$, which then all are part of the giant component in the whole graph. Hence, on $\CA_{\mathrm{hub}}$, the giant $|\CC_n^\sss{(1)}[1,\phi n)|$ increases by at least $\ell T_{n,\ell}$ when fully revealing the graph.  
   By the bounds on $T_{n,\ell}$  and the initial size $(1-\psi) \theta n $ of the giant inside the events, and \eqref{eq:hubs-lower-trunc} in \cref{claim:hubs-lower}, we obtain that 
    \[
    \begin{aligned}
    |\CC_n^\sss{(1)}|/n &\ge \theta(1-\psi) + \frac{1}{n}\sum_{\ell=1}^{\ell^\ast}\ell T_{n,\ell} \ge \theta(1-\psi) + (1-\psi)^2\sum_{\ell=1}^{\ell^\ast} \big(1-(1-p)^{h_\mathrm{lo}\ell}\big) \theta_\ell\\
    &\ge (1-\psi)^2\bigg(\theta + \sum_{\ell=1}^{\ell^\ast} \big(1-(1-p)^{h_\mathrm{lo}\ell}\big) \theta_\ell\bigg) > \rho,
    \end{aligned}
    \]
   so 
  \eqref{eq:comp-implication} holds. We bound from below the probability of the  intersection  using the law of total probability:
      \begin{equation} \label{eq:law-total-prob}
        \Prob\big(\CA_\mathrm{comp}\cap\CA_\mathrm{hubs}\cap\CA_\mathrm{conn}\big) = \Prob\big(\CA_\mathrm{comp}\cap\CA_\mathrm{hubs}\big)\Prob\big(\CA_\mathrm{conn}\mid\CA_\mathrm{comp}\cap\CA_\mathrm{hubs}\big).
    \end{equation}
   To study the first factor, we use the law of total probability again and then also $\Prob(A\cap B)\ge \Prob(A)-\Prob(\neg B)$ (see \eqref{eq:a-comp-7}-\eqref{eq:a-hub-7}):
\begin{align}
    \Prob\big(\CA_\mathrm{comp}\cap \CA_\mathrm{hubs}\big) &\ge
    \Prob\Big(\forall v\in \CV_{n}[Rn, \infty): v\sim \CC_n^\sss{(1)}[1, \phi n) \mid \CA_\mathrm{comp}\cap
    \big\{|\CV_{n}[Rn, \infty)|=h_\mathrm{lo}\big\}\Big) \nonumber \\
  &\qquad \cdot  \Big(\Prob\big(
    |\CV_{n}[Rn, \infty)|= h_\mathrm{lo}
    \big)-\Prob\big(\neg\CA_\mathrm{comp}\big)\Big).    \label{eq:upper-tail-lower-pr1}
       \end{align}
    For the first term in the row \eqref{eq:upper-tail-lower-pr1} we use that $|\CV_n[Rn,\infty)|\sim\mathrm{Poi}(n(Rn)^{-(\tau-1)})$ by the intensity in \eqref{eq:poisson-intensity}. Note here that the parameter tends to $0$ since $\tau>2$. Therefore, since $h_{\mathrm{lo}}$ is a $\rho$-dependent constant that does not depend on $n$,
    \begin{equation}\label{eq:lb-1st-term-7}
    \begin{aligned}
    \Prob\big(|\CV_{n}[Rn, \infty)|= h_\mathrm{lo}\big) &= \exp\big(\!-\!R^{-(\tau-1)}n^{-(\tau-2)}\big)\frac{\big(R^{-(\tau-1    )}n^{-(\tau-2)}\big)^{h_\mathrm{lo}}}{h_\mathrm{lo}!}\ge c_{R,\rho}n^{-(\tau-2)h_\mathrm{lo}}
    \end{aligned}
    \end{equation}
    for some constant $c_{R,\rho}>0$. We now bound $\Prob(\neg\CA_\mathrm{comp})$ in \eqref{eq:upper-tail-lower-pr1}. 
    \begin{equation}
    \Prob(\neg\CA_\mathrm{comp}) \le \Prob\Big( \exists \ell\le \ell^\ast: \ell S_{n,\ell}(\phi n)/n\le (1-\psi)\theta_\ell\Big) + \Prob(|\CC_n^\sss{(1)}[1,\phi n)|/n< \theta(1-\psi) ).
    \end{equation}
    Lemma~\ref{lemma:size-k-comp} and a union bound over the $\ell^\star$ many $\ell$ shows that the first term is at most $n^{-C}$ for arbitrary large $C$ when $\phi<\phi_0(\psi)$. For the second term, there are two cases. Either $\theta=0$ and then this probability is $0$; or, $\theta$ is non-zero. In the latter case, when additionally $\zeta_{\mathrm{long}}>0$, our result for the \emph{lower} tail of large deviations in \cref{thm:large-dev2} guarantees that this event has stretched exponentially decaying probability in $n$.  This is where we use that the model must be supercritical with $\zeta_{\mathrm{long}}>0$ when\footnote{Inspection of the proof of \cref{thm:large-dev2} shows that the sprinkling technique used in Section \ref{sec:biskup} needs supercriticality. The assumption $\zeta_{\mathrm{long}}>0$ is the technical condition of \cref{thm:large-dev2} that guarantees enough long edges.} $\theta>0$.
    So, under these conditions, we know that  $\Prob\big(\neg\CA_\mathrm{comp}\big)$ is of smaller order than the lower bound in \eqref{eq:lb-1st-term-7} on the first term in \eqref{eq:upper-tail-lower-pr1}. 

    We now analyze the conditional probability on the line above \eqref{eq:upper-tail-lower-pr1}. Take $R>\sqrt{d}/\beta$. Then a vertex with mark at least $Rn$ connects by an edge to any other vertex in $\Lambda_n$ with probability $p$ by \eqref{eq:connection-prob-gen}. The giant $\CC_n^{\sss{(1)}}[1,\phi n)$ has size at least $(1-\psi)\theta n$ on the event $\CA_{\mathrm{comp}}$ in \eqref{eq:a-comp-7} in the conditioning. So, by a union bound, the (conditional) probability that all the $h_\mathrm{lo}$ many vertices of mark at least $Rn$ connect by an edge to $\CC_n^{\sss{(1)}}[1,\phi n)$ is at least $1-h_\mathrm{lo}(1-p)^{\Theta(n)}$, which tends to $1$ as $n\to\infty$. Combining these with \eqref{eq:upper-tail-lower-pr1}--\eqref{eq:lb-1st-term-7} gives in \eqref{eq:law-total-prob}
    \begin{equation}\label{eq:ltp-t1}
\Prob\big(\CA_\mathrm{comp}\cap\CA_\mathrm{hubs}\big)\ge c_{R,\rho}n^{-(\tau-2)h_\mathrm{lo}}.    
    \end{equation}
    We show that the other term in  \eqref{eq:law-total-prob}, $
    \Prob\big(\CA_\mathrm{conn}\mid \CA_\mathrm{comp}\cap\CA_\mathrm{hubs}\big)$ tends to $1
    $
    as $n\to\infty$. Recall these events from \eqref{eq:a-comp-7}--\eqref{eq:a-conn-7} and that under the conditioning, $T_{n,\ell}$ denotes the number of size-$\ell$ components of $\CG_n[1,\phi n)$ that connect to one of the $h_{\mathrm{lo}}$ many vertices of mark at least $Rn$ (the hubs).
    As before, each vertex connects independently to these hubs with probability $p$. Further, the vertex set counted in $T_{n,\ell}$ is disjoint of the vertices in $\CC_{n}^{\sss{(1)}}[1,\phi n)$ so their connection to the hubs is independent of whether the vertices in $\CC_{n}^{\sss{(1)}}[1,\phi n)$ connect to the hubs, which is an event inside the conditioning $\CA_{\mathrm{comp}}\cap \CA_\mathrm{hubs}$. So, under the conditioning, 
 each component of size $\ell$ in $\CG_n[1,\phi n)$ connects to one of the $h_{\mathrm{lo}}$ hubs with probability $1-(1-p)^{h_\mathrm{lo} \ell}$ independently across size-$\ell$ components and across $\ell<\ell^\star$. Hence, given the number of size-$\ell$ components in $\CG_n[1,\psi n)$,  $(S_{n,\ell}(\phi n))_{\ell\le \ell^\star}$, the variables $T_{n,\ell}$ are binomially distributed with parameters  $S_{n,\ell}(\phi n)$ and $(1-(1-p)^{h_\mathrm{lo}\ell})$. Further, under the conditioning in $\CA_{\mathrm{comp}}$, each $S_{n,\ell}(\phi n)$ is at least $(1-\psi) n \theta_\ell/\ell$, so $T_{n,\ell}$ dominates a binomial random variable with parameters $(1-\psi) n \theta_\ell/\ell$ and $(1-(1-p)^{h_\mathrm{lo}\ell})$ for each $\ell\le \ell^\star$.
 A union bound over $\ell\le \ell^\star$ and then a Chernoff bound for each $\ell$  give that 
$\Prob\big(\CA_\mathrm{conn}\mid \CA_\mathrm{comp}\cap\CA_\mathrm{hubs}\big) \to 1
    $ as $n\to\infty$. Using this in \eqref{eq:law-total-prob} and \eqref{eq:ltp-t1} gives in \eqref{eq:comp-implication} that for some constant $c_{R,\rho}>0$, 
\begin{align*}\Prob(\CC_{n}^{\sss{(1)}}>\rho n) \ge  \Prob\big(\CA_\mathrm{comp}\cap\CA_\mathrm{hubs}\cap\CA_\mathrm{conn}\big)&\ge c_{R,\rho} n^{-(\tau-2)h_\mathrm{lo}}.\qedhere
    \end{align*}    
\end{proof}
We proceed to the upper bound of \cref{thm:large-dev-upper}. 
\begin{proposition}[Upper bound of the upper tail]\label{prop:upper-upper}
 Consider a KSRG satisfying Assumption \ref{assumption:main} on a PPP with $\tau<\infty$.
Fix $\rho\in(\theta, 1)$. There exists $A=A(\rho)>0$ such that for all $n\ge 1$,
 \begin{equation}\nonumber
\Prob\big(|\CC_n^\sss{(1)}|> \rho n\big)\le An^{-(\tau-2)h_\mathrm{up}}.
 \end{equation} 
 \end{proposition}
\begin{proof}
Given $\rho$, fix $\phi,\psi>0$ sufficiently small and $\ell^\star$ sufficiently large so that \eqref{eq:hubs-upper-trunc} holds for $\psi$ and $\ell^\star$.
    We write $\overline T_{n,\ell}$ for the number of components of size $\ell$ in $\CG_n[1,\phi n)$ that do \emph{not} connect by an edge to the vertices with mark at least $\phi n$, and define
    \begin{equation}\label{eq:a-events-up-7}
    \begin{aligned}
    \CA_\mathrm{comp}&:=\big\{|\CV_n|/n\le 1+\psi\}\cap \big\{\forall \ell\le \ell^\ast+1:  S_{n,\ell}(\phi n)\ge (1-\psi)n\theta_\ell/\ell\big\},\\
    \CA_\mathrm{hubs}&:=\big\{|\CV_{n}[\phi n, \infty)|\le h_\mathrm{up}-1\big\},\\
    \CA_\mathrm{conn}&:=\big\{\forall \ell\in[\ell^\ast]: \overline T_{n, \ell}> (1-\psi) (1-p)^{(h_\mathrm{up}-1) \ell}\big((1-\psi) n\theta_\ell/\ell\big) \big\}.
    \end{aligned}
    \end{equation}
    We show in the next paragraph that 
    \begin{equation}\label{eq:upper-containment}
        \CA_\mathrm{comp}\cap \CA_\mathrm{hubs}\cap \CA_\mathrm{conn} \subseteq \{ |\CC_{n}^{\sss{(1)}}| \le \rho n  \}. 
    \end{equation}   
    To show this, we bound the size of the complement of the giant on the intersection of the three events. Since $S_{n,\ell^\ast+1}(\phi n)\!\ge\!1$ for all sufficiently large $n$, the largest component has size at least $\ell^\ast$. Hence, vertices in components of size $\ell\le \ell^\star$ are not in the giant of $\CG_n[1,\psi n)$.   Combining the bounds in the events with  \cref{claim:hubs-lower}, we obtain for sufficiently small $\psi$ that
    \[
    1-|\CC_n^\sss{(1)}|/n \ge  |\CV_n\setminus \CC_n^\sss{(1)}|/n -\psi \ge \frac{1}{n}\sum_{\ell=1}^{\ell^\ast}\ell\overline T_{n,\ell} - \psi 
    \ge 
    (1-\psi)^2\sum_{\ell=1}^{\ell^\ast}\theta_\ell(1-p)^{(h_\mathrm{up}-1)\ell} - \psi > 1-\rho.
    \]
    Rearranging shows \eqref{eq:upper-containment}, which implies that $\{|\CC_n^\sss{(1)}|>\rho n\}\subseteq (\neg\CA_\mathrm{comp}) \cup (\neg \CA_\mathrm{hubs})\cup (\neg\CA_\mathrm{conn})$.
    Thus, it follows that    \begin{equation}\label{eq:upper-upper-pr1} 
    \begin{aligned}
\Prob\big(|\CC_n^\sss{(1)}|>\rho n\big) &\le \Prob\big(\neg \CA_\mathrm{comp}\big) + \Prob\big(\neg \CA_\mathrm{hubs}\big)+\Prob\big(\neg \CA_\mathrm{conn}\cap \CA_\mathrm{comp}\cap\CA_\mathrm{hubs}\big)\\
    &\le \Prob\big(\neg \CA_\mathrm{comp}\big) + \Prob\big(\neg \CA_\mathrm{hubs}\big)+\Prob\big(\neg \CA_\mathrm{conn}\mid \CA_\mathrm{comp}\cap\CA_\mathrm{hubs}\big).
    \end{aligned}
    \end{equation}
    By \cref{lemma:size-k-comp} and the concentration of Poisson random variables in \cref{lemma:poisson-1}, $\Prob\big(\neg\CA_\mathrm{comp}\big)=o(n^{-C})$ for any $C>0$ if $\phi=\phi(\psi)$ is sufficiently small. For $\Prob\big(\neg \CA_\mathrm{hubs}\big)$, we use similarly to \eqref{eq:lb-1st-term-7} that the number of hubs is Poisson. So, there exists a constant $C_{\phi,\rho}>0$ such that for all $n\ge 1$,
    \[
    \begin{aligned}
    \Prob\big(\neg \CA_\mathrm{hubs}\big) = \Prob\big(\mathrm{Poi}\big(n(\phi n)^{-(\tau-1)}\big)\ge h_\mathrm{up}\big) &= \exp\big(-\phi^{-(\tau-1)}n^{-(\tau-2)}\big)\sum_{j=h_\mathrm{up}}^\infty\frac{\phi^{-j(\tau-1)}n^{-j(\tau-2)}}{j!} \\
    &\le C_{\phi, \rho} n^{-(\tau-2)h_\mathrm{up}}.
    \end{aligned}
    \]
    For the conditional probability in~\eqref{eq:upper-upper-pr1}, with the corresponding events defined in \eqref{eq:a-events-up-7}, we argue similarly as in the proof of the lower bound. The probability that a single vertex in a size-$\ell$ component of $\CG_{n}[1, \phi n)$ connects by an edge to each of the at most $h_\mathrm{up}\!-\!1$ hubs is at most $p$, independently of each other under the conditioning. Hence, the probability that a component of size $\ell$ does not connect by an edge to one of the hubs is at least $(1-p)^{\ell(h_\mathrm{up}-1)}$, and connected components connect independently by an edge to the hubs. Thus, conditionally on $S_{n,\ell}(\phi n)$ and $\CA_\mathrm{comp}\cap\CA_\mathrm{hubs}$, $\overline T_{n,\ell}$ dominates a binomial random variable with parameters $S_{n,\ell}(\phi n)$ and $(1-p)^{\ell(h_\mathrm{up}-1)}$. On the event $\CA_{\mathrm{comp}}$, 
   $S_{n,\ell}(\phi n)\ge (1-\psi) n \theta_\ell /\ell$, so $\overline T_{n,\ell}$ dominates a binomial random variable with parameters  $(1-\psi) n \theta_\ell /\ell$ and $(1-p)^{\ell(h_\mathrm{up}-1)}$. A union bound over $\ell\le \ell^\star+1$ and then a Chernoff bound shows that    $\Prob\big(\neg \CA_\mathrm{conn}\mid \CA_\mathrm{comp}\cap\CA_\mathrm{hubs}\big)=\exp\big(-\Theta(n)\big)$. Substituting these three bounds into~\eqref{eq:upper-upper-pr1} finishes the proof.
\end{proof}

\begin{proof}[Proof of \cref{thm:large-dev-upper}]
    Recall $h_\mathrm{lo}$ and $h_\mathrm{up}$ from~\eqref{eq:underover-h}. The lower bound is given by Proposition~\ref{prop:upper-lower}, the upper bound by Proposition~\ref{prop:upper-upper}.
\end{proof}

\begin{appendices}
 \section{Preliminary lemmas}
 \begin{lemma}[Concentration of Poisson random variables {\cite{mitzenmacher2017probability}}]\label{lemma:poisson-1}
 For $x>1$, 
 \begin{equation}
  \Prob\big(\Poi(\lambda) \ge x\lambda\big)\le\exp(-\lambda(1+x\log(x)-x)),
 \end{equation}
 and for $x<1$,
  \begin{equation}
  \Prob\big(\Poi(\lambda) \le x\lambda\big)\le 
  \exp(-\lambda(1-x-x\log(1/x)).
 \end{equation}
\end{lemma}
 \section{Existence of a giant: postponed proofs}\label{app:linear-sized}
\subsection{Polynomially-sized components}
 The proof of \cref{lemma:ltld-weakest} uses a similar strategy as the proof of \cref{proposition:ltld-weaker}: we first prove the statement along a subsequence $(n_i)_{i\ge 1}$ using a multi-scale renormalization method, and then extend the result to all $n$. For the multi-scale renormalization, we follow the steps \emph{(Sub),(Con),(Mark)} outlined in \cref{sec:outline-biskup-renorm}. The initialization of the renormalization is different. Here, we shall use the uniqueness of the infinite cluster to construct large connected components in level-$0$ boxes. We use slightly overlapping boxes, so the events that two overlapping boxes contain a large connected component in their induced subgraph are no longer independent. Therefore, we cannot use Chernoff bounds as we did in the proof of \cref{proposition:ltld-weaker}. Markov's inequality can still be applied, but it leads to weaker bounds. 
We start by formally setting up the renormalization. 

We recall the parameter $\beta$ from the connectivity function in \eqref{eq:connection-prob-gen}, and the critical value $\beta_c$ from \eqref{eq:betac}. 
From \ref{def:ksrg-alt}, we recall that we can construct the vertex set $\CV$ as the union of two independent PPPs $\CV^\sss{\mathrm{base}}$ and $\CV^\sss{\mathrm{spr}}$  with respective spatial intensities $q$ and $1-q$, with $\beta q^d>\beta_c$ gives two independent PPPs. The choice of $q$ ensures that the graph $\CG^\sss{\mathrm{base}}=(\CV^\sss{\mathrm{base}}, \CE^\sss{\mathrm{base}})$ induced by vertices in $\CV^\sss{\mathrm{base}}$ is supercritical itself. 
    Let $\theta^\mathrm{\sss{base}}$ be the percolation probability of $\CG^\sss{\mathrm{base}}$. Lastly, w.l.o.g.\ we assume that $\alpha< \infty$, as otherwise $\zeta_\mathrm{hl}=(\tau-1)/\alpha-(\tau-2)<0$.     
\begin{definition}[renormalization scheme  with overlapping boxes]\label{def:renorm-overlap}
    Let $\zeta=\zeta_\mathrm{hl}>0$, and let $\varepsilon'>0$ be a sufficiently small constant. Let \begin{equation}\label{eq:berger-eta}
    \eta:=\lceil2(1/\varepsilon' -1)\rceil.
    \end{equation} 
    Given $\alpha\in(1,\infty)$, let $s=s(\alpha, \varepsilon')$ be the smallest integer satisfying $s^{\varepsilon'+\alpha }(s+1)^{-\alpha}\ge 1$ (so that $ \tilde s^{(\varepsilon'+\alpha)}(\tilde  s+1)^{-\alpha}\ge 1$ for all $\tilde s \ge s$).
Let $n_0\ge \underline{n}_0\ge 1$ be two large constants. Then, we iteratively define for $i\ge 1$,
    \begin{equation}\label{eq:berger-mi-ni}
    m_i:= (i+s)^{(2+\eta)d}, \qquad \underline n_i:=\underline n_{i-1}m_i=\underline n_0\prod_{\ell=1}^i (\ell+s)^{(2+\eta)d}, 
    \qquad 
    n_i:=\big(\underline n_i^{1/d}+ n_0^{1/d}-\underline n_0^{1/d}\big)^d.
    \end{equation}       
    Let $\CQ^\sss{\mathrm{sub}}_i(x)$ and $\CQ_i(x)$ denote boxes of volume $\underline n_i$ and $n_i$ centered at $x$.
    We define a subgraph $\widehat \CG_{i, x}:=(\widehat \CV_{i, x}, \widehat \CE_{i, x})$ of the graph $\CG$ on $\CQ^\sss{\mathrm{sub}}_i(x)$. We call $\widehat \CV_{i,x}$ the level-$i$ vertex set inside $\CQ_i(x)$ and define it as 
    \begin{equation}
        \widehat\CV_{i,x}:= \CV^\sss{\mathrm{base}}_{\CQ_i(x)}\cup\CV^\sss{\mathrm{spr}}_{\CQ^\sss{\mathrm{sub}}_i(x)}[1, 2\underline n_i^\gamma).\label{eq:def-xi-berger}
    \end{equation}
   Let $\CE^\sss{\mathrm{base}}_{i,x}$ be the edges of $\CG^{\sss{\mathrm{base}}}$ inside $\CQ_i(x)$. Denote then by $\CE^{\sss{\mathrm{spr\textnormal{-}base}}}_{i,x}$ the set of edges of $\CG$ inside $Q_i^{\sss{\mathrm{sub}}}(x)$ with one endpoint in  $\CV^\sss{\mathrm{spr}}[1, 2\underline n_i^\gamma)$ and one in $\CV^\sss{\mathrm{base}}$, formally $\{\{u, v\}\in\CE_{\CQ^\sss{\mathrm{sub}}_i(x)}: u\in \CV^\sss{\mathrm{spr}}_{\CQ^\sss{\mathrm{sub}}_i(x)}[1, 2\underline n_i^\gamma), v\in \CV^\sss{\mathrm{base}}_{\CQ^\sss{\mathrm{sub}}_i(x)}\}$.
   Then we say that the level-$i$ edge set inside $\CQ_i(x)$ is 
    \begin{equation}\label{eq:def-vertex-berger}
        \widehat\CE_{i,x}:= \CE^\sss{\mathrm{base}}_{i,x}\cup\CE^{\sss{\mathrm{spr\textnormal{-}base}}}_{i,x}.
    \end{equation}    
    We define the graphs $\widehat \CG_{i, x}:=(\widehat \CV_{i, x}, \widehat \CE_{i, x})$.
    Let $\rho_0:=\theta^\sss{\mathrm{base}}/2$, and define for $i\ge 1$ \begin{equation}\label{eq:berger-rho-i}
    \rho_i:=1/(i+s)^{2d}.
    \end{equation}
    We call a box $\CQ_i(x)$ $(i, \varepsilon')$-good  if the graph $\widehat \CG_{i, x}$ satisfies the event
    \begin{equation}\label{eq:i-good-def-berger}
    \begin{aligned}
       \CA_{\mathrm{i\textnormal{-}good}, \varepsilon'}(x):= \Big\{ \exists\text{comp.\ $\CC_{i, x}$ in }\widehat\CG_{i,x}&: |\CC_{i, x}\cap\CQ^\sss{\mathrm{sub}}_i(x)| \ge  \underline n_i\prod_{\ell=0}^i\rho_\ell, \\&\hspace{10pt}|\CV^\sss{\mathrm{spr}}_{\CQ^\sss{\mathrm{sub}}_i(x)}[\underline n_i^\gamma, 2\underline n_i^\gamma)\cap\CC_{i, x}| \ge \underline{n}_i^{\zeta-2\varepsilon'}\Big\}.
    \end{aligned}
    \end{equation}   
\end{definition}

The picture about the boxes $\CQ^\sss{\mathrm{sub}}_i(x), \CQ_i(x)$ one should have in mind is as follows: $\CQ_i(x)$, the box of volume $n_i$, is the ``fattening'' of $\CQ^\sss{\mathrm{sub}}_i(x)$ by at most an $n_0$-dependent constant. By construction, $\CQ^\sss{\mathrm{sub}}_i(x)$ can be partitioned into exactly  $\underline n_i/\underline n_0$ many non-overlapping small boxes $(\CQ^\sss{\mathrm{sub}}_0(x_j))_{j\ge 1}$  of volume $\underline n_0$, centered at some $(x_j)_{j\ge 1}$. If we take the union of the overlapping fatter boxes $(\CQ_0(x_j))_{j\ge 1}$ of volume $n_0>\underline n_0$, we obtain $\CQ_i(x)$. 
Since $\underline n_0, n_0$ are both constants, we obtain that 
    \begin{equation}\label{eq:ni-bigger-box}
        n_i=(1+o(1))\underline n_i,\qquad\text{as }i\to\infty.
    \end{equation}
The edge set at level $i$ consists of all edges in $\CE^\sss{\mathrm{base}}$ whose vertices are in the larger box $\CQ_i(x)$, and additionally, the edges from certain ``high-mark vertices'' in $\CV^\sss{\mathrm{spr}}$ to $\CV^\sss{\mathrm{base}}$ with both endpoints inside the smaller box $\CQ^\sss{\mathrm{sub}}_i(x)$. 

For the connected component that realizes the event $\CA_{\mathrm{i\textnormal{-}good}, \varepsilon'}(x)$, we have a lower bound on the proportion of vertices inside the smaller box $\CQ^\sss{\mathrm{sub}}_i(x)$, and also on the amount of vertices in the range $[\underline n_i^\gamma, 2\underline n_i^\gamma)$ in the sprinkled vertex set in $\CQ^\sss{\mathrm{sub}}_i(x)$.  These criteria in \eqref{eq:i-good-def-berger} do not require any information about vertices in the `box-annulus' $\CQ_i(x)\setminus \CQ^\sss{\mathrm{sub}}_i(x)$.

We can compute the density of the component $\CC_{i,x}$ inside the smaller box $\CQ^\sss{\mathrm{sub}}_i(x)$. By definition of $\underline n_i$ in \eqref{eq:berger-mi-ni} and $\rho_i$ in~\eqref{eq:berger-rho-i}, for any $\underline n_0$ sufficiently large, using that $\eta/(\eta+2)\le \eps'$ due to~\eqref{eq:berger-eta}, this density is
    \begin{equation}
    \prod_{\ell=0}^i\rho_\ell = \rho_0\prod_{\ell=1}^i (\ell+ s)^{-2d}= \rho_0(\underline n_i/\underline n_0)^{-2/(2+\eta)}\ge (\theta^{\sss{\mathrm{base}}}/2)\cdot  n_i^{-\varepsilon'}.\label{eq:berger-rho-bound}
    \end{equation}
    Hence, combined with~\eqref{eq:ni-bigger-box}, for all $i$ sufficiently large,
    \begin{equation}\label{eq:berger-sub-pr1}
    \Prob\big(|\CC_{n_i}^\sss{(1)}| \ge n_i^{1-\varepsilon'}\theta^\sss{\mathrm{base}}/4\big) \ge \Prob\big( \CA_{\mathrm{i\textnormal{-}good}, \varepsilon'}(0)\big).
    \end{equation}
  This is closely related to what we aim for in \cref{lemma:ltld-weakest}. In the next three claims, we will show that the right-hand side can be made arbitrarily close to $1$ for all $i$ when $\underline n_0$ is sufficiently large. Afterwards we prove \cref{lemma:ltld-weakest} to bound the left-hand side for arbitrary $n$. 
      \begin{claim}[Ergodicity and uniqueness]\label{claim:ergodic-unique}
        Consider a supercritical KSRG as in \cref{lemma:ltld-weakest} and Definition \ref{def:renorm-overlap}. Then the measure $\Prob$ is ergodic, and there is a unique infinite component.
        \end{claim}
        \begin{proof}
             We start with ergodicity. We construct the graph as a marked Poisson point process $\widetilde\CV$ such that the graph remains invariant under translation. These marks are `additional' to the vertex marks $w_v$. Namely, we equip each vertex $v=(x_v, w_v)\in\CV$ with a sequence of iid $\mathrm{Unif}[0,1]$ random variables $(U_{vi})_{i\ge 1}$. Given this marked PPP, we write $\CU(v)=(u_i)_{i\ge 1}$ for all vertices whose first spatial coordinate is strictly larger than the first spatial coordinate of $v$, increasingly ordered with respect to their Euclidean distance to $v$. We include an edge between $v$ and $u_i$ iff   $U_{vi}\le \mathrm{p}(v, u_i)$.
                This yields a graph that has the same distribution as in \cref{def:ksrg}. By the adaptation of~\cite[Proposition 8.13]{last2017lectures} to marked Poisson point processes (see also~\cite[Exercise 10.1]{last2017lectures}), such a process is ergodic: any event in the invariant $\sigma$-algebra has either probability $0$ or $1$.

    The uniqueness of the infinite cluster follows by the classical Burton-Keane argument~\cite{BurtonKeane}, which makes implicit use of ergodicity. We do \emph{not} sketch the Burton-Keane argument here, just highlight one point:  the argument requires that the expected size of the vertex boundary of $\Lambda_n$ (i.e., the number of vertices in $\Lambda_n$ with an edge leaving $\Lambda_n$) is of smaller order than the total number of vertices inside $\Lambda_n$. Formally, we have to show that 
    \begin{equation}\label{eq:vertex-boundary-app}
   \mathrm{VB}_n:= \E\big[|\{v\in\Lambda_n: v\text{ connected by an edge to }\Lambda_n^\complement\}|\big] 
    \end{equation}
 is $o(n)$.   Fix $\eps>0$ arbitrarily small, and let $\delta=\delta(\eps)$ be a small constant depending on $\eps$, such that the expected number of vertices in $\Lambda_n$ within distance $\delta n^{1/d}$ from the boundary $\partial\Lambda_n$ is at most $\eps n$. Since $\eps>0$ is arbitrary, we may bound from above the vertex boundary in \eqref{eq:vertex-boundary-app} by counting all these vertices in $\mathrm{VB}_n$. The remaining vertices are in $\Lambda_{n(1-\delta)^d}$, and they only contribute to the vertex boundary in \eqref{eq:vertex-boundary-app} if they have an edge of length $\delta n^{1/d}$. Thus, 
    \begin{equation}\label{eq:vertex-boundary-longer-app}
    \mathrm{VB}_n\le \eps n + \E\big[|\{v\in\Lambda_n: v\text{ incident to edge of length at least $\delta n^{1/d}$}\}|\big].
    \end{equation}
    We use translation invariance of the graph and Mecke's formula, and arrive at
      \[
  \mathrm{VB}_n\le \eps n +n (1-\delta)^d \Prob^\sss{0}\big(0\text{ incident to edge of length at least $\delta n^{1/d}$}\big).
    \]
    When $\alpha>1$ and $\tau>2$ in \cref{assumption:main}, the graph is locally finite, see also \cite{DeiHofHoo13, hirsch2017gilbertgraph, luchtrathThesis2022}, which implies that the probability on the left-hand side tends to $0$ as $n\to\infty$. Therefore, the expected size of the vertex boundary is $o(n)$. By the Burton-Keane argument, the uniqueness of the infinite cluster follows.
        \end{proof}
    \begin{claim}[Induction base for overlapping boxes]\label{claim:induction-base-ii}
        Consider a KSRG in the setting of \cref{lemma:ltld-weakest} and Definition \ref{def:renorm-overlap}. Let $\varepsilon'>0$ is a sufficiently small constant. For each $\delta>0$, there exists a constant $\underline n_0^\star>0$ such that for each $\underline n_0\ge \underline n_0^\star$ there exists $n_0>0$ such that for all $x\in\R^d$, 
        \[
        \Prob\big(\CQ_0(x) \text{ is $(0, \varepsilon')$-good}\big)\ge 1-\delta.
        \]
        \end{claim}
        \begin{proof}
    We estimate the probability of $\CA_{\mathrm{0\textnormal{-}good}, \varepsilon'}(x)$ defined in~\eqref{eq:i-good-def-berger}. We set $x=0$, so that $\CQ_{0}^{\sss{\mathrm{sub}}}(0)=\Lambda_{\underline n_0}$ while $\CQ_{0}(0)=\Lambda_{n_0}$. We write $\CC_\infty^{\sss{\mathrm{base}}}$ for the unique infinite component of $\CG^{\sss{\mathrm{base}}}$. By definition, $\theta^{\sss{\mathrm{base}}}=\Prob(0 \in \CC_\infty^{\sss{\mathrm{base}}})$, so that $\CC_\infty^\sss{\mathrm{base}}$ typically contains  a $\theta^{\sss{\mathrm{base}}}$-proportion of the vertices in $\Lambda_{\underline n_0}$ by translation invariance. By ergodicity (\cref{claim:ergodic-unique}), the proportion of vertices in $\Lambda_{\underline n_0}$ in $\CC_\infty^{\sss{\mathrm{base}}}$ converges in probability to $\theta^{\sss{\mathrm{base}}}$ as $\underline{n}_0\to\infty$. In \cref{def:renorm-overlap} we defined $\rho_0=\theta^{\sss{\mathrm{base}}}/2$.
    Hence, for  $\underline n_0$ sufficiently large, 
    \[
    \Prob\big(|\Lambda_{\underline n_0}\cap\CC_\infty^{\sss{\mathrm{base}}}|\ge \rho_0 \underline n_0 \big)\ge 1-\delta/4.
    \]
    Since the infinite component is unique (\cref{claim:ergodic-unique}) and the graph is locally finite, given $\delta$, there exists $n_0=n_0(\underline n_0)$ such that with probability at least $1-\delta/2$, the separate components that $\CC_\infty^{\sss{\mathrm{base}}}$ induces inside $\Lambda_{\underline n_0}$ are all in a single component inside the larger box $\Lambda_{n_0}$. Formally, let $n_0$ be so large that 
    \begin{equation}\label{eq:berger-init}
\Prob\big(|\Lambda_{\underline n_0}\cap\CC_\infty^{\sss{\mathrm{base}}}|\ge \rho_0 \underline n_0, \Lambda_{\underline n_0}\cap\CC_\infty^{\sss{\mathrm{base}}}\mbox{ in same component in }\CG_{n_0}^{\sss{\mathrm{base}}} \big)\ge 1-\delta/2
    \end{equation}
   holds. This also implies that the event $\CA_\mathrm{base}:=\{\exists\text{ comp. }\CC\text{ in }\CG_{n_0}: |\CC\cap \Lambda_{\underline n_0}|\ge \rho_0 \underline n_0\}$ occurs with probability at least $1-\delta/2$.
   We now condition on the realization of the graph $\CG^{\sss{\mathrm{base}}}$ satisfying $\CA_\mathrm{base}$. 
   Let $\CC$ be an arbitrary component in $\CG_{n_0}^{\sss{\mathrm{base}}}$  that the event $\CA_\mathrm{base}$ describes. This component only uses vertices and edges of $\CG_{n_0}^{\sss{\mathrm{base}}}$, which is also a subgraph of the graph $\widehat \CG_{0,x}$. The component $\widehat \CC\supseteq \CC$ in  $\widehat \CG_{i,0}$ thus satisfies the first condition of $\CA_{\mathrm{0\textnormal{-}good}, \eps'}$ in~\eqref{eq:i-good-def-berger}. For the second criterion of $\CA_{\mathrm{0\textnormal{-}good}, \varepsilon'}$ we will compute how many vertices of $\CV_{\underline n_0}^\sss{\mathrm{spr}}[\underline n_0^\gamma,2 \underline n_0^\gamma)$ connect to $\CC$ (and thus form $\widehat \CC$). Since we revealed a realization of $\CG^{\sss{\mathrm{base}}}$ satisfying $\CA_\mathrm{base}$, $\CC$ is known, and we obtain 
    \[
    \ind{\CA_\mathrm{base}}\Prob\big(\neg \CA_{\mathrm{0\textnormal{-}good}, \varepsilon'}(x)\mid \CG^\sss{\mathrm{base}}\big)\le \ind{\CA_\mathrm{base}}\Prob\big(|\{u\in\CV^\sss{\mathrm{spr}}[\underline n_0^\gamma,2 \underline n_0^\gamma): u\sim \CC\}|< \underline{n}_0^{\zeta-2\varepsilon'} \mid \CG^\mathrm{base}\big).
    \]
      The vertices in $\CV_{\underline n_0}^\sss{\mathrm{spr}}[\underline n_0^\gamma,2 \underline n_0^\gamma)$  form  a PPP with intensity $(1-q)\mathrm{Leb}(\cdot) \times (\tau-1)w^{-\tau} \mathrm dw$. We compute the mean and estimate it  from below using that $1-\gamma (\tau-1)=\zeta$:  \[
      \E\big[|\CV^\sss{\mathrm{spr}}_{\underline n_0}[\underline n_0^\gamma,2 \underline n_0^\gamma)| \big]\ge 
      (1-q)\underline n_0^{1-\gamma(\tau-1)}(1-2^{-(\tau-1)})\ge \underline n_0^{\zeta-\varepsilon'}.\]
    We estimate the mark of each vertex in $\CC\cap \Lambda_{\underline n_0}$ from below by one, the spatial distance by at most the diameter of $\Lambda_{\underline n_0}$, and the size of $\CC\cap \Lambda_{\underline n_0}$ by $\rho_0\underline n_0$. So, a single vertex in $\CV_{\underline n_0}^\sss{\mathrm{spr}}[\underline n_0^\gamma,2 \underline n_0^\gamma)$ connects to $\CC$ with probability at least 
       \[
       \begin{aligned}
    1-\Big(1-p\Big(1\wedge\beta\frac{\underline n_0^\gamma}{d^{d/2}\underline n_0}\Big)^\alpha\Big)^{\rho_0\underline n_0}&\ge 1-\exp\big(-p\beta^\alpha \rho_0d^{-\alpha d/2}\underline n_0^{\gamma\alpha-\alpha+1}\big)\\
    &=1-\exp\big(-p\beta^\alpha \rho_0d^{-\alpha d/2}\big)  \ge 2\underline n_0^{-\varepsilon'}.
    \end{aligned}\]
    Here we used that $\gamma=1-1/\alpha$ to obtain the second row so that powers of $\underline n_0$ vanish, and that the last inequality holds for all $\underline n_0$ sufficiently large. Given the vertex set, vertices in $\CV^\sss{\mathrm{spr}}$ connect independently by an edge to $\CC$. So, the total number of vertices in $\CV_{\underline n_0}^\sss{\mathrm{spr}}[\underline n_i^\gamma, 2\underline n_i^\gamma)$ dominates a Poisson random variable with mean $2\underline n_0^{\zeta-2\varepsilon'}$. By concentration for Poisson random variables (\cref{lemma:poisson-1}), there exists $c>0$ such that
    \[
    \ind{\CA_\mathrm{base}}\Prob\big( \CQ_0(x) \text{ is $(0, \varepsilon')$-bad}\big)\le \exp\big(-c\underline n_0^{\zeta-2\varepsilon'}\big).
    \]
    The right-hand side is at most $\delta/2$ if $\underline n_0$ is large. Together with~\eqref{eq:berger-init}, this finishes the proof.
        \end{proof}
  The key idea of \cref{claim:induction-base-ii} above is that the separate components of the infinite component in $\CG^\sss{\mathrm{base}}$ induced in a box $\Lambda_{\underline n_0}$ connect to each other in a larger box $\Lambda_{n_0}$. However, we had no control over the relation between $n_0$ and $\underline n_0$, and thus lose the density estimate of the connected component in the larger box. The multi-scale renormalization in the following claim makes this control possible, using that the added annulus in the `fattened' box  has constant width across levels, see \eqref{eq:ni-bigger-box}. 
    \begin{claim}[Iterative renormalization with overlapping boxes]\label{claim:induction-advance-ii}
         Consider a KSRG in the setting of \cref{lemma:ltld-weakest}, with the boxing scheme from \cref{def:renorm-overlap}. Assume $\varepsilon'>0$ is a sufficiently small constant. For each $\delta>0$, there exists a constant $\underline n_0^\star=\underline n_0^\star(\varepsilon', \delta)>0$ such that for each $\underline n_0\ge \underline n_0^\star$ there exists $n_0>0$ so that for all $x\in\R^d$ and all $i\ge 0$,
        \begin{equation}\label{eq:berger-box-bad-i}
        \Prob\big(\CQ_i(x) \text{ is $(i, \varepsilon')$-good}\big)\ge 1-\delta.
        \end{equation}
        \begin{proof}
            We follow the same steps \emph{(Sub), (Con)}, and \emph{(Mark)}, similar to the proof of \cref{claim:induction-advance}. By \cref{claim:induction-base-ii} the statement holds for $i=0$, so we may assume inductively that it holds until $i-1$ and that $i\ge 1$.

            \smallskip\noindent 
    \emph{Good subboxes.}
Consider $\CQ_i(x)$, and take a partition of the smaller box $\CQ^\sss{\mathrm{sub}}_i(x)$ into $m_{i}$ many subboxes of volume $\underline n_{i-1}$ with centers $(x_j)_{j\le m_i}$. We then obtain $\big((\CQ^\sss{\mathrm{sub}}_{i-1}(x_j), \CQ_{i-1}(x_j))\big)_{j\le m_i}$, where $(\CQ^\sss{\mathrm{sub}}_{i-1}(x_j))_{j\le m_i}$ are the subboxes forming the partitioning, and $(\CQ_{i-1}(x_j))_{j\le m_i}$ their fattened versions. Since two neighboring fattened subboxes overlap, the events that two fattened boxes are $(i-1, \varepsilon')$-good are not independent. Let $\CA_\mathrm{sub}$ be the event that at least $\rho_i m_i$ fattened boxes are $(i-1, \varepsilon')$-good.    
    By Markov's inequality and translation invariance,
    \begin{align}
    \Prob\big(\neg\CA_\mathrm{sub}\big) &= \Prob\big(|\{j: \CQ_{i-1}(x_j) \text{ is $(i, \varepsilon')$-bad} \}| \ge (1-\rho_i) m_i \big)\nonumber \\
    &\le \frac{1}{1-\rho_i}\Prob\big(\CQ_{i-1}(0) \text{ is $(i, \varepsilon')$-bad}\big)\nonumber\\&\le \Big(1+\frac{2}{(i+s)^{2d}}\Big)\Prob\big( \CQ_{i-1}(0) \text{ is $(i, \varepsilon')$-bad}\big),\label{eq:con-berger}
    \end{align}
    where for the second bound we used $\rho_i=(i+s)^{-2d}$, so that $1/(1-\rho_i)\le1+2(i+s)^{-2d}$ for all $i\ge 1$ and $d\ge 1$.
    \smallskip

    \noindent    
   \emph{Connection between subboxes.}
    We write $\mathrm{pre}$-$\widehat\CG_{i, x}$ for the union of the graphs $\cup_{j \le m_i} \widehat \CG_{i-1, x_j}$ inside the subboxes at level $i-1$, see~\eqref{eq:def-xi-berger} and~\eqref{eq:def-vertex-berger}, and assume that $\CA_\mathrm{sub}$ holds. Let $\mathrm{inter}$-$\widehat\CG_{i,x}$ denote the graph on the same vertex set as $\mathrm{pre}$-$\widehat\CG_{i, x}$, but including also all edges between the vertices in $\cup_{j \le m_i} \widehat \CV_{i-1, j}$. 
   \smallskip 
   
     \noindent\emph{The auxiliary graph.}
    Similar to below~\eqref{eq:pre-graph}, we re-label the subboxes so that the subboxes with indices $j\le  \rho_i m_i$ are all $(i-1, \varepsilon')$-good.  We consider the auxiliary graph $\CH_i=(\CV_{\CH_i}, \CE_{\CH_i})$: each vertex $j\in[ \rho_i m_i]=:\CV_{\CH_i}$ corresponds to a component $\CC_{i-1, x_j}$ that realizes the event $\CA_{i\textnormal{-}\mathrm{good}, \varepsilon'}(x_j)$ in the $j$-th subbox, see~\eqref{eq:i-good-def-berger}. To determine the edges of the auxiliary graph, we use the  vertices at level $i-1$ that belong to inter-$\CG_{i,x}$ as follows: we connect vertex $j$ and $k$ in $\CH_i$ if there is an edge between a sprinkled vertex in $\CC_{i-1,j}\cap \CQ_{i-1}^{\sss{\mathrm{sub}}}(x_j)$ to some base-vertex in $\CC_{i-1,k}\cap \CQ_{i-1}^{\sss{\mathrm{sub}}}(x_k)$ or vice versa: 
    \[ 
    \begin{aligned}
    \{j \sim_{\CH_i} k\} \ \Longleftrightarrow\ &\Big\{\big(\CC_{i-1,j}\cap\CV^\sss{\mathrm{spr}}_{\CQ^\sss{\mathrm{sub}}_{i-1}(x_{j})}[\underline n_{i-1}^\gamma, 2 \underline n_{i-1}^\gamma)\big)\sim_{\widehat\CG_{i,x}} \big(\CC_{i-1,k}\cap\CV^\sss{\mathrm{base}}_{\CQ^\sss{\mathrm{sub}}_{i-1}(x_{k})}\big)\Big\}\\
    &\, \cup \,
    \Big\{\big( \CC_{i-1,k}\cap\CV^\sss{\mathrm{spr}}_{\CQ^\sss{\mathrm{sub}}_{i-1}(x_{k})}[\underline n_{i-1}^\gamma, 2 \underline n_{i-1}^\gamma)\big)\sim_{\widehat\CG_{i,x}} \big(\CC_{i-1,j}\cap\CV^\sss{\mathrm{base}}_{\CQ^\sss{\mathrm{sub}}_{i-1}(x_{j})}\big)\Big\}.
    \end{aligned}
    \]
 Each good subbox $\CQ_{i-1}(x_j)$ is contained in $\CQ^\sss{\mathrm{sub}}_i(x)$,  and the connected component $\CC_{i-1,j}$ satisfies $|\CC_{i-1,j}\cap \CQ^\sss{\mathrm{sub}}_{i-1}(x_j)|\ge \underline n_{i-1}\prod_{\ell=0}^{i-1}\rho_\ell$. On $\CA_\mathrm{sub}$ there are at least $\rho_i m_i$ good subboxes. Thus, if $\CH_i$ is connected, then there is a component $\CC_{i,x}$ satisfying:
    \[
  |\CC_{i,x} \cap \CQ^\sss{\mathrm{sub}}_i(x)|  \ge \rho_i m_i\cdot \underline n_{i-1}\prod_{\ell=0}^{i-1}\rho_\ell = \underline n_i\prod_{\ell=0}^{i}\rho_\ell.
    \]           
    We bound from above the probability that $\CH_i$ is not connected by the probability that $\CH_i$ is not the complete graph.  We already lost sharp bounds in \eqref{eq:con-berger} as opposed to  \eqref{eq:first-good-crit} and \eqref{eq:ind-pr-con} where we used the connectivity of the Erd{\H o}s-R\'enyi random graph. So even though the edges in $\CH_i$ are conditionally independent given pre-$\widehat \CG_{i,x}$, we will use Markov's inequality to bound edge-presence events.
    We compute conditionally on the realization of $\mathrm{pre}\textnormal{-}\widehat\CG_{i, x}$, the union of subgraphs inside level-($i- 1$)-subboxes that
    \[
    \begin{aligned}
    \ind{\CA_\mathrm{sub}}\Prob\bigg(\nexists \text{component }\CC_{i,x} \text{ in }\mathrm{inter}\textnormal{-}\widehat\CG_{i,x}&: |\CC_{i,x}\cap \CQ^\sss{\mathrm{sub}}_i(x)|\ge\underline n_i\prod_{\ell=0}^i\rho_\ell \, \Big|\,  \mathrm{pre}\textnormal{-}\widehat\CG_{i, x}\bigg)\\ 
    &\le \ind{\CA_\mathrm{sub}}\Prob\big(\CH_i\mbox{ is not connected}\mid\mathrm{pre}\textnormal{-}\widehat\CG_{i, x}\big)\\
    &\le m_i^2 \max_{j, k}\ind{\CA_\mathrm{sub}}\Prob\big(j\nsim_{\CH_i}k\mid \mathrm{pre}\textnormal{-}\widehat\CG_{i, x}\big).
    \end{aligned}
    \]
    To bound the probability on the right-hand side, we use the two properties of $ \CC_{i-1, j}$ of a good $(i-1, \varepsilon')$ subbox in~\eqref{eq:i-good-def-berger} about its density in  $ \CQ^\sss{\mathrm{sub}}_{i-1}(x_j)$ and its level-$(i-1)$  high-mark vertices.
         Since the union of the smaller level-$(i-1)$-boxes forms $\CQ_i^\sss{\mathrm{sub}}(x)$ (see reasoning below \cref{def:renorm-overlap}), the maximal distance between two vertices in $\CQ_{i-1}(x_{j})$ and $\CQ_{i-1}(x_{k})$ is $\sqrt{d}\underline n_{i}^{1/d}$.
    Moreover, by the definition of the level-$(i-1)$ edge set in~\eqref{eq:def-vertex-berger}, none of the edges between $\CV^\sss{\mathrm{spr}}_{\CQ_{i-1}^\sss{\mathrm{sub}}(x_{j})}$ and $\CV^\sss{\mathrm{base}}_{\CQ_{i-1}^\sss{\mathrm{sub}}(x_{k})}$ have been revealed in pre-$\widehat\CG_{i,x}$ and they are thus conditionally independently present. We bound this probability from above by only requiring that there is no edge between $\CC_{i-1,j}\cap \CV^\sss{\mathrm{spr}}_{\CQ_{i-1}^\sss{\mathrm{sub}}(x_{j})}[\underline n_{i-1}^\gamma, 2\underline n_{i-1}^\gamma)$ and $\CC_{i-1,k}\cap \CV^\sss{\mathrm{base}}_{\CQ_{i-1}^\sss{\mathrm{sub}}(x_{k})}$, and we use that we have lower bounds on the sizes of these sets in \eqref{eq:i-good-def-berger}:  
     \[
    \begin{aligned}
\ind{\CA_\mathrm{sub}}\Prob\big(\CH_i\mbox{ is not connected}&\mid\mathrm{pre\textnormal{-}}\widehat\CG_{i, x}\big) \\
&\le m_i^2\prod_{u\in \CC_{i-1,j}\cap \CV^\sss{\mathrm{spr}}_{\CQ^\sss{\mathrm{sub}}_{i-1}(x_{j_1})}[\underline n_{i-1}^\gamma, 2 \underline n_{i-1}^\gamma)}\prod_{v \in \CC_{i-1,k}\cap\CV^\sss{\mathrm{base}}_{\CQ^\sss{\mathrm{sub}}_{i-1}(x_{j_2})}}(1-\mathrm{p}(u,v)) \\
    & \le m_i^2\Big(1-p\Big(1\wedge \beta\frac{\underline n_{i-1}^\gamma}{d^{d/2}\underline n_i}\Big)^\alpha\Big)^{\underline n_{i-1}^{1+\zeta-2\varepsilon'}\prod_{\ell=0}^{i-1}\rho_\ell}.
    \end{aligned}
    \]
    We use that $\underline n_i=m_i\underline n_{i-1}$ and $\gamma=1-1/\alpha$, so that the minimum is attained at the second term for any $n_0$ sufficiently large. Hence, 
    \begin{align}
    \ind{\CA_\mathrm{sub}}\Prob\big(\CH_i\mbox{ is not connected}\mid\mathrm{pre\textnormal{-}}\widehat\CG_{i, x}\big)
    &\le 
    m_i^2\exp\Big(-p\beta^\alpha m_i^{-\alpha}\underline n_{i-1}^{(\gamma-1)\alpha}\underline n_{i-1}^{\zeta+1-2\varepsilon'}\prod_{\ell=0}^{i-1}\rho_\ell\Big) \nonumber\\
    &=m_i^2\exp\Big(- m_i^{-\alpha}\underline n_{i-1}^{\zeta-2\varepsilon'}p\beta^\alpha\prod_{\ell=0}^{i-1}\rho_\ell\Big).\label{eq:berger-not-conn}
    \end{align}
    The expression $\prod_{\ell=0}^{i-1}\rho_\ell$ inside the exponent is at least $\underline n_i^{-\varepsilon'}$ by~\eqref{eq:berger-rho-bound}.
We would also like to replace the factor $m_i^{-\alpha}$ by $\underline n_{i-1}^{-\eps'}$. For this it is enough to show that $\underline n_{i-1}^{\varepsilon'}\ge m_i^\alpha$ with $m_i$ and $\underline n_i$ from~\eqref{eq:berger-mi-ni}. We show this by induction. For $i=1$, $m_1^\alpha=(1+s)^{(2+\eta)d\alpha}\le \underline n_0^{\eps'}$ holds for any $\underline n_0$ sufficiently large. We advance the induction, assuming $ \underline n_{i-2}^{\varepsilon'}\ge m_{i-1}^\alpha$. Using the recursion defining $\underline n_i$ and the induction hypothesis, 
 \[
    \underline n_{i-1}^{\varepsilon'} =  
    m_{i-1}^{\varepsilon'} \underline n_{i-2}^{\varepsilon'} \ge m_{i-1}^{\varepsilon'+\alpha}=m_{i}^\alpha \Big(\frac{m_{i-1}}{m_{i}}\Big)^\alpha m_{i-1}^{\varepsilon'}=m_{i}^\alpha \Big(\frac{i-1+s}{i+s}\Big)^{(2+\eta)d\alpha} (i-1+s)^{(2+\eta)d\varepsilon'}\ge m_{i}^\alpha,
    \]
    by the choice of $s$ just above~\eqref{eq:berger-mi-ni}.
     Substituting this bound into~\eqref{eq:berger-not-conn}, and discounting the prefactor $m_{i-1}^2$ and constant factors inside the exponential by another factor $\underline n_{i-1}^{\varepsilon'}$, we obtain

    \begin{equation}\label{eq:berger-not-conn-2}
     \ind{\CA_\mathrm{sub}}\Prob\bigg(\nexists \text{comp. }\CC_i \text{ in }\mathrm{inter\textnormal{-}}\widehat\CG_{i,x}: |\CC_i\cap \CQ^\sss{\mathrm{sub}}_i(x)|\ge\underline n_i\prod_{\ell=0}^i\rho_\ell \, \Big|\,  \mathrm{pre\textnormal{-}}\widehat\CG_{i, x}\bigg)
     \le 
     \exp\big(-\underline n_{i-1}^{\zeta-5\varepsilon'}\big).
    \end{equation}
    
    \smallskip\noindent 
    \emph{Sprinkling high-mark vertices.} Let us denote the opposite event of \eqref{eq:berger-not-conn-2} by
    \[ \CA_\mathrm{conn}:= \Big\{\exists \text{comp. }\CC_i \text{ in }\mathrm{inter\textnormal{-}}\widehat\CG_{i,x}: |\CC_i\cap \CQ^\sss{\mathrm{sub}}_i(x)|\ge\underline n_i\prod_{\ell=0}^i\rho_\ell \Big\}\].
 On $\CA_\mathrm{conn}$, $\widehat\CG_{i,x}$ satisfies the first criterion of being $(i, \varepsilon')$-good in~\eqref{eq:i-good-def-berger}. We now look at the second criterion of $\CA_{\mathrm{i\textnormal{-}good}, \varepsilon'}(0)$ in~\eqref{eq:i-good-def-berger},  conditioned on the graph $\mathrm{inter}\textnormal{-}\widehat\CG_{i, x}$ satisfying the event $\CA_\mathrm{conn}\cap\CA_\mathrm{sub}$. We reveal the sprinkled vertices  $\CV^{\sss{\mathrm{spr}}}[\underline{n}_i^\gamma, 2\underline{n}_i^\gamma)$ with location in $\CQ_i^\sss{\mathrm{sub}}(x)$. These vertices are independent of $\mathrm{inter\textnormal{-}}\widehat\CG_{i, x}$, which only contains vertices with mark below $2\underline n_{i-1}^\gamma$. This mark threshold is below $\underline n_i^\gamma$ for sufficiently large $\underline n_0$ and small $\varepsilon'$ see~(\ref{eq:berger-eta}--\ref{eq:berger-mi-ni}). The expected  number of revealed vertices is then $(1-q)\underline{n}_i^{1-\gamma(\tau-1)}(1-2^{-(\tau-1)})$ using~\cref{def:ksrg-alt}.

  We bound from below the probability that such a vertex connects by an edge to one of the at least $\underline{n}_i\prod_{\ell=0}^i\rho_\ell$ many vertices in $\CC_i\cap \CQ^\sss{\mathrm{sub}}_i(x)$ of $\mathrm{inter}\textnormal{-}\widehat\CG_{i, x}$. The probability that a single vertex of mark at least $\underline{n}_i^\gamma$ does not connect by an edge to $\CC_i$ is at most 
    \[
    \Big(1-p\Big(1\wedge \beta\frac{\underline{n}_i^\gamma}{d^{d/2}\underline{n}_i}\Big)^\alpha\Big)^{\underline{n}_i\prod_{\ell=0}^i\rho_\ell}
    \le 
    \exp\Big(-p\beta^\alpha d^{-\alpha d/2}\underline n_i^{(\gamma-1)\alpha+1}\prod_{\ell=0}^i\rho_\ell \Big)=
    \exp\Big(-p\beta^\alpha d^{-\alpha d/2}\prod_{\ell=0}^i\rho_\ell \Big),
    \]
    using again $\gamma=1-1/\alpha$ and cancellation of the powers of $\underline n_i$. Each vertex in $\CV^\sss{\mathrm{spr}}$ connects by an edge independently to the vertices in $\CC_i$. Thus, the total number of vertices from $\CV^\sss{\mathrm{spr}}[\underline n_i^\gamma, 2\underline n_i^\gamma)\cap\CQ^\sss{\mathrm{sub}}_i(x)$ connecting to $\CC_i$ dominates a Poisson random variable with mean 
    \[
   (1- q)(1-2^{-(\tau-1)})\underline{n}_i^{\zeta}\cdot \Big(1-\exp\Big(-p\beta^\alpha d^{-\alpha d/2}\prod_{\ell=0}^i\rho_\ell \Big)\Big)
    \ge 
    c\underline{n}_i^{\zeta}\prod_{\ell=0}^i\rho_\ell\ge c\rho_0\underline{n}_i^{\zeta-\varepsilon'}\ge 2\underline{n}_i^{\zeta-2\varepsilon'},
    \]
    where the second bound follows for some constant $c>0$, the third bound by~\eqref{eq:berger-rho-bound}, and the last bound if $\underline{n}_0$ is sufficiently large. We apply concentration for Poisson random variables (\cref{lemma:poisson-1}). We recall the definition of $\CA_{\mathrm{i\textnormal{-}good}, \varepsilon'}$ from~\eqref{eq:i-good-def-berger}, and combine it with~\eqref{eq:berger-not-conn-2} to obtain for $\underline n_0$ sufficiently large that
    \[
    \ind{\CA_\mathrm{sub}\cap\CA_\mathrm{conn}}\Prob\big(\neg \CA_{\mathrm{i\textnormal{-}good}, \varepsilon'}(0)\big)\le 2\exp\big(-\underline{n}_i^{\zeta-4\varepsilon'}\mid \mathrm{inter}\textnormal{-}\widehat\CG_{i, x}\big).
    \]
    We combine this bound with $\Prob\big(\neg\CA_\mathrm{sub}\big)\le \big(1+2(i+s)^{-2d}\big)\Prob\big(\neg\CA_{\mathrm{(i-1)\textnormal{-}good}, \varepsilon'}(0)\big)$ from~\eqref{eq:con-berger}. Then
    \begin{equation}\label{eq:recursive-i}
    \Prob\big(\neg \CA_{\mathrm{i\textnormal{-}good}, \varepsilon'}(0)\big) \le \big(1+2(i+s)^{-2d}\big)\Prob\big(\neg\CA_{\mathrm{(i-1)\textnormal{-}good}, \varepsilon'}(0)\big) + 2\exp\big(-\underline{n}_i^{\zeta-5\varepsilon'}\big).
    \end{equation}
    \smallskip
    
    \noindent 
    \emph{Combining everything: a recursive bound.}
     We finish the proof of  \eqref{eq:berger-box-bad-i} by some analysis.
   We show that the right-hand side of  \eqref{eq:recursive-i} is at most $\delta$ by working out the recursion in $i$.
       Fix $\delta'>0$ sufficiently small.    
    By \cref{claim:induction-base-ii}, $\Prob\big(\neg\CA_{\mathrm{0\textnormal{-}good}, \varepsilon'}(0)\big)\le \delta'$ for all $\underline n_0$ sufficiently large. We may assume inductively that for all $1 \le j\le i-1$, 
    \begin{equation*}\Prob(\neg\CA_{\mathrm{j\textnormal{-}good}, \varepsilon'}(0)) \le \delta' \prod_{\ell=1}^j \Big(1 + \frac{3}{(\ell+s)^{2d}}\Big).\end{equation*}
    Then, using this bound in \eqref{eq:recursive-i} we obtain
    \[
    \begin{aligned}
    \Prob(\neg\CA_{\mathrm{i\textnormal{-}good}, \varepsilon'}(0))    &\le 
    \delta' \Big(1+\frac{2}{(i+s)^{2d}}\Big)\prod_{\ell=1}^{i-1} \Big(1 + \frac{3}{(\ell+s)^{2d}}\Big) + 2\exp\big(-\underline n_i^{\zeta-5\varepsilon'}\big).
    \end{aligned}
    \]
    Showing that the right-hand side is at most $\delta'\prod_{\ell=1}^i \big(1 + \frac{3}{(\ell+s)^{2d}}\big)$ is equivalent to showing that 
    \begin{equation}\label{eq:beger-recursion-to-show}
    1+\frac{2}{(i+s)^{2d}}+\frac{2\exp\big(-\underline n_i^{\zeta-5\varepsilon'}\big)}{\delta'\prod_{\ell=1}^{i-1}(1+3/(\ell+s)^{2d})}\le 1 + \frac{3}{(i+s)^{2d}}.
    \end{equation}
    The product in the denominator on the left-hand side is bounded away from 0. Since $\underline{n}_i$ grows superpolynomially by its definition in~\eqref{eq:berger-mi-ni}, the bound~\eqref{eq:beger-recursion-to-show} is satisfied for all $i\ge 1$ if $\underline{n}_0$ is sufficiently large. This advances the induction, and we conclude that for all $i\ge 0$ and a proper choice of the constant $\delta'>0$,
    \[
    \Prob\big(\neg\CA_{\mathrm{i\textnormal{-}good}, \varepsilon'}(0)\big)\le \delta'\prod_{\ell=1}^\infty \Big(1 + \frac{3}{(\ell+s)^{2d}}\Big)=\delta.\qedhere
    \]
        \end{proof}
    \end{claim}
    
    We now prove \cref{lemma:ltld-weakest}. Similarly to the proof of Proposition \ref{proposition:ltld-weaker} on page \pageref{proof:prop-ltld-weaker}, this step generalizes the previous claim to arbitrary box size $n$ (instead of $n$ restricted along the subsequence $(n_i)_{i\ge 1}$).

    \begin{proof}[Proof of \cref{lemma:ltld-weakest}]
        Fix $\varepsilon'\in(0,\varepsilon/2)$. Let $c>0$ be a small constant. It suffices to show that for all $\delta$, there exists $n_\ast$ such that for all $n\ge n_\ast$,
        \begin{equation}\label{eq:berger-pr-1}
            \Prob\big(|\CC_n^\sss{(1)}|<n^{1-2\varepsilon'}\big)\le \delta.
        \end{equation}
        Consider the sequence $(n_i)_{i\ge 0}$ from \cref{def:renorm-overlap}, with $n_0$ given by \cref{claim:induction-advance-ii} such that 
        \begin{equation}\label{eq:berger-pr-2}
        \Prob\big(\neg\CA_{\mathrm{i\textnormal{-}good},\varepsilon'}(x)\big)\le \delta/4.
        \end{equation}
        for all $i\ge 0$. Assume $n\gg n_0$, and define  $m$, $i_\ast$, and $\tilde n$ as in~\eqref{eq:boxing-tilde-n}, i.e., 
        \begin{equation}\label{eq:berger-i-star}
            \begin{aligned}
                i_\ast=\max\{i: n_i<n\}, \qquad m:=\lfloor(n/n_{i_\ast})^{1/d}\rfloor^d, \qquad \tilde n:=n_{i_\ast}m.
            \end{aligned}
        \end{equation} 
        We will assume that $i_\ast$ is a sufficiently large constant.
        We consider a box $\Lambda_{\tilde n}$ of volume $\tilde n$ inside $\Lambda_n$ and partition this smaller box into \emph{disjoint} subboxes of volume $n_{i_\ast}$. By construction, $\Lambda_{\tilde n}$ covers a uniformly positive proportion of the larger box $\Lambda_n$, i.e., exactly as in \eqref{def:ratio-nu}, there exists an absolute constant $\nu>0$ such that 
        \begin{equation}\label{eq:berger-pr-3}
        m\ge \nu (n/n_{i_\ast}).
        \end{equation}
        Moreover, we may assume without loss of generality that $m$ is at least a large constant $m_0$, as otherwise the statement is immediately implied by~\eqref{eq:berger-sub-pr1} and \cref{claim:induction-advance-ii}.
        We denote the boxes of the partition of $\Lambda_{\tilde n}$ by $\CQ_{i_\ast}(x_1),\ldots, \CQ_{i_\ast}(x_m)$. By~\eqref{eq:berger-pr-2} and Markov's inequality,
        \begin{equation}\label{eq:berger-markov-last}
       \Prob(\CA_{\mathrm{sub}}):= \Prob\bigg(\sum_{j\in[m]}\ind{\CQ_{i_\ast}(x_j)\text{ is $(i_\ast,\varepsilon')$-good}}\ge  m/2\bigg)\ge 1-\delta/2.
        \end{equation}
        We reveal the subgraphs $\widehat\CG_{i_\ast, x_j}$ from \cref{def:renorm-overlap} in the subboxes, and denote their union by pre-$\widehat\CG_{\tilde n}$. We condition the graph on satisfying $\CA_\mathrm{sub}$: each good subbox contains a large component $\CC_j$ of size at least $n_{i_\ast}^{1-\varepsilon'}$. If the components $\CC_j$ are all connected by an edge to each other, then $\CG_{\tilde n}$ contains a connected component of size at least 
        \[
         (1/2) m \underline n_{i_\ast}\prod_{\ell=1}^{i_\ast}\rho_\ell \ge 
         (1/2) m \underline n_{i_\ast}^{1-\varepsilon'} \ge   (1/4) m  n_{i_\ast}^{1-\varepsilon'}  \ge (\nu/4)n n_{i_\ast}^{-\varepsilon'}\ge n^{1-2\varepsilon'}
             \]
        using~\eqref{eq:berger-rho-bound}, \eqref{eq:berger-pr-3}, and the fact that  $n_i=\underline n_i(1+o(1))$ as $i\to\infty$ for all $n$ is sufficiently large. 
        
        We thus need to study the connectivity between the large connected components $\CC_j$ in the good subboxes. By the definition of $(i, \eps')$ goodness in \eqref{eq:i-good-def-berger}, $\CC_{j}$ has size at least $n_{i_\ast}^{1-\varepsilon'}$, and contains at least $\underline{n}_{i_\ast}^{\zeta-2\varepsilon'}$ vertices of mark at least $\underline{n}_{i_\ast}^{\gamma}$. The edges between subboxes have not yet been revealed. The largest possible distance between two vertices in $\Lambda_{\tilde n}$ is $\sqrt{d}\tilde{n}^{1/d}=\sqrt{d}\tilde{m n_{i_\ast}}^{1/d}$.  Thus,       
        the probability that the large components $\CC_{j}$ and $\CC_{k}$ in two good subboxes $j$, $k$ are not connected by an edge in $\CG_n$ is at most 
        \[
        \begin{aligned}
        \prod_{u\in\CC_{j}, v\in\CC_{k}}\hspace{-10pt}\Big(1-\mathrm{p}(u,v)\Big) &\le \Big(1-p\Big(\beta\frac{\underline n_{i_\ast}^\gamma}{d^{d/2}n_{i_\ast}m}\Big)^\alpha\Big)^{2\underline n_{i_\ast}^{1+\zeta-3\varepsilon'}}\\
        &\le 
        \exp\Big(-p\beta^\alpha d^{-\alpha d/2}\underline n_{i_\ast}^{\gamma\alpha + 1-\alpha+\zeta-3\varepsilon'}m^{-\alpha}\Big)
        \le \exp\Big(-\underline n_{i_\ast}^{\zeta-4\varepsilon'}m^{-\alpha}\Big),
        \end{aligned}
        \]
        having used $\gamma=1-1/\alpha$ in the last step, combined with a bound of $\underline n_i^{-\varepsilon'}$ to account for the constant prefactors.
        Considering the at most $m^2$ potential edges, and combining it with~\eqref{eq:berger-markov-last}, we obtain
        \[
        \Prob\big(|\CC_n^\sss{(1)}|<n^{1-2\varepsilon'}\big)\le \delta/2 + m^2\exp\Big(-\underline n_{i_\ast}^{\zeta-4\varepsilon'}m^{-\alpha}\Big).
        \]
        By~\eqref{eq:berger-i-star}, $m$ is at most $m_{i_\ast}$ defined in~\eqref{eq:berger-mi-ni}. By the same reasoning as below~\eqref{eq:berger-not-conn}, it follows that $m^{-\alpha}\ge \underline n_{i_\ast}^{-\eps}$. Thus, if $\underline n_0$ is a sufficiently large constant,  the second term on the right-hand side is smaller than $\delta/2$. This finishes the proof.
    \end{proof}

\subsection{Useful bounds for the iterative renormalization}\label{app:useful}
\begin{proof}[Proof of \cref{claim:useful}]
  We first discuss \eqref{eq:mi-ni-doubly-exp}. Its proof is standard and based on  simple estimates relating to integer parts and induction, so we leave it to the reader.

    For the first inequality (Sub) in~\eqref{eq:mi-ni-sub} we substitute $n_{i}=m_{i}n_{i-1}$ on its right-hand side. Simplification gives that (Sub)-inequality is equivalent to  
    $
    \varepsilon_{i}m_{i}^{1-\zeta+\delta}\ge C
    $
    for all $i\ge 1$.
    This follows for any sufficiently large $n_0$ since $m_{i}$ is doubly exponentially increasing, $\varepsilon_{i}=(i+1)^{-2}$, and $1-\zeta=\gamma(\tau-1)>0$. For the second inequality (Con) in~\eqref{eq:mi-ni-con} we define the constant
    \[ 
    \widetilde C:=C/\prod_{j=1}^\infty (1-\varepsilon_j)= C/\prod_{j=2}^\infty (1-1/j^{2}),
    \]
    which does not depend on $n_0$. It is elementary to check that $C/\rho_{i-1}^2\le \widetilde C n_0^{\delta/2}$ by definition of $\rho_i$ in~\eqref{eq:rho-i}.
     Hence, dividing both sides with $\rho_{i-1}^2$ in \eqref{eq:mi-ni-con}, and using this inequality, it is  sufficient to show that for all $n_0$ sufficiently large
    \[
    m_{i}^{1-\alpha}n_{i-1}^\zeta\ge \widetilde C n_0^{\delta/2}n_{i}^{\zeta-\delta}.
    \]
    Using that $n_{i}=m_{i}n_{i-1}$, after rearranging this is equivalent to $n_{i-1}^\delta \ge \widetilde C n_0^{\delta/2} m_i^{\zeta-\delta+ \alpha-1}$.
   By \eqref{eq:def-mi-ni}, $m_i\le n_{i-1}^{\xi_\delta}$ with $\xi_\delta=(\delta/2)/(\zeta+\alpha-1-\delta/2)$, and raising then both sides to the power $2/\delta$ gives that \eqref{eq:mi-ni-con} holds if 
    $n_{i-1}^2\ge 
    {\widetilde C}^{2/\delta} n_0 n_{i-1}^{1-\xi_\delta}$. This inequality holds for all $i\ge 1$ when $n_0$ is sufficiently large (the choice of $n_0$ depends on $\widetilde C$). 

    For the bound (Mark) in~\eqref{eq:mi-ni-mark} we have to show that $
    \rho_in_i^\delta\ge C$.
    The infinite product of $(1-1/j^2)$ is strictly positive in~\eqref{eq:rho-i}, so $\rho_i=\Theta(\rho_0)=\Theta\big(n_0^{-\delta/4}\big)$ as $n_0\to\infty$. Since $(n_i)_{i\ge 0}$ is increasing, (Mark) follows again for sufficiently large $n_0$.

    The bounds in~\eqref{eq:mi-ni-1} follow since $\varepsilon_i=(i+1)^{-2}$ for $i\ge 1$ and $m_i$ and $n_i$ increase doubly exponentially in $i$, while the powers appearing in \eqref{eq:mi-ni-1} are all positive (e.g. $1-\zeta-\delta>0$). The same is true for \eqref{eq:mi-ni-1-gamma}.
We turn to the last bound~\eqref{eq:useful-boxing}. We substitute $n=n_im$ and rearrange to obtain that \eqref{eq:useful-boxing} is equivalent to $n_i^\delta\ge m^{\zeta-\delta+\alpha-1}$. Since $m\le m_{i+1}$ by assumption, this inequality holds for all such $m$ if $n_i^\delta\ge m_{i+1}^{\zeta-\delta+\alpha-1}$. Since $m_{i+1}\le n_i^{\xi_\delta}$ by \eqref{eq:def-mi-ni}, it follows by definition of $\xi_\delta$ in~\eqref{eq:xi-delta} that this last inequality holds if $n_i^\delta\ge n_i^{\delta/2-\xi_\delta\delta/2}$, which clearly holds since $\xi_\delta > 0$.
    \end{proof}
\subsection{Poly-logarithmic mark-thresholds to belong to the giant}\label{app:verif}
  \begin{proof}[Proof of \cref{lemma:prerequisite-verif}]
  \emph{High-low regime}.  Assuming \eqref{eq:hyp-lh-verif}, we will adapt the proof of the upper bound on $|\CC_n^\sss{(2)}|$ from Section~\ref{sec:second-upper} to show \eqref{eq:lemma-prer-ver-2} and \eqref{eq:lemma-prer-ver-1}.  Let $\zeta=\zeta_\mathrm{hl}$, $\gamma=\gamma_\mathrm{hl}$, $\eta=\eta_\mathrm{hl}$. Since we assume $\zeta_\mathrm{hl}>0$, we may assume that $\alpha<\infty$. 
  
  Take  $k=k_n=(A_1\log n)^{1/(\zeta-\delta)}$ for a sufficiently large $A_1 > 0$ and partition $\Lambda_n$ into boxes of volume $k$, denoted by $\CQ_1,\ldots, \CQ_{n/k}$ similarly to Section~\ref{sec:second-upper}, and labeled so that boxes with consecutive labels share a $(d-1)$-dimensional side.   
  We write $\CC_{k,j}^\sss{(1)}[1, A'k^{\gamma})$ for the largest connected component induced by the vertices in $\CQ_j\times[1, A'k^{\gamma})$, where $A'$ is the constant from the hypothesis~\eqref{eq:hyp-lh-verif}.
Now, we construct a backbone with the following revealment scheme (again similar to Section~\ref{sec:second-upper}):
      \begin{enumerate}
          \item\label{item:app-1} For each $j=1,\ldots, n/k$, reveal all edges inside box $j$ belonging to $\CV_{k,j}[1, A'k^{\gamma})$.
          \item\label{item:app-2} Let $\CU_{j}=\{v \in \CV_{k,j}[A'k^\gamma, 2A'k^{\gamma}) |\, v \sim \CC_{k,j}^\sss{(1)}[1,A'k^{\gamma})\}$, call the local giant $\mathrm{LG}_j:= \CU_j\cup \CC_{k,j}^\sss{(1)}[1,A'k^{\gamma})$ and reveal all edges between any pair of vertices in $\cup_{j \le n/k} \mathrm{LG}_j$.
          \item\label{item:app-3} Reveal all vertices of mark at least $2A'k^\eta=2A'(A_1\log n)^{\eta/(\zeta-\delta)}$ and their edges towards vertices in  $\cup_{j \le n/k}\mathrm{LG}_j$.
      \end{enumerate}
      Define a box $\CQ_j$ then to be $1$-good if, for some $\delta_1 > 0$, it satisfies the event
      \begin{equation}\label{eq:app-good}
      \big\{|\CC_{k,j}^\sss{(1)}[1, A'k^{\gamma})|> \rho k\big\}\cap\big\{
          |\CU_j| \ge \delta_1k^{\gamma}\big\}.
          \end{equation}
      We call (stage-2-)$\CG_n$ $2$-good if
      \begin{enumerate}
          \item[(2a)] all boxes $\CQ_1,\ldots, \CQ_{n/k}$ are 1-good,
          \item[(2b)] $|\{ u \in \CU_j : u \sim \mathrm{LG}_{j+1}\}| \ge 1$ for all $j < n/k$, i.e., each box contains a box-wedging vertex.
      \end{enumerate}
      We call (stage-3-)$\CG_n$ $3$-good if
      \begin{enumerate}
          \item[(3a)] the stage-$2$-$\CG_n$ is $2$-good,
          \item[(3b)] all vertices of mark at least $2A'k^\eta$ are connected by an edge to a vertex in $\cup_{j \le n/k} \mathrm{LG}_j$.
      \end{enumerate}
        By the hypothesis~\eqref{eq:hyp-lh-verif}, for any fixed $1 \le i \le n/k$,
  $$
\Prob\big(|\CC_{k,j}^\sss{(1)}[1, A'k^{\gamma})|\le \rho k \big) \le A_2\exp\big(- \tfrac{1}{A'} k^{\zeta-\delta}\big), 
$$
and by the same calculations as in \cref{claim:1good}, there exists $\varepsilon>0$ such that
$$
\Prob\big(|\CU_j|\le \delta_1 k^{\zeta} \mid |\CC_{k,j}^\sss{(1)}[1, A'k^{\gamma})|> \rho k \big) \le \exp(-\varepsilon k^{\zeta}).
$$
Let us denote by 2a-good the event that condition (2a) is satisfied. Then,  by a union bound over all $n/k$ boxes, for some $\varepsilon>0$,
$$
\Prob\big(\text{$\CG_n$ is 2a-good}\big) \ge 1-2(n/k)\exp\big(-\varepsilon k^{\zeta-\delta}\big).
$$
We estimate the probability of having a box-wedging vertex using the same calculations as in \cref{claim:second-2-good}. Again for some $\varepsilon > 0$,
$$
\Prob\big(|\{ u \in \CU_j : u \sim \mathrm{LG}_{j+1}\}|=0 \mid \CQ_j, \CQ_{j+1} \mbox{ are 1-good}\big) \le \exp(-\varepsilon k^{\zeta}).
$$ Hence, by a union bound over all $n/k$ boxes, there exists $\varepsilon>0$ such that
\[\Prob\big(\text{$\CG_n$ is 2-good}\mid\text{$\CG_n$ is 2a-good}\big)\ge 1-2(n/k)\exp(-\varepsilon k^{\zeta}).\] 
Thus, there exists a constant $\varepsilon'>0$ such that for $n$ sufficiently large, using $k=(A_1\log n)^{1/(\zeta-\delta)}$,
\[
\Prob\big(\CG_n\text{ is 2-good}\big)\ge 1-(n/\varepsilon' k)\exp\big(-\varepsilon' k^{\zeta-\delta}\big)\ge 1-\exp\big(-\varepsilon'A_1\log n+\log(n))\big).
\]
Recall now $C>0$ from the statement of \cref{lemma:prerequisite-verif}. Then, for all $A_1=A_1(\varepsilon)$ sufficiently large, 
\begin{equation}
\Prob\big(\CG_n\text{ is 2-good}\big)\ge1-o(n^{-C}).\label{eq:verif-2-good}
\end{equation}
In the first two stages \eqref{item:app-1}--\eqref{item:app-2} of the revealment scheme, we only revealed vertices with mark below $2A'k^{\gamma}=2A'(A_1\log n)^{\gamma/(\zeta-\delta)}$ to obtain stage-2-$\CG_n$. Moreover, conditional on $\CG_n$ being 2-good, the local giants in all the subboxes are connected to each other by box-wedging vertices, and their size is at least $\rho k$ by \eqref{eq:app-good}. So,
$$|\CC_{n}^\sss{(1)}[1,2A'k^{\gamma})| =|\CC_{n}^\sss{(1)}[1,2A'(A_1\log n)^{\gamma/(\zeta-\delta)})|\ge \sum_{j\le n/k}|\mathrm{LG}_j|\ge \rho n.
$$
Combining this with \eqref{eq:verif-2-good} gives that
\[
\Prob\big(|\CC_{n}^\sss{(1)}[1,2A'(A_1\log n)^{\gamma/(\zeta-\delta)})|\ge \rho n\big) \ge 1-o(n^{-C}).
\]
This proves~\eqref{eq:lemma-prer-ver-2} for $A:=2A'A_1^{\gamma/(\zeta-\delta)}$.

We proceed to the proof of~\eqref{eq:lemma-prer-ver-1}. We use the same revealment scheme as above in \eqref{item:app-1}--\eqref{item:app-3}. Choosing $k:=(A_1\log n)^{1/(\zeta-\delta)})$, for \eqref{eq:lemma-prer-ver-1} it suffices if all vertices of mark at least $2A'k^{\eta}$ connect by an edge to $\CC_{n}^\sss{(1)}[1,2A'k^\gamma)$ with probability at least $1-o(n^{-C})$.

 We first show that vertices with marks at least $2A'k^\eta$ have not been revealed in the stage-2-$\CG_n$ in \eqref{item:app-1}--\eqref{item:app-2} where marks up to $2A'k^\gamma$ are revealed. To this end, we recall that $\eta=1-\gamma(\tau-1)/\alpha$ from~\eqref{eq:eta-def}. We then need to check that $\eta\ge\gamma$, equivalently, $1-\gamma(\tau-1)/\alpha \ge\gamma$. Elementary rearrangements yield that this inequality holds whenever $1-\gamma(\tau-1)=\zeta_\mathrm{hl}\ge0$. This holds by our assumption in the statement of Lemma \ref{lemma:prerequisite-verif}. In conclusion, any vertex of mark $\ge 2A'k^\eta$ is not part of stage-$2$-$\CG_n$.
 
Excluding first the event that stage-$2$-$\CG_n$ is not $2$-good, we can then assume that stage-$2$-$\CG_n$ is $2$-good and condition on its realization, and also on the realization of vertices with mark at least $2A'k^\eta$. We obtain for the complement of the event in \eqref{eq:lemma-prer-ver-1} that
\begin{align}
\Prob&\big(\nexists\CC: \CC \not\nsupseteq \CV_n[2A'k^{\eta}, \infty), |\CC|\ge \rho n\big) \label{eq:prer-verif-3} \\&\le 
\E\Big[\Prob\big(\exists v\in \CV_n[2A'k^{\eta}, \infty): v\nsim  \CC_{n}^\sss{(1)}[1,2A'k^\gamma)\mid  \CG_n \mbox{ 2-good}, \CV_n[2A'k^\eta,\infty)\big)\Big] 
+\Prob(\CG_n \mbox{ not 2-good}).\nonumber
\end{align}
The second term on the right-hand side is at most $o(n^{-C})$ by the previous calculations leading to~\eqref{eq:verif-2-good} for any $A_1$ sufficiently large.
For the first term, we consider all possible connections of a high-mark vertex $u$ to the lower-mark vertices in the `local giant' in the subbox containing $u$. By a union bound over all vertices of mark at least $2A'k^{\eta}$, and recalling that the maximum distance between two vertices in the same box is at most $\sqrt{d}k^{1/d}$, we have
\begin{align}
\Prob\big(\nexists\CC: \CC \not\nsupseteq \CV_n[2A'k^{\eta}, \infty), |\CC|\ge \rho n\big) 
&\le  \E\bigg[|\CV_n[2A'k^{\eta}, \infty)|\Big(1 - p \Big(1\wedge \beta \frac{2A'k^\eta}{d^{d/2}k}\Big)^\alpha\Big)^{\rho  k}\bigg] + o(n^{-C}), \nonumber\\
&\le n\exp\big(-p\beta^\alpha d^{-\alpha d/2}\rho (2A')^\alpha k^{1-\alpha(1-\eta)}\big)+o(n^{-C}),\label{eq:eta-zeta}
\end{align}
where we used that the minimum is attained at the second term for $k$ sufficiently large (this follows because $\eta=1-\gamma(\tau-1)/\alpha<1$). We compute that $1-\alpha(1-\eta)=1-\gamma(\tau-1)=\zeta$ and substituting $k=(A_1\log n)^{1/(\zeta-\delta)}$, we see that the right-hand side is at most $o(n^{-C})$ if $A_1$ is sufficiently large.
Thus,~\eqref{eq:lemma-prer-ver-1} follows for the hl-regime for some constant $A>0$.

\emph{Low-low regime.} We only sketch the adaptation of the method to the low-low regime, where the calculations are less involved. Set now $\zeta=\zeta_\mathrm{ll}=2-\alpha$, $\gamma=\gamma_\mathrm{ll}=(\alpha-1)/(\tau-1)$, set $k=k_n=(A_1 \log n)^{1/(\zeta-\delta)}$ and partition $\Lambda_n$ into boxes $\CQ_1, \ldots, \CQ_{n/k}$ of volume $k$. Set also $\eta=\eta_\mathrm{ll}=1/\alpha$. 
We modify slightly the revealment scheme in \eqref{item:app-1}--\eqref{item:app-3} as follows: there are no sets $\CV_j$; a box $\CQ_j$ is called $1$-good if  $|\CC_{k,j}^\sss{(1)}[1,A'k^\gamma)| \ge \rho k$; the stage-2-$\CG_n$ is called $2$-good if all subboxes $\CQ_1, \ldots, \CQ_{n/k}$ are $1$-good and additionally there are box wedging vertices between each consecutive box:
$$
  |\{u\in\CC_{k,j}^\sss{(1)}[1,A'k^\gamma): u\sim \CC_{k,j+1}^\sss{(1)}[1,A'k^\gamma)\}|\ge 1 \mbox{ for all } i<n/k.$$
  By calculations similar to those in the proof of \cref{prop:second-upper} when $\zeta_\mathrm{ll}>0$ on page~\pageref{proof:lowlow-second}, one can show that there exists $\varepsilon_2>0$ such that for all $n$ sufficiently large
  \[
  \Prob\big(\CG_n\text{ is $2$-good}\big)\ge 1-n\exp\big(-\varepsilon_2k^{\zeta-\delta}\big)\ge 1-o(n^{-C}),
  \]
  substituting $k=(A_1\log n)^{1/(\zeta-\delta)}$ for any sufficiently large constant $A_1$. Now~\eqref{eq:lemma-prer-ver-2} follows as in~\eqref{eq:verif-2-good}.

We turn to~\eqref{eq:lemma-prer-ver-1}, for which we first prove that vertices of mark above $2A' k^{\eta_{\mathrm{ll}}}$ have not been revealed yet in stage-$2$-$\CG_n$. For this we need that $\eta_\mathrm{ll}\ge \gamma_\mathrm{ll}$. When $\zeta_\mathrm{ll}> \max(\zeta_\mathrm{hl}, 0)$, it follows that $\alpha<\tau-1$ and $\alpha\in(1,2)$ by definition of $\zeta_\mathrm{ll}$ and $\zeta_\mathrm{hl}$ in~\eqref{eq:zeta-ll} and~\eqref{eq:zeta-lh}, respectively. For such values of $\alpha$, $\alpha(\alpha-1)<\tau-1$ holds, which implies that $\eta_\mathrm{ll}=1/\alpha>(\alpha-1)/(\tau-1)=\gamma_\mathrm{ll}$. Thus, no vertices of mark at least $2A'k^{\eta_\mathrm{ll}}$ have been exposed in the stage-2-$\CG_n$. Similar to~\eqref{eq:prer-verif-3} and~\eqref{eq:eta-zeta}, we obtain by a union bound over all vertices of mark at least $2A'k^{\eta_\mathrm{ll}}$, 
\begin{align*}
\Prob\big(\nexists\CC: \CC \not\nsupseteq \CV_n[2A'k^{\eta_\mathrm{ll}}, \infty), |\CC|\ge \rho n\big) 
&\le  \E\bigg[|\CV_n[2A'k^{\eta_\mathrm{ll}}, \infty)|\Big(1 - p \Big(1\wedge \beta \frac{2A'k^\eta}{d^{d/2}k}\Big)^\alpha\Big)^{\rho k}\bigg] +o(n^{-C})\\
&\le n\exp\big(-p\beta^\alpha d^{-\alpha d/2}\rho (2A')^\alpha k^\zeta\big)+o(n^{-C}).
\end{align*}
Here we used that  $\eta=1/\alpha$ implies that the minimum is again attained at the second term, and we then used that $\zeta=2-\alpha$. Substituting $k=(A_1\log n)^{1/(\zeta-\delta)}$ yields that when $A_1$ is large
\[
\Prob\big(\nexists\CC: \CC \not\nsupseteq \CV_n[2A'k^{\eta_\mathrm{ll}}, \infty), |\CC|> \rho n\big)= o(n^{-C}),
\]
showing~\eqref{eq:lemma-prer-ver-1} for some $A>0$.
The corresponding statements for the Palm version $\Prob^\sss{x}$ of $\Prob$ for any $x\in\R^d$, are proven in exactly the same way.
\end{proof}

 \section{Postponed proofs: upper tail}
 \subsection{Concentration of degrees and crossing edges}\label{sec:upper-deg-cross}
 \begin{proof}[Proof of \cref{claim:total-deg}]
 We write $\underline w=\underline w_n$. To bound \eqref{eq:total-deg}, we first distinguish whether there are more or fewer than $\lceil C/\psi\rceil$ such vertices, and in the latter case we condition on their number: 
\[ 
\begin{aligned}\Prob\Bigg(\sum_{v\in\CV_n[\underline{w}, \phi n)}\hspace{-10pt}\mathrm{deg}_v[1,\phi n) &\ge \psi n\Bigg)\le \Prob\big(|\CV_n[\underline{w}, \phi n)|> \lceil C/\psi\rceil\big)\\&+ \E\bigg[\ind{|\CV_n[\underline{w}, \phi n)|\le
\lceil C/\psi\rceil}\Prob\Bigg(\sum_{v\in\CV_n[\underline{w}, \phi n)}\hspace{-10pt}\mathrm{deg}_v[1,\phi n)\ge \psi n\, \Big|\, |\CV_n[\underline{w}, \phi n)|\Bigg)\bigg]  .
\end{aligned}
\]
For the first probability, the intensity of $\CV$ in~\eqref{eq:poisson-intensity} yields that  $|\CV_n[\underline w, \phi n)| \preccurlyeq \mathrm{Poi}(n \underline{w}^{-(\tau-1)}) = \mathrm{Poi}(n^{-\psi})$. Hence, $\Prob(|\CV_n[\underline{w},\phi n)|> \lceil C/\psi\rceil)  =\Theta(n^{-\psi(\lceil C/\psi\rceil +1)})=o(n^{-C})$.
We turn to the term inside the expectation. For the sum $\sum_{v\in\CV_n[\underline{w}, \phi n)}\mathrm{deg}_v[1,\phi n)$ to be at least $\psi n$, by the pigeon-hole principle, there must be a vertex among the at most $\lceil C/\psi \rceil$  vertices in $\CV_n[\underline{w}, \phi n)$ that has total degree at least $\psi n/\lceil(C/\psi)\rceil\ge 2n\psi^2/(3C)$ (we may assume $C$ large so that this inequality holds). Since at most  $\lceil C/\psi \rceil$ edges go to other vertices in $\CV_n[\underline{w}, \phi n)$,  this vertex must have at least $2n\psi^2/(3C) - \lceil C/\psi \rceil \ge n\psi^2/(2C)$  edges towards vertices with mark below $\underline w$ (the bound holds for $n$ large).  We take a union bound over these high-mark vertices, and consider the worst-case location $y$ of such a vertex to estimate the second term above:
 \begin{equation}\label{eq:total-deg-pr}
 \begin{aligned}
\Prob\Bigg(\sum_{v\in\CV_n[\underline{w}, \phi n)}&\hspace{-12pt}\mathrm{deg}_v[1,\phi n) \ge \psi n\Bigg)\le o(n^{-C})\\
&+ \lceil \tfrac{C}{\psi}\rceil \sup_{(y, w)\in\Lambda_n\times[\underline w, \phi n)}\Prob^y\Big(\mathrm{deg}_{(y, w)}[1,\underline w) \ge \tfrac{\psi^2}{2C}n\mid (y,w)\in\CV_n[\underline w, \phi n)\Big).
\end{aligned}
\end{equation}
Since $\CV_n[1,\underline w)$ is independent of $\CV_n[\underline w, \infty)$, the PPP restricted to vertices of mark at most $\underline w$ is independent of the location $y$ of a high-mark vertex $v=(y,w)$. Since edges are present conditionally independently, conditionally on $v=(y,w)$, the neighbors of $v$ in the graph form an inhomogeneous PPP on $\Lambda_n\times[1,\underline w)\subseteq \R^{d+1}$  with intensity $p\big((y,w), (x,z)\big)\big(\rd x\times F_W(\rd z)\big)$.
By integrating this intensity over $\Lambda_n\times[1,\underline w)$ (see~\cite[Proposition 2.1]{luchtrathThesis2022} for an example) the reader may verify that there exists a constant $c=c(d,\alpha, \tau, \sigma, \beta, p)>0$ such that  for \emph{any} vertex $v=(y,w)$ with location $y\in\Lambda_n$ and $w\ge \underline w$
\[
\mathrm{deg}_{(y, w)}[1,\underline w)\preccurlyeq
\begin{dcases}
    \mathrm{Poi}(cw),&\text{if }\sigma< \tau-1, \\
    \mathrm{Poi}\big((cw\log(n/w)\big),&\text{if }\sigma=\tau-1,\\
    \mathrm{Poi}\big(cw^{(\tau-1)/\sigma}n^{1-(\tau-1)/\sigma}\big),&\text{if }\sigma>\tau-1.
\end{dcases}
\]
 The intensity is increasing in $w$ in all three cases, so the supremum in~\eqref{eq:total-deg-pr} is at $w=\phi n$. Substituting this value, the parameter of the Poisson distribution is order $O(n)$ in all three cases with constant prefactor at most $c_\phi:=c\max(\phi, \phi^{(\tau-1)/\sigma},\phi\log(1/\phi))$. Clearly $c_\phi$ tends to zero as $\phi\!\downarrow\! 0$. 
 We 
  conclude that 
\[
\Prob\Bigg(\sum_{v\in\CV_n[\underline{w}, \phi n)}\hspace{-10pt}\mathrm{deg}_v[1,\phi n) \ge \psi n\Bigg)\le o(n^{-C})+ \lceil \tfrac{C}{\psi}\rceil \Prob\Big(\Poi\big(nc_\phi \big)\ge n\frac{\psi^2}{2C}\Big).
\]
 By concentration for Poisson random variables (\cref{lemma:poisson-1}), the second term decays exponentially in $n$ for any $\phi$ sufficiently small, finishing the proof. 
\end{proof}

 \begin{proof}[Proof of \cref{claim:high-deg}]
We first condition on the set of vertices with mark at least $\underline w$ and apply a union bound over these vertices: 
\[
\Prob\big(\exists v\in \CV_n[\underline w, \infty): \mathrm{deg}_v[1,\underline w) <\ell\big) \le\E\bigg[\sum_{ v\in \CV_n[\underline w, \infty)}\Prob\Big( \mathrm{deg}_v[1,\underline w) <\ell\mid \CV_n[\underline w,\infty)\Big)\bigg].
\]
We analyze a single summand: since the PPP $\CV_n[1,\underline w)$ is independent of $\CV_n[\underline w, \infty)$, and edges are present conditionally independently, it follows that 
\begin{equation}\label{app:low-deg}
\Prob\big( \mathrm{deg}_v[1,\underline w) <\ell\mid \CV_n[\underline w,\infty)\big) \le \sup_{(y,w)\in \Lambda_n\times [\underline w, \infty)}
\Prob\big( \mathrm{deg}_v[1,\underline w) <\ell\mid v=(y,w)\in \CV\big).
\end{equation}
The neighbors of the vertex $v$ form an independent inhomogeneous PPP on $\Lambda_n\times[1,\underline w)$  with intensity $p\big((y,w), (x,z)\big)\big(\rd x\times F_W(\rd z)\big)$.
By integrating the intensity over $\Lambda_n\times[1,\underline w)$ (see~\cite[Proposition 2.1]{luchtrathThesis2022} for an example) there exists a constant $c=c(d,\alpha, \tau, \sigma, \beta, p)>0$ such that $\mathrm{deg}_{v}[1,\underline w)\succcurlyeq \mathrm{Poi}(c (w\wedge n))$ for \emph{any} $y\in\Lambda_n$. This is also true when $y$ is close to the boundary of $\Lambda_n$, since at least a $d$-dependent constant fraction of the ball of volume $w\wedge n$ centered at any $y$ falls inside $\Lambda_n$. Since $w\ge \underline w$, the supremum in~\eqref{app:low-deg} is attained at $w=\underline w$, so we obtain  that 
\[
\Prob\big(\exists v\in \CV_n[\underline w, \infty): \mathrm{deg}_v[1,\underline w) <\ell\big) \le \E[|\CV_n[\underline w,\infty)|]\Prob\big(\mathrm{Poi}(c\underline w)<\ell\big)\le n\Prob\big(\mathrm{Poi}(c\underline w)<\ell\big).
\]
By concentration inequalities for Poisson random variables (see \cref{lemma:poisson-1}), the probability on the right-hand side is of order $o(n^{-C-1})$  when $\underline w\ge C_1\log n$, where $C_1=C_1(C,c)$ is a sufficiently large constant.
\end{proof}

\begin{proof}[Proof of \cref{claim:crossing-edges}]
We distinguish three types of vertices inside $\Lambda_n$ that may be connected by an edge to a vertex outside $\Lambda_n$: (1) those close to the boundary of $\Lambda_n$, (2) those with mark at least a very large constant not close to the boundary, and (3) the rest. By the pigeon-hole principle, in order for the event inside of the probability of the statement~\eqref{eq:crossing-edges} to hold, at least one of these three disjoint sets must have size at least $(\psi/3)n$. We write  $\tilde n=(1-\psi/4)n$ and $A_{0}:=\Lambda_n\setminus\Lambda_{\tilde n}$ for an annulus inside $\Lambda_n$. Formally, the three cases become
\begin{equation}\label{eq:app-3-term-cut}
\begin{aligned}
\Prob\big(|\{v\in \CV_{n}[1,\underline w): v&\sim \Lambda_n^\complement\times[1,\underline w)\}|\ge \psi n\big)\\&\le
\Prob\big(|\CV_{A_0}|\ge (\psi/3) n\big)
+
\Prob\big(|\CV_{\tilde n}[(4/\psi)^{1/(\tau-1)},\infty)|\ge (\psi/3) n\big)\\
&\hspace{15pt}+
\Prob\big(|\{v\in \CV_{\tilde n}[1,(4/\psi)^{1/(\tau-1)}): v\sim \Lambda_n^\complement\times[1,\underline w)\}|\ge (\psi/3) n\big).
\end{aligned}
\end{equation}
The size of the annulus in the first term and the mark-truncation value $(4/\psi)^{1/(\tau-1)}$ in the second term are both chosen  so that $\E[|\CV_{A_0}|]=\E[|\CV_{\tilde n}[(4/\psi)^{1/(\tau-1)},\infty)|]= (\psi/4)n$.
By concentration inequalities for Poisson variables (\cref{lemma:poisson-1}), the first two terms on the right-hand side decay exponentially in $n$ and are thus $o(n^{-C})$. In the remainder, we focus on the third term. 

When $\alpha=\infty$, the third term equals zero for $n$ sufficiently large: by the connection probability $\mathrm{p}$ in~\eqref{eq:connection-prob-gen}, a vertex with mark at most $(4/\psi)^{1/(\tau-1)}$ cannot be connected by an edge to a vertex of mark at most $\underline w=o(n)$ at distance $\Omega(n^{1/d})$. 

We now study the third term in \eqref{eq:app-3-term-cut} when $\alpha<\infty$. We use a discretization argument: we decompose $\Lambda_n^\complement\times [1, \underline w)$ into space-mark annuli and count edges according to where their endpoint falls. Set for $i\ge 1$ and $j\le \lceil\log_2(\underline w)\rceil$,
\begin{equation}\label{eq:app-aij-def}
\begin{aligned}
A_{i,j}&:=(\Lambda_{2^in}\setminus\Lambda_{2^{i-1}n})\times[2^j,2^{j+1}),\\
\E[\CV\cap A_{i,j} ] &= 2^{i-1}n\cdot (2^{-j(\tau-1)}-2^{-(j+1)(\tau-1)})\le 2^{i-j(\tau-1)-1}n. 
\end{aligned}
\end{equation}
To avoid too many crossing edges, a good event is when none of the sets $A_{i,j}$ contains more than $2(\log n)^2$ times the expectation many vertices. Truncating also at $1$, we define this event as
\begin{equation}
\CA_\mathrm{out}:=\bigcap_{i\ge 1}\bigcap_{j\le \lceil\log_2(\underline w)\rceil}\CA_{i,j},\quad \mbox{with }\quad 
\CA_{i,j}:=\Big\{|\CV\cap A_{i,j}|\le (\log n)^2\max(2^{i-j(\tau-1)}n, 1\big)\Big\}.\label{eq:event-cross}
\end{equation}
Let us introduce the event that the number of vertices inside $\Lambda_{\tilde n}=\Lambda_{(1-\psi/4)n}$ does not exceed $n$:
\begin{equation}
\CA_\mathrm{in}:=\{|\CV_{\tilde n}[1,(4/\psi)^{1/(\tau-1)})|\le n\}\qquad\mbox{and}\qquad \CA=\CA_\mathrm{in}\cap\CA_\mathrm{out}.\label{eq:event-cross-in}
\end{equation}
We first use concentration inequalities for Poisson random variables to show that $\Prob\big(\neg\CA)$ decays superpolynomially. 
To make another case distinction below, we let 
\begin{equation}\label{eq:j-star}
\begin{aligned}
    j_\ast(i)&:=\max\big\{j\in[\lceil\log_2\underline w\rceil]: 2^{i-j(\tau-1)-1}n\ge (\log n)^2\big\},\\
    C_1&:= \inf_{C'\ge 2\psi}\big\{\forall i\ge C_1\log_2 n: j_\ast(i)= \lceil \log_2 \underline w\rceil\big\}.
    \end{aligned}
\end{equation}
 The number of vertices in $A_{i,j}$ is Poisson distributed with expectation at most $2^{i-j(\tau-1)-1}n$, see~\eqref{eq:app-aij-def}. 
  Therefore, by a union bound over all space-mark annuli,
\[
\begin{aligned}
\Prob\big(\neg\CA_\mathrm{out}\big)&\le \sum_{i\ge 1}\sum_{j=1}^{j_\ast(i)}\Prob\big(\mathrm{Poi}(2^{i-j(\tau-1)-1}n)\ge (\log n)^22^{i-j(\tau-1)}n\big) \\
&\hspace{15pt}+ \sum_{i=1}^{C_1\log_2 n}\sum_{j= j_\ast(i)+1}^{\lceil\log_2 \underline w\rceil }\Prob\Big(\mathrm{Poi}\big(2^{i-j(\tau-1)-1}n\big)\ge (\log n)^2\max\big(2^{i-j(\tau-1)}n, 1\big)\Big).
\end{aligned}
\]
We use that the function  $f(\lambda):=\Prob(\Poi(\lambda)>2(\log n)^2 (\lambda\vee 1))$ is increasing on the interval $[0,1]$, maximal at $\lambda=1$, and then decreasing on $[1, \infty)$. So, the summands on the second line can be all bounded from above by $\Prob\big(\mathrm{Poi}(1)\ge (\log n)^2\big)=O\big(1/(\log n)^2!\big)$. Since $\underline w=n^{(1+\psi)/(\tau-1)}$, there are $O((\log n)^2)$ summands, so the second row  decays faster than any polynomial.
Turning to the first row, since the means here are all above $1$, for each fixed $i$, the largest summand is $j=j_\ast(i)\le \lceil\log_2 \underline w\rceil\le\lceil\log_2 n\rceil$. We then employ concentration for Poisson random variables (\cref{lemma:poisson-1}) and obtain that  there exists $c>0$ such that 
\[
\Prob\big(\neg\CA_\mathrm{out}\big)\le  o(n^{-C})+ \sum_{i\ge 1}\Big(\lceil\log_2 n\rceil \exp\big(-c2^{i-j_\ast(i)(\tau-1)}n\big)\Big).
\]
If $i\ge\lceil C_1\log_2 n\rceil$, then $j_\ast(i)=\lceil\log_2 w\rceil$ by definition in~\eqref{eq:j-star}. For smaller $i$, again by definition of $j_\ast(i)$, the argument of the exponential is always at least $c(\log n)^2$. Therefore, 
\[
\Prob\big(\neg\CA_\mathrm{out}\big)\le o(n^{-C}) +  C_1\lceil\log_2 n\rceil^2\exp\big(-c(\log n)^2\big) + \lceil\log_2 n\rceil\sum_{i\ge \lceil C_1\log_2 n\rceil}\exp\big(-c2^{i-\lceil\log_2 \underline w\rceil(\tau-1)}n\big).
\]
The second term decays superpolynomially and is absorbed by the $o(n^{-C})$ term. Since $\underline w\le n^{(1+\psi)/(\tau-1)}$ with $(1+\psi)/(\tau-1)<1$, we obtain for a constant $c'>0$ that 
\[
\Prob\big(\neg\CA_\mathrm{out}\big)\le o(n^{-C}) +\lceil\log_2 n\rceil \sum_{i\ge \lceil C_1\log_2 n\rceil}\exp\big(-c'2^{i}n^{-\psi}\big).
\]
Since $C_1>2\psi$ by definition in~\eqref{eq:j-star}, it follows that $2^i\ge n^{2\psi}$ for all $i$. Therefore, the sum decays faster than any polynomial. We recall that $\CA=\CA_\mathrm{in}\cap\CA_\mathrm{out}$ by~\eqref{eq:event-cross-in}. By \cref{lemma:poisson-1}, also $\Prob\big(\neg \CA_\mathrm{in}\big)=\exp(-\Omega(n))$. 
This proves that 
\begin{equation}\label{eq:app-bad-a}
\Prob\big(\neg \CA\big)=o(n^{-C}).
\end{equation}
 We return to the third term in~\eqref{eq:app-3-term-cut} and show that it decays  exponentially in $n$ when we condition on any  realization of the vertex set $\CV$ that satisfies $\CA$. Our first goal is to show that the probability that any fixed vertex $v\in\CV_{\tilde n}[1,(4/\psi)^{1/(\tau-1)})$ is connected by an edge to a vertex outside $\Lambda_n$ tends to zero, conditionally on $\CV$ that satisfies $\CA$.
Using the upper bounds on $|\CV\cap A_{i,j}|$ in~\eqref{eq:event-cross} and the connection probability $\mathrm{p}$ from~\eqref{eq:connection-prob-gen}, Markov's inequality yields for any $v\in \Lambda_{\tilde n}\times[1,(4/\psi)^{1/(\tau-1)})$,
\[
\begin{aligned}
\Prob\big(\exists u&\in \Lambda_n^\complement\times[1,\underline w): v\sim u \mid \CV, \CA\big)\\&\le
p\beta^\alpha(\log n)^2\sum_{i\ge 1}\sum_{j\in[\lceil\log_2\underline w\rceil]}\big(2^{i-j(\tau-1)}n\vee 1\big)\Big(\frac{\kappa_\sigma((C/\psi)^{1/(\tau-1)}, 2^{j+1})}{\|\Lambda_{\tilde n}-A_{i,j}\|^d}\Big)^\alpha.
\end{aligned}
\]
By the definition of $\kappa_\sigma$ in~\eqref{eq:kernels} the numerator is at most $c_12^{j\alpha}$ for some constant $c_1=c_1(C,\psi, \sigma, \alpha)$.
There exists a constant $c_2>0$ such that $\|\Lambda_{\tilde n}-A_{i,j}\|\ge (c_22^in)^{1/d}$ for all $i\ge 1$, and hence there exists another constant $c>0$ such that 
\[
\begin{aligned}
\Prob\big(\exists u\in \Lambda_n^\complement\times [1,\underline w): v\sim u \mid \CV, \CA\big)\le
c(\log n)^2n^{-\alpha}\sum_{i\ge 1}\sum_{j\in[\lceil\log_2\underline w\rceil]}\big(2^{i-j(\tau-1)}n\vee 1\big)2^{(j-i)\alpha} .
\end{aligned}
\]
Studying the maximum inside the summation, $2^{i-j(\tau-1)}n<1$ if $j\in[(i+\log_2 n)/(\tau-1), \lceil\log_2\underline w\rceil]$. Since $ \log_2 \underline w \le \log_2 n$, this interval is non-empty only when $i<C_2\log_2 n$ for some constant $C_2>0$.  
Using now that $x\vee 1\le x + \ind{x\le 1}$ with these criteria, we obtain the upper bound
\begin{equation}\label{eq:app-crossing-444}
\begin{aligned}
\Prob\big(\exists u\in \Lambda_n^\complement\times[1,\underline w): v\sim u \mid \CV, \CA\big)&\le
c(\log n)^2n^{1-\alpha}\sum_{i\ge 1}\sum_{j\in[\lceil\log_2\underline w\rceil]}2^{-i(\alpha-1)+j(\alpha-(\tau-1))} \\
&\hspace{15pt}
+c(\log n)^2n^{-\alpha}\sum_{i=1}^{\lceil C_2\log_2 n\rceil}\sum_{j=\lceil (i+\log_2 n)/(\tau-1)\rceil}^{\lceil \log_2 \underline w\rceil}2^{(j-i)\alpha}.
\end{aligned}
\end{equation}
We continue to bound the second line. Recall that $\underline w\le n^{(1+\psi)/(\tau-1)}$ where $(1+\psi)/(\tau-1)<1$. Therefore, there exists $\varepsilon>0$ such that $2^j<n^{1-\varepsilon}$ for all $j\le \lceil\log_2 \underline w\rceil$ when $n$ is sufficiently large. Furthermore, we bound $2^{-i\alpha}\le 1$. Consequently, there exists $C_3>0$ such that 
\[
\begin{aligned}
\Prob\big(\exists u\in \Lambda_n^\complement\times[1,\underline w): v\sim u \mid \CV, \CA\big)&\le
c(\log n)^2n^{1-\alpha}\sum_{i\ge 1}2^{-i(\alpha-1)}\sum_{j\in[\lceil\log_2\underline w\rceil]}2^{j(\alpha-(\tau-1))} \\
&\hspace{15pt}
+C_3(\log n)^4n^{-\varepsilon\alpha}.
\end{aligned}
\]
We now bound the first line in \eqref{eq:app-crossing-444} on the right-hand side. Since $\alpha>1$, the sum over $i$ is finite. The sum over $j$ is finite when $\alpha<\tau-1$, and at most logarithmic when $\alpha=\tau-1$. Since $\alpha > 1$, in both of these cases the right-hand side on the first line is at most $O(n^{-\tilde\varepsilon})$ for some $\tilde\varepsilon > 0$. When $\alpha>\tau-1$, we use again that $2^j<n^{1-\varepsilon}$, so the maximal term is at most $n^{(1-\varepsilon)(\alpha-(\tau-1))}$. When multiplied with $n^{1-\alpha}$, this is at most $n^{-(\tau-2)-\varepsilon(\alpha-(\tau-1))}$. We conclude that in all cases there exist $\tilde\varepsilon, C_4>0$ such that 
\[
\Prob\big(\exists u\in \Lambda_n^\complement[1,\underline w): v\sim u \mid \CV, \CA\big)\le
C_4(\log n)^3n^{-\tilde\varepsilon} + C_2(\log n)^4n^{-\varepsilon \alpha}.
\]
Assuming $n$ is sufficiently large, the right-hand side is at most $\psi/6$.
By definition of $\CA_\mathrm{in}$ in~\eqref{eq:event-cross-in}  there are at most $n$ vertices inside $\CV_{\tilde n}[1,(4/\psi)^{1/(\tau-1)})$. Recall that edges are present conditionally independently on $\CV$. Therefore, conditionally on any realization of $\CV$ that satisfies $\CA$,
\[
|\{v\in \CV_{\tilde n}[1,(4/\psi)^{1/(\tau-1)}): v\sim \Lambda_n^\complement\times[1,\underline w)\}|\succcurlyeq \mathrm{Bin}(n,\psi/6).
\]
Returning to the third term in~\eqref{eq:app-3-term-cut}, we obtain that 
\[
\begin{aligned}
\Prob\big(|\{v\in \CV_{\tilde n}[1,(4/\psi)^{1/(\tau-1)})&: v\sim \Lambda_n^\complement\times[1,\underline w)\}|\ge (\psi/3) n\big) \\
&\le 
\Prob\big(\neg \CA) + \E\big[\ind{\CA}\Prob\big(\mathrm{Bin}(n,\psi/6)\ge (\psi/3)n\mid \CV,\CA\big)\big].
\end{aligned}
\]
The first term is of order $o(n^{-C})$ by~\eqref{eq:app-bad-a}. The second term decays exponentially in $n$ by a Chernoff bound. Thus all terms in~\eqref{eq:app-3-term-cut} are of order $o(n^{-C})$, finishing the proof. 
\end{proof}

 \subsection{Truncation of the generating function}\label{sec:app-trunc-gen}
\begin{proof}[Proof of \cref{claim:hubs-lower}]
        We start with the first bound. Using elementary rearrangements of~\eqref{eq:hubs-lower-trunc}, we leave it to the reader to verify that it is sufficient that there exists $\psi'>0$ and $\ell_0$ such that 
        \[
        \rho-\theta + \psi'<\sum_{\ell=1}^{\ell_0}\big(1-(1-p)^{h_\mathrm{lo} \ell}\big)\theta_\ell = \sum_{\ell=1}^{\ell_0}\theta_\ell - \sum_{\ell=1}^{\ell_0}(1-p)^{h_\mathrm{lo} \ell}\theta_\ell.
        \]
            Since $\sum_{\ell=1}^{\infty}\theta_\ell = 1-\theta$,  we may choose $\ell_0$ is sufficiently large (depending on $\psi'$) so that $\sum_{\ell=1}^{\ell_0}\theta_\ell=\Prob^\sss{0}\big(|\CC(0)|\le \ell_0\big)> 1-\theta-\psi'$. Rearranging both sides, it is sufficient to show that we can choose $\psi'$ satisfying 
        \begin{equation}
        2\psi' <1-\rho-  \sum_{\ell=1}^{\ell_0}(1-p)^{h_\mathrm{lo} \ell}\theta_\ell.
        \label{eq:lower-trunc-pr}
        \end{equation}
        If $p=1$, then the sum is $0$ and the bound is satisfied for any $\psi'\in(0,(1-\rho)/2)$.
        Now we assume that $p<1$, so that $\mathrm{hubs}(\rho)$, defined in~\eqref{eq:hubs-gen}, satisfies the equation 
        \begin{equation}
        \E^\sss{0}\big[(1-p)^{\mathrm{hubs}(\rho)|\CC(0)|}\big]=\sum_{\ell=1}^{\infty} (1-p)^{\mathrm{hubs}(\rho) \ell } \theta_\ell = 1-\rho.\label{eq:hubs-equal}
        \end{equation} 
        Now the definition of $h_{\mathrm{lo}}(\rho)$ matters: 
   If $\mathrm{hubs}(\rho)$ is an integer, then $h_{\mathrm{lo}}(\rho)=\lceil \mathrm{hubs}(\rho)\rceil +1 = \mathrm{hubs}(\rho)+1$  by definition of $h_\mathrm{lo}$ in~\eqref{eq:underover-h}. If $\mathrm{hubs}(\rho)$ is not an integer, then $h_{\mathrm{lo}}(\rho)=\lceil \mathrm{hubs}(\rho)\rceil$, and in both cases we have $h_{\mathrm{lo}}(\rho)>\mathrm{hubs}(\rho)$. So, using the infinite sum on the right-hand side of \eqref{eq:lower-trunc-pr} in place of $1-\rho$ gives that~\eqref{eq:lower-trunc-pr} holds if
        \[
        2\psi' < \sum_{\ell=1}^{\ell_0}\big((1-p)^{\mathrm{hubs}(\rho)\ell}-(1-p)^{h_\mathrm{lo} \ell}\big).
        \]
        Since $\mathrm{hubs}(\rho)<h_\mathrm{lo}$, each summand is nonnegative, so the above bound is satisfied if $\psi'$ is chosen as, for instance, $\psi':= \big((1-p)^{\mathrm{hubs}(\rho)}- (1-p)^{h_\mathrm{lo}}\big)/2>0$.

        We turn to the proof of~\eqref{eq:hubs-upper-trunc} for which we first assume $p<1$.
        Using elementary rearrangements,  it suffices to find for an $\psi'>0$ sufficiently small, and an $\ell_0$ large, so that for all $\ell^\ast\ge \ell_0$,
        \begin{equation}\label{eq:trunc-upper-pr}\sum_{\ell=1}^{\ell^\ast}\theta_\ell(1-p)^{(h_\mathrm{up}-1)\ell}  > 1-\rho + \psi'.
        \end{equation}
        Since we assumed $p<1$, the equality~\eqref{eq:hubs-equal} is satisfied. Because $h_\mathrm{up}=\lceil\mathrm{hubs}(\rho)\rceil$, it holds that 
        $h_{\mathrm{up}}-1< \mathrm{hubs}(\rho)$. Since  the left-hand side in~\eqref{eq:hubs-equal} is decreasing in $\mathrm{hubs}(\rho)$, it follows that 
        \begin{equation}\nonumber
        \sum_{\ell=1}^{\infty}(1-p)^{(h_\mathrm{up} - 1)\ell}\theta_\ell>1-\rho.
        \end{equation}
        Thus, there exists $\delta>0$, and $\ell_0>0$ such that for all $\ell^\ast\ge \ell_0$
        \[
        \sum_{\ell=1}^{\ell^\ast}(1-p)^{(h_\mathrm{up} - 1)\ell}\theta_\ell>1-\rho+\delta.
        \]
        We combine this bound with~\eqref{eq:trunc-upper-pr}, and choose $\psi'=\delta$. This finishes the proof of~\eqref{eq:hubs-upper-trunc} when $p<1$. 
        When $p=1$, $h_\mathrm{up}=1$ by its own definition in~\eqref{eq:underover-h} and the definition of $\mathrm{hubs}(\rho)=1$ in~\eqref{eq:hubs-gen}. Hence, $(1-p)^{(h_\mathrm{up}-1)\ell}=1$ for all $\ell\ge 1$. The proof follows as before since $\sum_{\ell=1}^{\infty}\theta_\ell = 1-\theta$,  choose $\ell_0$ is sufficiently large (depending on $\psi'$) so that $\sum_{\ell=1}^{\ell_0}\theta_\ell=\Prob^\sss{0}\big(|\CC(0)|\le \ell_0\big)> 1-\theta-\psi'>1-\rho$.
       \end{proof}

\subsection*{Acknowledgements}
The work of JJ and JK has been partly supported through grant NWO 613.009.122.  The work of DM has been partially supported by grant Fondecyt grant 1220174 and by grant GrHyDy ANR-20-CE40-0002.
\end{appendices}

\end{document}

%% file: headers-dev.tex
\usepackage{accents}
\usepackage{algorithm}
\usepackage{algpseudocode}
\usepackage{amsfonts}
\usepackage{amsmath}
\usepackage{amssymb}
\usepackage{amsthm}
\usepackage{appendix}
\usepackage[english]{babel}
\usepackage{bm}
\usepackage{bbm}
\usepackage{caption}
\usepackage{changepage}
\usepackage[usenames,dvipsnames]{color}
\usepackage{dsfont}
 \usepackage{enumerate}
\usepackage{esvect}
\usepackage[T1]{fontenc}
\usepackage{graphicx}
\usepackage[linkcolor=blue]{hyperref}
\usepackage[capitalise]{cleveref}
\usepackage[utf8]{inputenc}
\usepackage{mathtools}
\usepackage{nccmath}
\usepackage{relsize}
\usepackage{subcaption}

\theoremstyle{plain}
\newtheorem{theorem}{Theorem}[section]
\newtheorem{metatheorem}[theorem]{Meta-theorem}
\newtheorem{corollary}[theorem]{Corollary}
\newtheorem{assumption}[theorem]{Assumption}
\newtheorem{claim}[theorem]{Claim}
\newtheorem{proposition}[theorem]{Proposition}

\newtheorem{lemma}[theorem]{Lemma}

\newtheorem{observation}[theorem]{Observation}

\newtheorem{definition}[theorem]{Definition}

\theoremstyle{definition}
\newtheorem{remark}[theorem]{Remark}

\numberwithin{equation}{section}

\newcommand{\R}{{\mathbb R}}
\newcommand{\Z}{{\mathbb Z}}
\newcommand{\N}{{\mathbb N}}
\newcommand{\E}{\mathbb E}
\newcommand{\Prob}{\mathbb{P}}
\newcommand\eps{\varepsilon}
\newcommand{\CA}{{\mathcal A}}

\newcommand{\CC}{{\mathcal C}}

\newcommand{\CE}{{\mathcal E}}

\newcommand{\CG}{{\mathcal G}}
\newcommand{\CH}{{\mathcal H}}

\newcommand{\CQ}{{\mathcal Q}}

\newcommand{\CU}{{\mathcal U}}
\newcommand{\CV}{{\mathcal V}}

\newcommand{\CZ}{{\mathcal Z}}
\newcommand{\re}{{\mathrm e}}

\newcommand{\Poi}{{\mathrm{Poi}}}

\newcommand{\ind}[1]{\mathds{1}_{\{#1\}}}
\newcommand{\Ind}[1]{\mathbbm{1}{\{#1\}}}
\newcommand{\rd}{\mathrm{d}}
\newcommand{\sss}[1]{{\scriptscriptstyle #1}}